\documentclass[11pt]{amsart}
\usepackage{amsthm}
\usepackage{amsmath}
\usepackage{amssymb}
\makeindex
\usepackage{stmaryrd}
\usepackage{verbatim}
\usepackage{mathrsfs}
\usepackage{mathabx}
\usepackage{eucal}
\usepackage{color}
\let\mathcal=\CMcal
\usepackage{hyperref}
\usepackage{devanagari}
\DeclareSymbolFont{yhlargesymbols}{OMX}{yhex}{m}{n}
\DeclareMathAccent{\yhwidehat}{\mathord}{yhlargesymbols}{"62}
\usepackage{scalerel}

\newcommand\reallywidehat[1]{\arraycolsep=0pt\relax%
\begin{array}{c}
\stretchto{
  \scaleto{
    \scalerel*[\widthof{\ensuremath{#1}}]{\kern-.5pt\bigwedge\kern-.5pt}
    {\rule[-\textheight/2]{1ex}{\textheight}} 
  }{\textheight} %
}{0.5ex}\\           
#1\\                 
\rule{-1ex}{0ex}
\end{array}
}
\allowdisplaybreaks[4]
\begin{document}
\def\sbt{\raisebox{1.2pt}{$\scriptscriptstyle\,\bullet\,$}}

\def\alp{\alpha}
\def\bet{\beta}
\def\gam{\gamma}
\def\del{\delta}
\def\eps{\epsilon}
\def\zet{\zeta}
\def\tht{\theta}
\def\iot{\iota}
\def\kap{\kappa}
\def\lam{\lambda}
\def\sig{\sigma}
\def\ome{\omega}
\def\vep{\varepsilon}
\def\vth{\vartheta}
\def\vpi{\varpi}
\def\vrh{\varrho}
\def\vsi{\varsigma}
\def\vph{\varphi}
\def\Gam{\Gamma}
\def\Del{\Delta}
\def\Tht{\Theta}
\def\Lam{\Lambda}
\def\Sig{\Sigma}
\def\Ups{\Upsilon}
\def\Ome{\Omega}
\def\vka{\varkappa}
\def\vDe{\varDelta}
\def\vSi{\varSigma}
\def\vTh{\varTheta}
\def\vGm{\varGamma}
\def\vOm{\varOmega}
\def\vPi{\varPi}
\def\vPh{\varPhi}
\def\vPs{\varPsi}
\def\vUp{\varUpsilon}
\def\vXi{\varXi}

\def\bdse{\boldsymbol{e}}
\def\bdsf{\boldsymbol{f}}
\def\bdsg{\boldsymbol{g}}
\def\bdsh{\boldsymbol{h}}
\def\bdsF{\boldsymbol{F}}
\def\bdscF{\boldsymbol{\calf}}
\def\bdscU{\boldsymbol{\calu}}

\def\bdalp{\boldsymbol{\alp}}
\def\bdTht{\boldsymbol{\Tht}}
\def\bdsig{\boldsymbol{\sig}}
\def\bdpi{\boldsymbol{\pi}}
\def\bdome{\boldsymbol{\ome}}
\def\bdmu{\boldsymbol{\mu}}
\def\bdkap{\boldsymbol{\kap}}
\def\ulQ{\underline{Q}}
\def\ulk{\underline{k}}
\def\ulu{\underline{u}}
\def\ulv{\underline{v}}
\def\uleps{\underline{\eps}}
\def\ulchi{\underline{\chi}}
\def\ulkap{\underline{\kap}}

\def\frka{{\mathfrak a}}    \def\frkA{{\mathfrak A}}
\def\frkb{{\mathfrak b}}    \def\frkB{{\mathfrak B}}
\def\frkc{{\mathfrak c}}    \def\frkC{{\mathfrak C}}
\def\frkd{{\mathfrak d}}    \def\frkD{{\mathfrak D}}
\def\frke{{\mathfrak e}}    \def\frkE{{\mathfrak E}}
\def\frkf{{\mathfrak f}}    \def\frkF{{\mathfrak F}}
\def\frkg{{\mathfrak g}}    \def\frkG{{\mathfrak G}}
\def\frkh{{\mathfrak h}}    \def\frkH{{\mathfrak H}}
\def\frki{{\mathfrak i}}    \def\frkI{{\mathfrak I}}
\def\frkj{{\mathfrak j}}    \def\frkJ{{\mathfrak J}}
\def\frkk{{\mathfrak k}}    \def\frkK{{\mathfrak K}}
\def\frkl{{\mathfrak l}}    \def\frkL{{\mathfrak L}}
\def\frkm{{\mathfrak m}}    \def\frkM{{\mathfrak M}}
\def\frkn{{\mathfrak n}}    \def\frkN{{\mathfrak N}}
\def\frko{{\mathfrak o}}    \def\frkO{{\mathfrak O}}
\def\frkp{{\mathfrak p}}    \def\frkP{{\mathfrak P}}
\def\frkq{{\mathfrak q}}    \def\frkQ{{\mathfrak Q}}
\def\frkr{{\mathfrak r}}    \def\frkR{{\mathfrak R}}
\def\frks{{\mathfrak s}}    \def\frkS{{\mathfrak S}}
\def\frkt{{\mathfrak t}}    \def\frkT{{\mathfrak T}}
\def\frku{{\mathfrak u}}    \def\frkU{{\mathfrak U}}
\def\frkv{{\mathfrak v}}    \def\frkV{{\mathfrak V}}
\def\frkw{{\mathfrak w}}    \def\frkW{{\mathfrak W}}
\def\frkx{{\mathfrak x}}    \def\frkX{{\mathfrak X}}
\def\frky{{\mathfrak y}}    \def\frkY{{\mathfrak Y}}
\def\frkz{{\mathfrak z}}    \def\frkZ{{\mathfrak Z}}

\def\cal{\fam2}
\def\cala{{\cal A}}
\def\calb{{\cal B}}
\def\calc{{\cal C}}
\def\cald{{\cal D}}
\def\cale{{\cal E}}
\def\calf{{\cal F}}
\def\calg{{\cal G}}
\def\calh{{\cal H}}
\def\cali{{\cal I}}
\def\calj{{\cal J}}
\def\calk{{\cal K}}
\def\call{{\cal L}}
\def\calm{{\cal M}}
\def\caln{{\cal N}}
\def\calo{{\cal O}}
\def\calp{{\cal P}}
\def\calq{{\cal Q}}
\def\calr{{\cal R}}
\def\cals{{\cal S}}
\def\calt{{\cal T}}
\def\calu{{\cal U}}
\def\calv{{\cal V}}
\def\calw{{\cal W}}
\def\calx{{\cal X}}
\def\caly{{\cal Y}}
\def\calz{{\cal Z}}

\def\AA{{\mathbb A}}
\def\BB{{\mathbb B}}
\def\CC{{\mathbb C}}
\def\DD{{\mathbb D}}
\def\EE{{\mathbb E}}
\def\FF{{\mathbb F}}
\def\GG{{\mathbb G}}
\def\HH{{\mathbb H}}
\def\II{{\mathbb I}}
\def\JJ{{\mathbb J}}
\def\KK{{\mathbb K}}
\def\LL{{\mathbb L}}
\def\MM{{\mathbb M}}
\def\NN{{\mathbb N}}
\def\OO{{\mathbb O}}
\def\PP{{\mathbb P}}
\def\QQ{{\mathbb Q}}
\def\RR{{\mathbb R}}
\def\SS{{\mathbb S}}
\def\TT{{\mathbb T}}
\def\UU{{\mathbb U}}
\def\VV{{\mathbb V}}
\def\WW{{\mathbb W}}
\def\XX{{\mathbb X}}
\def\YY{{\mathbb Y}}
\def\ZZ{{\mathbb Z}}

\def\bfa{{\mathbf a}}    \def\bfA{{\mathbf A}}
\def\bfb{{\mathbf b}}    \def\bfB{{\mathbf B}}
\def\bfc{{\mathbf c}}    \def\bfC{{\mathbf C}}
\def\bfd{{\mathbf d}}    \def\bfD{{\mathbf D}}
\def\bfe{{\mathbf e}}    \def\bfE{{\mathbf E}}
\def\bff{{\mathbf f}}    \def\bfF{{\mathbf F}}
\def\bfg{{\mathbf g}}    \def\bfG{{\mathbf G}}
\def\bfh{{\mathbf h}}    \def\bfH{{\mathbf H}}
\def\bfi{{\mathbf i}}    \def\bfI{{\mathbf I}}
\def\bfj{{\mathbf j}}    \def\bfJ{{\mathbf J}}
\def\bfk{{\mathbf k}}    \def\bfK{{\mathbf K}}
\def\bfl{{\mathbf l}}    \def\bfL{{\mathbf L}}
\def\bfm{{\mathbf m}}    \def\bfM{{\mathbf M}}
\def\bfn{{\mathbf n}}    \def\bfN{{\mathbf N}}
\def\bfo{{\mathbf o}}    \def\bfO{{\mathbf O}}
\def\bfp{{\mathbf p}}    \def\bfP{{\mathbf P}}
\def\bfq{{\mathbf q}}    \def\bfQ{{\mathbf Q}}
\def\bfr{{\mathbf r}}    \def\bfR{{\mathbf R}}
\def\bfs{{\mathbf s}}    \def\bfS{{\mathbf S}}
\def\bft{{\mathbf t}}    \def\bfT{{\mathbf T}}
\def\bfu{{\mathbf u}}    \def\bfU{{\mathbf U}}
\def\bfv{{\mathbf v}}    \def\bfV{{\mathbf V}}
\def\bfw{{\mathbf w}}    \def\bfW{{\mathbf W}}
\def\bfx{{\mathbf x}}    \def\bfX{{\mathbf X}}
\def\bfy{{\mathbf y}}    \def\bfY{{\mathbf Y}}
\def\bfz{{\mathbf z}}    \def\bfZ{{\mathbf Z}}

\def\scra{{\mathscr A}}
\def\scrb{{\mathscr B}}
\def\scrc{{\mathscr C}}
\def\scrd{{\mathscr D}}
\def\scre{{\mathscr E}}
\def\scrf{{\mathscr F}}
\def\scrg{{\mathscr G}}
\def\scrh{{\mathscr H}}
\def\scri{{\mathscr I}}
\def\scrj{{\mathscr J}}
\def\scrk{{\mathscr K}}
\def\scrl{{\mathscr L}}
\def\scrm{{\mathscr M}}
\def\scrn{{\mathscr N}}
\def\scro{{\mathscr O}}
\def\scrp{{\mathscr P}}
\def\scrq{{\mathscr Q}}
\def\scrr{{\mathscr R}}
\def\scrs{{\mathscr S}}
\def\scrt{{\mathscr T}}
\def\scru{{\mathscr U}}
\def\scrv{{\mathscr V}}
\def\scrw{{\mathscr W}}
\def\scrx{{\mathscr X}}
\def\scry{{\mathscr Y}}
\def\scrz{{\mathscr Z}}

\def\Zp{{\ZZ_p}}
\def\CE{{\mathcal{E}}}
\def\CF{{\mathcal{F}}}
\def\CA{{\mathcal{A}}}
\def\CK{{\mathcal{K}}}
\def\CL{{\mathcal{L}}}
\def\CO{{\mathcal{O}}}
\def\fp{{\frkp}}
\def\fg{{\frkg}}
\def\fk{{\frkk}}
\def\fh{{\frkh}}
\def\ra{{\rightarrow}}
\def\Spec{{\mathrm Spec}}
\def\tIg{{\widetilde{Ig}}}
\def\hSh{{\widehat{Sh}}}
\def\CV{{\mathcal{V}}}
\def\tCV{{\tilde{\CV}}}
\newcommand\dnv{\text{\dn{t}}}

\newcommand{\isoarrow}{{~\overset\sim\longrightarrow~}}

\def\phm{\phantom}
\def\smallstrut{\vphantom{\vrule height 3pt }}
\def\bdm #1#2#3#4{\left(
\begin{array} {c|c}{\ds{#1}}
 & {\ds{#2}} \\ \hline
{\ds{#3}\vphantom{\ds{#3}^1}} &  {\ds{#4}}
\end{array}
\right)}
\newcommand{\powerseries}[1]{\llbracket{#1}\rrbracket}

\def\norm#1{\lVert#1\rVert}
\newcommand{\pair}[2]{\langle#1, #2\rangle}

\def\GL{\mathrm{GL}}
\def\BC{\mathrm{BC}}
\def\PGL{\mathrm{PGL}}
\def\PG{\mathrm{P}G}
\def\SL{\mathrm{SL}}
\def\Mp{\mathrm{Mp}}
\def\GSp{\mathrm{GSp}}
\def\PGSp{\mathrm{PGSp}}
\def\Sp{\mathrm{Sp}}
\def\St{\mathrm{St}}
\def\SU{\mathrm{SU}}
\def\SO{\mathrm{SO}}
\def\U{\mathrm{U}}
\def\GU{\mathrm{GU}}
\def\O{\mathrm{O}}
\def\Mat{\mathrm{M}}
\def\Spec{\mathrm{Spec}}
\def\Gal{\mathrm{Gal}}
\def\Tr{\mathrm{T}}
\def\Nr{\mathrm{N}}
\def\tr{\mathrm{tr}}
\def\Ad{\mathrm{Ad}}
\def\As{\mathrm{As}}
\def\Pet{\mathrm{Pet}}
\def\new{\mathrm{new}}
\def\can{\mathrm{can}}
\def\Wh{\mathrm{Wh}}
\def\FC{\mathrm{FC}}
\def\FJ{\mathrm{FJ}}
\def\Fj{\mathrm{Fj}}
\def\Sym{\mathrm{Sym}}
\def\Her{\mathrm{Her}}
\def\sym{\mathrm{sym}}
\def\Lie{\mathrm{Lie}}
\def\supp{\mathrm{supp}}
\def\proj{\mathrm{proj}}
\def\Hom{\mathrm{Hom}}
\def\End{\mathrm{End}}
\def\Inj{\underline{\mathrm{Inj}}}
\def\Isom{\underline{\mathrm{Isom}}}
\def\Ker{\mathrm{Ker}}
\def\Res{\mathrm{Res}}
\def\res{\mathrm{res}}
\def\cusp{\mathrm{cusp}}
\def\temp{\mathrm{temp}}
\def\Irr{\mathrm{Irr}}
\def\rank{\mathrm{rank}}
\def\sgn{\mathrm{sgn}}
\def\diag{\mathrm{diag}}
\def\Wd{\mathrm{Wd}}
\def\nd{\mathrm{nd}}
\def\ord{\mathrm{ord}}
\def\Cl{\mathrm{Cl}}
\def\Hol{\mathrm{Hol}}
\def\Pet{\mathrm{Pet}}
\def\Pic{\mathrm{Pic}}
\def\Pol{\mathrm{Pol}}
\def\ar{\mathrm{ar}}
\def\ab{\mathrm{ab}}
\def\cls{\mathrm{cls}}
\def\d{\mathrm{d}}
\def\La{\langle}
\def\Ra{\rangle}
\newcommand{\one}{1\hspace{-0.25em}{\rm{l}}}
\def\addchar{{\boldsymbol\psi}}
\def\Abs{{\boldsymbol\alp}}
\def\bvep{{\boldsymbol \vep}}
\def\bnu{{\boldsymbol \nu}}
\def\bj{{\boldsymbol j}}
\def\cyc{\boldsymbol\varepsilon_{\rm cyc}}
\def\rmH{{\rm H}}
\def\vFJ{\overrightarrow{\FJ}}

\def\trs{\,^t\!}
\def\tri{\,^\iot\!}
\def\iu{\sqrt{-1}}
\def\oo{\hbox{\bf 0}}
\def\ono{\hbox{\bf 1}}
\def\smallcirc{\lower .3em \hbox{\rm\char'27}\!}
\def\AAf{\AA_\bff}
\def\thalf{\tfrac{1}{2}}
\def\bsl{\backslash}
\def\wtl{\widetilde}
\def\til{\tilde}
\def\Ind{\operatorname{Ind}}
\def\ind{\operatorname{ind}}
\def\cind{\operatorname{c-ind}}
\def\beq{\begin{equation}}
\def\eeq{\end{equation}}
\def\d{\mathrm{d}}
\def\lddots{\mathinner{\mskip1mu\raise1pt\vbox{\kern7pt\hbox{.}}\mskip2mu\raise4pt\hbox{.}\mskip2mu\raise7pt\hbox{.}\mskip1mu}}
\newcommand{\1}{1\hspace{-0.25em}{\rm{l}}}
\newcommand{\Ts}{{T^{sym}}}

\newcounter{one}
\setcounter{one}{1}
\newcounter{two}
\setcounter{two}{2}
\newcounter{thr}
\setcounter{thr}{3}
\newcounter{fou}
\setcounter{fou}{4}
\newcounter{fiv}
\setcounter{fiv}{5}
\newcounter{six}
\setcounter{six}{6}
\newcounter{sev}
\setcounter{sev}{7}

\newcommand{\shp}{\rm\char'43}

\def\lddots{\mathinner{\mskip1mu\raise1pt\vbox{\kern7pt\hbox{.}}\mskip2mu\raise4pt\hbox{.}\mskip2mu\raise7pt\hbox{.}\mskip1mu}}

\makeatletter
\def\varddots{\mathinner{\mkern1mu
    \raise\p@\hbox{.}\mkern2mu\raise4\p@\hbox{.}\mkern2mu
    \raise7\p@\vbox{\kern7\p@\hbox{.}}\mkern1mu}}
\makeatother

\def\today{\ifcase\month\or
 January\or February\or March\or April\or May\or June\or
 July\or August\or September\or October\or November\or December\fi
 \space\number\day, \number\year}

\makeatletter
\def\varddots{\mathinner{\mkern1mu
    \raise\p@\hbox{.}\mkern2mu\raise4\p@\hbox{.}\mkern2mu
    \raise7\p@\vbox{\kern7\p@\hbox{.}}\mkern1mu}}
\makeatother

\def\today{\ifcase\month\or
 January\or February\or March\or April\or May\or June\or
 July\or August\or September\or October\or November\or December\fi
 \space\number\day, \number\year}

\makeatletter
\@addtoreset{equation}{section}
\def\theequation{\thesection.\arabic{equation}}

\theoremstyle{plain}
\newtheorem{theorem}{Theorem}[section]
\newtheorem*{theorem_a}{Theorem A}
\newtheorem*{theorem_b}{Theorem B}
\newtheorem*{theorem_c}{Theorem C}
\newtheorem*{theorem_e}{Theorem}
\newtheorem*{corollary_a}{Corollary A}
\newtheorem{lemma}[theorem]{Lemma}
\newtheorem{proposition}[theorem]{Proposition}
\theoremstyle{definition}
\newtheorem{definition}[theorem]{Definition}
\newtheorem{conjecture}[theorem]{Conjecture}
\newtheorem{corollary}[theorem]{Corollary}
\theoremstyle{remark}
\newtheorem{remark}[theorem]{Remark}

\renewcommand{\thepart}{\Roman{part}}
\setcounter{tocdepth}{1}
\setcounter{section}{0} 

\title{$p$-adic $L$-functions for $\U(2,1)\times\U(1,1)$}
\author{Michael Harris}
\author{Ming-Lun Hsieh}
\author{Shunsuke Yamana}
\date{\today}
\address{Department of Mathematics, Columbia University, New York, NY 10027, USA}
\email{harris@math.columbia.edu}
\address{Institute of Mathematics, Academia Sinica, Taipei 10617, Taiwan
}
\email{mlhsieh@math.sinica.edu.tw}
\address{Department of Mathematics, Graduate School of Science, Osaka Metropolitan University, 3-3-138 Sugimoto, Sumiyoshi-ku, Osaka 558-8585, Japan}
\email{yamana@omu.ac.jp}

\begin{abstract}
We construct the five-variable $p$-adic $L$-function attached to Hida families on $\U(2,1)\times\U(1,1)$, interpolating the square-root of Rankin-Selberg $L$-values in the \emph{shifted piano} range. Our construction  relies on a new theta operator and its $p$-adic variation which  plays a role analogous to the classical Ramanujan-Serre theta operator in Hida's $p$-adic Rankin-Selberg method.
The interpolation formula, including the modified Euler factors at $p$ and at the real place, is consistent with the conjectural shape of $p$-adic $L$-functions predicted by Coates and Perrin-Riou.
\end{abstract}

\maketitle
\tableofcontents

\section{Introduction}\label{sec:1}

In the pioneering paper \cite{BDP}, Berolini, Darmon and Prasanna established a formula between the $p$-adic Abel-Jacobi image of generalizaed Heegner cycles and the values outside the $p$-adic interpolation of some anticyclotomic $p$-adic $L$-function. 
The aim of this project is to carry out  the first step toward a generalization of this work to the diagonal cycles of unitary groups and $p$-adic $L$-functions (cf. \cite{RSZ}). Using the $p$-adic Rankin-Selberg method, we construct a five-variable $p$-adic $L$-function interpolating a square root of the algebraic part of central values of the $L$-series attached to a pair of Hida families on the quasi-split unitary groups $\U(2,1)$ and $\U(1,1)$ in the shifted piano position. 
In contrast, the $p$-adic $L$-function in the piano position, in the sense of \cite[Definition 5.6]{GHL}, was constructed in \cite{HY}.

Moreover, we obtain the explicit interpolation formulae, which is consistent with the conjectural framework described in \cite{CPR,Coe,Coe2}.
A crucial novelty of our construction is the introduction of a new theta operator and its connection with complex differential operators. 

Let $E$ be an imaginary quadratic field. 
Throughout this paper we fix a prime number $p>3$ which {\it splits} in $E$ and an isomorphism $\iot_p:\CC\stackrel{\sim}{\to}\overline{\QQ}_p$, where $\overline{\QQ}_p$ is a fixed algebraic closure of $\QQ_p$. 
We write $\frkp$ for the prime ideal induced by the restriction of $\iot_p$ to $E$. 
We denote the rings of ad\`{e}les of $\QQ$ and $E$) by $\AA$ and $\EE$, and the rings of finite ad\`{e}les by $\widehat{\QQ}$ and $\widehat{E}$, respectively. 


\subsection{Galois representations associated to Hida families}\label{ssec:11}

Let $x\mapsto x^c$ be the non-trivial automorphism of $E$. 
Fix a Hermitian matrix $T$ in $\GL_n(E)$ of signature $(r,s)$. 
For $g\in\Mat_n(E)$ we define $g^\ddagger:=T^{-1}{}\trs g^c T$. 
The unitary group $\U(r,s)$ associated with $T$ is the algebraic group defined over $\QQ$ by setting 
\[\U(r,s)(R)=\{g\in \Mat_n(E\otimes_{\QQ} R)\mid g^\ddagger g=1\}\]
for any $\QQ$-algebra $R$. 

Fix a finite extension $F$ of $\QQ_p$ and denote its maximal compact subring by $\calo$. 
For each positive integer $n$, let $T_n\subset \GL_n$ be the diagonal torus. 
Let 
\[\Lam_n:=\calo\powerseries{T_n(\ZZ_p)}=\varprojlim_{m\geq 1}\calo\left[T_n(\ZZ/p^m\ZZ)\right] \]
be the completed group algebra, and $\bfI_n$ a semi-local and normal $\Lam_n$-algebra finite and flat over $\Lam_n$. 
We say that an $\calo$-algebra homomorphism $\ulQ :\bfI_n\to\overline{\QQ}_p$ is locally algebraic if its restriction to $T_n(\ZZ_p)$ is of the form $\ulQ(z_1,\dots,z_n)=\prod_{i=1}^nz_i^{k_{Q_i}}\eps_{Q_i}(z_i^{})$ with $(k_{Q_1},\dots,k_{Q_n})\in\ZZ^n$ and characters $\eps_{Q_i}:\ZZ_p^\times\rightarrow \overline{\QQ}_p^\times$ of finite order. 
We call $k_{\ulQ}=(k_{Q_1},\dots,k_{Q_n})\in\ZZ^n$ the weight of $\ulQ$ and $\eps_{\ulQ}=(\eps_{Q_1},\dots,\eps_{Q_n})$ the finite part of $\ulQ$. 
Let $\frkX_{\bfI_n}$ be the set of locally algebraic points in $\Spec\bfI_n(\overline{\QQ}_p)$. 
We say that $\ulQ\in\frkX_{\bfI_n}$ is dominant if $k_{Q_1}\leq k_{Q_2}\leq\dots\leq k_{Q_n}$. 
Let $\frkX_{\bfI_n}^+$ be the subset of locally algebraic points of dominant weights. 
Put 
\[\frkX_{\bfI_n}^\cls:=\{\ulQ\in \frkX_{\bfI_n}^+\mid k_{Q_n}-k_{Q_r}>r\}.\]

Let $N$ be a positive integer only divisible by primes $q\neq p$ split in $E$. 
Choose an ideal $\frkN$ of the ring $\frko_E$ of integers of $E$ such that $\frko_E/\frkN\simeq\ZZ/N\ZZ$.  
This ideal $\frkN$ shall be referred to as the tame level. 
Hida theory produces the $\Lam_n$-module $\bfS_\ord^{\U(r,s)}(\frkN,\Lam_n)$ of ordinary $\Lam_n$-adic cusp forms, which is free of finite rank equipped with a faithful action of the universal ordinary Hecke algebra $\bfT^{\U(r,s)}_{\ord}(N,\Lam_n)$ for the unitary group $\U(r,s)$. 
Roughly speaking, ordinary $\Lam_n$-adic cusp forms are $p$-adic families of $p$-ordinary cusp forms on $\U(r,s)$ invariant by the mirabolic subgroup of level $\frkN$. 
An $\bfI_n$-adic Hida family $\bdsf$ on $\U(r,s)$ is a non-zero Hecke eigenform in 
\[\bfS_\ord^{\U(r,s)}(\frkN,\bfI_n):=\bfS_\ord^{\U(r,s)}(\frkN,\Lam_n)\otimes_{\Lam_n}\bfI_n, \]
which induces a $\bfI_n$-algebra homomorphism $\lam_{\bdsf}:\bfT^{\U(r,s)}_{\ord}(N,\bfI_n)\to \bfI_n$. 

Denote the absolute Galois group of a field $L$ by $\Gam_L$ and its cyclotomic character by $\cyc$. 
Let $\frkm$ be a maximal ideal of $\bfI_n$. 
To each $\bfI_n$-adic Hida family $\bdsf$, one can associate the residual semisimple Galois representation $\bar\rho_{\bdsf}: \Gamma_E\to \GL_n(\bfI_n/\frkm)$. 
If $\bar\rho_{\bdsf}$ is absolutely irreducible, then we can further obtain the Galois representation $\rho_{\bdsf}:\Gamma_E\to \GL_n(\bfI_n)$ unramified outside primes dividing $Np$ and primes $l$ where $\U(r,s)(\QQ_l)$ is ramified (see \cite[Propositions 2.28, 2.29]{Geraghty} for more details). 
Denote by $V_{\bdsf}$ the free $\bfI_n$-module of rank $n$ on which $\Gamma_E$ acts via $\rho_{\bdsf}$. 

For each $\ulQ\in \frkX^\cls_{\bfI_n}$, the specialization $V_{\bdsf_{\ulQ}}:=V_{\bdsf}\otimes_{\bfI_n,\ulQ}\overline{\QQ}_p$ is the geometric $p$-adic Galois representation associated with some automorphic representation $\bdpi_{\ulQ}\simeq\otimes_v'\bdpi_{\ulQ,v}^{}$ of $\U(r,s)(\AA)$.  
Let $\frkX_{\bdsf}^\temp$ be the set of points $\ulQ\in\frkX_{\bfI_n}^\cls$ such that $\bdpi_{\ulQ}$ is everywhere tempered. 
The representation $V_{\bdsf}$ is \emph{conjugate self-dual} in the sense that
\[V_{\bdsf}^\vee\simeq V_{\bdsf}^c\otimes\cyc^{n-1}.\]
Moreover, by the local description of $p$-adic Galois representations \cite[Corollary 2.33]{Geraghty} at $p$ combined with \cite[Lemma 7.2]{TU99}, there exists a filtration $\{\mathrm{Fil}_i(V_{\bdsf}|_{\Gam_{E_\frkp}})\}_{i=1}^n$ of $\Gam_{E_\frkp}$-stable lattices 
\[\{0\}=\mathrm{Fil}_0(V_{\bdsf}|_{\Gam_{E_\frkp}})\subset \mathrm{Fil}_1(V_{\bdsf}|_{\Gam_{E_\frkp}})\subset \dots\subset \mathrm{Fil}_n(V_{\bdsf}|_{\Gam_{E_\frkp}})=V_{\bdsf}|_{\Gam_{E_\frkp}}\]
such that for every $\ulQ\in\frkX^+_{\bfI_n}$, the specialization $V_{\bdsf_{\ulQ}}|_{\Gam_{E_\frkp}}$ is Hodge-Tate and each graded piece \[\mathrm{gr}_i(V_{\bdsf_{\ulQ}}|_{\Gam_{E_\frkp}}):=\mathrm{Fil}_i(V_{\bdsf_{\ulQ}}|_{\Gam_{E_\frkp}})/\mathrm{Fil}_{i-1}(V_{\bdsf_{\ulQ}}|_{\Gam_{E_\frkp}})\] 
has Hodge-Tate weights 
\begin{align*}
-k_{Q_i}-s-i+1&\;\; (1\leq i\leq r), & 
-k_{Q_i}+r-i+1&\;\; (r+1\leq i\leq n). 
\end{align*}
Here the Hodge-Tate weight of $\QQ_p(1)$ is 1 in our convention. 
Likewise there exists a filtration $\{\mathrm{Fil}_i(V_{\bdsf}|_{\Gam_{E_{\frkp^c}}})\}_{i=1}^n$ of $\Gam_{E_{\frkp^c}}$-stable lattices in $V_{\bdsf}$ such that such that for every $\ulQ\in\frkX^+_{\bfI_n}$, each graded piece 
\[\mathrm{gr}_i(V_{\bdsf_{\ulQ}}|_{\Gam_{E_{\frkp^c}}})=\mathrm{Fil}_i(V_{\bdsf_{\ulQ}}|_{\Gam_{E_{\frkp^c}}})/\mathrm{Fil}_{i-1}(V_{\bdsf_{\ulQ}}|_{\Gam_{E_{\frkp^c}}})\] 
is Hodge-Tate of weight \begin{align*}
k_{Q_{n-i+1}}-r-i+1&\;\; (1\leq i\leq s), & 
k_{Q_{n-i+1}}+s+1-i&\;\; (s+1\leq i\leq n). 
\end{align*}


\subsection{The $L$-series for $\U(2,1)\times \U(1,1)$}\label{ssec:13}

Define the skew Hermitian matrices $S$ and $S'$ by 
\begin{align*}
S&=\begin{bmatrix} 0 & 0 & 1 \\ 0 & \sqrt{-D_E} & 0 \\ -1 & 0 & 0 \end{bmatrix}, & 
S'&=\begin{bmatrix} 0 & 1 \\ -1 & 0 \end{bmatrix}, 
\end{align*} 
where $D_E$ is the absolute value of the discriminant of $E$. 

Let $G=\U(S)$ and $H=\U(S')$ be the quasi-split unitary groups attached to $T$ and $T'$. 
Let $\calg=\GU(S)$ and $\calh=\GU(S')$ be the groups of unitary similitudes attached to $T$ and $T'$. 
Fix prime-to-$p$ natural numbers $N$ and $N'$ that satisfy the following condition: 
\beq\label{H1}
\text{all the prime factors of $NN'$ are split in $E$.} \tag{splt} 
\eeq
Fix a decomposition $N'\frko_E=\frkN'\frkN^{\prime c}$ such that $\frkN'+\frkN^{\prime c}=\frko_E$. 

Though it may be more natural to work with Shimura varieties attached to the unitary groups, which are not of PEL type (cf. \cite{H21,RSZ}), we start with Hida families 
\begin{align*}
\bdsf&\in\bfS^\calg_{\ord}(\frkN,\chi,\bfI_3), &
\bdsg&\in\bfS^\calh_{\ord}(\frkN',\chi',\bfI_2)
\end{align*}
on the unitary similitude groups $\calg$ and $\calh$, and restrict them to unitary groups $G$ and $H$. 
We further assume that the residual Galois representations $\bar\rho_{\bdsf}$ and $\bar\rho_{\bdsg}$ are both absolutely irreducible. Let $V_{\bdsf}$ and $V_{\bdsg}$ be the Galois representations of $\Gam_E$ associated with the Hida families $\bdsf$ and $\bdsg$ respectively. 
Consider the tensor product representation $V_{\bdsf\bdsg}:=V_{\bdsf}^{}\otimes V_{\bdsg}^\vee$ of rank six over the five variable Iwasawa algebra $\bfI_3\widehat\otimes_{\calo} \bfI_2$, where $V_{\bdsg}^\vee$ is the dual representation of $V_{\bdsg}$. 

Define the induced representation $\bfV$ of $\Gam_\QQ$ by 
\[\bfV:=\Ind_{\Gamma_E}^{\Gamma_\QQ}(V_{\bdsf\bdsg}^{}\otimes\cyc). \]
For each prime number $q$ we denote the Weil-Deligne group of $\QQ_q$ by $W_{\QQ_q}$. 
For each $\calq=(Q,Q')\in\frkX_{\bfI_3}^+\times\frkX_{\bfI_2}^+$, let $\bfV_\calq$ be the specialization of $\bfV$ at $\calq$ and define the complex $L$-series of the $p$-adic Galois representation $\bfV_\calq$ by the Euler product 
\[L(\bfV_\calq,s)=\prod_q L_q(\bfV_\calq,s) \]
of the local $L$-factors attached to $\mathrm{WD}_q(\bfV_\calq)\otimes_{\overline{\QQ}_p^{},\iot_p^{-1}}\CC$, where $\mathrm{WD}_q(\bfV_\calq)$ is the Weil-Deligne representation of $W_{\QQ_q}$ over $\overline{\QQ}_p$ associated to $\bfV_\calq$. 
Let \begin{align*}
(\lam_{Q_1},\lam_{Q_2},\lam_{Q_3})&=(-k_{Q_1},-k_{Q_2}-1;-k_{Q_3}+1);\\ 
(\mu_{Q_1'},\mu_{Q_2'})&=\biggl(-k_{Q_1'}-\frac{1}{2};-k_{Q_2'}+\frac{1}{2}\biggl),
\end{align*}
be the Harish-Chandra parameters of $\bdpi_{\ulQ}^{}$ and $\bdsig^\vee_{\ulQ'}$ respectively. Define the archimedean $L$-factor of $\bfV_\calq$ by 
\[\Gam(\bfV_\calq^{},s)=\prod_{i=1,2,3}\prod_{j=1,2}\Gamma_\CC\biggl(s+\frac{1}{2}+|\lam_{Q_i^{}}-\mu_{Q_j'}|\biggl),\]
where $\Gamma_{\CC}(s)=2(2\pi)^{-s}\Gamma(s)$. 
We are interested in the algebraic part of the value of $L(\bfV_\calq^{},s)$ at $s=0$. 
Note that $L(\bfV_\calq,0)$ are central values as $\bfV^\vee\otimes\cyc\simeq\bfV$ is self-dual. 
With the assumption \eqref{H1}, the specializations of the Hecke eigensystems $\ulQ\circ\lam_{\bdsf}$ and $\ulQ'\circ\lam_{\bdsg}$ determine unique unitary automorphic representations $\bdpi_{\ulQ}$ and $\bdsig_{\ulQ'}$ of $G(\AA)$ and $H(\AA)$, and we have
\[\Gam(\bfV_\calq^{},s)L(\bfV_\calq^{},s)=L\biggl(s+\frac{1}{2},\bdpi_{\ulQ}^{}\times\bdsig_{\ulQ'}^\vee\biggl).\] 
The right hand side is the {\it complete} automorphic $L$-function defined by 
\[L(s,\bdpi_{\ulQ}^{}\times\bdsig_{\ulQ'}^\vee)=L^{\GL}(s,\BC(\bdpi_{\ulQ}^{})\times \BC(\bdsig_{\ulQ'}^\vee)),\]
where $\BC(\bdpi_{\ulQ}^{})$ (resp. $\BC(\bdsig_{\ulQ'}^\vee)$) is the functorial lift of $\bdpi_{\ulQ}^{}$ (resp. $\bdsig_{\ulQ'}^\vee$) to an automorphic representation of $\Res_{E/\QQ}\GL_3(\AA)$ (resp. $\Res_{E/\QQ}\GL_2(\AA)$) (cf. \cite[Corollary 5.3]{Lab}). 
Here $\Res_{E/\QQ}\GL_n$ denotes the general linear group over an imaginary quadratic field $E$, regarded as an algebraic group over $\QQ$ by restricting scalars.
The $L$-function in the right hand side has been defined by the Rankin-Selberg convolution whose local and global analytic theories were established by Jacquet, Piatetski-Shapiro and Shalika in \cite{JPSS2}. 

The set of critical points in the shifted piano range is defined by 
\[\frkX^{\rm crit}_\bfV=\{(\ulQ,\ulQ')\in \frkX_{\bdsf}^{\temp}\times \frkX_{\bfI_2}^\cls\mid k_{Q_1^{}}\leq k_{Q_1'}\leq k_{Q_2^{}},\;k_{Q_3^{}}\leq k_{Q_2'}\}. \]


\subsection{The algebraicity of central values}\label{ssec:12}


Fix $(\ulQ,\ulQ')\in\frkX^{\rm crit}_\bfV$. 
Put 
\begin{align*}
\pi&:=\bdpi_{\ulQ}, & \ulk&=(k_1^{},k_2^{};k_3^{}):=k_{\ulQ}, \\
\sig&:=\bdsig_{\ulQ'}, & \ulk'&=(k_1';k_2'):=k_{\ulQ'}.
\end{align*}
Let $L(s,\sig,\Ad)$ and $L(s,\pi,\Ad)$ be the {\it complete} adjoint $L$-functions of $\sig$ and $\pi$, respectively. 
These are the Asai and twisted Asai $L$-functions of $\BC(\sig)$ and $\BC(\pi)$ (cf. Remark \ref{rem:71}). 
With the assumption \eqref{H1} we denote the conductor of $\pi$ by $\frkN_\pi$, and the conductor of $\sig$ by $\frkN_\sig$ (cf. \S \ref{ssec:82}).


Let $f^\circ\in S^G_{\ulk}(\frkN_\pi,\chi,\overline{\QQ})$ be a vector-valued \emph{holomorphic} $\overline{\QQ}$-rational newform associated to $\pi$ (see Definition \ref{def:57}). 
A perfect bilinear pairing $(\,,\,)_{\frkN_\pi}$ on the space of modular forms is introduced in \S \ref{ssec:711}.
We write $\Ome_\infty$ for the complex period of a CM elliptic curve introduced in \cite[(4.4 b)]{HT93} (see \S \ref{ssec:67} and (\ref{tag:65})).
For $a,b\in \CC$ we write $a\sim b$ if $b\neq 0$ and $a/b\in \overline{\QQ}$. 
We will prove the following theorem in in Section \ref{ssec:76}.   
\begin{theorem}
Notation being as above, we have 
\[L\left(\frac{1}{2},\pi\times\sig^\vee\right)\sim\biggl(\frac{\Ome_\infty}{\pi}\biggl)^{2(k_1'+k_2'-k_1-k_2-k_3^{})}\frac{L(1,\pi,\mathrm{Ad})}{(f^\circ,f^\circ)_{\frkN_\pi}}L(1,\sig,\mathrm{Ad})^2. \]
\end{theorem}
 
In Appendix \ref{sec:d} we will see that this theorem is compatible with Theorem 5.19 of \cite{GHL}.


\subsection{The periods}\label{ssec:14}

Let $\addchar:\AA/\QQ\to\CC^\times$ be the additive character with the archimedean component $\addchar_\infty(z)=e^{2\pi\sqrt{-1}z}$ and $\addchar_p:\QQ_p\to\CC^\times$ the local component of $\addchar$ at the prime number $p$. 
Define the additive character $\addchar^E$ of $\EE/E$ by $\addchar^E(x)=\addchar(\Tr_{E/\QQ}(x/\sqrt{-D_E}))$. 

To make our interpolation formula more precise, we explain the periods for critical $L$-values associated with the Galois representation $\bfV_{\calq}$. 
Let $g^\circ\in S^H_{\ulk'}(\frkN_\sig,\chi',\overline{\QQ})$ be the normalized newform associated to $\sig$. 
Put
\[\Ome(\sig):=D_E^{-1}2^2(-2\sqrt{-1})^{k_2'-k_1'}\|g^\circ\|_{\frkN_\sig}^2\cale(\sig_p,\Ad,\addchar_p),\]
where $\|\cdot\|_{\frkN_\sig}$ denotes the normalized Petersson norm and $\cale(\sig_p,\Ad,\addchar_p)$ is the modified Euler $p$-factor for the adjoint motive attached to $\sig$, defined in Definition \ref{def:82}. 
The period $\Ome(\sig)$ is analogous to Hida's canonical period considered in \cite[Definition 3.12]{MH}. 
Corollary \ref{cor:c1} shows that 
\[L(1,\sig,\mathrm{Ad})\sim\Ome(\sig). \]

Let $f^\circ$ be the Picard newform associated with $\bdsf_{\ulQ}$ in the sense of Definition \ref{def:41}. 
In addition to $\Omega(\sig)$ we introduce a non-zero complex number $\Xi_{\bdsf_{\ulQ}}$ in \S \ref{ssec:711}. 
Roughly speaking, the period $\Xi_{\bdsf_{\ulQ}}$ is given by the ratio 
\[\Xi_{\bdsf_{\ulQ}}\sim\frac{L(1,\pi,\Ad)}{(f^\circ,f^\circ)_{\frkN_\pi}}. \]  

Let $f^{\rm gen}$ be the normalized vector-valued \emph{generic} newform in the global $L$-packet of $\pi$ with respect to $\addchar^E$.  
In \cite[Theorem 4.1]{Chen25}, Shih-Yu Chen shows that the Petersson norm $\La f^{\rm gen},f^{\rm gen}\Ra^{}_{\frkN_\pi}$ of $f^{\rm gen}$ is equal to the adjoint $L$-value $L(1,\pi,\Ad)$, using the Lapid-Mao conjecture. 
Therefore, $\Xi_{\bdsf_{\ulQ}}$ is essentially the ratio $\frac{\La f^{\rm gen},f^{\rm gen}\Ra^{}_{\frkN_\pi}}{(f^\circ,f^\circ)_{\frkN_\pi}}$ between the Petersson norms of normalized generic forms and holomorphic forms in the $L$-packet of $\pi$. The form $f^{\rm gen}$ is normalized so that the valued of the $\psi^E$-Whittaker function of highest weight component of $f^{\rm gen}$ at $\diag({\sqrt{-D_E},1,-\sqrt{-D_E}^{-1}})$ is given by \[K_{k_2-k_3+2}(4\sqrt{2}\pi),\]
where $K_s(z)$ is the modified $K$-Bessel function.



\subsection{The modified Euler factors}\label{ssec:15}
For $\calq=(\ulQ,\ulQ')\in\frkX^{\rm crit}_\bfV$ we view the algebraic number 
\[\biggl(\frac{2\pi\sqrt{-1}}{\Ome_\infty}\biggl)^{2(k_{Q_1'}+k_{Q_2'})-2(k_{Q_1}+k_{Q_2}+k_{Q_3^{})}}\frac{\Gam(\bfV_\calq^{},0)L(\bfV_\calq^{},0)}{\Xi_{\bdsf_{\ulQ}}^{}\Ome(\bdsig_{\ulQ'})^2}\] 
through the embedding $\iota_p$, as a $p$-adic number. 
We are interested in the $p$-adic behavior of this ratio when $\calq$ varies in $\frkX^{\rm crit}_\bfV$. A general conjecture on $p$-adic $L$-functions predicts that, after inserting a suitably modified Euler factor at $p$ (recalled below), this ratio varies $p$-adic analytically.

We consider the  rank three $\Gam_{E_\frkp}$-invariant and $\Gam_{E_{\frkp^c}}$-invariant subspaces of $V_{\bdsf\bdsg}$ by 
\begin{align*}
\mathrm{Fil}^+_\frkp\bfV_{\bdsf\bdsg}&=\mathrm{Fil}_1V_{\bdsf}^{}|_{\Gam_{E_\frkp}}\otimes V_{\bdsg}^\vee|_{\Gam_{E_\frkp}}+V_{\bdsf}^{}|_{\Gam_{E_\frkp}}\otimes \mathrm{Fil}_1V_{\bdsg}^\vee|_{\Gam_{E_\frkp}}; \\
\mathrm{Fil}^+_{\frkp^c}\bfV_{\bdsf\bdsg}&=\mathrm{Fil}_2V_{\bdsf}^{}|_{\Gam_{E_{\frkp^c}}}\otimes\mathrm{Fil}_1V_{\bdsg}^\vee|_{\Gam_{E_{\frkp^c}}},
\end{align*}
and define the six dimensional $\Gam_{\QQ_p}$-invariant subspace of $\bfV$ by 
\begin{align*}
\mathrm{Fil}^+\bfV&=\left(\mathrm{Fil}^+_\frkp\bfV_{\bdsf\bdsg}\oplus\mathrm{Fil}^+_{\frkp^c}\bfV_{\bdsf\bdsg}\right)\otimes\cyc.
\end{align*}
The pair $(\mathrm{Fil}^+\bfV,\frkX^{\rm crit}_\bfV)$ satisfies the Panchishkin condition in \cite[p.~217]{Gre} in the sense that all the Hodge-Tate numbers of $\mathrm{Fil}^+\bfV_\calq$ are positive but none of the Hodge-Tate numbers of $\bfV_\calq/\mathrm{Fil}^+\bfV_\calq$ is positive for $\calq\in\frkX^{\rm crit}_\bfV$. 

For each $p$-adic representation $V$ of $\Gamma_{\QQ_p}$, recall that the $\gamma$-factor $\gamma(V)$ is defined by 
\[\gamma(V)=\varepsilon({\rm WD}_p(V),\addchar_p)\frac{L_p(V^\vee,1)}{L_p(V,0)},\]
where $\varepsilon({\rm WD}_p(V),\addchar_p)$ is the local constant defined in \cite[\S 4]{Deligne73}. 
Following the recipe of Coates and Perrin-Riou, we define the modified factor at $p$ by 
\[\cale_p(\mathrm{Fil}^+\bfV_\calq)
=\frac{1}{\gam(\mathrm{Fil}^+\bfV_\calq)L_p(\bfV_\calq,0)}. \]
In the theory of $p$-adic $L$-functions we also need the modified factor $\cale_\infty(\bfV_\calq)$ at the archimedean place as observed by Deligne. 
It is explicitly given by
\[\cale_\infty(\bfV_\calq)=(-\sqrt{-1})^{k_{Q_1'}+k_{Q_2'}-k_{Q_1}^{}-k_{Q_2}^{}-k_{Q_3}^{}}. \]


\subsection{Five-variable $p$-adic $L$-functions}\label{ssec:16}

Given a ring $R$, we write $\mathrm{Frac}(R)$ for the total ring of fractions of $R$. 
Let $\Ome_\frkp$ be the $p$-adic CM period attached to an elliptic curve with complex multiplication by $E$ (cf. \S \ref{ssec:67} and \cite[(4.4a)]{HT93}). 
Put $\frkX_{\bdsf}'=\{\ulQ\in\frkX^{\rm cls}_\calr\;|\;\ulQ(\bdsf)\neq 0\}$. 
The subset $\frkX_{\bdsf}''$ consists of $\ulQ\in\frkX_{\bdsf}'$ such that $\bdsf_{\ulQ}$ is new outside $p$. 
For a technical reason we consider its subset 
\[\frkY^{\rm crit}_\bfV=\{(\ulQ,\ulQ')\in \frkX^{\rm crit}_\bfV\mid k_{Q_1^{}}=k_{Q_1'}\}. \]

\begin{theorem}\label{thm:11}
Let $M$ be the conductor of $\chi'$. 
We assume \eqref{H1} and assume that 
\beq
M^2\text{ is divisible by }N', \tag{$H_3$}
\eeq 
Then there is a unique element $\scrl_{\bdsf,\bdsg}\in\bfI_3\widehat{\otimes}_\calo\mathrm{Frac}(\bfI_2)$ such that 
\[\frac{\calq(\scrl_{\bdsf,\bdsg}^2)}{\Ome_\frkp^{2m_{\calq}}}
=\biggl(\frac{2\pi\sqrt{-1}}{\Ome_\infty}\biggl)^{2m_{\calq}}\frac{\Gam(\bfV_\calq^{},0)L(\bfV_\calq^{},0)}{\Xi_{\bdsf_{\ulQ}}\Ome(\bdsig_{\ulQ'})^2}\cale_\infty(\bfV_\calq)\cale_p(\mathrm{Fil}^+\bfV_\calq) \]
for $\calq=(\ulQ,\ulQ')\in\frkY^{\rm crit}_\bfV\cap(\frkX_{\bdsf}''\times\frkX_{\bfI_2}^\cls)$, 
where 
\[m_\calq=k_{Q_1'}+k_{Q_2'}-(k_{Q_1}+k_{Q_2}+k_{Q_3}). \]
\end{theorem}

\begin{remark}\label{rem:12}
\begin{enumerate}
\item We note that the global root number 
\[\vep\biggl(\frac{1}{2},\bdpi_{\ulQ}^{}\times\bdsig_{\ulQ'}^\vee\biggl)=+1\text{ for }(\ulQ,\ulQ')\in\frkX^{\rm crit}_\bfV\] for $(\ulQ,\ulQ')\in\frkX^{\rm crit}_\bfV$ by (splt) and Remark \ref{rem:b1}(\ref{rem:b12}). 
\item\label{rem:121} The assumption ($H_3$) is technical. 
We denote the central character of $\bdsig_{\ulQ',l}$ by $\ome_{\bdsig_{\ulQ',l}}$. 
When $N'$ is odd, we can choose an auxiliary automorphic character $\vrh_\AA=\prod_v\vrh_v$ of $\U(1)(\AA)$ of finite order which ramifies precisely at prime factors of $N'$ so that 
\[c(\bdsig_{\ulQ',l}\otimes\vrh_l)=2c(\ome_{\bdsig_{\ulQ',l}\otimes\vrh_l})\] 
for all $\ulQ'\in\frkX^\cls_{\bfI_2}$ and prime factors $l$ of $N'$ by stability of the epsilon factor. 
Fix $\calq_0^{}=(\ulQ_0,\ulQ'_0)\in\frkY^{\rm crit}_\bfV\cap(\frkX_{\bdsf}'\times\frkX_{\bfI_2}^\cls)$. 
Let $f_\vrh$ and $g_\vrh$ be $p$-stabilized newforms associated to $\bdpi_{\ulQ_0}\otimes\vrh_\AA^{}$ and $\bdsig_{\ulQ'_0}\otimes\vrh_\AA^{}$. 
Theorem \ref{thm:41} allows us to lift $f_\vrh$ and $g_\vrh$ to Hida families $\bdsf_\vrh$ and $\bdsg_\vrh$ to which Theorem \ref{thm:11} can apply. 
The specialization $\calq_0^{}(\scrl_{\bdsf_\vrh,\bdsg_\vrh})$ involves the central value $\Gam(0,\bfV_{\calq_0})L(0,\bfV_{\calq_0})$. 

\item\label{rem:122} If $\ulQ\in\frkX_{\bfI_3}^\cls$ satisfies $k_{Q_1}<k_{Q_2}$, then the weights of $\bdpi_{\ulQ,\infty}$ are regular, and so by the endoscopic classification of cuspidal automorphic representations on $\U(2,1)$ \cite{Rog90} (see also \cite[3.2.3, p.~618]{BC04}) combined with the Ramanujan conjecture \cite[Theorem 1.2]{Cara12} we have $\ulQ\in\frkX_{\bdsf}^\temp$. 
\item\label{rem:123} Hsieh and Yamana \cite{HY} have constructed five-variable $p$-adic $L$-functions for the Rankin-Selberg convolution $\BC(\bdpi_{\ulQ}^{})\times \BC(\bdsig_{\ulQ'}^\vee)$ in the piano case, where the weights satisfy the interlacing relation 
\[k_{Q_1^{}}\leq k_{Q_1'}\leq k_{Q_2^{}}\leq k_{Q_2'}\leq k_{Q_3^{}}, \]
using Hida families on definite unitary groups $\U(3)\times\U(2)$. 
\item We expect the interpolation formula in Theorem \ref{thm:11} holds for all $(\ulQ,\ulQ')\in \frkX_{\bfV}^{\rm crit}$, but at this moment, we only obtain the interpolation formula at $\frkY^{\rm crit}_\bfV$ due to our limited understanding of the integral properties of vector-valued Fourier-Jacobi coefficients. \end{enumerate}
\end{remark}


\subsection{One-variable $p$-adic $L$-functions}\label{ssec:17}

Let $\ulk=(k_1^{},k_2^{};k_3^{})$ be a tuple of integers such that $k_1^{}\leq k_2^{}\leq k_3-2$. 
Let $\pi\simeq\otimes_v'\pi_v^{}$ be an irreducible tempered automorphic representation of $G(\AA)$ which is unramified at $p$ and all the non-split primes and such that $\pi_\infty$ is a discrete series or limit of discrete series with Blattner parameter $-\ulk$. 
We identify $\pi_p$ with an irreducible principal series $I(\nu_p,\rho_p,\mu_p)$ with unramified characters $\nu_p,\rho_p,\mu_p$ of $\QQ_p^\times$. 
Fix a $\overline{\QQ}$-rational $p$-stabilized newform $f_\pi^{\ord}\in S^G_{\ulk}(p\frkN_\pi,\chi,\overline{\QQ})$ associated to $\pi$ with respect to $\nu_p,\rho_p,\mu_p$, by which we mean that 
\begin{align*}
U_p(\alp_1)f_\pi^{\ord}&=p\nu_p(p)^{-1}f_\pi^{\ord}, & 
U_p(\bet_1)f_\pi^{\ord}&=p\mu_p(p)f_\pi^{\ord}
\end{align*}
(see Proposition \ref{prop:81}). 
When $\nu_p(p),\rho_p(p),\mu_p(p)$ are mutually different, one can obtain $f_\pi^{\ord}$ by applying the $p$-stabilization to a $\overline{\QQ}$-rational newform $f_\pi^\circ\in S^G_{\ulk}(\frkN_\pi,\chi,\overline{\QQ})$ associated to $\pi$ (cf. \cite{MY14,HY3}). 

Take a finite extension $F$ of $\QQ_p$ in such a way that $f_\pi^{\ord}\in S^G_{\ulk}(p\frkN_\pi,\chi,F)$ via $\iot_p$. 
Let $\calo$ be the maximal compact subring of $F$ and $\vpi$ a prime element of $\calo$. 
We may normalize $f_\pi^{\ord}$ so that $S^G_{\ulk}(p\frkN_\pi,\chi,\calo)$ contains $f_\pi^{\ord}$ but does not contain $\vpi^{-1}f_\pi^{\ord}$. 

Let $\bfI$ be a semi-local normal ring finite and flat over $\Lam=\calo\powerseries{\ZZ_p^\times}$. 
Let $\bdsg'\in\bfS_\ord^\calh(\frkN',\chi',\bfI)$ be a one-variable cuspidal Hida family on $\calh$ such that $\bdsg'_Q\in\bdse' S_{(k_1;k_Q)}^\calh(p^\ell\frkN',\chi'\eps_Q^\downarrow,\bfI(Q))$ for $Q\in\frkX_\bfI$ with $k_Q-k_1\geq 2$, where the character $\eps_Q^\downarrow$ is defined in (\ref{tag:81}).  
We denote the Galois representation of $\Gam_E$ associated to $\pi$ by $V_\pi$. 
Put 
\[\bfV':=\Ind_{\Gamma_E}^{\Gamma_\QQ}(V_\pi^{}\otimes V_{\bdsg'}^{}\otimes\cyc). \]
We write $\bdsig_Q\simeq\otimes'_v\bdsig_{Q,v}^{}$ for the irreducible cuspidal automorphic representation of $H(\AA)$ associated to $\bdsg'_Q$ in the sense of \S \ref{ssec:54}. 
Let 
\[\cale_p(\mathrm{Fil}^+\bfV_Q')=\cale(\pi_p^{},\bdsig_{Q,p}^\vee)\] 
be the modified factor relative to $\mu_p$ (see Definition \ref{def:72}). 
To remove the assumption ($H_3$), we use the simplified period $\Ome^{(N')}(\bdsig_Q)$ introduced in Definition \ref{def:74}. 

\begin{theorem}\label{thm:12}
Notation and assumption being as above, if all the prime factors of $N'$ are odd and split in $E$, then there exists a unique element $\call_{f_\pi^{\ord},\bdsg'}\in\mathrm{Frac}\bfI$ such that 
\[\frac{Q(\call_{f_\pi^{\ord},\bdsg'}^2)}{\Ome_\frkp^{2m_Q'}}=\biggl(\frac{2\pi\sqrt{-1}}{\Ome_\infty}\biggl)^{2m_Q'}\frac{\Gam(\bfV_Q',0)L(\bfV_Q',0)}{\Xi_{f_\pi^{\ord}}\Ome^{(N')}(\bdsig_Q)^2}\cale_\infty(\bfV_Q')\cale_p(\mathrm{Fil}^+\bfV_Q') \]
for $Q\in\frkX_\bfI$ with $k_Q\geq k_3$, where 
\[m_Q'=k_Q^{}-(k_2^{}+k_3^{}). \]
\end{theorem}

\begin{remark}
\begin{enumerate}
\item The most surprising and confusing point in our construction of the theta element $\call_{f_\pi^{\ord},\bdsg'}$ is that it does not impose any ordinary hypothesis on $\pi$, even though  the Panchishkin condition requires $\pi$ to be $\frkp^c$-ordinary, namely, a $\Gam_{E_\frkp}$-stable lattice $\mathrm{Fil}_1V_\pi$ with Hodge-Tate weight $-k_1-1$ to exists. 
Note that 
\[\mathrm{Fil}_2V_\pi^c=\{u\in V_\pi^c\;|\;\La v,u\Ra=0\text{ for }v\in \mathrm{Fil}_1V_\pi\}, \]
where $\La\;,\;\Ra:V_\pi^{}\otimes V_\pi^c\to\cyc^{-2}$ is a perfect pairing. 
Actually, we need the $\frkp^c$-ordinary hypothesis when we vary $\eps_{Q_1'}$
\item When $k_1=k_2=k_3-2$, the range of interpolation includes the points $Q$ such that $k_Q=k_1+2$, i.e., the restriction of $\bdsg'_Q$ to $\GL_2(\AA)$ is an elliptic cusp form of weight $2$. 
The specialization of the five-variable $p$-adic $L$-function at such points will be interesting (cf. \cite{DD14,DD17}). 
\end{enumerate}
\end{remark}

We give a brief outline of the construction of $\call_{f^{\ord}_\pi,\bdsg'}$ and a proof of the interpolation formula. 
The main ingredients are followings:
\begin{itemize}
\item geometric theory of Fourier-Jacobi coefficients of Picard modular forms;
\item Shimura's theory of differential operators;
\item Bannai-Kobayashi theory of algebraic CM theta functions;
\item the Ichino--Ikeda--Neal-Harris (IINH) formula. 
\end{itemize}
Below we describe these ingredients and their role in more detail. 


\subsection{Fourier-Jacobi coefficients of Picard modular forms}

For simplicity we suppose that $k_1=k_2=0$ and $\pi_p$ is unramified. 
Put $k:=k_3$ and $\ulk_0=(0,0;k)$. 
The Hermitian symmetric domain associated to $G(\RR)$ is 
\begin{align*}
\frkD&=\biggl\{Z=\begin{bmatrix} \tau \\ w \end{bmatrix}\in\CC^2\;\biggl|\;\eta(Z)>0\biggl\}, & 
\eta(Z):=\sqrt{-1}(\tau^c-\tau)-\frac{ww^c}{\sqrt{D_E}}. 
\end{align*}
The group $G(\RR)$ acts on $\frkD$ by $\alp\begin{bmatrix} Z \\ 1 \end{bmatrix}=\begin{bmatrix} \alp Z \\ 1 \end{bmatrix}\mu(\alp,Z)$ for $\alp\in G(\RR)$ and $Z\in\frkD$. 
Put $q:=e^{2\pi\sqrt{-1}\tau}$.
The space $S_{\ulk_0}^G(p^\ell\frkN,\chi,\CC)$ consists of functions $f:\frkD\times G(\widehat{\QQ})\to\CC$ which are holomorphic in $Z$, satisfy 
\[f(\gam Z,\gam g u)=\chi(u)^{-1}\mu(\gam,Z)^kf(Z,g) \]
for $\gam\in G(\QQ)$, $g\in G(\widehat{\QQ})$, $u\in K_0(p^\ell\frkN)$, and have Fourier-Jacobi expansions 
\[f\biggl(\begin{bmatrix} \tau \\ w \end{bmatrix},\bfd(a)\biggl)=\sum_{m\in\scrs_\frka^+}\FJ^m_\frka(w,f)q^m \]
for finite id\`{e}les $a\in\widehat{E}^\times$ such that $a_\frkl=1$ for $\frkl|pN$, where $\scrs_\frka^+:=(\frka\frka^c)^{-1}\cap\QQ^\times_+$, $\bfd(a):=\diag(a^c,1,a^{-1})$ and $\frka$ is the fractional ideal of $\frko_E$ associated to $a$. 
It is important to note that Fourier-Jacobi coefficients are CM theta series, namely, there is a lattice $L_0\subset E$ such that for every $l\in\frka L_0$ 
\[\FJ^m_\frka(w+l,f)=e^{\frac{2\pi m}{\sqrt{D_E}}\bigl(w+\frac{l}{2}\bigl)l^c}\FJ^m_\frka(w,f). \]

Let $\overline{\ZZ}_{(p)}$ denote the integral closure of $\ZZ_{(p)}=\QQ\cap\ZZ_p$ in $\overline{\QQ}$. 
Let $f\in S_{\ulk_0}^G(p^\ell\frkN,\chi,\CC)$ be $\overline{\ZZ}_{(p)}$-integral in the sense of Definition \ref{def:33}. 
Then $\FJ^m_\frka(w,f)$ can be identified with an integral section of a pullback of Poincar\`{e} bundle of a CM elliptic curve over $\overline{\ZZ}_{(p)}$ by Proposition \ref{prop:31} (cf. \cite{Lan12}). 

The holomorphic function $w\mapsto \FJ^m_\frka\bigl(\frac{w}{\Ome_\infty},f\bigl)$ has algebraic Taylor coefficients, and moreover, the coefficients $\biggl(\dfrac{\Ome_\frkp}{\Ome_\infty}\biggl)^n\dfrac{\partial^n\FJ^m_\frka}{\partial w^n}(0,f)$ are $p$-integral by the work \cite{BK10Duke} of Bannai and Kobayashi. 
We associate to this $p$-integral power series the $p$-adic measure $\d\mu(\widetilde{\FJ}_\frka^m(f)_0^{\Ome_\infty})$ on $\ZZ_p$. 


\subsection{Shimura's differential operator}

Shimura has introduced a theory of nearly homomorphic modular forms and differential operators on Hermitian symmetric domains in his paper \cite{Shimura86}. 
His differential operator $\del_{\ulk}$ preserves near holomorphy and $\overline{\QQ}$-rationality. 
At the same time, the first named author gave an equivalent (and more general) construction in \cite{Harris86}, using algebraic geometry. 

If $f\in S_{\ulk_0}^G(p^\ell\frkN,\chi,\overline{\QQ})$, then $\del_{\ulk_0}f$ is a $\overline{\QQ}$-rational nearly holomorphic modular form valued in the dual space of the global holomorphic tangent space $\frkT=\CC\frac{\partial}{\partial\tau}\oplus\CC\frac{\partial}{\partial w}$ of $\frkD$. 
Let $\frkH$ be the upper half plane. 
We regard $\frkH$ as a subdomain of $\frkD$ via the map $\tau\mapsto\begin{bmatrix} \tau \\ 0 \end{bmatrix}$ which is compatible with the embedding $\iot:H\hookrightarrow G$ defined by  
\[\iota\left(\begin{bmatrix} a & b \\ c & d \end{bmatrix}\right)
=\begin{bmatrix}a & 0 & b\\
0 & 1 & 0 \\
c & 0 & d \end{bmatrix}. \]

Appendix \ref{sec:a} computes iteration of this operator and evaluates it at $w=0$ to produce a nearly holomorphic cusp form on $\frkH$ of weight $(0,k+n)$. 
More precisely, we have 
\begin{align*}
\del_{\ulk_0}^nf\biggl(\frac{\partial}{\partial w},\dots,\frac{\partial}{\partial w}\biggl)\biggl|_\frkH
&=\frac{\eta(Z)^{-k}}{(2\pi\sqrt{-1})^n}\left(\frac{\partial}{\partial w}-\frac{w^c}{\sqrt{-D_E}}\frac{\partial}{\partial\tau}\right)^n\eta(Z)^kf(Z)\biggl|_{w=0} \\
&\equiv(2\pi\sqrt{-1})^{-n}\frac{\partial^nf}{\partial w^n}\biggl(\begin{bmatrix} \tau \\ 0 \end{bmatrix}\biggl)\pmod{y^{-1}},
\end{align*}
where we write $y$ for the imaginary part of $\tau$. 
We denote the holomorphic projection from the space of nearly holomorphic modular forms on $\frkH$ by $\Hol$, and Hida's ordinary projection on $p$-adic modular forms on $H$ by $\bdse'$. 
We deduce an important identity 
\[\bdse'\Hol\biggl(\del_{\ulk_0}^nf\biggl(\frac{\partial}{\partial w},\dots,\frac{\partial}{\partial w}\biggl)\biggl|_H\biggl)=\bdse'\biggl((2\pi\sqrt{-1})^{-n}\frac{\partial^nf}{\partial w^n}\biggl|_H\biggl) \]
from the simple relation stated in \cite[(3), p.~311]{Hid93}. 

Shimura's differential operators can be reformulated algebro geometrically in terms of the Gauss-Manin connection (cf. Section 6 of \cite{Eischen12}). 


\subsection{The $\Theta_\frkp$-operator}

Given $w\in\EE$ and $z\in\AA$, we define an element of $G(\AA)$ by 
\[\bfn(w,z)=\begin{bmatrix} 1 & w^c & z+\frac{1}{2}\sqrt{-D_E}ww^c \\ 0 & 1 & \sqrt{-D_E}w \\ 0 & 0 & 1 \end{bmatrix}. \]
Fix a sufficiently small ideal $\frkc$ of $\frko_E$ prime to $\frkp$. 
Take a set $\calx_n$ of representatives for $\frkp^{-n}\frkc/\frkc$. 
Section \ref{sec:6} associates the operator 
\beq
\bdTht_\frkp^\phi f(Z,g)=\frac{1}{p^n}\sum_{x\in\calx_n}\sum_{z\in (\ZZ_p/p^n\ZZ_p)^\times}\phi(z)\addchar_p(zx_\frkp)f(Z,g\bfn(x_\frkp,0)) \label{tag:11}
\eeq
to an $\calo$-valued function $\phi$ on $(\ZZ_p/p^n\ZZ_p)^\times$. 
Here we view $x$ as a principal ad\`{e}le and denote its $\frkp$-component by $x_\frkp$. 
Building on the work \cite{BK10Duke} of Bannai-Kobayashi, we show that $\bdTht_\frkp^\phi$ can be interpreted in terms of $p$-adic translation by $\frkp$-power torsion points by the pull-back of theta functions to the formal group of a CM elliptic curve and see that 
\[\int_{\ZZ_p^\times}\phi(x)\,\d\mu(\widetilde{\FJ}_\frka^m(f))=\FJ^m_\frka(0,\bdTht_\frkp^\phi f) \]
for locally constant functions $\phi$ on $\ZZ_p^\times$. 
In particular, the right hand side of (\ref{tag:11}) is independent of $n$ and bounded. 
Furthermore, we have 
\[\int_{\ZZ_p}x^n\,\d\mu(\widetilde{\FJ}_\frka^m(f))=\biggl(\frac{\Ome_\frkp}{\Ome_\infty}\biggl)^n\frac{\partial^n\FJ^m_\frka}{\partial w^n}(0,f)\]
for non-negative integers $n$. 
We obtain an element $\calj^m_\frka(f)\in\Lam$ such that for $Q\in\frkX_\Lam$ with $k_Q\geq k$ 
\[Q(\calj^m_\frka(f))=\biggl(\frac{\Ome_\frkp}{\Ome_\infty}\biggl)^{k_Q-k}\frac{\partial^{k_Q-k}\FJ^m_\frka}{\partial w^{k_Q-k}}(0,\bdTht_\frkp^{\eps_Q}f). \]

Now we can define formal power series $\calj_\frka(f)\in\Lam_2\powerseries{q}$ by 
\[\calj_\frka(f)=\lim_{j\to\infty}\sum_{m\in\scrs_\frka^+}\calj^{mp^{j!}}_\frka(f)q^m. \]
The set $\{\calj_\frka(f)\}_{\frka\in C_E}$ forms a $p$-adic family $\calj(f)$ of $p$-ordinary cusp forms on $H$ such that 
for $Q\in\frkX_\Lam$ with $k_Q\geq k$
\[Q(\calj(f))=\biggl(\frac{2\pi\sqrt{-1}\Ome_\frkp}{\Ome_\infty}\biggl)^{k_Q-k}\bdse'\Hol\biggl(\del_{\ulk_0}^{k_Q-k}\bdTht_\frkp^{\eps_Q}f\biggl(\frac{\partial}{\partial w},\dots,\frac{\partial}{\partial w}\biggl)\biggl|_H\biggl). \]


\subsection{The Ichino-Ikeda formula}

Let $\bdsg'\in\bfS_\ord^\calh(\frkN',\chi',\bfI)$ be the one-variable cuspidal Hida family on $\calh$. 
Note that the restriction of $\bdsg'$ to $\GL_2$ is a usual Hida family of elliptic cusp forms. 
We denote its restriction to $H$ also by $\bdsg'$ and write $\1_{\bdsg'}\in\bfT^H_{\ord}(N',\bfI)\otimes_\bfI\mathrm{Frac}\bfI$ for the idempotent corresponding to $\bdsg'$. 
For simplicity we assume that $N'=1$ and $\chi'=1$. 

Define $\vsi^{(p)}_{}=(\vsi^{(p)}_l)\in G(\widehat{\QQ})$ by 
\begin{align*}
\vsi&=\begin{bmatrix}
1 & 0 & 0 \\ 
0 & 0 & 1 \\
0 & 1 & 0   
\end{bmatrix}, &
\vsi^{(p)}_l&=\begin{cases}
\imath_\frkl^{-1}(\vsi) &\text{if $l$ splits in $E$ and differs from $p$, }\\
\ono_3 &\text{otherwise. } 
\end{cases} 
\end{align*}
Define the $\frkp^c$-depletion of $f_\pi^{\ord}$ by 
\[[\bfU_{\frkp^c}^0f_\pi^{\ord}](Z,g)=\frac{1}{p\mu_p(p)}\sum_{y\in(\ZZ_p/p\ZZ_p)^\times}\sum_{z\in \ZZ_p/p\ZZ_p}f_\pi^{\ord}\left(Z,g\imath_\frkp^{-1}\left(\begin{bmatrix} p & y & z \\ 0 & 1 & 0 \\ 0 & 0 & 1 \end{bmatrix}\right)\right). \] 
Let $\calf(Z,g)=[\bfU_{\frkp^c}^0f_\pi^{\ord}](Z,g\vsi^{(p)})$. 
Then $\calj(\calf)\in \bfS_\ord^H(\frko_E,\bfI)$. 
Writing the spectrum decomposition of $J(\calf)$, we define the theta element by  
\[\call_{f_\pi^{\ord},\bdsg'}:=\text{the first Fourier coefficient of }\1_{\bdsg'} \calj(\calf). \]
The evaluation of $\call_{f_\pi^{\ord},\bdsg'}$ at $Q\in\frkX_\Lam$ with $k_Q\geq k$ is roughly given by
\[\biggl(\frac{2\pi\sqrt{-1}\Ome_\frkp}{\Ome_\infty}\biggl)^{k_Q-k}\frac{\Big\La\del_{\ulk_0}^{k_Q-k}\bdTht_\frkp^{\eps_Q}\calf\Big(\frac{\partial}{\partial w},\dots,\frac{\partial}{\partial w}\Big)\Big|_H,\bdsg'_Q\Big\Ra}{\|\bdsg'_Q\|_{\frkN_{\bdsig_Q}}^2}, \]
where $\La\;,\;\Ra$ is the Petersson pairing on $H$. 

The Ichino-Ikeda conjecture has been formulated for orthogonal groups in \cite{II} as a refinement of the global Gross--Prasad conjecture formulated in \cite{GP92}. 
The unitary analogue was formulated in \cite{GGP,NHarris} and proved in \cite{Z,BLZZ,BCZ} (cf. Theorem \ref{thm:71}). 
The pairing above is the Gan-Gross-Prasad period integral, and by the celebrated Ichino--Ikeda--Neal-Harris (IINH) formula, its square is written as the product of the central $L$-value  $L\bigl(\frac{1}{2},\pi\times\bdsig_Q^\vee\bigl)$ and local zeta integrals. 

A big chunk of this paper is to make explicit evaluations of these local zeta integrals including at $p$ and the archimedean place. 
Appendix \ref{sec:b} computes the archimedean integral, which proves the algebraicity of the $L$-value mentioned above, together with Shimura's theory on $\overline{\QQ}$-rationality of nearly holomorphic modular forms. 
We will compute the local integral at $p$ in Section \ref{sec:8}. 
A local key ingredient is the splitting lemma which has been proved in \cite{LM,Z} (see Lemma \ref{lem:81}). 
This lemma relates the local integral to a square of the JPSS integral at split primes. 

When $N'>1$, we will apply the operator $\bfU^{\chi'}_{\frkN'}$ (see \S \ref{ssec:63}). 


\subsection{Theta operators}

We conclude this introduction by remarking that in addition to the explicit evaluation of local integrals, a novelty of this work is the introduction of a new theta operator $\bdTht_\frkp^\phi$. 
Since $\phi\mapsto\bdTht_\frkp^\phi f$ is a $p$-adic measure on $\ZZ_p^\times$ valued in the space of $p$-adic modular forms, we can speak of operators $\bdTht_\frkp^n:=\bdTht_\frkp^{\phi_n}$ for integers $n$, where $\phi_n(z)=z^n$ $(z\in\ZZ_p^\times)$. 

This theta operator acts directly on Fourier-Jacobi coefficients, different from the one studied in \cite{SG16,SG19}, and plays a role analogous to the classical Ramanujan-Serre theta operator $\theta$ in Hida's $p$-adic Rankin-Selberg method for $\GL_2\times\GL_2$. 
For instance, if $g$ and $h$ are elliptic modular forms, then Hida has proved the miraculous identity 
\[\bdse_{\rm ord}(g\cdot\tht^n h)=\bdse_{\rm ord}\Hol(g\cdot\del^nh) \]
for non-negative integers $n$, where $\bdse_{\rm ord}$ is the Hida ordinary projector on $p$-adic modular forms on $\GL(2)$ and $\delta$ is the Maass-Shimura differential operator (see \cite[(2), p.~330]{Hid93}). 
We discover an analogous identity 
\[\frac{1}{\Ome_\frkp^n}\bdse'(\bdTht_\frkp^nf|_H)=\biggl(\frac{2\pi\sqrt{-1}}{\Ome_\infty}\biggl)^n\bdse'\Hol\biggl(\del_{\ulk_0}^nf\biggl(\frac{\partial}{\partial w},\dots,\frac{\partial}{\partial w}\biggl)\biggl|_H\biggl). \]

Our one-variable $p$-adic $L$-function in Theorem \ref{thm:12} roughly interpolates the value $\1_{\bdsg_k'}\bdse'(\bdTht_\frkp^{k-k_3}f|_H)$, where $\bdsg_k'$ is an elliptic modular of weight $k$ in a Hida family of elliptic modular forms. 
In a future work on the $p$-adic IINH formula for $\U(2,1)\times\U(1,1)$ we shall evaluate this one-variable $p$-adic $L$-function at $k=k_3-1$ and relate this value to the $p$-Abel-Jacobi image of the diagonal cycle in the product of a Picard modular surface and a modular curve. 

One may also ask about the relation between our theta operator and the $p$-adic differential operators constructed in \cite{EFMV}. 
The space of differential operators that share homogeneity properties with these operators is necessarily one-dimensional, so the two types of operators differ by a scalar multiple.  This can be determined by relating the Serre-Tate coordinates used at an ordinary point in \cite[\S 2.7.1]{EFMV} with the local coordinates at a boundary point (which also belongs to the ordinary locus).  This comparison also determines the relation between the element $\call_{f_\pi^{\rm ord},\bdsg'} \in \mathrm{Frac}\bfI$, constructed in Theorem 1.4, with the $L(F, \tau)$ defined in \cite[Theorem 6.6]{MHarris}, under a Gorenstein condition, as an element of the localized Hecke algebra.   We hope to clarify this relation in a future article.

\subsection*{Acknowledgement}
Yamana is partially supported by JSPS Grant-in-Aid for Scientific Research (C)23K03055, (B)19H01778, (A)22H00096. 
Harris was partially supported by NSF grants DMS-2001369 and DMS-2302208. 
This work was partly supported by MEXT Promotion of Distinctive Joint Research Center Program JPMXP0723833165 and Osaka Metropolitan University Strategic Research Promotion Project (Development of International Research Hubs). 


\section*{Notation}

We will fix some notation that will be used in what follows. 
Besides the standard symbols $\ZZ$, $\QQ$, $\RR$, $\CC$, $\ZZ_p$, $\QQ_p$ we denote by $\ZZ_{(p)}$ the valuation ring $\ZZ_p\cap\QQ$, by $\RR_+^\times$ and $\QQ^\times_+$ the groups of strictly positive real and rational numbers, and by $\CC^1$ the group of complex numbers of absolute value $1$. 
Given a field $F$, we write $\overline{F}$ for an algebraic closure of $F$. 
We denote the integral closure of $\ZZ_{(p)}$ in $\overline{\QQ}$ by $\overline{\ZZ}_{(p)}$. 
For a ring $R$ we put $R_{(p)}=R\otimes_\ZZ\ZZ_{(p)}$. 

Let $\AA$ be the rational ad\`{e}les ring and $\widehat{\QQ}$ its finite part. 
We denote by $\widehat{\ZZ}=\prod_q\ZZ_q$ the maximal compact subring of $\widehat{\QQ}$, and by $\widehat{\QQ}^{(p)}$ the restricted product $\prod'_q\QQ_q^{}$ over rational primes $q\neq p$.  
Given a place $v$ of $\QQ$, we write $\QQ_v$ for the completion of $\QQ$ with respect to $v$. 
We shall regard $\QQ_v^{}$, $\widehat{\QQ}^{(p)}$ (resp. $\QQ_v^\times$) as subgroups of $\AA$ (resp. $\AA^\times$) in a natural way. 
We denote by the formal symbol $\infty$ the real place of $\QQ$ and do not use $q,l$ for the infinite place. 

Define $\addchar_v:\QQ_v\to\CC^1$ by $\addchar_\infty(z)=e^{2\pi\sqrt{-1}z}$ for $z\in\RR$ and by $\addchar_p(x)=\addchar_\infty(-y)$ for $x\in\QQ_p$ with $y\in\ZZ[p^{-1}]$ such that $y-x\in\ZZ_p$. 
Then $\addchar=\prod_v\addchar_v$ defines a character of $\AA/\QQ$. 
We associate to $m\in\QQ$ the global additive character $\addchar^m$ defined by $\addchar^m(z)=\addchar(mz)$ for $z\in\AA$. 
For $a\in\QQ_q^\times$ we define an additive character $\addchar_q^a$ of $\QQ_q^{}$ by $\addchar_q^a(z)=\addchar_q(az)$ for $z\in\QQ_q$. \index{$\addchar$}

Let $\d z_v$ be the Haar measure on $\QQ_v$ self-dual with respect to the pairing $(z_v^{},z_v')\mapsto\addchar_v(z_v^{}z_v')$. 
Note that $\int_{\ZZ_q}\d z_q=1$ for each rational prime $q$ and that $\d z_\infty$ is the usual Lebesgue measure on $\RR$. 
Let $\d z=\prod_v\d z_v$.  
Then $\d z$ is the Haar measures on $\AA$ such that $\AA/\QQ$ has volume 1. 
Let $\d^\times t=\prod_v\d^\times t_v$ be the Haar measure on $\AA^\times$, where $\d^\times t_v$ is the Haar measure on $\QQ_v^\times$ normalized by $\int_{\ZZ_q^\times}\d^\times t_q=1$ if $v = q <\infty$, and $\d^\times t_\infty=\frac{\d t_\infty}{|t_\infty|}$. 

For positive integers $m,n$ and a commutative ring $R$ we denote $\Mat_{m,n}(R)$ the $R$-module of all $m\times n$-matrices with entries in $R$, by $\GL_m(R)$ the group of all invertible elements of $\Mat_m(R)= \Mat_{m,m}(R)$, by $T_m(R)$ the subgroup of diagonal matrices in $\GL_m(R)$, by $B_m(R)$ the subgroup of upper triangular matrices in $\GL_m(R)$, by $N_m(R)$ the subgroup of upper unipotent triangular matrices in $\GL_m(R)$. 
If $x_1,\dots,x_k$ are square matrices, then $\diag[x_1,\dots,x_k]$ denotes the matrix with $x_1,\dots,x_k$ in the diagonal blocks and 0 in all other blocks.

Let $E$ be an imaginary quadratic field of discriminant $-D_E$ with integer ring $\frko_E$. 
Put $\del=\sqrt{-D_E}$. 
Then $\frkd_E=\del\frko_E$ is the different of $E/\QQ$.  
Set 
\begin{align*}
\EE&=E\otimes_\QQ\AA, & 
E_v&=E\otimes_\QQ\QQ_v, &
\frko_{E,q}&=\frko_E\otimes_\ZZ\ZZ_q. 
\end{align*}
Denote by $\eps_{E/\QQ}=\prod_v\eps_{E_v/\QQ_v}$ the quadratic character of $\QQ^\times\bsl\AA^\times$ associated to $E$. 
We write $x\mapsto x^c$ for complex conjugation on $E$, i.e., the generator of $\Gal(E/\QQ)$, which induces an involution on $R\otimes_\ZZ E$ by $(r\otimes x)^c=r\otimes x^c$ for a ring $R$. 
Let $\trs x\in\Mat_{n,m}(R\otimes_\ZZ E)$ be the transpose of a matrix $x=(x_{ij})\in\Mat_{m,n}(R\otimes_\ZZ E)$ and put $x_{}^c=(x_{ij}^c)$. 
 
Once and for all we fix an odd rational prime $p$ that is split in $E$. 
Fix an algebraic closure $\overline{\QQ}$ of $\QQ$ containing $E$.
We can identify $E$ with the subfield of $\CC$ in exactly two ways. 
We fix an embedding $\iot_\infty:\overline{\QQ}\hookrightarrow\CC$ and an isomorphism $\iot_p:\CC\stackrel{\sim}{\to}\CC_p$, where $\CC_p$ is the completion of an algebraic closure of $\QQ_p$.   
We denote the restriction of $\iot_\infty$ to $E$ by the same symbol and write $\iot_\infty^c$ for the other injection of $E$ into $\CC$. 
We view $\iot_p\circ\iot_\infty$ as a $p$-adic place of $E$ and write $\frkp$ for the corresponding prime ideal of $\frko_E$. 
Let $\ord_p$ be the $p$-adic valuation on $\CC_p$ normalized so that $\ord_pp=1$. 

Given an algebraic number field $L$, we regard $L$ as a subfield in $\CC$ (resp. $\CC_p$) via $\iot_\infty$ (resp. $\iot_p\circ\iot_\infty$). 
We denote by $\frko_L$ the ring of integers of $L$, by $\Gam_L=\Gal(\overline{\QQ}/L)$ the absolute Galois group, and by $L_\ab$ the maximal abelian extension of $L$.
If $X$ is a scheme over $L$, then $\mathrm{R}_{L/\QQ}X$ is the restriction of scalar of $X$ from $L$ to $\QQ$. 

If $R$ is any $\frko_{E,(p)}$-algebra and $M$ is any $R$-module with a commuting $\frko_E$-action, then $M$ becomes an $\frko_E\otimes R$-module and we have a canonical type decomposition $M=M(\vSi)\oplus M(\vSi^c)$, where $M(\vSi)=e_\vSi M$ is the part of $M$ on which $\frko_E$ acts via the canonical homomorphism $\frko_E\hookrightarrow\frko_{E,(p)}\to R$, and $M(\vSi^c)=e_{\vSi^c}M$ is the part on which it acts via the conjugate homomorphism. 
The idempotents $e_\vSi,e_{\vSi^c}\in \frko_E\otimes_\ZZ\frko_E$ are explicitly given in (1.2) of \cite{SG16}. 
Since $p$ is split in $E$, we can view $\ZZ_p$ as an $\frko_{E,(p)}$-algebra via $\iot_p\circ\iot_\infty$. 
For an $\frko_E$-module $M$ we put 
\begin{align*}
M_p&=M\otimes_\ZZ\ZZ_p, & 
M_\frkp&=e_\vSi M_p, &
M_{\frkp^c}&=e_{\vSi^c}M_p. 
\index{$e_\vSi,e_{\vSi^c}$}
\end{align*}



\part{Geometric Theory}\label{part:1}

\section{Shimura varieties for unitary groups}


\subsection{Unitary groups}\label{ssec:21}

Let $r,s$ be non-negative integers such that $r\geq s$. 
Put $n=r+s$ and $t=r-s$. 
When $t\geq 1$, we fix $\gam_0\in\GL_t(E)$ such that 
\begin{enumerate}
\item[(C$_1$)] $-\sqrt{-1}\iot_\infty(\gam_0)$ is a positive definite Hermitian matrix;
\item[(C$_2$)] $\gam_0\in\GL_t(\frko_{E,(p)})$; 
\end{enumerate}
Let $V=E^n$ be an $E$-vector space equipped with a skew Hermitian form 
\begin{align*}
S_{\gam_0}(x,y)&=xS_{\gam_0}\trs y^c, & 
S_{\gam_0}&=\begin{bmatrix} 0 & 0 & -\ono_s \\ 0 & -\gam_0 & 0 \\ \ono_s & 0 & 0 \end{bmatrix}. 
\end{align*}
Define the alternating pairing $\La\;,\;\Ra_{r,s}$\index{$\La\;,\;\Ra_{r,s}$} on $V$ by $\La x,y\Ra_{r,s}=\Tr_{E/\QQ}(S_{\gam_0}(x,y))$. 
Then $\La\;,\;\Ra_{r,s}$ is Hermitian in the sense that $\La ax,y\Ra_{r,s}=\La x,a^c y\Ra_{r,s}$ for all $x,y\in V$ and $a\in E$. 
Let $\calg=\GU(V)=\GU(S_{\gam_0})$ be the unitary similitude group attached to the skew Hermitian space $(V,S_{\gam_0})$. 
As an algebraic group over $\QQ$, the $R$-points of $\calg$\index{$\calg$} are given by 
\[\calg(R)=\{g\in\GL_n(E\otimes_\QQ R)\;|\;gS_{\gam_0}\trs g^c=\nu(g)S_{\gam_0}\text{ for some }\nu(g)\in R^\times\} \]
for a $\QQ$-algebra $R$. 
The morphism $\nu:\calg\to\GG_{m/\QQ}$ is called the similitude character\index{$\nu$}. 
The unitary group $G=\U(V)=\U(S_{\gam_0})$\index{$G$} is defined by
\[G(R)=\{g\in\calg(R)\;|\;\nu(g)=1\}. \]

\begin{remark}\label{rem:20}
If $n$ is even, then $\nu(\calg(\AA))=\AA^\times$. 
If $n$ is odd, then $\nu(\calg(\AA))=\Nr_{E/\QQ}(\EE^\times)$ and hence $\calg(\AA)=\EE^\times G(\AA)$ (cf. Proposition 5.3 of \cite{Shimura97}). 
\end{remark}

Let $\{e_i\}_{i=1}^n$ be the standard basis of $V$. 
Put 
\begin{align*}
W_0&=\oplus_{i=s+1}^rEe_i, & 
Y_E&=\oplus_{i=1}^sEe_i, & 
X_E&=\oplus_{i=r+1}^nEe_i, & 
W&=Y_E\oplus X_E. 
\end{align*}
Then $Y_E$ and $X_E$ are maximal totally isotropic subspaces, which are non-degenerately paired by $S_{\gam_0}$. 
Let $H=\U(J_s)$ and $\calh=\GU(J_s)$, where $J_s=\begin{bmatrix} 0 & \ono_s \\ -\ono_s & 0 \end{bmatrix}$. \index{$\calh$}\index{$H$}
The natural inclusion $\iot:H\stackrel{\sim}{\to}\U(W)\subset G$ is given by 
\beq
\iot\left(\begin{bmatrix} A & B \\ C & D \end{bmatrix}\right)=\begin{bmatrix} A & 0 & B \\ 0 & \ono_t & 0 \\ C & 0 & D \end{bmatrix}, \label{tag:21}
\eeq
where $A,B,C,D$ have size $s$. \index{$\iot$}

Define the alternating pairing $\La\;,\;\Ra_{t,0}$ on $W_0$ by $\La x,y\Ra_{t,0}=-\Tr_{E/\QQ}(x\gam_0\trs y^c)$. 
Let $G_0=\U(W_0)$\index{$G_0$} and $\calg_0=\GU(W_0)$\index{$\calg_0$} be the unitary and unitary similitude groups of the definite skew Hermitian space $(W_0,\gam_0)$. 
Let $\calp=\calm\caln$\index{$\calp$} be stabilizer of $X_E$ in $\calg$, where $\caln$\index{$\calm$}\index{$\caln$} is the unipotent radical of $\calp$ and $\calm=\GL(X_E)\times\calg_0$ is the standard Levi subgroup of $\calp$. 
Here we regard $\GL(X_E)\times\calg_0$ as a subgroup of $\calg$ by the morphism 
\beq
(A,g_0)\mapsto\begin{bmatrix} A & & \\ & g_0 & \\ & & \nu(g_0)(\trs A^{-1})^c\end{bmatrix}. \label{tag:22}
\eeq


\subsection{Lattices and polarization}\label{ssec:22}
 
Fix an $\frko_E$-lattice $L_0$ in $W_0$ on which we impose the following conditions:
\begin{enumerate}
\item[($C_3$)] $L_0$ is integral in the sense that $\La L_0,L_0\Ra_{t,0}\subset\ZZ$;
\item[($C_4$)] $L_0\otimes\ZZ_p=\sum_{i=s+1}^r\frko_{E,p}\;e_i$. \index{$L_0$}
\end{enumerate}
Let $Y=\oplus_{i=1}^s\frko_Ee_i$ and $X^\vee=\oplus_{i=r+1}^n\frko_Ee_i$ be the standard $\frko_E$-lattices in the $E$-vector spaces $Y_E$ and $X_E$. 
Let 
\[M=Y\oplus L_0\oplus X^\vee\] 
be the $\frko_E$-lattice in $V$. 
Then $\La M,M\Ra_{r,s}\subset \ZZ$ and $M_p$ is self-dual with respect to the alternating form $\La\;,\;\Ra_{r,s}$ by (C$_2$). 

An ordered polarization of $M_p=M\otimes_\ZZ\ZZ_p$ is a pair $\Pol_p=\{N^{-1},N^0\}$ of sublattices of $M_p$ which are maximal totally isotropic submodules of $M_p$ and which are dual to each other with respect to $\La\;,\;\Ra_{r,s}$ and satisfy 
\begin{align*}
\rank N^{-1}_\frkp&=\rank N^0_{\frkp^c}=r, & 
\rank N^{-1}_{\frkp^c}&=\rank N^0_\frkp=s. 
\end{align*}
We define the standard polarization $\Pol_p^0=\{M^{-1},M^0\}$ by 
\begin{align*}
M^{-1}&=Y_\frkp\oplus L_{0,\frkp}\oplus Y_{\frkp^c}, &
M^0&=X^\vee_{\frkp^c}\oplus L_{0,\frkp^c}^{}\oplus X^\vee_\frkp,  
\end{align*}
following \cite{MH2}. 
Observe that 
\begin{align*}
M^{-1}_\frkp&=Y_\frkp\oplus L_{0,\frkp}, &
M^{-1}_{\frkp^c}&=Y_{\frkp^c}, &
M^0_\frkp&=X^\vee_\frkp, & 
M^0_{\frkp^c}&=X^\vee_{\frkp^c}\oplus L_{0,\frkp^c}^{}. 
\end{align*}

Given a lattice $L$ in $V$ and $g\in\GL_n(\widehat{E})$, we denote by $Lg$ the lattice in $V$ such that $(Lg)_q=L_qg_q$ for every rational prime $q$. 
Put 
\begin{align*}
K^0&=\{g\in\calg(\widehat{\QQ})\;|\;Mg=M\}, &  
K^0_q&=\{g\in\calg(\QQ_q)\;|\;M_qg=M_q\}. 
\end{align*}
Note that $K^0=\prod_qK^0_q$. \index{$K^0$}


\subsection{Abelian schemes}

An abelian variety $A$ is an algebraic group over a field $k$ whose underlying variety is complete and connected. 
Given a line bundle $\frkL$ on $A$, we define a morphism $\phi_\frkL:A\to A^\vee$ by $\phi_\frkL(x)=T_x^*\frkL\otimes\frkL^{-1}$, where $A^\vee=\Pic^0(A)$ denotes the dual abelian variety, $T_x$ is the translation by $x$ and $T_x^*$ denotes the pullback of the line bundle $\frkL$ by $T_x$. 
A polarization of $A$ is a homomorphism $\lam:A\to A^\vee$ such that $\lam=\phi_\frkL$ for some ample line bundle $\frkL$ on $A$ over $\bar k$. 
The poincar\'{e} bundle $P_A$ is a line bundle on $A\times A^\vee$ with the following properties:
\begin{enumerate}
\item For every $\alp\in A^\vee(\bar k)$ the line bundle $P_A|_{A\times\{\alp\}}$ represents that element of
$\Pic^0(A)(\bar k)$ given by $\alp$; 
\item $P_A|_{\{0\}\times A^\vee}$ is trivial. 
\end{enumerate}
A morphism $f:A\to B$ of abelian varieties gives a dual morphism $f^\vee:B^\vee\to A^\vee$, 
which satisfies $(1\times f^\vee)^*P_A\simeq (f\times 1)^*P_B$. 
Put $\End^0(A)=\End(A)\otimes\QQ$. 
The Rosati involution associated to a polarization $\lam$ is the involution on $\End^0(A)$ defined by $f\mapsto\lam^{-1}\circ f^\vee\circ\lam$. 

An abelian scheme $\scra$ over a base scheme $S$ is a smooth group scheme whose geometric fibres are abelian varieties. 
We denote by $\scra^\vee$ the dual abelian scheme.
A polarization of $\scra$ is a homomorphism $\lam:\scra\to\scra^\vee$ such that for every geometric point $s$ the induced homomorphism $\lam_s:\scra_s\to\scra_s^\vee$ is a polarization of $\scra_s$. 

A semi-abelian variety is a commutative group variety which is the extension of an abelian variety by a connected torus. 
A semi-abelian scheme is a smooth group scheme whose geometric fibres are semi-abelian varieties. 

When $S$ is locally noetherian and the prime number $p$ is invertible on it, we write $T_p(A)$ for the $p$-adic Tate module of $A$. 
When $S$ is a $\ZZ_{(p)}$-scheme, we similarly write $T^{(p)}(A)=\prod_{q\neq p}T_q(A)$. 

The Rosati involution of $\End(A)^0:=\End(A)\otimes\QQ$ induced by $\lam$ sends $f\in\End(A)^0$ to $\lam^{-1}\circ\hat f\circ\lam$. 
The Weil pairing $e_N:A[N]\times A^\vee[N]\to\bdmu_N$ composed with $\lam$ gives an alternating pairing $\La\;,\;\Ra_\lam:T^{(p)}(A)\times T^{(p)}(A)\to\widehat{\ZZ}^{(p)}$ after taking the limit with respect to $N$. 


\subsection{Shimura varieties associated to $\GU(r,s)$}

For a finite set $\Box$ of prime numbers we put 
\begin{align*}
\ZZ_{(\Box)}&=\QQ\cap\bigcap_{q\in\Box}\ZZ_q, &
\widehat{\ZZ}^{(\Box)}&=\prod_{q\notin\Box}\ZZ_q, &
\widehat{\QQ}^{(\Box)}&={\prod_{q\notin\Box}}'\QQ_q.
\end{align*} 
Let $U\subset K^0$ be an open compact subgroup of $\calg(\widehat{\QQ}^{(\Box)})$. 

Put $\scro=\frko_{E,(p)}$. 
An $\scro$-algebra $R$ is called a base ring, and similarly a scheme $S$ over $\Spec\scro$ is called a base scheme. 
Let $S$ be a locally noetherian connected $\scro$-scheme and $\bar s$ a geometric point of $S$. 
An $S$-quadruple $\underline{A}=(A,\iot,\bar\lam,\bar\eta^{(\Box)})_S$ of level $U^{(\Box)}$ consists of the following data:
\begin{itemize}
\item $A$ is an abelian scheme of dimension $r+s$ over $S$, 
\item $\iot:\frko_E\hookrightarrow\End_S(A)\otimes_\ZZ\ZZ_{(\Box)}$ is an embedding of $\ZZ_{(\Box)}$-algebra, 
\item $\lam:A\to A^\vee$ is a prime-to-$\Box$ polarization over $S$ and $\bar\lam$ the $\ZZ_{(\Box),+}$-orbit of $\lam$, where $\ZZ_{(\Box),+}=\ZZ_{(\Box)}\cap\QQ_+^\times$, 
\item $\bar\eta^{(\Box)}=\eta^{(\Box)}U$ is a $\pi_1(S,\bar s)$-invariant $U^{(\Box)}$-orbit of isomorphisms of $\frko_E$-modules $\eta^{(\Box)}:M\otimes\widehat{\ZZ}^{(\Box)}\stackrel{\sim}{\to} T^{(\Box)}(A_{\bar s}):=H_1(A_{\bar s},\widehat{\ZZ}^{(\Box)})$. 
\end{itemize}
Here we define $\eta^{(\Box)}u(x)=\eta^{(\Box)}(xu^{-1}\nu(u))$ for $u\in U$. 

In addition, the $S$-quadruple $\underline{A}$ satisfies the following conditions:
\begin{enumerate}
\item[(K$_1$)] The Rosati involution induced by $\lam$ is given by $\iot(a)\mapsto\iot(a^c)$, 
\item[(K$_2$)] $\eta^{(\Box)}$ takes the pairing $\La\;,\;\Ra_{r,s}$ on $M$ to a $\widehat{\QQ}^{(\Box)\times}$-multiple of $\La\;,\;\Ra_\lam$, 
\item[(K$_3$)] The determinant condition: for $b\in\frko_E$ we have 
\[\det(X-\iot(b)|_{\Lie A})=(X-\iot_\infty^c(b))^r(X-\iot_\infty^{}(b))^s\in\calo_S[X]. \]
\end{enumerate}
Two $S$-quadruples $\underline{A}=(A,\iot,\lam,\bar\eta^{(\Box)})_S$ and $\underline{A}'=(A',\iot',\lam',\bar\eta^{\prime(\Box)})_S$ are equivalent
if there is an isomorphism $\phi:A\to A'$ such that 
\begin{align*}
\phi\circ\iot&=\iot'\circ\phi, &
\phi^\vee\circ\bar\lam'\circ\phi&=\bar\lam, & 
\bar\eta^{\prime(\Box)}&=\phi\circ\bar\eta^{(\Box)}. 
\end{align*}
See Section 1.4 of \cite{Lan13} for the comparison with another moduli problem. 

Assume that $U$ is neat in the sense of Definition 1.4.1.8 of \cite{Lan13}. 
Let $\frkS_U^{(\Box)}(S)$ be the collection of equivalence classes of $S$-quadruple $\underline{A}=(A,\iot,\lam,\bar\eta^{(\Box)})_S$ of level $U^{(\Box)}$. 
Then $\frkS_U^{(\Box)}$ becomes a contravariant functor from the category of locally Noetherian $\scro$-schemes to the category of sets. 

The reflex field associated to the Shimura datum is contained in $E$ as $E/\QQ$ is normal. 
If $\Box=\emptyset$, then the functor $S\mapsto \frkS_U^{(\emptyset)}(S)$ is represented by a quasi-projective scheme $S_\calg(U)_{/E}$ over $E$ by the theory of Shimura and Deligne. 
If $\Box=\{p\}$ and $U=K^0_p\times U^{(p)}$, where $U^{(p)}$ is an open compact subgroup of $\calg(\widehat{\QQ}^{(p)})$, then $\frkS_U^{(p)}$ is representable by a quasi-projective scheme $S_\calg(U)_{/\scro}$ over $\scro$ by the work \cite{Kottwitz92} of Kottwitz. \index{$S_\calg(U)$}

Fix a smooth toroidal compactification $\overline{S}_\calg(U)_{/\scro}$ of $S_\calg(U)_{/\scro}$, which is a proper smooth scheme over $\scro$ containing $S_\calg(U)_{/\scro}$ as an open dense subscheme. 
The abelian scheme $\scra/S_\calg(U)$ extends with the $\frko_E$-action to a semi-abelian scheme $\scrg$ over the toroidal compactification $\overline{S}_\calg(U)$ by Theorem 6.4.1.1 of \cite{Lan13}. 
In addition, the universal quadruple $\underline{\scra}$ of level $U^{(p)}$ over $S_\calg(U)$ extends to a quadruple $\underline{\scrg}=(\scrg,\iot,\lam,\eta)$ over $\overline{S}_\calg(U)$. 


\subsection{Igusa schemes associated to $\GU(r,s)$}
Define the isomorphism $\imath_\frkp:G(\QQ_p)\to\GL_n(\QQ_p)$ by $\imath_\frkp(g)=g|_{V_\frkp}$. \index{$\imath_\frkp$}
We write $\imath_\frkp(g)=\begin{bmatrix} A & B \\ C & D \end{bmatrix}$ according to the standard filtration of  $M_\frkp=M_\frkp^{-1}\oplus M_\frkp^0$ defined in \cite[(1.6)]{MH}. 
Let $U=K^0_p\times U^{(p)}$ be a neat open compact subgroup of $\calg(\widehat{\QQ})$. 
For positive integers $\ell$ we define 
\begin{align*}
U^\ell&=\biggl\{u\in U\;\biggl|\;\imath_\frkp(u_p)\equiv\begin{bmatrix} \ono_r & * \\ 0 & \ono_s \end{bmatrix}\pmod{p^\ell}\biggl\}, \index{$U^\ell$}\\
U^\ell_1&=\{u\in U\;|\;\imath_\frkp(u_p)\in N_n(\ZZ_p)\pmod{p^\ell}\}. \index{$U^\ell_1$}
\end{align*} 

We introduce a notion of the Igusa scheme, which is a model of $S_\calg(U^\ell)$ over $\scro$ (see \cite[Lemma 2.3]{MH}). 
Let $\Pol_p=\{N^{-1},N^0\}$ be an ordered polarization of $M_p$. 
An $S$-quintuple $(\underline{A},j)_S=(A,\iot,\lam,\bar\eta^{(p)},j)_S$ of level $K^\ell$ with respect to $\Pol_p$ consists of an $S$-quadruple $\underline{A}$ of level $K^{(p)}$ and a level-$p^\ell$ structure $j$ of $A$ with respect to $\Pol_p$, which is a monomorphism $j:\bdmu_{p^\ell}\otimes_\ZZ N^0\hookrightarrow A[p^\ell]$ as $\frko_E$-group schemes over $S$. 

Two $S$-quintuples $(\underline{A},j)_S$ and $(\underline{A}',j')_S$ are equivalent if there is an isomorphism $\phi:\underline{A}\stackrel{\sim}{\to}\underline{A}'$ such that $\phi\circ j=j'$. 
Let $\frkI_{K,\ell,\Pol_p}^{(p)}(S)$ be the set of equivalence classes of $S$-quintuple $\underline{A}=(A,\iot,\lam,\bar\eta^{(p)})_S$ of level $U^\ell$ with respect to $\Pol_p$. 
The functor $S\mapsto\frkI_{U,\ell,\Pol_p}^{(p)}(S)$ is relatively representable over $S_\calg(U)_{/\scro}$. 
Recall the standard polarization $\Pol_p^0=\{M^{-1},M^0\}$ of $M_p$.  
Let $I^0_\calg(U^\ell)_{/\scro}$ be the scheme representing $\frkI_{K,\ell,\Pol_p^0}^{(p)}$ over $\scro$. \index{$I^0_\calg(U^\ell)$}


\begin{definition}
Put \[\bfH=\GL(M^0,\frko_{E,p})
=\GL_r(\ZZ_p)\times\GL_s(\ZZ_p). \index{$\bfH$}\] 
The group $\bfH$ acts on $I^0_\calg(U^\ell)$ over $\overline{S}_\calg(U)$ by $h\cdot j(m)=j(mh)$. 
Let 
\begin{align*}
\bfN&=N_r(\ZZ_p)\times N_s(\ZZ_p), & 
\bfT&=T_r(\ZZ_p)\times T_s(\ZZ_p) \index{$\bfN$}\index{$\bfT$}
\end{align*}
be subgroups of $\bfH$ with respect to the $p$-adic basis of $M^0$ defined in \S \ref{ssec:22}.  
Define $I^0_\calg(U^\ell_1)=I^0_\calg(U^\ell)/\bfN$. 
\end{definition}


\subsection{Hermitian symmetric domains}\label{ssec:26}
 
We hereafter suppose that $s\geq 1$.  
We follow the notation of \cite{Shimura97,MH2}. 
The Hermitian symmetric domain associated to $\calg(\RR)$ is given by 
\[\frkD_{r,s}=\biggl\{Z=\begin{bmatrix} \tau \\ w \end{bmatrix}\in\Mat_{r,s}(\CC)\biggl|\;\tau\in\Mat_s(\CC),\;w\in\Mat_{t,s}(\CC),\;\eta(Z)>0\biggl\}, \index{$\frkD_{r,s}$}\]
where 
\[\eta(Z)=\sqrt{-1}(\trs\tau^c-\tau-\trs w^c\gam_0^{-1}w). \]
Let $\calg(\RR)^+$ be the subgroup of $\calg(\RR)$ consisting of elements whose similitude factor is positive. 
The action of $\calg(\RR)^+$ on $\frkD_{r,s}$ is given by 
\begin{align*}
\alp(Z)&=\begin{bmatrix} (a\tau+bw+c)(h\tau+lw+d)^{-1} \\ (g\tau+ew+f)(h\tau+lw+d)^{-1} \end{bmatrix}, &
\alp&=\begin{bmatrix} a & b & c \\ g & e & f \\ h & l & d \end{bmatrix}\in\calg(\RR)^+
\end{align*}
with $a,d\in\Mat_s(\CC)$ and $e\in\Mat_t(\CC)$. 
The automorphy factors are defined by 
\begin{align*}
\lam(\alp,Z)&=\begin{bmatrix} h^c\trs\tau+d^c & h^c\trs w-\gam_0^cl^c \\ -(\gam_0^{-1}g)^c\trs\tau-(\gam_0^{-1}f)^c & -(\gam_0^{-1}g)^c\trs w+(\gam_0^{-1}e\gam_0)^c\end{bmatrix}, \\
\mu(\alp,Z)&=h\tau+lw+d, \qquad\qquad
J(\alp,Z)=(\lam(\alp,Z),\mu(\alp,Z)). \index{$\lam(\alp,Z)$}\index{$\mu(\alp,Z)$}\index{$J(\alp,Z)$}
\end{align*}

\begin{remark}
If $r+s$ is odd, then $\calg(\RR)^+=\calg(\RR)$ (see Remark \ref{rem:20}). 
\end{remark}

We define a maximal compact subgroup $\GL_m(\CC)$ by
\[\U_T=\{g\in\GL_m(\CC)\;|\;\trs gTg^c=T\} \]
for a positive definite Hermitian matrix $T$ of size $m$. 
Following \cite{Shimura97,Shimura00}, we define the origin of $\frkD_{r,s}$ and the standard maximal compact subgroup of $G(\RR)$ by
\begin{align*}
\bfi&=\begin{bmatrix} \sqrt{-1}\ono_s \\ \oo_{t,s}\end{bmatrix}\in\frkD_{r,s}, &
\calk_\infty&=\{\alp\in G(\RR)\;|\;\alp(\bfi)=\bfi\}. 
\end{align*}\index{$\bfi$}\index{$\calk_\infty$}
Put $T'=\begin{bmatrix} \ono_s & \\ & \frac{\sqrt{-1}}{2}\gam_0^c\end{bmatrix}$. \index{$T'$}
The isomorphism $\calk_\infty\simeq\U_{T'}\times\U_{{\bf 1}_s}$ is given by 
\beq
\alp\mapsto J^-(\alp,\bfi):=(\overline{\lam(\alp,\bfi)},\mu(\alp,\bfi)) \label{tag:23}
\eeq 
by (\ref{tag:a5}). \index{$J^-(\alp,\bfi)$}
Recall that 
\beq
\det\alp\cdot \det\lam(\alp,Z)=\det\mu(\alp,Z) \label{tag:24}
\eeq
(see (6.3.9) of \cite{Shimura97}). 


\subsection{Complex uniformization}\label{sec:uniformization}

Put $\calg(\QQ)^+=\calg(\QQ)\cap\calg(\RR)^+$. 
Let $U$ be a neat open compact subgroup of $\calg(\widehat{\QQ})$. 
We denote the point represented by $(Z,g)\in\frkD_{r,s}\times \calg(\widehat{\QQ})$ in the quotient 
\[M_\calg(\frkD_{r,s},U)=\calg(\QQ)^+\bsl\frkD_{r,s}\times \calg(\widehat{\QQ})/U, \]
which is a complex manifold, by $[Z,g]\in M_\calg(\frkD_{r,s},U)$. 

We now recall the construction of the universal quadruple $\underline{\scra}$ of level $U^{(p)}$ over $S_\calg(U)/\CC$. 
Define the $E\otimes_\QQ\CC$-module $\CC^{r,s}:=\CC^r\oplus\CC^s$ so that $e_+(\CC^{r,s})=\CC^s$ and $e_-(\CC^{r,s})=\CC^r$, where $\CC^r$ and $\CC^s$ are regarded as spaces of row vectors. 
We define the map $c_{r,s}:\CC^{r,s}\to\CC^{r,s}$ by
\begin{align*}
(u,u')c_{r,s}&=(u^c,u') &(u&\in\CC^r,\;u'\in\CC^s), 
\end{align*}
where $u^c=(u_1^c,\dots,u_r^c)$ is the complex conjugation of $u=(u_1,\dots,u_r)\in\CC^r$. 

The embedding $\iot_\infty$ induces an isomorphism $E\otimes_\QQ\RR\simeq\CC$ by which we regard $V\otimes_\QQ\RR$ as a $\CC$-vector space of row vectors according to
the $E$-basis $\{e_i\}_{i=1}^n$. 
Assuming that $s>0$, we put 
\[B(Z)=\begin{bmatrix} \trs\tau^c & \trs w^c & \tau \\ 0 & \gam_0 & w  \\ \ono_s & 0 & \ono_s\end{bmatrix}. \index{$B(Z)$}\]
For each $Z\in\frkD_{r,s}$ we define the map $p_Z:V\otimes_\QQ\RR\simeq\CC^{r,s}$ by $p_Z(v)=vB(Z)c_{r,s}$. 
Then $V$ acts on $(Z,u)\in\frkD_{r,s}\times\CC^{r,s}$ by 
\[v\cdot(Z,u)=(Z,p_Z(v)+u). \]
We define a left action of $\calg$ on $V$ by $g*v:=vg^{-1}\nu(g)$. 
Put 
\begin{align*}
M_{[g]}&=g*M, & 
M_{[g]}(Z)&=p_Z(M_{[g]}). 
\end{align*}
We associate to each point $(Z,g)\in\frkD_{r,s}\times \calg(\widehat{\QQ})$ a $\CC$-quadruple $\underline{\cala}(V)_g(Z)=(\scra(V)_g(Z),\bar\lam_Z,\iot_\CC,\bar\eta_g)$ of level $U$ in the following way: 
\begin{itemize}
\item the abelian variety $\scra(V)_g(Z):=\CC^{r,s}/M_{[g]}(Z)\simeq V\otimes_\QQ\RR/M_{[g]}$; 
\item $\bar\lam_Z$ is the $\QQ^\times_+$-orbit of the polarization induced by the unique Riemann form $\La\;,\;\Ra_Z$ on $\CC^{r,s}$ such that $\La p_Z(v),p_Z(v')\Ra_Z=\La v,v'\Ra_{r,s}$ for $v,v'\in V$; 
\item $\iot_\CC:\frko_E\hookrightarrow\End_\CC\scra(V)_g(Z)$ is induced by the $\frko_E$-action on $V$ via $p_Z$;
\item the level structure $\eta_g:M\otimes\widehat{\ZZ}\stackrel{\sim}{\to}M_{[g]}\otimes\widehat{\ZZ}^{(p)}=H_1(\scra(V)_g(Z),\widehat{\ZZ})$ is defined by $\eta_g(x)=g*x$ for $x\in M$. 
\end{itemize}

Since 
\beq
p_{\alp(Z)}(v)=p_Z(v\alp) J(\alp,Z)^{-1}=p_Z(\alp^{-1}*v)\nu(\alp)J(\alp,Z)^{-1} \label{tag:25}
\eeq
for $\alp\in \calg(\RR)^+$ by (\ref{tag:a4}), we see that 
\[\underline{\cala}(V)_{\gam g}(\gam(Z))\sim\underline{\cala}(V)_g(Z)\]
for $\gam\in\calg(\QQ)^+$. 
Actually, $\underline{\cala}(V)_g(Z)$ is equivalent to $\underline{\cala}(V)_{g'}(Z')$ if and only if $[g,Z]=[g',Z']$. 
Moreover, a $\CC$-quadruple is equivalent to $\underline{\cala}(V)_g(Z)$ for some $(Z,g)$ (see Theorem 4.8 of \cite{Shimura00}). 
Therefore the set of complex points of $S_\calg(U)$ is identified with $M_\calg(\frkD_{r,s},U)$.

We define the prime-to-$p$ level structure $\eta_g^{(p)}$ similarly. 
As for the level structure at $p$ we fix a primitive $p^\ell$th root $\zet=e^\frac{2\pi\sqrt{-1}}{p^\ell}$ and define
\[j_\zet:M^0\otimes\bdmu_{p^\ell}\simeq M^0\otimes\ZZ/p^\ell\ZZ\hookrightarrow\scra(V)_g(Z)[p^\ell]=M_{[g]}\otimes\ZZ/p^\ell\ZZ\]
by $j_\zet(x^0)=g*x^0$ for $x^0\in M^0$, where $\zet$ is used to pin down an isomorphism $\zet^{-1}:\bdmu_{p^\ell}\simeq\ZZ/p^\ell\ZZ$. 


\subsection{Igusa schemes associated to $\U(r,s)$}

We consider a subvariety associated to an open subgroup of $\calg(\widehat{\QQ})$ containing $G(\widehat{\QQ})$.  
We follow the notion introduced in Section 2.5 of \cite{MH} (cf. Appendix A.1 of \cite{MHarris}). 
Let $\Cl_\QQ^+(U)$ be a set of representatives of the group $\QQ^\times_+\bsl\widehat{\QQ}^\times/\nu(U)$ in $\widehat{\QQ}^{(p)\times}$. 
For $\bfc\in\Cl_\QQ^+(U)$ we consider the following condition:
\beq 
\text{$\eta^{(p)}$ takes the pairing $\La\;,\;\Ra_{r,s}$ on $M$ to $u\La\;,\;\Ra_\lam$ for some $u\in\bfc\nu(U)$.} \tag{K$_{2,\bfc}$}
\eeq 
We consider the functor 
\[\frkI_{U,\ell,\Pol_p^0,\bfc}^{(p)}(S)=\biggl\{(A,\iot,\lam,\bar\eta^{(p)},j)_S\;\biggl|\;
\begin{matrix} (A,\iot,\lam,\bar\eta^{(p)})_S\in\frkS_U^{(p)}(S) \\ 
\lam\text{ and }\eta^{(p)}\text{ satisfy (K$_{2,\bfc}$)}
\end{matrix}\biggl\}/\sim\] 
It is proved in \cite[p.~136]{Hid04} that the equivalence class of $(A,\iot,\lam,\bar\eta^{(p)},j)_S$ is independent of the choice of $\lam$ in $\bar\lam$. 
Then $\frkI_{U,\ell,\Pol_p^0,\bfc}^{(p)}$ is represented by a scheme over $\frko_{E_U,(p)}$, which we denote by $I^0_G(U,\bfc)$. 
The set of complex points of $I^0_G(U,\bfc)$ is given by $M_G(\frkD_{r,s},{^\bfc U})$, where we write $^\bfc U=g_\bfc^{} Ug_\bfc^{-1}\cap G(\widehat{\QQ})$, taking an element $g_\bfc\in\calg(\widehat{\QQ})$ such that $\nu(g_\bfc)=\bfc$. 
As explained in \S 4.2.1 of \cite{Hid04} for Hilbert modular varieties, we have 
\[S_\calg(U)=\bigsqcup_{\bfc\in\Cl_\QQ^+(U)}S_G(U,\bfc). \] 
We will write $I^0_G(U)$ for $I^0_G(U,\bfc)$ if $\bfc=1$. 


\subsection{Morphisms between Igusa schemes}\label{ssec:pullback}

Put $L_1=Y\oplus X^\vee$. 
Recall the decompositions 
\begin{align*}
(V,\La\;,\;\Ra_{r,s})&=(W,\La\;,\;\Ra_{s,s})\oplus (W_0,\La\;,\;\Ra_{t,0}), &
M&=L_1\oplus L_0.  
\end{align*}
Define the hyperspecial open compact subgroups of $\calg_0(\QQ_p)$ and $H(\QQ_p)$ by   
\begin{align*}
U^{L_0}_p&=\{g\in\calg_0(\QQ_p)\;|\;(L_0\otimes\ZZ_p)g=L_0\otimes\ZZ_p\}, \\
U^{L_1}_p&=\{h\in H(\QQ_p)\;|\;(L_1\otimes\ZZ_p)h=L_1\otimes\ZZ_p\}. 
\end{align*}
Define the subgroup $\calh'$ of $\calh\times\calg_0$ by 
\[\calh'=\{(h,g)\in\calh\times\calg_0\;|\;\nu(h)=\nu(g)\}. \]
We naturally view $\calh'$ as a subgroup of $\calg$. 
When we view $H$ and $\calg_0$ as subgroups of $\calg$ by the morphisms (\ref{tag:21}) and (\ref{tag:22}), the subgroup $\calh'$ is identified with $H\calg_0$. 
Take open compact subgroups $U_{L_0}^{}=U^{L_0}_p\times U_{L_0}^{(p)}$ of $\calg_0(\widehat{\QQ})$ and $U_{L_1}=U^{L_1}_p\times U_{L_1}^{(p)}$ of $H(\widehat{\QQ})$ so that $U_{L_1}U_{L_0}\subset U$. 
For simplicity we let $t=1$.
Define an open compact subgroup $\calu_{L_1}$ of $\calh(\widehat{\QQ})$ by 
\[\calu_{L_1}=\biggl\{\begin{bmatrix} \ono_s & \\ & u\ono_s\end{bmatrix}k\;\biggl|\;u\in\nu(U_{L_0}),\;k\in U_{L_1}\biggl\}. \]

Then we have a natural morphism 
\[i_{W,W_0}:I_H^0(\calu_{L_1}^\ell)_{/\scro}\times I_{G_0}^0(U_{L_0}^\ell)_{/\scro}\to I_G^0(U^\ell)_{/\scro}\]
(see (2.12) of \cite{MH}).  
Fix $\underline{E}_0^{}=(E_0^{},\iot_0^{},\lam_0^{},\bar\eta_0^{(p)},j_0^{})_S\in I^0_{G_0}(U_{L_0})$. 
We restrict $i_{W,W_0}$ to obtain a morphism $\iot_W:I^0_H(\calu_{L_1}^\ell)_{/\scro}\to I^0_G(U^\ell)_{/\scro}$, which is given by 
\[\iot_W([(B,\iot,\lam,\eta^{(p)}\calu_{L_1}^\ell,j)_S])=[(B\times E_0^{},\iot\times\iot_0^{},\lam\times\lam_0^{},(\eta_{}^{(p)}\times\eta_0^{(p)})U^\ell,j\times j_0)_S]. \index{$\iot_W$}\]

We give an explicit expression of the morphism $\iot_W$ in terms of the complex uniformization constructed in \S \ref{sec:uniformization}.  
Let $E_0$ be an elliptic curve with CM by $E$ together with a complex uniformization $E_0(\CC)\stackrel{\sim}{\to}\CC(\vSi^c)/(L_0\gam_0)^c$ by which we define the level structure $\eta_0:L_0\otimes\widehat{\ZZ}\stackrel{\sim}{\to}H_1(E_0^{},\widehat{\ZZ})$. 
As is well-known, $E_0^{}$ is defined over $\overline{\QQ}$, and extends to an abelian scheme over $\overline{\ZZ}_{(p)}$. 

Recall the embedding $\iot:H\hookrightarrow G$ defined in (\ref{tag:21}). 
Put $\scrh_\calg=\GL_r\times\GL_s$. 
Define the embeddings $\jmath:\frkD_{s,s}\hookrightarrow\frkD_{s+1,s}$ and $\iot':\scrh_\calh\hookrightarrow\scrh_\calg$ by 
\begin{align*}
\jmath(\tau)&=\begin{bmatrix} \tau \\ \oo \end{bmatrix}, &
\iot'(A,B)&=\biggl(\begin{bmatrix} A & \\ & 1 \end{bmatrix},B\biggl),  
\end{align*}
where $A,B$ have size $s$. \index{$\iot'$}\index{$\jmath$}
Then for $h\in H(\RR)$
\begin{align}
\iot(h)\jmath(\tau)&=\jmath(h\tau), &
J(\iot(h),\jmath(\tau))&=\iot'(J(h,\tau)). \label{tag:26}
\end{align}
Since 
\begin{align*}
\scra(V)_{\iot(h)}(\jmath(\tau))&=\scra(W)_h^{}(\tau)\times E_0^{}, & 
\lam_{\jmath(\tau)}&=\lam_\tau^{}\times\lam_0^{}, &
\eta_{\iot(h)}&=\eta_h\times\eta_0 
\end{align*}
for $\tau\in\frkD_{s,s}$ and $h\in H(\widehat{\QQ})$, we find that 
\beq
\iot_W(\underline{\cala}(W)_h(\tau))=\underline{\cala}(V)_{\iot(h)}(\jmath(\tau)). \label{tag:27}
\eeq


\section{Modular forms on unitary groups}


\subsection{Geometric modular forms}

We recall the theory of geometric modular forms and $p$-adic modular forms on unitary groups from the viewpoint of Katz and Hida. 
In order to introduce modular forms of non-scalar weight, we consider the algebraic group $\scrh=\GL_r\times\GL_s$ over $\ZZ$. \index{$\scrh$}

The Hodge bundle $\ome=\ome_{\scrg/\overline{S}_\calg(U)}:=e^*\Ome_{\scrg/\overline{S}_\calg(U)}$ is the pull-back of the relative cotangent sheaf $\Ome_{\scrg/\overline{S}_\calg(U)}$ of the universal semi-abelian scheme $\scrg$ over $\overline{S}_\calg(U)$ via the zero section $e$ of $\scrg/\overline{S}_\calg(U)$. 
Then $\ome$ is a locally free coherent $\calo_{\overline{S}_\calg(U)}$-module. 
The structure morphism $\frko_E\to\calo_{\overline{S}_\calg(U)}$ makes $\ome$ into an $\frko_E$-module. 
The embedding $\iot$ makes the relative tangent sheaf $\frkt_{\scrg/\overline{S}_\calg(U)}$ into an $\frko_E$-module. 
The action of $\frko_E$ on $\bdome\in\ome$ is defined by $(a\bdome)(l)=\bdome(\iot(a^c)l)$ for $a\in\frko_E$ and $l\in e^*\frkt_{\scrg/\overline{S}_\calg(U)}$. 

We can decompose $\ome=\ome(\vSi)\oplus\ome(\vSi^c)$.  
Define 
\begin{align*}
\cale_{\underline{\scrg}}^+&:=\Isom(\calo^r_{\overline{S}_\calg(U)},\ome(\vSi)), &
\cale_{\underline{\scrg}}^-&:=\Isom(\calo^s_{\overline{S}_\calg(U)},\ome(\vSi^c)). 
\end{align*}
Then $\cale_{\underline{\scrg}}^{}:=\cale_{\underline{\scrg}}^+\times\cale_{\underline{\scrg}}^-$ is an $\scrh$-torsor over $\overline{S}_\calg(U)$ with the structure map $\bar\pi:\cale_{\underline{\scrg}}\to\overline{S}_\calg(U)$. 
Let $S=\Spec R$ be a base scheme. 
An $R$-point in $\cale_{\underline{\scrg}}$ is a pair $(\underline{A},\bdome)$ which consists of an $R$-point $\underline{A}=(A,\iot,\lam,\bar\eta^{(p)})_S$ in $\overline{S}_\calg(U)$ and $\bdome=(\bdome^+;\bdome^-)$ with  
\begin{align*} 
\bdome^+&\in\cale_{\underline{A}}^+=\mathrm{Isom}(e_\vSi(\frko_E\otimes R^r),\ome_{A/R}(\vSi)), \\ 
\bdome^-&\in\cale_{\underline{A}}^-=\mathrm{Isom}(e_{\vSi^c}(\frko_E\otimes R^s),\ome_{A/R}(\vSi^c)). 
\end{align*}
We will write 
\begin{align*}
\bdome_i^+&=\bdome^+(e_i)\quad (1\leq i\leq r), & 
\bdome_i^-&=\bdome^-(e_i)\quad (r+1\leq i\leq r+s). 
\end{align*}
Giving an element $\bdome^+$ (resp. $\bdome^-$) is equivalent to giving an ordered basis $\{\bdome_i^+\}$ of $\ome_{A/R}(\vSi^c)$ (resp. $\{\bdome_i^-\}$ of $\ome_{A/R}(\vSi)$). 
The action of $h=(h^+,h^-)\in \scrh$ on $\bdome=(\bdome^+;\bdome^-)$ is given by $h\bdome(v^+,v^-)=(\bdome^+(v^+h^+),\bdome^-(v^-h^-))$. 

Let $(\rho,\call)$ be an algebraic representation of $\scrh$ defined over $\scro$.

\begin{definition} 
A geometric modular form $f$ of weight $\rho$ and level $U$ defined over $\scro$ is a rule, defined for all $\scro$-algebras $R$, which associates to a pair $(\underline{A},j,\bdome)$ of an $R$-quintuple $(\underline{A},j)$ together with $\bdome\in\cale_{\underline{A}}$ an element $f(\underline{A},j,\bdome)\in \call(R)$, satisfying the following conditions: 
\begin{enumerate}
\item[(G$_1$)] $f(\underline{A},j,\bdome)$ depends only on the $R$-isomorphism class of $(\underline{A},j,\bdome)$; 
\item[(G$_2$)] $f(\underline{A},j,h\bdome)=\rho(\trs h^{-1})f(\underline{A},j,\bdome)$ for every $h\in\scrh(R)$;
\item[(G$_3$)] $f(\underline{A}_{/R'},j_{/R'},\bdome_{/R'})=\vph(f(\underline{A},j,\bdome))$ for every $\scro$-algebra homomorphism $\vph:R\to R'$; 
\item[(G$_4$)] if $r=s=1$, then 
\[f(\frko_E\otimes\mathrm{Tate}(q),\iot_\can^{},\frko_E\otimes\lam_\can^{},\frko_E\otimes\eta_\can^{(p)},\frko_E\otimes j_\can^{},\frko_E\otimes\bdome_\can^{})\in R\powerseries{q}, \]
where $(\mathrm{Tate}(q),\lam_\can,\eta_\can^{(p)},j_\can,\bdome_\can)$ is the Tate elliptic curve $\GG_m/q^\ZZ$ with the canonical level structure $\eta_\can^{(p)}$ and the canonical differential $\bdome_\can$ over $\ZZ(\!(q)\!)$. 
\end{enumerate} 
\end{definition}

We sometimes write $f(\underline{A},j,\{\bdome^+_i\};\{\bdome^-_i\})$ instead of $f(\underline{A},j,\bdome)$.
We define 
\[\ome_\rho=\cale_{\underline{\scrg}}\times^\rho\call:=\cale_{\underline{\scrg}}\times\call/\sim, \]
where the relation $\sim$ is given by $(\bdome,l)\sim(h\bdome,\rho(\trs h^{-1})l)$ for $h\in\scrh$. 
For a base ring $R$ flat over $\ZZ_{(p)}$ we set 
\[\bfM_\rho^\calg(U^\ell_1,R):=\rmH^0(I^0_\calg(U^\ell_1)_{/R},\ome_\rho). \index{$\bfM_\rho^\calg(U^\ell_1,R)$}\]
A geometric modular form associates to every $\scro$-algebra $R$ an element of $\bfM_\rho^\calg(U^\ell_1,R)$, which is compatible with base change. 
If $R$ is a field, then  
\[\bfM_\rho^\calg(U^\ell_1,R)=\rmH^0(\overline{S}_\calg(U^\ell_1)_{/R},\ome_\rho)\] 
by the Koecher principle unless $r=s=1$ (cf. \cite[Proposition 4.5]{Shimura97}). 

Let $\cald$ be the Cartier divisor $\overline{S}_\calg(U^\ell_1)-S_\calg(U^\ell_1)$, equipped with its structure of reduced closed subscheme. 
The $R$-module of cusp forms on $\calg$ over $R$ is the submodule
\[\bfS_\rho^\calg(U^\ell_1,R):=\rmH^0(I^0_\calg(U^\ell_1)_{/R},\ome_\rho(-\cald)). \index{$\bfS_\rho^\calg(U^\ell_1,R)$}\]
The diagonal torus $\bfT$ of $\bfH$ acts on $\bfM_\rho^\calg(U^\ell_1,R)$ by 
\[[t].f(\underline{A},j,\bdome)=f(\underline{A},t\cdot j,\bdome). \]


\subsection{Classical modular forms on $\GU(r,s)$}\label{ssec:23}

Let $\Her^+_s(\QQ)$ (resp. $\Her^0_s(\QQ)$) be the set of positive definite (resp. positive semi-definite) Hermitian matrices of size $s$ over $E$. 

\begin{definition}\label{def:21}
Let $U$ be an open compact subgroup of $\calg(\widehat{\QQ})$. 
We call a function $f:\frkD_{r,s}\times \calg(\widehat{\QQ})\to\call(\CC)$ a modular form on $\calg$ of weight $\rho$ and level $U$ if it is holomorphic in $Z$ and satisfies 
\begin{align*}
f(\gam Z,\gam g u)&=\rho(\nu(\gam)^{-1}J(\gam,Z))f(Z,g) &
(\gam&\in \calg(\QQ)^+,\; g\in \calg(\widehat{\QQ}),\;u\in U), 
\end{align*}
and has a Fourier-Jacobi expansion 
\[f\biggl(\begin{bmatrix} \tau \\ w \end{bmatrix},\bet\biggl)=\sum_{B\in\Her_s^0(\QQ)}\overrightarrow{\FJ}^B_\bet(\trs w,f)e^{2\pi\sqrt{-1}\tr(B\tau)}. \]
A modular form $f$ is called a cusp form if $\overrightarrow{\FJ}^B_\bet(\trs w,f)=0$ unless $B\in\Her^+_s(\QQ)$. 
Let $M_\rho^\calg(U,\CC)$ (resp. $S_\rho^\calg(U,\CC)$) be the space of modular (resp. cusp) forms on $\calg$ of weight $\rho$ and level $U$. \index{$M_\rho^\calg(U,\CC),S_\rho^\calg(U,\CC)$}
\end{definition}

For $g\in \calg(\widehat{\QQ})$ and a function $\calf$ on $\frkD_{r,s}\times \calg(\widehat{\QQ})$ we set 
\[[r(g)\calf](Z,\bet)=\calf(Z,\bet g). \]
For a character $\chi$ of $U$ we set
\[S_\rho^\calg(U,\chi,\CC)=\{f\in S_\rho^\calg(U,\CC)\;|\;r(u)f=\chi(u)^{-1}f\text{ for }u\in U\}. \index{$S_\rho^\calg(U,\chi,\CC)$}\]

\begin{remark}\label{rem:22}
Put $\Gam^\bet=\calg(\QQ)^+\cap \bet U\bet^{-1}$ for $\bet\in\calg(\widehat{\QQ})$. 
Let $f\in M_\rho^\calg(U,\CC)$. 
Define a function $f_\bet$ on $\frkD_{r,s}$ by $f_\bet(Z)=f(Z,\bet)$. 
Then $f_\bet$ satisfies 
\[f_\bet(\alp Z)=\rho(\nu(\alp)^{-1}J(\alp,Z))f_\bet(Z) \]
for $\alp\in\Gam^\bet$. 
Let $\calb$ be a finite subset of $\calg(\widehat{\QQ})$ such that $\calg(\widehat{\QQ})=\bigsqcup_{\bet\in\calb}\calg(\QQ)\bet U$. 
One can identify the modular form $f$ with the collection of holomorphic functions $\{f_\bet\}_{\bet\in\calb}$ satisfying the functional equations above. 
\end{remark}

We denote the first $r$ complex coordinates of $\CC^{r,s}$ by $z_V^{(1)}(\vSi),\dots,z_V^{(r)}(\vSi)$, and the remaining $s$ complex coordinates of $\CC^{r,s}$ by $z_V^{(r+1)}(\vSi^c),\dots,z_V^{(r+s)}(\vSi^c)$. 
Let $\d\underline{z}_V^{}=\d\underline{z}_V^{}(\vSi)\cup\d\underline{z}_V^{}(\vSi^c)$ be an ordered basis of $\Ome_{\scra(V)_g(Z)}$, where 
\begin{align*}
\d\underline{z}_V^{}(\vSi)&=\{\d z_V^{(i)}(\vSi)\}_{i=1}^r, & 
\d\underline{z}_V^{}(\vSi^c)&=\{\d z_V^{(r+j)}(\vSi^c)\}_{j=1}^s. 
\end{align*}
We define $\bdome_g(Z)=(\bdome^+_g(Z),\bdome^-_g(Z))\in\cale_{\underline{\cala}(V)_g(Z)}$ by 
\begin{align*}
\bdome^+_g(Z)_i&=2\pi\sqrt{-1}\d z_V^{(i)}(\vSi) & &(1\leq i\leq r), \\
\bdome^-_g(Z)_i&=2\pi\sqrt{-1}\d z_V^{(i)}(\vSi^c) & &(r+1\leq i\leq r+s). 
\end{align*}

We see from (\ref{tag:25}) that for $\bet\in\calg(\QQ)^+$
\[\bdome_{\bet g}(\bet Z)=\nu(\alp)\trs J(\bet,Z)^{-1}\bdome_g(Z). \]
It follows that if $U$ is neat and $f\in\bfM^\calg_\rho(U,\CC)$, then 
\begin{align*}
f(\underline{\cala}(V)_{\bet g}(\bet Z),\bdome_{\bet g}(\bet Z))
=&f(\underline{\cala}(V)_g(Z),\nu(\alp)\trs J(\bet,Z)^{-1}\bdome_g(Z)) \\
=&\rho(\nu(\alp)^{-1}J(\alp,Z))f(\underline{\cala}(V)_g(Z),\bdome_g(Z)). 
\end{align*} 
Thus the function $(Z,g)\mapsto f(\underline{\cala}(V)_g(Z),\bdome_g(Z))$ is a modular form of weight $\rho$ and level $U$. 
Actually, this is a one-to-one correspondence. 

\begin{lemma}[{cf. \cite[Lemma 3.7]{MH}, \cite[Lemma 3.7]{CEFMV}}]\label{lem:cpx}
If $U$ is neat, then the association 
\[f\mapsto f(Z,g):=f(\underline{\cala}(V)_g(Z),\bdome_g(Z)) \]
defines an isomorphism $\bfM^\calg_\rho(U^\ell_1,\CC)\stackrel{\sim}{\to}M_\rho^\calg(U^\ell_1,\CC)$. 
\end{lemma} 

\begin{definition}\label{def:33}
Let $\calu=K^0_p\times\calu^{(p)}$ be a not (necessarily neat) open compact subgroup of $\calg(\widehat{\QQ})$ containing $U$. 
The space $\bfM^\calg_\rho\bigl(\calu^\ell_1,\CC\bigl)$ is defined as the subspace of $\bfM^\calg_\rho\bigl(U^\ell_1,\CC\bigl)$ that is the inverse image of $M^\calg_\rho\bigl(\calu^\ell_1,\CC\bigl)$. 
For an $\scro$-subalgebra $R$ of $\CC$ we regard $\bfM^\calg_\rho\bigl(\calu^\ell_1,R\bigl)$ as a submodule of $\bfM^\calg_\rho\bigl(\calu^\ell_1,\CC\bigl)$. 
We define $M_\rho^\calg\bigl(\calu^\ell_1,R\bigl)$ as the image of $\bfM^\calg_\rho\bigl(\calu^\ell_1,R\bigl)$ under the isomorphism in Lemma \ref{lem:cpx}, and define the submodule $S_\rho^\calg\bigl(\calu^\ell_1,R\bigl)$ similarly. \index{$M_\rho^\calg(U^\ell_1,R)$}
\end{definition}


\subsection{Algebraic representations of $\GL_r\times\GL_s$}\label{subsec:alg}

Let $\scrt=T_r\times T_s$ be the diagonal torus of $\scrh$ and $N_\scrh=N_r\times N_s$ a unipotent subgroup of $\scrh$. \index{$N_\scrh$}
Put $E'=E(\sqrt{-1})$. 
We enlarge the base ring $\scro$ to $\scro'=\frko_{E',(p)}$. \index{$\scro'$}
Then $T'\in\GL_r(\scro')$ by (C$_2$) (see \S \ref{ssec:26} for the definition of $T'$). 

Let $R$ be a an $\scro'$-algebra. 
Let $\ulk=(k_1,\dots k_r;k_{r+1},\dots,k_{r+s})\in\ZZ^{r+s}$ be a tuple with $k_1\leq k_2\leq \dots\leq k_r$ and $k_{r+1}\leq\dots\leq k_{r+s}$. 
We define an algebraic character $\ulk$ of $\scrt$ by 
\[\ulk(t)=t_1^{k_1}\cdots t_r^{k_r}\cdot t_{r+1}^{k_{r+1}}\cdots t_{r+s}^{k_{r+s}}\] 
for $t=\diag(t_1,\dots,t_{r+s})\in\scrt$. 
The irreducible algebraic representation $L_{\ulk}(R)$ can be realized as the schematical induction module 
\[L_{\ulk}(R)=\{f\in R[\GL_r\times\GL_s]\;|\;f(tug)=\ulk(t)f(g)\;(t\in\scrt,\; u\in N_\scrh)\}, \index{$L_{\ulk}(R)$}\]
where $R[\GL_r\times\GL_s]$ is the ring of polynomial functions on $\GL_r\times\GL_s$ with coefficients in $R$. 
Put $\Bbbk=(k_r,\dots,k_1;k_{r+s},\dots,k_{r+1})\in\ZZ^{r+s}$. 
Let $\scrh(R)$ act on $L_{\ulk}(R)$ by 
\[\rho_{\Bbbk}(h_1,h_2)f(g_1,g_2)=f(g_1h_1,g_2h_2). \]

Define the distinguished functional $l_{\ulk}:L_{\ulk}(R)\to R$ by $f\mapsto f(1)$. 
Then $l_{\ulk}$ is the vector of maximal weight $-\ulk$ in $L_{\ulk}^\vee(R):=\Hom (L_{\ulk}(R),R)$. 
On the other hand, we define $\bfv_{\ulk}\in \ZZ[\GL_r]$ by 
\[\bfv_{\ulk}=(\det X)^{-k_r}(\det Y)^{-k_{r+s}}\prod_{l=1}^{r-1} (\det X_{(l)})^{k_{l+1}-k_{l}}\prod_{m=1}^{s-1} (\det Y_{(l)})^{k_{m+r+1}-k_{m+r}}, \]
where $X=(X_{i,j})\in \Mat_r$ and $X_{(l)}:=(X_{r-l+i,j})_{1\leq i,j\leq l}$. 
We find that $\bfv_{\ulk}\in L_{\ulk^\vee}(\ZZ)^{N_\scrh}$ is a vector of the highest weight $-\ulk$, where we set 
\[\ulk^\vee=(-k_r,\dots,-k_1;-k_{r+s},\dots,-k_{r+1}). \]

If $R$ is an integral domain of characteristic zero and $F$ is its fractional field, then $L_{\ulk}(F)$ is an irreducible representation of $\scrh(F)$ and there is a unique isomorphism $ L_{\ulk^\vee}(F)\simeq L_{\ulk}^\vee(F)$ sending $\bfv_{\ulk}$ to $l_{\ulk}$.  Let  
\[\ell_{\ulk}:L_{\ulk^\vee}(F)\otimes L_{\ulk}(F)\to F \index{$\ell_{\ulk}$}\]
be the canonical invariant pairing such that $l_{\ulk}(f)=\ell_{\ulk}(\bfv_{\ulk}\otimes f)$. 

Suppose that $T'\in \GL_r(R)$. 
We define representations $\rho_{\ulk}$ of $\scrh(R)$ on $L_{\ulk}(R)$ by $\rho_{\ulk}(\alp,\bet)=\rho_\Bbbk(T^{\prime-1}\trs\alp^{-1}T',\bet)$ for $\alp\in\GL_r(R)$ and $\bet\in\GL_s(R)$. \index{$\rho_{\ulk}$}
When $R=\CC$, we view $\rho_{\ulk}$ as a representation of $\calk_\infty$ by pulling back via the homomorphism $J^-(-,\bfi):\calk_\infty\to\scrh(\CC)$. 
By definition we have $\rho_{\ulk}(J(\alp,\bfi))=\rho_{\Bbbk}(J^-(\alp,\bfi))$ for $\alp\in \calk_\infty$. 


\subsection{Automorphic forms on $\GU(r,s)$}\label{ssec:34}

We write $\scra(\calg)$ for the space of automorphic forms on $\calg$, \index{$\scra(\calg)$}i.e., smooth, $K$-finite, $\frkz$-finite, slowly increasing functions on $\calg(\QQ)\bsl\calg(\AA)$, where $\frkz$ is the centre of the universal enveloping algebra of the complexfied Lie algebra of $\calg(\RR)$. 
We call an autormophic form a cusp form if it is rapidly decreasing. 
We denote the space of cusp forms on $\calg$ by $\scra^0(\calg)$. \index{$\scra^0(\calg)$}

We associate to a $\call(\CC)$-valued function $f$ on $\frkD_{r,s}\times \calg(\widehat{\QQ})$ satisfying the relation in Definition \ref{def:21} a function $\vPh_\rho(f):\calg(\QQ)\bsl \calg(\AA)\to\call(\CC)$ by 
\begin{align*}
\varPhi_\rho(f)(g)&=\rho(\nu(g_\infty)J(g_\infty,\bfi)^{-1})f(g_\infty(\bfi),\hat g) 
\end{align*}
for $g=(g_\infty,\hat g)$ with $g_\infty\in \calg(\RR)^+$ and $\hat g\in \calg(\widehat{\QQ})$. \index{$\varPhi_\rho,\varPhi_{\ulk}$}
Note that 
\[\varPhi_\rho(f)(g\alp)=\rho(\nu(\alp)J(\alp,\bfi)^{-1})\varPhi_\rho(f)(g) \]
for $g\in \calg(\AA)$ and $\alp\in\calk_\infty$. 

When $R$ is an $\scro'$-algebra, $\rho=\rho_{\ulk}$ and $\call=L_{\ulk}(R)$, we will write 
\begin{align*}
\ome_{\ulk}&=\ome_{\rho_{\ulk}}, & 
\bfM_{\ulk}^\calg(U,R)&=\bfM_{\rho_{\ulk}}^\calg(U,R), &
M_{\ulk}^\calg(U,R)&=M_{\rho_{\ulk}}^\calg(U,R), \\
\vPh_{\ulk}&=\vPh_{\rho_{\ulk}}, &
\bfS_{\ulk}^\calg(U,R)&=\bfS_{\rho_{\ulk}}^\calg(U,R), &
S_{\ulk}^\calg(U,R)&=S_{\rho_{\ulk}}^\bfG(U,R) 
\end{align*}
for an open compact subgroup $U$ of $\calg(\widehat{\QQ})$. \index{$M_{\ulk}^\calg(U,R)$}\index{$S_{\ulk}^\calg(U,R)$}
We associate to $f\in M_{\ulk}^\calg(U,\CC)$ and $\bfv\in L_{\ulk^\vee}(\CC)$ a scalar valued cusp form $\vPh_{\ulk}(f)_\bfv\in\scra(\calg)$ defined by $\vPh_{\ulk}(f)_\bfv(g)=\ell_{\ulk}(\bfv\otimes\vPh_{\ulk}(f)(g))$. \index{$\vPh_{\ulk}(f)_\bfv$}
If $f\in M_{\ulk}^\calg(U,\CC)$ is not zero, then $\bfv\mapsto\vPh_{\ulk}(f)_\bfv$ gives a $\calk_\infty$-equivariant embedding $L_{\ulk^\vee}(\CC)\hookrightarrow\scra(\calg)$. 

Let $\pi\simeq\otimes_v'\pi_v$ be an irreducible cuspidal automorphic representation of $\calg(\AA)$. 
If $f^\circ\in S_{\ulk}^\calg(U,\overline{\QQ})$ gives an embedding $L_{\ulk^\vee}(\CC)\hookrightarrow\pi$ and if $f^\circ$ corresponds to a new vector of $\pi_q$ at all rational primes $q$ in a suitable sense (cf. Definition \ref{def:81} and \cite{AOY}), then we call $f^\circ$ a $\overline{\QQ}$-rational newform associated to $\pi$. \index{$f^\circ$}

\begin{remark}\label{rem:21}
We define the highest weight part $f_0$ of $f$ by $f_0(Z)=l_{\ulk}(f(Z))$. \index{$f_0,\vPh_{\ulk}(f)_0$}
Put $\vPh_{\ulk}(f)_0=\vPh_{\ulk}(f)_{\bfv_{\ulk}}$. 
Let $\calt_\infty$ be the subgroup which consists of $\alp\in \calk_\infty$ such that $J^-(\alp,\bfi)$ is diagonal. 
Observe that for $g\in \calg(\AA)$ and $t_\infty\in\calt_\infty$ 
\[\vPh_{\ulk}(f)_0(gt_\infty)=\ulk(J^-(t_\infty,\bfi))^{-1}\vPh_{\ulk}(f)_0(g). \]
\end{remark}


\subsection{$p$-adic modular forms}

Let $\WW=\WW(\overline{\FF}_p)$ be the ring of Witt vectors associated with $\overline{\FF}_p$. 
We view $\WW$ as an $\scro$-algebra. 
Put $\WW_m=\WW/p^m\WW$ for positive integers $m$. 
We will identify $\bfH$ (resp. $\bfN$, $\bfT$) with the group $\scrh(\ZZ_p)$ (resp. $N_\scrh(\ZZ_p)$, $\scrt(\ZZ_p)$) with respect to the $p$-adic basis of $M^0$ fixed in \S \ref{ssec:22}. 
Let $F$ be a number field which contains $E'$. 
We denote by $\scro_F=\frko_{F,(p)}$ the localization of $\frko_F$ at $p$ and the $\frkp$-adic completion of $\scro_F$ by $\calo$. 

The ordinary locus $T_{0,m}$ of $\overline{S}_\calg(U)_{/\WW_m}$ is the complement of the zero set of a lift of some power of the Hasse invariant. 
Since $p$ splits in $E$, it is known that $T_{0,m}$ is open and dense in $\overline{S}_\calg(U)_{/\WW_m}$ (cf. \cite{Wed99}). 
Set $T_{\ell,m}:=I_\calg(U^\ell)_{/\WW_m}$ for positive integers $\ell$. 
As we let $\ell$ vary, we obtain a tower of finite \'{e}tale and Galois coverings of $T_{0,m}$, called the Igusa tower. 
Then $T_{\infty,m}:=\varprojlim_\ell T_{\ell,m}$ is a Galois cover of $T_{0,m}$ with Galois group $\bfH$. 
We define 
\begin{align*}
V_{\ell,m}&=\rmH^0(T_{\ell,m/\calo},\calo_{T_{\ell,m}}). 
\end{align*}
The module of $p$-adic modular forms on $\calg$ of level $U^{(p)}$ is 
\[V(\calg,U)=\varprojlim_m \varinjlim_\ell V_{\ell,m}^\bfN. \index{$V(\calg,U)$}\]

Let $\ell\geq m$. 
Each element $f\in V_{\ell,m}$ can be viewed as a function $(\underline{A},j)\mapsto f(\underline{A},j)\in\WW_m$, where $(\underline{A},j)$ is a $\WW_m$-quintuple. 
Each element $h\in\scrh(\ZZ/p^\ell\ZZ)$ acts on elements $f\in V_{\ell,m}$ via $(h\cdot f)(\underline{A},j)=f(\underline{A},hj)$. 
When $A$ is $p$-ordinary, we get an isomorphism $j^\vee:A[p^\ell]^{\rm{et}}\simeq(\ZZ/p^\ell\ZZ)\otimes N^{-1}$ by taking the dual of $j:\bdmu_{p^\ell}\otimes_\ZZ N^0\hookrightarrow A[p^\ell]$ as $\frko_E$-group schemes. 
Since $\ome_{A/\calo}\simeq A[p^\ell]^{\rm{et}}\otimes\calo$ (see \cite[(9)]{CEFMV}), we can view $j^\vee$ as an element $\bdome(j)\in\cale_{\underline{A}}$ and define the canonical morphism
\begin{align*}
\rmH^0(T_{\ell,m}{}_{/\calo},\ome_{\ulk})&\to V_{\ell,m}^\bfN, & 
f&\mapsto\widehat{f}_0(\underline{A},j)=l_{\ulk}(f(\underline{A},j,\bdome(j))). 
\end{align*}
We call $\widehat{f}_0$ the $p$-adic avatar of $f$. \index{$\widehat{f}_0$}

Let $\zet$ be a character of $\scrt(\ZZ/p^\ell\ZZ)$. 
We define the locally algebraic character $\ulk_\zet$ of $\bfT$ by $\ulk_\zet(t)=\ulk(t)\zet(t)$. 
Assume that $\calo$ contains the values of $\zet$. 
Put  
\begin{align*}
\bfM_{\ulk}^\calg(U^\ell_1,\calo)&:=\bfM_{\ulk}^\calg(U^\ell_1,\scro_F)\otimes_{\scro_F}\calo, \\
\bfM_{\ulk_\zet}^\calg(U^\ell_0,\calo)&:=\{f\in\bfM_{\ulk}^\calg(U^\ell_1,\calo)\;|\;[t].f=\zet(t)f\text{ for }t\in\bfT\}. 
\end{align*}
We therefore have a map \index{$\bfM_{\ulk_\zet}^\calg(U^\ell_0,\calo)$}
\begin{align*}
\bfM_{\ulk_\zet}^\calg(U^\ell_0,\calo)&\to V(\calg,U)[\ulk_\zet], & 
f&\mapsto \widehat{f}_0,   
\end{align*}
where $V(\calg,U)[\ulk_\zet]$ is the $\ulk_\zet$-eigenspace of the torus $\bfT$ (cf. Remark \ref{rem:21}). 



\subsection{The $\calu_p$-operators}\label{ssec:25}

Define elements $\alp_i,\bet_j\in G(\QQ_p)$ by 
\begin{align*}
\alp_i&=\begin{bmatrix} \ono_{r+s-i} & \\ & p\ono_i \end{bmatrix}, &
\bet_j&=\begin{bmatrix} p^{-1}\ono_j & \\ & \ono_{r+s-j} \end{bmatrix}
\end{align*}
for $i=1,2,\dots,s$ and $j=1,2,\dots,r$. 
Put $\cali_p=\imath_\frkp^{-1}(U_1^\ell)$. 
Following Proposition 3.15\footnote{We here let $a_i=-k_i$ $(1\leq i\leq r)$ and $b_j=k_{r+j}$ $(1\leq j\leq s)$. 
Since the similitude group $\GU(r,s)$ is considered in \cite{MH2}, the factor $\ulk(\alp_v)^{-1}$ should be replaced by $\ulk(\alp_v)$. } of \cite{MH2}, we define the normalized $U_p$-operators 
\[\calu_p(\alp_i),\,\calu_p(\beta_j)\in C_c^\infty(\cali_p\bsl G(\QQ_p)/\cali_p)\] 
by 
\begin{align*}
\calu_p(\alp_i)&:=p^{-ri+\sum_{l=1}^ik_{r+s+1-l}}[\cali_p\imath_\frkp^{-1}(\alp_i)\cali_p], \\
\calu_p(\bet_j)&:=p^{-sj-\sum_{l=1}^jk_l}[\cali_p\imath_\frkp^{-1}(\beta_j)\cali_p]. \index{$\calu_p(\bet_1)$}
\end{align*}

Let $\bfT^\calg(N,R)$ be the universal Hecke algebra over a commutative ring $R$ defined by 
\[\bfT^\calg(N,R)=\bigotimes'_{l\nmid pN} C_c^\infty(\calk_l\backslash\calg(\QQ_l)/\calk_l,R)\bigotimes R\left[\calu_p(\alp_i),\calu_p(\beta_j)\right]_{\substack{i=1,\dots,s,\\j=1\dots,r}}. \]
As is well-known, $\bfT^\calg(N,\CC)$ is commutative and acts naturally on $M_{\ulk}^\calg(U_1^\ell,\CC)$. 
More explicitly, we have  
\begin{align*}
[\calu_p(\alp_i)f](Z,\bet)&=\frac{1}{p^{ri-\sum_{l=1}^ik_{r+s+1-l}}}\sum_{u\in N_{r+s}(\ZZ_p)/\alp_i^{-1}N_{r+s}(\ZZ_p)\alp_i^{}}f(Z,\bet \imath_\frkp^{-1}(u\alp_i^{-1})), \\
[\calu_p(\bet_j)f](Z,\bet)&=\frac{1}{p^{sj+\sum_{l=1}^jk_l}}\sum_{u\in N_{r+s}(\ZZ_p)/\bet_j^{-1}N_{r+s}(\ZZ_p)\bet_j^{}}f(Z,\bet \imath_\frkp^{-1}(u\bet_j^{-1}))
\end{align*}
for $f\in M_{\ulk}^\calg(U_1^\ell,\CC)$. 
The Hecke algebra $\bfT^\calg(N,\calo)$ also acts on the module $\bfS_{\ulk}^\calg(U_1^\ell;\calo)$ of cusp forms over $\calo$ and on the module of $p$-adic modular forms. 
We put $\calu_p=\prod_{i=1}^s\calu_p(\alp_i)\prod_{j=1}^r\calu_p(\bet_j)$ and define the projector 
\[\bdse=\lim_{m\to\infty}\calu_p^{m!}. \index{$\bdse$}\]
If $f\in M_{\ulk}^\calg(U_1^\ell,\CC)$, then we call $f$ an $p$-ordinary eigenform if $f$ is an eigenvector of $\calu_p$ with eigenvalue $a_p(f)$ such that $\iot_p(a_p(f))\in\overline{\ZZ}_p^\times$. 

\begin{definition}\label{def:23}
Let $\pi$ be an irreducible cuspidal automorphic representation of $\calg(\AA)$ generated by $\Phi_{\ulk}(f)_0$ with a Hecke eigenform $f\in S_{\ulk}^\calg(U,\overline{\QQ})$. 
Denote  the Harish-Chandra parameter of $\pi_\infty$ by $(\lam_1,\lam_2,\dots,\lam_n)$. 
Note that $\lam_j=-k_j+\frac{r-s+1}{2}-j$ for $1\leq j\leq r$ and $\lam_{r+i}=-k_{r+i}+\frac{r+s+1}{2}-i$ for $1\leq i\leq s$. 
Assume that $\pi_p|_{\GL_n(\QQ_p)}$ is a subquotient of a principal series $I(\mu_1,\mu_2,\dots,\mu_n)$ of $\GL_n(\QQ_p)$ with locally algebraic characters $\iot_p\circ\mu_i:\QQ_p^\times\to\CC_p^\times$. 
Put $\alp_i=\iot_p(\mu_i(p))$. 
We order $\mu_i$ so that $\ord_p\alp_1\geq\ord_p\alp_2\geq\cdots\geq\ord_p\alp_n$. 
We say that $\pi$ is $p$-ordinary with respect to $\iot_p$ if $\ord_p\alp_i=-\lam_{n-i+1}$ for $i=1,2,\dots,n$ (cf. Conjecture 4.1 of \cite{CPR}). 
\end{definition}

The ordinary vector $h_{\pi_p}^\ord\in\pi_p$ constructed in \S \ref{ssec:84} has the eigenvalue whose $p$-adic order is smallest. 
Proposition \ref{prop:81} shows that $h_{\pi_p}^\ord$ is the unique simultaneous eigenvector whose $\calu(\alp_i)$ and $\calu(\bet_j)$ eigenvalues are $p$-units for $1\leq i\leq s$ and $1\leq j\leq r$. 

\subsection{Theta functions}\label{ssec:28}

Let $W_0=E^t$. 
Put $\bfW_0:=W_0\otimes_\QQ\RR=\CC^t$. 
Let $\gam_0\in\Mat_t(E)$ be a matrix which satisfies $\trs\gam_0=\gam_0=-\gam_0^c$ and such that $\iot_\infty(-\sqrt{-1}\gamma_0)$ is positive negative.
We regard $W_0$ as a vector space over $\QQ$ and equip $\bfW_0$ with the real linear alternating form defined by 
\[\ll w_1,w_2\gg=\mathrm{T}_{E/\QQ}(w_2^c\gam_0^{-1}\trs w_1). \]


Let $L_0$ be an $\frko_E$-lattice in $W_0$ as in \S \ref{ssec:22}. 
Put 
\begin{align*}
L_0^c&=\{l^c\;|\;l\in L_0^{}\}, & 
L_0'&=L_0^c\gam_0^{}. \index{$L_0'$}
\end{align*}
For simplicity we will impose a stronger assumption 
\beq
\ll\;,\;\gg\text{ takes values in $2\ZZ$ on $L_0'$} \tag{$C_3^+$}
\eeq
(cf. (\ref{tag:33})). 
Define the positive definite Hermitian form on $\bfW_0$ by 
\[H_0(w_1,w_2)=2w_2^c\sqrt{-1}\gam_0^{-1}\trs w_1. \]
Note that 
\[E_0(w_1,w_2):={\rm Im}H_0(w_1,w_2)=\ll w_1,w_2\gg. \index{$E_0$, $H_0$}\]

Fix a fractional ideal $\frka$ of $\frko_E$ and a positive rational number $m$ such that $m\frka\frka^c\subset\frko_E$. 
Let $L_\frka$ be a lattice contained in $\frka L_0'$. 
Then $mE_0$ takes values in $2\ZZ$ on $L_\frka$. 
In particular, it is a Riemann form on $\bfW_0$ relative to $L_\frka$. 

Given $l\in L_\frka$, we define the function $e_l^m:\bfW_0\to\CC^\times$ by 
\[e_l^m(w)=e^{\pi mH_0\bigl(w+\frac{l}{2},l\bigl)}. \]
We define a line bundle $\frkL_\frka^m$ on the complex abelian variety $\calz_\frka^\circ(\CC):=\bfW_0/L_\frka$ by $L_\frka\backslash (\bfW_0\times \CC)$ with the action of $l\in L_\frka$ given by
\[l\cdot (w,x)=(w+l,e_l^m(w)x).\]

Let $(\rho,\call)$ be a finite dimensional representation of $\GL_{t+1}$.  
For $l\in\bfW_0$ and any $\call(\CC)$-valued function $\Tht$ on $\bfW_0$ we define 
\[[A_{m,\rho}(l)\Tht](w)=e^{-\pi mH_0\bigl(w+\frac{l}{2},l\bigl)}\rho\left(\begin{bmatrix} 1 & 0 \\ -\gam_0^{-1}\trs l^c & \ono_t \end{bmatrix}\right)\Tht(w+l). \index{$A_{m,\rho}(l)$}
\]
The operators $A_{m,\rho}(l)$ fulfill the basic commutation relation 
\beq
A_{m,\rho}(l_1)A_{m,\rho}(l_2)=e^{\pi\sqrt{-1}mE_0(l_2,l_1)}A_{m,\rho}(l_1+l_2). \label{tag:33}
\eeq
For any scalar-valued function $\tht$ on $\bfW_0$ we define 
\[[A_m(l)\tht](w)=e^{-\pi mH_0\bigl(w+\frac{l}{2},l\bigl)}\tht(w+l). \index{$A_m(l)$}\]

\begin{definition}[Theta functions]\label{def:24} 
The space $\bfT^m(L_\frka,\CC)$ consists of holomorphic functions $\tht$ on $\bfW_0$ which satisfy the functional equations $\tht(w+l)=e_l^m(w)\tht(w)$ for every $l\in L_\frka$. 
By definition we can identify $\bfT^m(L_\frka,\CC)$ with the space $\rmH^0(\calz_\frka^\circ(\CC),\frkL_\frka^m)$ of global sections. 
If $t=0$, then we define $\bfT^m(L_\frka,\CC)=\CC$. 

More generally, the space $\bfT^m_\rho(L_\frka,\CC)$ consists of $\call(\CC)$-valued holomorphic functions $\Tht$ on $\bfW_0$ such that $A_{m,\rho}(l)\Tht=\Tht$ for all $l\in L_\frka$. 
\end{definition}

Put $N_\scrh^1=\biggl\{\begin{bmatrix} 1 & 0 \\ * & \ono_t\end{bmatrix}\in\GL_{t+1}\biggl\}$. \index{$N_\scrh^1$}
Let $\call(R)_{N_\scrh^1}$ be the module of $N_\scrh^1(R)$-coinvariants for the action of $\rho$ in $\call(R)$. 
For $\Tht\in\bfT^m_\rho(L_\frka,\CC)$ we write $\Tht_{N_\scrh^1}(w)$ for the image of $\Tht(w)$ in $\call(\CC)_{N_\scrh^1}$. 
Then the function $w\mapsto\Tht_{N_\scrh^1}(w)$ belongs to $\bfT^m(L_\frka,\CC)\otimes_\CC\call(\CC)_{N^1_\scrh}$. 


\subsection{Analytic Fourier-Jacobi coefficients}\label{ssec:27}
 
In the rest of this section we shall restrict ourselves to the simple case $s=1$. 
Let $\calp=\calm\caln$ be the standard maximal parabolic subgroup of $\calg$ whose unipotent radical $\caln$ consists of matrices of the form 
\[\bfn(w,z)=\begin{bmatrix} 1 & -w^c\gamma_0^{-1} & z-\frac{1}{2}w^c\gam_0^{-1}\trs w \\ 0 & \ono_t & \trs w \\ 0 & 0 & 1 \end{bmatrix}\]
with $w\in W_0$ and $z\in\QQ$. \index{$\bfn(w,z)$}
Observe that 
\beq
\bfn(w_1,z_1)\bfn(w_2,z_2)=\bfn\left(w_1+w_2,z_1+z_2+\frac{1}{2}\ll w_1,w_2\gg\right). \label{tag:31}
\eeq
Thus $\caln$ is the Heisenberg group associated with $(W_0,\ll\;,\;\gg)$. 
\[g_0^{-1}\bfn(w,z)g_0^{}=\bfn(w\trs g_0^{-1},\nu(g_0) z) \]
We form the group $\cals=\calg_0\ltimes\caln$.  
For $\xi\in\EE^\times$ we put 
\[\bfd(\xi)=\begin{bmatrix} \xi^c & 0 & 0 \\ 0 & \ono_t & 0 \\ 0 & 0 & \xi^{-1}\end{bmatrix}\in G(\AA). \]


For $\vPh\in\scra(\calg)$ we put 
\begin{align*}
\vPh^m_\bet(b)&=\int_{\QQ\bsl\AA}\vPh(\bfn(0,z)b\bet)\overline{\addchar^m(z)}\,\d z
\end{align*}
for $b\in\calp(\AA)$, $\bet\in\calg(\widehat \QQ)$ and a positive rational number $m$. 


Given a finite id\`{e}le $\alp=(\alp_\frkq)\in\widehat{E}^\times$, we write $\alp\frko_E$ for the fractional ideal $\frka$ of $\frko_E$ such that $\frka_\frkq=\alp_\frkq\frko_{E_\frkq}$ for all prime ideals $\frkq$ of $\frko_E$. 
We will consider Fourier-Jacobi coefficients at the cusp labeled by $\bfd(\alp)\in\calm(\widehat{\QQ})$. 

\begin{definition}
Take a finite set $\Box$ of prime numbers so that $U\supset K^0_l$ if $l\notin\Box$. 
Let $\frka$ be a prime-to-$\Box$ fractional ideal of $\frko_E$.  
Take a finite id\`{e}le $\alp\in\widehat{E}^\times$ so that $\alp\frko_E=\frka$ and $\alpha_q=1$ for all $q\in\Box$. 
For $f\in S_{\ulk}^\calg(U,\CC)$ we write 
\begin{align*}
\overrightarrow{\FJ}^m_\frka(w,f)&=\overrightarrow{\FJ}^m_{\bfd(\alp)}(w,f), &
\FJ^m_\frka(f)_0^{}(w)&=l_{\ulk}\left(\overrightarrow{\FJ}^m_{\frka}(w,f)\right). 
\end{align*}
It is clear that $\overrightarrow{\FJ}^m_\frka(w,f)$ does not depend on the choice of $\alp$. \index{$\FJ^m_\frka(f)_0^{}$}
\end{definition} 


\begin{lemma}\label{lem:21}
Let $f\in M_\rho^\calg(U,\CC)$. 
We have 
\begin{multline*}
\vPh_\rho(f)_{g_0\bfd(\xi)\bet}^m(\bfn(w,0))\\
=e^{-\pi m\bigl(2\nu(g_0)^{-1}\xi\xi^c+\frac{1}{2}H_0(w,w)\bigl)}\rho\left(\begin{bmatrix} \xi^c & 0 \\ -\trs g_0^{}\gam_0^{-1}\trs w^c & \trs g_0 \end{bmatrix},\xi\right)\overrightarrow{\FJ}^m_\bet(w,f) 
\end{multline*} 
for $w\in \bfW_0$, $\xi\in E_\infty^\times$, $\bet\in\calg(\widehat{\QQ})$ and $g_0\in\calg_0(\RR)^+$. 
\end{lemma}

\begin{proof}
Let $x\in\RR$. 
Put
\[\alp=\bfn(w,x)g_0\bfd(\xi)=\begin{bmatrix} \xi^c & -w^c\gam_0^{-1}g_0^{} & \frac{\nu(g_0)}{\xi}(x-\frac{1}{2}{w^c\gam_0^{-1}\trs w}) \\ 0 & g_0 & \frac{\nu(g_0)}{\xi}\trs w \\ 0 & 0 & \frac{\nu(g_0)}{\xi}\end{bmatrix}. \]
We remind the reader of the assumption $\trs\gam_0^{}=\gam_0^{}=-\gam_0^c$. 
Since 
\begin{align*}
\alp(\bfi)&=\begin{bmatrix} \frac{\xi\xi^c}{\nu(g_0)}\sqrt{-1}-\frac{w^c\gam_0^{-1}\trs w}{2}+x \\ \trs w \end{bmatrix}, \\
\lam(\alp,\bfi)&=\nu(g_0)\begin{bmatrix} \frac{1}{\xi^c} & 0 \\ \frac{1}{\xi^c}\gam_0^{-1}\trs w^c & \trs g_0^{-1} \end{bmatrix}, &
\mu(\alp,\bfi)&=\frac{\nu(g_0)}{\xi},  
\end{align*}
we see that $\vPh_\rho(f)(\bfn(w,x)g_0\bfd(\xi)\bet)$ equals 
\[\rho\left(\begin{bmatrix} \xi^c & 0 \\ -\trs g_0^{}\gam_0^{-1}\trs w^c & \trs g_0 \end{bmatrix},\xi\right)f\biggl(\begin{bmatrix} \frac{\xi\xi^c}{\nu(g_0)}\sqrt{-1}-\frac{w^c\gam_0^{-1}\trs w}{2}+x \\ \trs w \end{bmatrix},\bet\biggl) \]
by definition. 
Now one can easily check the stated identity. 
\end{proof}
Given a $\ZZ$-lattice $L$, we write $\widehat{L}=L\otimes_\ZZ\widehat{\ZZ}$. 
When $L$ is small, the set  
\begin{align*}
\caln(L)
&=\{\bfn(l,z)\;|\;l\in\widehat{L},\;z\in\widehat{\ZZ}\} \index{$\caln(L)$}
\end{align*}
is an open compact subgroup of $\caln(\widehat{\QQ})$. 
It is worth noting that $\caln(L_0')\subset K^0$. 
For an integral ideal $\frka$ of $\frko_E$ we put 
\[\scrs_\frka^+=(\frka\frka^c)^{-1}\cap\QQ^\times_+. \index{$\scrs_\frka^+$}\]

\begin{proposition}\label{prop:FJ}
Let $f\in M_\rho^\calg(U,\CC)$. 
Assume that $U$ contains $\caln(L_0')$. 
Let $a\in\widehat{E}^\times$. 
Put $\frka=a\frko_E$. 
\begin{enumerate}
\item $\overrightarrow{\FJ}^m_\frka(w,f)=0$ unless $m\in\scrs_\frka^+\cup\{0\}$. 
\item $\overrightarrow{\FJ}^m_\frka(w,f)\in\bfT^m_\rho(\frka L_0',\CC)$ and $\FJ^m_\frka(f)_0^{}\in\bfT^m(\frka L_0',\CC)$ for $m\in\QQ_+^\times$.  
\end{enumerate}
\end{proposition}

\begin{proof}
If $z\in\widehat{\Nr_{E/\QQ}(\frka)}$, then 
\begin{align*}
\overrightarrow{\FJ}^m_\frka(w,f)
&=\overrightarrow{\FJ}^m_{\bfd(a)\bfn(0,\Nr_{E/\QQ}(a)^{-1}z)}(w,f)\\
&=\overrightarrow{\FJ}^m_{\bfn(0,z)\bfd(a)}(w,f)
=\addchar^m(z)\overrightarrow{\FJ}^m_\frka(w,f), 
\end{align*}
which proves (1). 
Given $l\in W_0\otimes_\QQ\AA$, we denote its finite and infinite parts by $l_\bff$ and $l_\infty$. 
If $l\in W_0$, then 
\[\vPh_\rho(f)_{\bfd(a)}^m(\bfn(w+l_\infty,0))=e^{-\pi\sqrt{-1}m\ll l_\infty,w\gg}\vPh_\rho(f)_{\bfd(a)}^m(\bfn(w-l_\bff,0)) \]
by (\ref{tag:31}). 
It follows from Lemma \ref{lem:21} that  
\[\overrightarrow{\FJ}^m_{\bfn(-l_\bff,0)\bfd(a)}(w,f)=e^{-\pi mH_0\bigl(w+\frac{l_\infty}{2},l_\infty\bigl)}\rho\left(\begin{bmatrix} 1 & 0 \\ -\gam_0^{-1}\trs l_\infty^c & \ono_t \end{bmatrix}\right)\overrightarrow{\FJ}^m_\frka(w+l_\infty,f). \]
If $l\in\frka L_0'$, then since $a^{-1}l_\bff\in\widehat{L_0'}$, we have $\bfn(-a^{-1}l_\bff,0)\in U$ by assumption, which completes the proof of (2). 
\end{proof}


\subsection{Integral theta functions}\label{ssec:39}

Suppose that $t>0$. 
By the standard theory of abelian varieties with complex multiplication there is a CM abelian variety $B_0^{}$ defined over a number field such that $B_0^{}(\CC)=\bfW_0^{}/L_0'$ after extending scalars to $\CC$ via $\iot_\infty$. 
Assume that $B_0^{}$ can be extended to an abelian scheme $\scrb_0^{}$ over $\scro$. 

We write $\bfW_0^*$ for the space of $\CC$-antilinear forms on $\bfW_0$. 
Define the $\CC$-linear map $\Lam_0:\bfW_0^{}\to\bfW_0^*$ by $\Lam_0(w):w'\mapsto H_0(w,w')$. 
Let $\scrb_0^\vee$ be the dual abelian scheme of $\scrb_0^{}$. 
It is important to note that $\scrb_0^\vee(\CC)=\bfW_0^*/L_0^{\prime\perp}$, where $L_0^{\prime\perp}:=\{l\in \bfW_0^*\;|\; {\rm Im} l(L_0')\subset\ZZ\}$. 
Denote by $\lambda_0:\scrb_0^{}\to \scrb_0^\vee$ the polarization such that $\lam_{0/\CC}$ is induced by $\Lam_0$. 

Let $\scrp_0$ be the Poincar\'e bundle over $\scrb_0^{}\times \scrb_0^\vee$. 
After scalar extension to $\CC$ the complex line bundle $\scrp_{0/\CC}$ is characterized by the Riemann form 
\[E_{\scrp_0}((z_1,l_1),(z_2,l_2))={\rm Im}l_1(z_2)-{\rm Im}l_2(z_1)\]
for $(z_i,l_i)\in\bfW_0^{}\times\bfW_0^*$ $(i=1,2)$ relative to $L_0'\times L_0^{\prime\perp}$. 
Set 
\[L_\frka=\frka L_0'\cap\Lam_0^{-1}(\frkd_E\frka L_0^{\prime\perp}). \]
Let $\calz_\frka^\circ$ be the CM abelian scheme such that $\calz_\frka^\circ(\CC)=\bfW_0/L_\frka$. \index{$\calz_\frka^\circ$}
Given a positive rational number $m$ such that $m\Nr_{E/\QQ}(\frka)\subset \ZZ$, we write $m=\sum_i\Tr_{E/\QQ}(a_i^{}b_i^c)$ with $a_i\in \frka^{-1}$ and $b_i\in \frka^{-1}\frkd_E^{-1}$. 
Then the map 
\begin{align*}
[m]&:\bfW_0^{}\to \bfW_0^{}\times\bfW_0^*, & 
w&\mapsto \sum_i(a_iw,\Lam_0(b_iw))
\end{align*}
induces a morphism
\[[m]: \calz_\frka^\circ\rightarrow\scrb_0^{}\times\scrb_0^\vee. 
\]
We redefine the line bundle $\frkL_\frka^m$ by  
\[\frkL_\frka^m:=[m]^*\scrp_0. \]
The Riemann form associated with the complex line bunle $\frkL_\frka^m{}_{/\CC}$ is given by 
\begin{align*}
E_{\frkL_\frka^m}(w_1,w_2)
&=\sum_iE_{\scrp_0}((a_iw_1,\lam_0(b_iw_1)),(a_iw_2,\lam_0(b_iw_2)))\\
&=\sum_i(E_0(b_iw_1,a_iw_2)-E_0(b_iw_2, a_iw_1))
=mE_0(w_1,w_2). 
\end{align*}
We consider the $\scro$-module of global sections $\rmH^0(\calz_\frka^\circ,\frkL_\frka^m)\subset \bfT^m(L_\frka,\CC)$. 

\begin{definition}\label{def:25}
For an $\scro$-algebra $\calo$ we put 
\[\bfT^m(L_\frka,\calo):=\rmH^0(\calz_\frka^\circ,\frkL_\frka^m)\otimes_\scro\calo. \index{$\bfT^m(L_\frka,\calo)$}\]
This provides an integral structure of $\bfT^m(L_\frka,\CC)$ in the sense that  
\[\bfT^m(L_\frka,\calo)\otimes_\calo\CC=\bfT^m(L_\frka,\CC)\]
when $\calo$ is a subring of $\CC$. 
When $t=0$, we set $\bfT^m(L_\frka,\calo)=\calo$. 
\end{definition}





\subsection{Arithmetic Fourier-Jacobi coefficients}\label{ssec:310}

Recall that 
\begin{align*}
Y&=\frko_E^{}e_1, & 
X^\vee&=\frko_Ee_n. 
\end{align*}
Let $\alp$ be a finite id\`{e}le $\alp\in\widehat{E}^\times$ such that $\alp_q=1$ for all $q\in\Box$. 
Put   
\begin{align*}
\bet&=\bfd(\alp), & 
\frka&=\alp\frko_E, &  
Y_\bet^{}&=Y\bet^{-1}=(\frka^c)^{-1} e_1, & 
X_\bet^\vee&=X^\vee \bet^{-1}=\frka e_n.  
\end{align*}
Let 
\[X_\bet^{}=\{y\in Y_E\;|\;\La y,X_\bet^\vee\Ra_{r,1}\subset\ZZ\}=\frkd_E^{-1}(\frka^c)^{-1}e_1\] 
be the $\ZZ$-dual of $X_\bet^\vee$. 
We denote by $i:Y_\bet=(\frka^c)^{-1} e_1\hookrightarrow X_\bet=\frkd_E^{-1}(\frka^c)^{-1} e_1$ the inclusion map. 
Let $I_\bet$ be the subgroup in $X_\bet\otimes_\ZZ Y_\bet$ generated by
\begin{align*}
&x\otimes i(x')-x'\otimes i(x); &
&xb\otimes y-x\otimes b^cy & 
&(x,x'\in X_\bet,\;y\in Y_\bet,\;b\in\frko_E)
\end{align*}
Let $\scrs_\bet=S(X_\bet\otimes_\ZZ Y_\bet)$ be the maximal free quotient of the group $X_\bet\otimes_\ZZ Y_\bet/I_\bet$. 
The $\ZZ$-dual $\scrs_\bet^\vee=\Hom_\ZZ(\scrs_\bet,\ZZ)$ is the module of $\ZZ$-valued symmetric and Hermitian bilinear forms on $\frkd_E^{-1}(\frka^c)^{-1}\times(\frka^c)^{-1}$. 
Therefore $\scrs_\bet^\vee$ is isomorphic to the fractional ideal $\Nr_{E/\QQ}(\frka)=(\frka\frka^c)\cap\QQ$, and hence $\scrs_\bet$ is the fractional ideal $\Nr_{E/\QQ}(\frka)^{-1}$. 
Let 
\[\scrr^0_{\bet,\ell}=\biggl\{\sum_{0\leq m\in p^{-\ell}\scrs_\bet}\tht^m\,q^m\;\Big|\;\tht^m\in\bfT^m(L_\frka,\scro)\biggl\} \]
be the ring of formal $q$-series with $m$th coefficients in $\bfT^m(L_\frka,\scro)$. 

Put $U_\calp^\bet=\bet U\bet^{-1}\cap\calg_0(\widehat{\QQ})$. 
Let $(\scrb_0,\bar\lam_0,\iot_0,\bar\eta)$ be the quadruple associated to $[\ono_{\calg_0}]\in S_{\calg_0}(U_\calp^\bet)$, where $\ono_{\calg_0}$ is the identity element of $\calg(\widehat{\QQ})$. 
Following \S 2.7.4 of \cite{MH2}, we consider the group scheme 
\[\calz_\frka
=\underline{\Hom}_{\frko_E}(X_\bet,\scrb_0^\vee)\times_{\underline{\Hom}_{\frko_E}(Y_\bet,\scrb_0^\vee)}\underline{\Hom}_{\frko_E}(Y_\bet,\scrb_0), \]
which consists of pairs $(c,c^\vee)\in\underline{\Hom}_{\frko_E}(X_\bet,\scrb_0^\vee)\times\underline{\Hom}_{\frko_E}(Y_\bet,\scrb_0)$ such that $c(i(y))=\lam_0(c^\vee(y))$ for $y\in Y_\bet$. 
After scalar extension to $\CC$ we have  
\[\calz_\frka(\CC)=\{(w',w)\in\bfW_0^*/\frkd_E\frka L_0^{\prime\perp}\times\bfW_0/\frka L_0'\;|\;w'-\Lam_0(w)\in\frka L_0^{\prime\perp}\}. \]
We redefine $\calz_\frka^\circ$ as the connected component of $\calz_\frka^{}$. \index{$\calz_\frka^\circ$}
Then $\calz_\frka^\circ(\CC)=\bfW_0/L_\frka$ in view of the exact sequence 
\[0\to \bfW_0/L_\frka\to\calz_\frka(\CC)\to\frka L_0^{\prime\perp}/(\frkd_E\frka L_0^{\prime\perp}+\Lam_0(\frka L_0'))\to 0. \] 

Let $\underline{\scrm}_\bet$ be an $\scrr^0_\bet$-quintuple of level $U$, called the Mumford quadruple at the cusp labeled by $\bet$. 
Fix a $p^\ell$-level structure $\jmath_\scrm$ on $\scrm_\bet$ lifted from the tautological $p^\ell$-level structure $\jmath_{\scrb_0}$ on $\scrb_0$ and the natural
one on $X_\bet^\vee\otimes\GG_m$. 
Then $(\underline{\scrm}_\bet,\jmath_\scrm)$ is an $\scrr^0_{\bet,\ell}$-quintuple, called the Mumford quintuple. 

Let $\d\underline{z}_{W_0}^{}=\d\underline{z}_{W_0}^{}(\vSi)$ be the ordered basis of $\Ome_{\scrb_0}$. 
There are an $\frko_E\otimes_\ZZ\scro$-basis $\bdome_{\scrb_0}$ of $\Ome_{\scrb_0}$ and a pair $(\Ome_{\calg_0,\infty},\Ome_{\calg_0,\frkp})\in\GL_t(\CC)\times\GL_t(\WW)$ such that 
\[\Ome_{\calg_0,\infty}\d\underline{z}_{W_0}^{}=\bdome_{\scrb_0}=\Ome_{\calg_0,\frkp}\bdome(\jmath_{\scrb_0}) \index{$\Ome_\infty$}\index{$\Ome_\frkp$}\]
(see \S 2.8 of \cite{MH2}), where $\WW=\WW(\overline{\FF}_p)$ is the $p$-adic completion of the ring of integers of the maximal unramified extension of $\QQ_p$. 
Assuming that $\Ome_{\calg_0,\infty}$ and $\Ome_{\calg_0,\frkp}$ can be taken to be diagonal, we write 
\begin{align*}
\Ome_{\calg_0,\infty}&=\begin{bmatrix} \Ome_{1,\infty} & & \\ & \ddots & \\ & & \Ome_{t,\infty}\end{bmatrix}, & 
\Ome_{\calg_0,\frkp}&=\begin{bmatrix} \Ome_{1,\frkp} & & \\ & \ddots & \\ & & \Ome_{t,\frkp}\end{bmatrix}. 
\end{align*}

There is an exact sequence of $\frko_E$-modules 
\[0\to\Ome_{\scrb_0}\to\Ome_{\scrm_\bet}\to\Ome_{X_\bet^\vee\otimes\GG_m}\to 0. \]
Let $\d^\times t$ be the canonical $\frko_E$-basis of $\Ome_{X_\bet^\vee\otimes\GG_m}$. 
We denote a canonical lifting of $e_{\vSi^c}\d^\times t$ in $\Ome_{\scrm_\bet}$ by $\d^\times t_-$. 
Choose a lifting $\d^\times t'_+$ of $e_\vSi\d^\times t$ in $\Ome_{\scrm_\bet}$. 
Then $\bdome_\scrm:=(\d^\times t'_+,\bdome_{\scrb_0};\d^\times t_-)\in\cale_{\underline{\scrm}_\bet}$. 

When $f\in\bfM^\calg_{\ulk}(U^\ell_1,R)$ with an $\scro$-subalgebra $R$ of $\CC$, we can consider the analytic Fourier-Jacobi expansion of $f(Z,g):=f(\underline{\cala}(V)_g(Z),\bdome_g(Z))\in M_{\ulk}^\calg(U^\ell_1,R)$ (see Lemma \ref{lem:cpx}). 
We have the following important result:  

\begin{proposition}\label{prop:31}
Let $f\in\bfM^\calg_{\ulk}(U^\ell_1,R)$. 
Then 
\[\prod_{j=1}^t\biggl(\frac{\Ome_{j,\infty}}{2\pi\sqrt{-1}}\biggl)^{k_{j+1}}\cdot\; \FJ^m_\frka(f)_0^{}\in\bfT^m(L_\frka,R). \]
\end{proposition}

\begin{proof}
Let $h\in\bfH$. 
We evaluate $f$ at $(\underline{\scrm}_\bet,h^{-1}\jmath_\scrm,\bdome_\scrm)$ to obtain 
\[f(\underline{\scrm}_\bet,h^{-1}\jmath_\scrm,\bdome_\scrm)\in L_{\ulk}(\scrr^0_{\bet,\ell})\otimes_\scro R. \]
Its image $f(\underline{\scrm}_\bet,\jmath_\scrm,\bdome_\scrm)_{N_\scrh^1}$ in $L_{\ulk}(\scrr^0_{\bet,\ell})_{N_\scrh^1}\otimes_\scro R$ only depends on the choice
of $\bdome_{\scrb_0}$. 
We write 
\[f(\underline{\scrm}_\bet,h^{-1}\jmath_\scrm,\bdome_\scrm)_{N_\scrh^1}=\sum_m\overrightarrow{\FJ}^m_{\bet,h}(f)_{N_\scrh^1}\,q^m, \]
where  
\[\overrightarrow{\FJ}^m_{\bet,h}(f)_{N_\scrh^1}\in L_{\ulk}(R)_{N_\scrh^1}\otimes_\scro\bfT^m(L_\frka,\scro). \]

We presently have the following relation between analytic and arithmetic Fourier-Jacobi coefficients: 
\beq
\overrightarrow{\FJ}^m_{h\bet}(f)_{N_\scrh^1}=\rho_{\ulk}\biggl(\begin{bmatrix} 1 & \\ & (2\pi\sqrt{-1})^{-1}\Ome_{\calg_0,\infty}\end{bmatrix},1\biggl)\overrightarrow{\FJ}^m_{\bet,h}(f)_{N_\scrh^1} \label{tag:35}
\eeq 
(cf. \cite{Lan12}, \cite[(3.12)]{MH2}). 
We complete the proof by letting $h=1$ and applying the evaluation map $l_{\ulk}$. 
\end{proof}

We define the Fourier-Jacobi expansion of a $p$-adic modular form $\calf\in V(\calg,U)$ at the $p$-adic cusp $(\bet,h)$ by 
\[\calf(\underline{\scrm}_\bet,h^{-1}\jmath_\scrm)=\sum_m\widehat{\FJ}^m_{\bet,h}(\calf)\,q^m\in\scrr^0_{\bet,\infty}, \]
where $\scrr^0_{\bet,\infty}=\varprojlim_\ell\scrr^0_{\bet,\ell}\otimes_\scro\scro_\frkp$. 
When $\calf$ is the $p$-adic avatar of a classical modular form, we have the following relation between analytic and $p$-adic Fourier-Jacobi coefficients:

\begin{corollary}\label{cor:31}
Let $f\in\bfM^\calg_{\ulk}(U^\ell_1,R)$. 
Then 
\[\widehat{\FJ}^m_{\bet,1}(\widehat{f}_0)=\prod_{j=1}^t\biggl(\frac{\Ome_{j,\infty}}{2\pi\sqrt{-1}\Ome_{j,\frkp}}\biggl)^{k_{j+1}}\cdot\; \FJ^m_\frka(f)_0^{}. \]
\end{corollary}

\begin{proof}
Since $\bdome(\jmath_\scrm)=(\d^\times t'_+,\bdome(\jmath_{\scrb_0});\d^\times t_-)$, we have 
\begin{align*}
\widehat{f}_0(\underline{\scrm}_\bet,h^{-1}\jmath_\scrm)
&=l_{\ulk}(f(\underline{\scrm}_\bet,h^{-1}\jmath_\scrm,\bdome(\jmath_\scrm)))\\
&=l_{\ulk}(f(\underline{\scrm}_\bet,h^{-1}\jmath_\scrm,(\d^\times t'_+,\Ome_{\calg_0,\frkp}^{-1}\bdome_{\scrb_0};\d^\times t_-)))\\
&=l_{\ulk}\biggl(\rho_{\ulk}\biggl(\begin{bmatrix} 1 & \\ & \Ome_{\calg_0,\frkp}\end{bmatrix},1\biggl)f(\underline{\scrm}_\bet,\jmath_\scrm,\bdome_\scrm)\biggl)\\
&=\prod_{j=1}^t\Ome_{j,\frkp}^{-k_{j+1}}\cdot\;l_{\ulk}(f(\underline{\scrm}_\bet,h^{-1}\jmath_\scrm,\bdome_\scrm)), 
\end{align*}
from which we get $\widehat{\FJ}^m_{\bet,h}(\widehat{f}_0)=\prod_{j=1}^t\Ome_{j,\frkp}^{-k_{j+1}}\cdot\;l_{\ulk}(\overrightarrow{\FJ}^m_{\bet,h}(f))$ for $h\in\bfH$. 
The proof is complete by (\ref{tag:35}). 
\end{proof}




\section{Hida theory for $\GU(r,1)$}


In this section we continue to assume $s=1$ and $n=r+1$. 

\subsection{Open compact subgroups}
Let $l$ be a split prime, i.e., $l\frko_E=\frkl\frkl^c$. 
Then $V\otimes_\QQ\QQ_l=V_\frkl^{}\oplus V_{\frkl^c}$, where $V_\frkl$ and $V_{\frkl^c}$ are vector spaces over $\QQ_l$ and in duality under $\La\;,\Ra_{r,s}$. 
The projection $g\mapsto g|_{V_\frkl}$ gives an isomorphism $\imath_\frkl:G(\QQ_l)\stackrel{\sim}{\to}\GL_n(\QQ_l)$. 
Fix a maximal compact subgroup $\calk=\prod_q\calk_q^{}$ of $\calg(\widehat{\QQ})$ such that $\imath_\frkl(\calk_l)=\GL_{n}(\frko_{E,\frkl})$ for every split prime $l$. 

Define open subgroups of $\GL_n(\frko_{E,\frkl})$ by \index{$\calk_0^{(n)}(\frkp^\ell)$}\index{$\calk_1^{(n)}(\frkl^k)$}
\begin{align*}
\calk_0^{(n)}(\frkl^k)&=\{(g_{ij})\in\GL_n(\frko_{E,\frkl})\;|\;g_{nj}\in \frkl^k\text{ for }j=1,2,\dots,n-1\}, \\ 
\calk_1^{(n)}(\frkl^k)&=\{(g_{ij})\in\calk_0^{(n)}(\frkl^k)\;|\;g_{nn}-1\in \frkl^k\}. 
\end{align*}
Fix a positive integer $N$ whose prime factors are split in $E$. 
We take an ideal $\frkN$ of $\frko_E$ such that $\frko_E/\frkN\simeq\ZZ/N\ZZ$. 
Let  
\begin{align*}
\calk_0(\frkN)&=\{(g_q)\in\calk\;|\; \imath_\frkl(g_l^{})\in\calk_0^{(n)}(N\frko_{E,\frkl})\text{ for }l|N\}, \\ \calk_1(\frkN)&=\{(g_q)\in\calk\;|\; \imath_\frkl(g_l^{})\in\calk_1^{(n)}(N\frko_{E,\frkl})\text{ for }l|N\}
\end{align*} 
be open compact subgroups of $\calg(\widehat{\QQ})$. 

For each positive integer $\ell$ we define open subgroups of $\GL_n(\frko_{E,\frkp})$ by \index{$\cali_0^{(n)}(\frkp^\ell)$}
\begin{align*}
\cali_0^{(n)}(\frkp^\ell)&=\{(g_{ij})\in\GL_n(\frko_{E,\frkp})\;|\;g_{ij}\in \frkp^\ell\text{ for }i>j\}, \\ 
\cali_1^{(n)}(\frkp^\ell)&=\{(g_{ij})\in\cali_0^{(n)}(\frkp^\ell)\;|\;g_{ii}-1\in \frkp^\ell\text{ for }i=1,2,\dots,n\} 
\end{align*}
and open compact subgroups of $\calg(\widehat{\QQ})$ by \index{$\calk_0(p^\ell\frkN)$, $\calk_1(p^\ell\frkN)$}
\begin{align*}
\calk_0(p^\ell\frkN)&=\{(g_q)\in\calk_0(\frkN)\;|\; \imath_\frkp(g_p^{})\in \cali_0^{(n)}(\frkp^\ell)\}, \\ 
\calk_1(p^\ell\frkN)&=\{(g_q)\in\calk_1(\frkN)\;|\; \imath_\frkp(g_p^{})\in \cali_1^{(n)}(\frkp^\ell)\}. 
\end{align*}

Let $\calo$ be a $p$-adic ring. 
Put $\Lam=\calo\powerseries{\ZZ_p^\times}$. 
We say that an $\calo$-algebra homomorphism $Q:\Lam\to\overline{\QQ}_p$ is locally algebraic if its restriction to $\ZZ_p^\times$ is of the form $Q(t)=t^{k_Q}\eps_Q(t)$ with $k_Q\in\ZZ$ and a finite order character $\eps_Q$. 
Let $\frkX_\Lam$ be the set of locally algebraic homomorphisms of $\Lam$. 

Put $\Lam_n=\calo\powerseries{T_n(\ZZ_p)}$. \index{$\Lam_n$}
Let $\calr$ be a normal ring finite flat over $\Lam_n$. 
We say that an $\calo$-algebra homomorphism $\ulQ :\calr\to\overline{\QQ}_p$ is locally algebraic if its restriction to $T_n(\ZZ_p)$ is of the form $\ulQ(t)=\prod_{i=1}^nQ_i(t_i)$ with $Q_i\in\frkX_\Lam$ for $t=\diag(t_1,\dots,t_n)$. 
We write $\ulQ=(Q_1,Q_2,\dots,Q_n)$ and call $k_{\ulQ}:=(k_{Q_1},\dots,k_{Q_n})\in\ZZ^n$ the weight of $\ulQ$. \index{$k_{\ulQ}$, $\eps_{\ulQ}$}
Let $\frkX_\calr$ be the set of locally algebraic homomorphisms of $\calr$. 

For $\ulQ\in\frkX_\calr$ we put $\ell=\max\{1,c(\eps_{Q_1}),c(\eps_{Q_2}),\dots,c(\eps_{Q_n})\}$ and define finite order characters of $\calk_0(p^\ell\frkN)$ by
\[\eps_{\ulQ}(t_p)=\eps_{Q_1}(t_{p,1})\eps_{Q_2}(t_{p,2})\cdots\eps_{Q_n}(t_{p,n}), \]
where $\imath_\frkp(t_p)=\diag(t_{p,1},t_{p,2},\dots,t_{p,n})$ with $t_{p,i}\in\ZZ_p^\times$. 
Let $\wp_{\ulQ}$ be the ideal of $\calr$ corresponding to $\ulQ$ and $\calr(\ulQ)$ the image of $\calr$ under $\ulQ$. \index{$\wp_{\ulQ}$, $\calr(\ulQ)$}



\subsection{$\Lam_n$-adic families on $\GU(r,1)$}\label{ssec:29}

Let $F$ be a number field which contains $E'$ and the values of $\chi$. 
We denote the $p$-adic closure of $\iot_p(\iot_\infty(\frko_F))$ in $\CC_p$ by $\calo$. 
Put $\scro_F=\frko_{F,(p)}$. 
For a character $\chi$ of $\calk_0(p^\ell\frkN)$ we define 
\begin{align*}
S_{\ulk}^\calg(p^\ell\frkN,\chi,\scro_F)&=\{f\in S_{\ulk}^\calg(\calk_1(p^\ell\frkN),\scro_F)\;|\;r(u)f=\chi(u)^{-1}f\;(u\in\calk_0(p^\ell\frkN))\}, \\
S_{\ulk}^\calg(p^\ell\frkN,\chi;\calo)&=S_{\ulk}^\calg(p^\ell\frkN,\chi;\scro_F)\otimes\calo. 
\end{align*}
For a local ring $R$ finite flat over $\calo$ we set 
\[\bdse S_{\ulk}^\calg(p^\ell\frkN,\chi;R)=\bdse S_{\ulk}^\calg(p^\ell\frkN,\chi;\calo)\otimes R. \]

Let $\Box$ be the set of prime factors of $pN$. 
We denote the class number of $E$ by $h_E$. 
Fix a set $C_E:=\{\frka_1,\frka_2,\dots,\frka_{h_E}\}$ of representatives of the ideal class group of $E$ which consists of prime-to-$pN$ fractional ideals of $\frko_E$. 
Take a finite id\`{e}le $\alp_i\in\widehat{E}^\times$ so that $\alp_i\frko_E=\frka_i$ and $\alpha_q=1$ for all $q\in\Box$ for $i=1,2,\dots,h_E$. 

We consider the finite set 
\[\underline{\cald}(U)=\calg(\QQ)\bsl\calg(\widehat{\QQ})/U  \]
of double cosets. 
We can take the set $\cald(U)$ of representatives for $\underline{\cald}(U)$ from $\calm(\widehat{\QQ}^\Box)$ by strong approximation in special unitary groups. 
For $f\in S_{\ulk}^\calg(p^\ell\frkN,\chi;\calo)$ and $g\in\cald(U)$ we write 
\[\FJ^m_g(f)_0^{}(w)=l_{\ulk}\left(\overrightarrow{\FJ}^m_g(w,f)\right)=\ell_{\ulk}\left(\bfv_{\ulk}\otimes\overrightarrow{\FJ}^m_g(w,f)\right). \]
One can similarly define the CM abelian scheme $\calz_g^\circ$ and the line bundle $\frkL_g^m$ and extend Proposition \ref{prop:31} and Corollary \ref{cor:31} to Fourier-Jacobi coefficients $\FJ^m_g(f)_0^{}$ at those cusps (see \S 2.7.4 of \cite{MH2}). 

\begin{remark}\label{rem:42}
If $n$ is even and $\nu(U)=\widehat{\ZZ}^\times$, then we can take $\cald(U)$ from the set $\{\bfd(\alp_i)\}_{i=1}^{h_E}$. 
If $n$ is odd and $f\in S_{\ulk}^\calg(p^\ell\frkN,\chi;\calo)$ has central character, then $f$ is uniquely determined by $\overrightarrow{\FJ}_{\frka_i}^m(w,f)$ for $m\in\scrs_{\frka_i}^+$ and $i=1,2,\dots,h_E$ (see Remark \ref{rem:20}). 
\end{remark}


Taking Corollary \ref{cor:31} into account, we define the notion of $p$-adic families of cusp forms in terms of Fourier-Jacobi coefficients.  

\begin{definition}
An ordinary $\Lam_n$-adic cusp form $\bdsf$ on $\calg$ of tame level $\frkN$ and nebentypus $\chi$ is a collection of theta functions $\Tht_g^m\in\rmH^0(\calz_g^\circ,\frkL_g^m)\otimes_\scro\Lam_n^{}$ indexed by $g\in\cald(U)$ and $0<m\in\scrs_g$ 
 such that for each sufficiently regular weight $\ulQ\in\frkX_{\Lam_n}$ there is a unique ordinary Picard cusp form 
\[\bdsf_{\ulQ}\in \bdse S_{k_{\ulQ}}^\calg(p^{\ell_{\ulQ}}\frkN,\chi\eps_{\ulQ};\Lam_n(\ulQ))\] 
which satisfies
\[\prod_{j=1}^{r-1}\biggl(\frac{\Ome_{j,\infty}}{2\pi\sqrt{-1}\Ome_{j,\frkp}}\biggl)^{k_{Q_{j+1}}}\cdot\;\FJ_g^m(\bdsf_{\ulQ})_0^{}=\ulQ(\Tht_g^m). \]

\end{definition}
\begin{remark}
\begin{enumerate}
\item The above definition of $\Lam_n$-adic forms is a restatement of \cite[Definition 4.24]{MH2}. 
\item When $r=2$, we will write 
\begin{align*}
\Omega_\infty&=\Omega_{\calg_0,\infty}\in \CC^\times, & 
\Omega_\frkp&=\Omega_{\calg_0,\frkp}\in \WW^\times
\end{align*} 
for the complex and $p$-adic CM periods. \index{$\Ome_\infty$}\index{$\Ome_\frkp$}
We remark that $(\Omega_\infty,\Omega_\frkp)$ are precisely the periods defined in \cite[(4.4 a,b) page 211]{HT93} up to a $p$-adic unit. 
\end{enumerate}
\end{remark}


\subsection{Classicality}

Let $\bfS^\calg(\frkN,\chi,\Lambda_n)$ be the $\Lambda_n$-module of ordinary $\Lambda_n$-adic cusp forms on $\calg$ of tame level $\frkN$ and nebentypus $\chi$, which is equipped with the action of $\Lambda_n$-adic Hecke algebra $\bfT(N,\chi,\Lambda_n)$. 
The ordinary projector $\bdse=\displaystyle\lim_{m\to\infty}\calu_p^{m!}$ converges in $\End_{\Lambda_n}(\bfS^\calg(\frkN,\chi,\Lambda_n))$. 
Put 
\[\bfS_\ord^\calg(\frkN,\chi,\Lambda_n)=\bdse\bfS^\calg(\frkN,\chi,\Lambda_n). \]
Let $\bfT_{\ord}(N,\chi,\Lam_n)$ be the image of $\bfT(N,\chi,\Lambda_n)$ in $\End_{\Lambda_n}(\bfS_\ord^\calg(\frkN,\chi,\Lambda_n))$. 

For a normal ring $\calr$ finite flat over $\Lam_n$ we put 
\[\bfS^\calg_{\ord}(\frkN,\chi,\calr)=\bfS^\calg_{\ord}(\frkN,\chi,\Lambda_n)\otimes_{\Lam_n}\calr. \index{$\bfS^\calg_{\ord}(\frkN,\chi,\calr)$}\]
The set of classical points is defined by 
\[\frkX^\cls_\calr=\{\ulQ\in\frkX_\calr^{}\;|\;k_{Q_1}\leq\dots\leq k_{Q_r}<k_{Q_{r+1}}-r\}. \index{$\frkX^\cls_\calr$}\]

\begin{theorem}[Control theorem]\label{thm:41}
\begin{enumerate}
\item $\bfS_\ord^\calg(\frkN,\chi,\calr)$ is an $\calr$-module of finite rank. 
\item For $\ulQ\in \frkX^{\rm cls}_\calr$ the specialization induces an isomorphism 
\begin{align*}
\bfS_\ord^\calg(\frkN,\chi,\calr)/\wp_{\ulQ}&\stackrel{\sim}{\to} \bdse S_{k_{\ulQ}}^\calg(p^\ell\frkN,\chi\eps_{\ulQ};\calr(\ulQ)), \\
\bdsf&\mapsto \bdsf_{\ulQ}. \index{$\bdsf_{\ulQ}$}
\end{align*}
\end{enumerate}
\end{theorem}

\begin{proof} 
Part (1) is proved by Hida (see also Theorem 4.21 of \cite{MH2}). 
Part (2) is also due to Hida when $k_{\ulQ}$ is very regular (c.f. Corollary 4.23 of \cite{MH2}). 
In general, the specialization $\bdsf_{\ulQ}$ of $\bdsf$ at $\ulQ\in \frkX^\cls_\calr$ is an ordinary $p$-adic modular form. 
Ordinary $p$-adic modular forms of cohomological weight are overconvergent by Proposition 5.4 and Th\'eor\`eme A.3 of \cite{P12BSMF}, and any overconvergent modular form of weight $k_{\ulQ}=(k_{Q_1},\dots,k_{Q_r};k_{Q_{r+1}})$ with $k_{Q_{r+1}}-k_{Q_r}>r$ is classical by Th\'eor\`eme and Remarque, p.977 of \cite{BPS}. 
We therefore find that $\bdsf_{\ulQ}\in \bdse S_{k_{\ulQ}}^\calg(p^\ell\frkN,\chi\eps_{\ulQ};\calr(\ulQ))$ if $k_{Q_{r+1}}-k_{Q_r}>r$. 
\end{proof}

\begin{definition}[Hida families] 
We call a non-zero $\calr$-adic cusp form $\bdsf\in\bfS^\calg_{\ord}(\frkN,\chi,\calr)$ an $\calr$-adic Hida family if $\bdsf_{\ulQ}$ is a simultaneous eigenform of Hecke operators away from $pN$ for $\ulQ\in\frkX^{\rm cls}_\calr$. 
Let 
\[\frkX_{\bdsf}'=\{\ulQ\in\frkX^{\rm cls}_\calr\;|\;\ulQ(\bdsf)\neq 0\}\index{$\frkX_{\bdsf}'$}\] 
be a Zariski dense subset of the rigid generic fiber of $\mathrm{Spf}\calr$. 
\end{definition}

\begin{lemma}\label{lem:41}
Assume that the space of cuspidal automorphic forms on $\calg$ has multiplicity one. 
If $\bdsf\in \bfS^\calg_{\ord}(\frkN,\chi,\calr)$ is an $\calr$-adic Hida family, then $\vPh_{k_{\ulQ}}(\bdsf_{\ulQ})_0^{}$ (see Remark \ref{rem:21}) generates an irreducible $p$-ordinary cuspidal automorphic representation of $\calg(\AA)$ for $\ulQ\in\frkX_{\bdsf}'$. 
\end{lemma}

\begin{proof}
Let $\til\pi\simeq\otimes_v'\til\pi_v^{}$ be an irreducible constituent of the cuspidal automorphic representation generated by $\vPh_{k_{\ulQ}}(\bdsf_{\ulQ})_0^{}$. 
If $l$ and $Np$ are coprime, then $\til\pi_l$ is determined by the Hecke eigenvalues of $\bdsf_{\ulQ}$. 
Thus $\til\pi$ belongs to the A-packet associated to these eigenvalues. 
Moreover, $\til\pi_q$ is uniquely determined for each prime factor $q$ of $N$ by the assumption on $N$ as the associated local A-packet is a singleton for each split prime. 
Since $\til\pi_\infty$ is the holomorphic discrete series with highest weight $-k_{\ulQ}$, the equivalence class $\til\pi$ is determined by $\bdsf_{\ulQ}$. 
Thus $\vPh_{k_{\ulQ}}(\bdsf_{\ulQ})_0^{}$ generates an irreducible representation of $\calg(\AA)$ by the multiplicity one for unitary groups. 
\end{proof}

\begin{remark}\label{rem:41}
When $n=2$, the multiplicity one holds in view of the accidental isomorphism (see Remark \ref{rem:51}). 
Moreover, the group $\U(1,1)$ has multiplicity one (see \S 6.2 of \cite{LM15}). 
When $n=3$, the multiplicity one holds by Theorem 13.3.1 of \cite{Rog90} and Proposition \ref{prop:51}(\ref{prop:512}) below. 
\end{remark}

If $\bdsf$ is a Hida family, then $\bdsf_{\ulQ}$ is a simultaneous eigenform of the operators $\calu_p(\alp_1)$ and $\calu_p(\bet_j)$ $(1\leq j\leq r)$ for $\ulQ\in\frkX_{\bdsf}'$, and we denote by $\til\bdpi_{\ulQ}$ the cuspidal automorphic representation of $\calg(\AA)$ associated to $\bdsf_{\ulQ}$, which is $p$-ordinary with respect to $\iot_p$ (see Definition \ref{def:23}). \index{$\til\bdpi_{\ulQ}$}
Let $\frkN_{\til\bdpi_{\ulQ}}$ be the largest ideal of $\frko_E$ such that there exists a non-zero automorphic form in $\til\bdpi_{\ulQ}$ fixed by $\calk_1(\frkN_{\til\bdpi_{\ulQ}})$. \index{$N_\pi,\frkN_\pi$}
Such an automorphic form corresponds to an essential vector at each split prime $l$ (see Definition \ref{def:81}). 

\begin{definition}\label{def:41}
The subset $\frkX_{\bdsf}''$ consists of $\ulQ\in\frkX_{\bdsf}'$ such that $\bdsf_{\ulQ}$ is new outside $p$, namely, $\bdsf_{\ulQ}$ corresponds to an essential vector at every split prime $l\neq p$. \index{$\frkX_{\bdsf}''$}
Let us fix a factorization $\til\bdpi_{\ulQ}^{}\simeq\otimes_v'\til\bdpi_{\ulQ,v}^{}$. 
For $\ulQ\in\frkX_{\bdsf}''$ we can write $\vPh_{k_{\ulQ}}(\bdsf_{\ulQ})_0^{}=\otimes_vW_v$ by multiplicity one for new and ordinary vectors (see \cite[Theorem 5.3]{Hid04}), where $W_p=h_{\til\bdpi_{\ulQ},p}^{\ord}\in\til\bdpi_{\ulQ,p}$ is the normalized ordinary vector. 
We call $\bdsf_{\ulQ}^\circ\in S_{k_{\ulQ}}^\calg(\frkN_{\til\bdpi_{\ulQ}},\chi\eps_{\ulQ};\calr(\ulQ))$ the newform associated to $\bdsf_{\ulQ}^{}$ if $\vPh_{k_{\ulQ}}(\bdsf_{\ulQ}^\circ)_0^{}=W_p^\circ\otimes\bigotimes_{v\neq p}W_v^{}$, where $W_p\in\til\bdpi_{\ulQ,p}$ is the normalized essential vector. \index{$\bdsf_{\ulQ}^\circ$}
\end{definition}

\subsection{Algebraic representations of $\GL_2\times\GL_1$}\label{ssec:31}

In our later discussion we will specifically consider the case in which $r=2$ and $s=1$. 
We give an explicit expression of the schematical module $L_{\ulk}$ and the pairing $\ell_{\ulk}$ introduced in \S \ref{subsec:alg} in this case. 

For a commutative ring $R$ of characteristic $0$ and a non-negative integer $\kap$ we write $\call_\kap(R)$ for the space of two variable homogeneous polynomials of degree $\kap$ over $R$. 
The group $\GL_2(R)$ acts on this space $\call_\kap(R)$ by 
\[\rho_\kap(\alp)P(X,Y)=P((X,Y)\alp). \]
For $\lam\in\ZZ$ we will write $\rho_{(\kap+\lam,\lam)}=\rho_\kap\otimes\det^\lam$. 
Define the perfect pairing 
\[\ell_\kap:\call_\kap(R)\otimes \call_\kap(R)\to R\] 
by 
\beq
\ell_\kap(X^{j}Y^{\kap-j}\otimes X^{\prime i}Y^{\prime \kap-i})
=\begin{cases}
(-1)^i\binom{\kap}{i}^{-1} &\text{if $i+j=\kap$, }\\
0 &\text{if $i+j\neq \kap$. }
\end{cases}
\label{tag:41}
\eeq

Let $\gam''$ be a positive definite Hermitian maxtrix of size $2$. 
We view $\U_{\gam''}$ as a subgroup of $\GL_2(\CC)$ and regard $\call_\kap(\CC)$ as an irreducible representation of $\U_{\gam''}$ of dimension $\kap+1$. 
If $\sig$ is an irreducible representation of $\U_{\gam''}$ of dimension $\kap+1$, then there is an integer $\lam$ such that $\sig\simeq\rho_{(\kap+\lam,\lam)}$. 
Let 
\begin{align*}
\ulk&=(k_1,k_2;k_3)\in\ZZ^3, &
&k_1\leq k_2\leq k_3. 
\end{align*}
We write $\call_{\ulk}(R)$ for an $R$-module on which $\GL_2(R)$ acts as the representation $\rho_{(k_2,k_1)}$. \index{$\call_{\ulk}(R)$}
We define the representation $\rho_{\ulk}$ of $\GL_2(R)\times R^\times$ on $\call_{\ulk}(R)$ by 
\[\rho_{\ulk}(\alp,\bet)=\rho_{(k_2,k_1)}(\trs\alp^{-1})\bet^{k_3}
=(\det\alp)^{-k_1}\rho_\kap(\trs\alp^{-1})\bet^{k_3}\] 
for $\alp\in\GL_2(R)$ and $\bet\in R^\times$, where we put $\kap=k_2-k_1$. \index{$\rho_{\ulk}$}

When we view $\call_{\ulk}(\CC)$ as a representation of the maximal compact subgroup $\calk_\infty$ of $\U(2,1)$, it has highest weight $(k_2,k_1;k_3)$ (cf. (\ref{tag:a2})). 
Define the representation $\rho_{\ulk}^\vee$ of $\GL_2(R)\times R^\times$ on $\call_{\ulk}^\vee(R)=\call_\kap^{}(R)$ by 
\begin{align*}
\rho_{\ulk}^\vee(\alp,\bet)
&=(\det\alp)^{k_1}\rho_\kap^{}(J\alp J^{-1})\bet^{-k_3}, & 
J&=\begin{bmatrix} 
0 & 1 \\ -1 & 0 
\end{bmatrix}. 
\end{align*}
We write $\ell_{\ulk}=\ell_\kap$. 
Then for $\alp\in\GL_2(R)$ and $\bet\in R^\times$ and $P,Q\in \call_\kap(R)$ 
\[\ell_{\ulk}(\rho_{\ulk}^{}(\alp,\bet)P\otimes \rho_{\ulk}^\vee(\alp,\bet)Q)=\ell_{\ulk}(P\otimes Q). \index{$\ell_{\ulk}$}\]

\begin{remark}
When $r=2$, the isomorphism $\call_{\ulk}(R)\stackrel{\sim}{\to}L_{\ulk}(R)$ is given by 
\[P\mapsto f_P(g)=(\det g)^{k_1}P((0,1)g). \]
Observe that the isomorphism $\call_{\ulk^\vee}(R)\stackrel{\sim}{\to}L_{\ulk^\vee}(R)$ sends $X^{k_2-k_1}$ to $\bfv_{\ulk}$. 
\end{remark}

Let  
\[\bfP_{\ulk}=(X\otimes Y'-Y\otimes X')^\kap\in \call_{\ulk}^{}(R)\otimes\call^\vee_{\ulk}(R) \index{$\bfP_{\ulk}$}\]
be the vector $\bfP_{\ulk}$ fixed by the diagonal action of $\GL_2(R)$. 
One can easily show that
\beq
\ell_{\ulk}(\bfP_{\ulk})=\dim\rho_{\ulk}=\kap+1. \label{tag:34}
\eeq


\subsection{The restriction of Picard modular forms} 

Recall the morphism $\iot_W:I^0_H(\calu_{L_1}^\ell)\to I^0_G(U^\ell)$ constructed in \S \ref{ssec:pullback}. 
Let $\rho$ be an irreducible representation of $\scrh=\GL_r\times\GL_1$. 
Assume that the restriction of $\rho$ from $\scrh$ to $\scrh'=\GL_{r-1}\times\GL_1$ has an irreducible quotient isomorphic to $\rho'$. 
We denote by $\ome_{\rho'}$ the automorphic sheaf of weight $\rho'$ over $I^0_H(\calu_{L_1}^\ell)$. 
Fix a projection $\vpi_{\rho,\rho'}:\rho\to\rho'$. 
The morphism $\vpi_{\rho,\rho'}$ induces a morphism $r_{\rho,\rho'}:\iot_W^*\ome_\rho\to\ome_{\rho'}$ of sheaves over $I^0_H(\calu_{L_1}^\ell)$. 
See Section 6 of \cite{CEFMV} for a general framework for the unitary Shimura varieties. 
Section 2 of \cite{MHarris} studies the pullback 
\[r_{\rho,\rho'}\circ\iot_W^*:\bfM_\rho^\calg(U^\ell_1,R)=\rmH^0(I^0_\calg(U^\ell_1)_{/R},\ome_\rho)\to\rmH^0(I^0_H((\calu_{L_1})_1^\ell)_{/R},\ome_{\rho'})\] 
of global sections of the automorphic sheaves, which corresponds to the restriction $M_\rho^\calg(U^\ell_1,\CC)\to M_{\rho'}^H((\calu_{L_1})_1^\ell,\CC)$ in view of (\ref{tag:27}) if $R=\CC$. 

We will explicate this map in the case when 
\begin{align*} 
r&=2, & 
\rho&=\rho_{\ulk}, & \ulk&=(k_1,k_2;k_3), \\
& &
\rho'&=\rho_{\ulk'}, & \ulk'_i&=(k_1+i;k_3)
\end{align*} 
for $i=0,1,2,\dots,k_2-k_1$. 
Henceforth we will abbreviate
\begin{align*}
\kap&:=k_2-k_1, &
\frkD&:=\frkD_{2,1}, & 
\frkH&:=\frkD_{1,1}. \index{$\frkD,\frkH$}
\end{align*}
 
We write $f\in\bfM_{\ulk}^\calg(U^\ell_1,R)$ as 
\[f(\underline{A},j,\bdome)=\sum_{i=0}^\kap f_i(\underline{A},j,\bdome)X^iY^{\kap-i} \]
for $R$-quintuples $(\underline{A},j)\in I^0_G(U^\ell)$. 
Let 
\[f(Z,g):=f(\underline{\cala}(V)_g(Z),\bdome_g(Z))\in M_{\ulk}^\calg(U^\ell_1,R)\]
(see Lemma \ref{lem:cpx}). 
For $\bet\in\calg(\widehat{\QQ})$ and $m\in\QQ^\times_+$ we will write 
\[\overrightarrow{\FJ}^m_\beta(w,f)=\sum_{i=0}^\kap\FJ^m_\bet(f)_i(w)X^iY^{\kap-i}. \]
We now consider the pullback of $f$ via the morphism $\iot_W$. 

\begin{definition}
Fix a fractional ideal $\frka$ of $E$. 
Let $E_0(\CC)\stackrel{\sim}{\to}\CC(\vSi^c)/\iot_\infty^c(\frka)$ be a complex elliptic curve, which can be descended to an elliptic curve $E_0$ over a number field $F$ by the theory of complex multiplication. 
Since $E_0$ has potentially good reduction everywhere at all finite places of $F$ by Proposition VII.5.5 and Corollary C.11.1.1(b) of \cite{Sil09}, by enlarging the base field $F$, we can further descend $E_0$ to an abelian scheme $\scre_0$ over the ring $\frko_{F,\frkP}$ obtained by the localization $\frko_F$ of at some prime ideal $\frkP$ above $p$. 
Fix a N\'{e}ron differential $\bdome_0$ of $\scre_0$ defined over $\frko_{F,\frkP}$. 
The complex period $\Ome_\infty$ is an element of $\CC^\times/\overline{\QQ}$ defined by $\bdome_0=\Ome_\infty\d w$, where $\d w$ is the standard invariant differential on $\CC(\vSi^c)/\iot_\infty^c(\frka)$. 
\end{definition}

Define $\iot_W':\cale_{\underline{B}}\to\cale_{\iot_W(\underline{B})}$ by $\iot_W'(\bdome^+;\bdome^-)=(\bdome^+,\bdome_0;\bdome^-)$. 

\begin{proposition}\label{prop:32}
Let $f\in\bfM_{\ulk}^\calg(U^\ell_1,R)$. 
If $R$ is an $\scro$-subalgebra of $\CC$, then 
\[(2\pi\sqrt{-1}/\Ome_\infty)^{i-k_2}\FJ^m_\bet(f)_i(0)\in R\]
for $m\in\QQ^\times_+$, $\bet\in\calg(\widehat{\QQ})$, and $i=0,1,2,\dots,k_2-k_1$. 
\end{proposition}

\begin{proof}
Let $R$ be a $\overline{\ZZ}_{(p)}$-algebra. 
For an $R$-point $(\underline{B},\jmath)\in I^0_H(\calu_{L_1}^\ell)$ and one forms $\bdome^\pm\in\cale^\pm_{\underline{B}}$ defined over $R$, we have 
\[f_i(\iot_W(\underline{B},\jmath),\bdome^+,\bdome_0;\bdome^-)\in R. \]
Given $a,b,c\in R^\times$, we have 
\begin{align*}
&f(\iot_W(\underline{B},\jmath),a\bdome^+,b\bdome_0;c\bdome^-)\\
=&\rho_{\ulk}\biggl(\begin{bmatrix} a^{-1} & 0 \\ 0 & b^{-1} \end{bmatrix},c^{-1}\biggl)f(\iot_W(\underline{B},\jmath),\bdome^+,\bdome_0;\bdome^-)\\
=&\sum_{i=0}^{k_2-k_1} a^{k_1+i}b^{k_2-i}c^{-k_3}f_i(\iot_W(\underline{B},\jmath),\bdome^+,\bdome_0;\bdome^-)X^iY^{\kap-i}. 
\end{align*}
Thus $((\underline{B},\jmath),\bdome)\mapsto f_i(\iot_W^{}(\underline{B},\jmath),\iot_W'(\bdome))$ is a geometric modular form on $H$ of weight $(k_1+i;k_3)$. 

It follows from (\ref{tag:27}) that 
\begin{align*}
f_i(\jmath(\tau),\iot(h))&=f_i(\underline{\cala}(V)_{\iot(h)}(\jmath(\tau)),\bdome_{\iot(h)}(\jmath(\tau)))\\
&=f_i(\iot_W(\underline{\cala}(W)_h(\tau)),\bdome_h^+(\tau)_1,2\pi\sqrt{-1}\d w;\bdome_h^-(\tau)_2)\\
&=(2\pi\sqrt{-1}/\Ome_\infty)^{k_2-i}f_i(\iot_W^{}(\underline{\cala}(W)_h(\tau)),\iot_W'(\bdome_h(\tau))). 
\end{align*}
Let $(\underline{\scrm}_h,\jmath_{\scrm_h})$ be the Mumford quintuple at the cusp labeled by $h$, and $\bdome_{\scrm_h}=(\d^\times t'_+;\d^\times t_-)\in\cale_{\underline{\scrm}_h}$. 
Since the analytic and algebraic Fourier expansions coincide, we have 
\[(2\pi\sqrt{-1}/\Ome_\infty)^{i-k_2}\sum_m\FJ^m_{\iot(h)}(0)_iq^m=f_i(\iot_W^{}(\underline{\scrm}_h,\jmath_{\scrm_h}),\iot_W'(\bdome_{\scrm_h}))\in R\powerseries{q}, \]
which completes our proof. 
\end{proof}


\part{Analytic Theory}\label{part:2}


\section{Modular forms on $\U(1,1)$ and $\U(2,1)$}\label{sec:5}


\subsection{The basic setting}\label{ssec:51}

In this section we let $s=1$ and $\gam_0=\del$ in the setting of \S \ref{ssec:26}. 
Thus $G$ is a quasi-split unitary group in three variables. 
Let 
\begin{align*}
G&=\{g\in\Res_{E/\QQ}\GL_3\;|\;gS_\del\trs g^c=S_\del\}, & S_\del&=\begin{bmatrix} 0 & 0 & -1 \\ 0 & -\del & 0 \\ 1 & 0 & 0 \end{bmatrix}, \\ 
H&=\{h\in\Res_{E/\QQ}\GL_2\;|\;hJ\trs h^c=J\}, & J&=\begin{bmatrix} 0 & 1 \\ -1 & 0 \end{bmatrix}
\end{align*}
be quasi-split unitary groups in three and two variables. \index{$G$}\index{$H$}
The associated unitary similitude groups are denoted by $\calg=\GU(T)$ and $\calh=\GU(J)$. \index{$\calg$}\index{$\calh$}

We define the embedding $\iot:H\hookrightarrow G$ by 
\[\iot\left(\begin{bmatrix} a & b \\ c & d \end{bmatrix}\right)=\begin{bmatrix} a & 0 & b \\ 0 & 1 & 0 \\ c & 0 & d \end{bmatrix}. \index{$\iot$}\]
We regard $H$ as the subgroup of $G$ via $\iot$. 
Recall the embedding $\jmath:\frkH\hookrightarrow\frkD$ defined by $\tau\mapsto\begin{bmatrix} \tau \\ 0 \end{bmatrix}$ is compatible with the embedding $\iot$ in the sense of (\ref{tag:26}). \index{$\jmath$}
The reader is reminded that for $Z=\begin{bmatrix} \tau \\ w \end{bmatrix}\in\frkD$
\begin{align*}
\eta(Z)&=\sqrt{-1}(\tau^c-\tau-w^c\del^{-1} w), & 
\eta(\tau)&=\sqrt{-1}(\tau^c-\tau). 
\end{align*}

We will frequently add $'$ to the notation for various objects to indicate that they are attached to $T'$. 
Recall the standard Borel subgroups $\calp=\calm\caln$ of $\calg$ and $\calp'=\calm'\caln'$ of $\calh$. \index{$\calp$}\index{$\calm$}\index{$\caln$}
We denote the Borel subgroups of $G$ and $H$ by $\bfP=\bfM\caln$ and $\bfP'=\bfM'\caln'$, and the additive and multiplicative groups in one variable over $\QQ$ by $\GG_a=\Mat_1$ and $\GG_m=\GL_1$. \index{$\bfP$}\index{$\bfM$}
Define isomorphisms 
\begin{align*}
&\bfm:\Res_{E/\QQ}\GG_m\times\Res_{E/\QQ}\GG_m\stackrel{\sim}{\to}\calm, & 
&\bfm':\Res_{E/\QQ}\GG_m\times\GG_m\stackrel{\sim}{\to}\calm', \\
&\bfn:\Res_{E/\QQ}\GG_a\ltimes\GG_a\stackrel{\sim}{\to}\caln, & 
&\bfn':\GG_a\stackrel{\sim}{\to}\caln'
\intertext{by}
&\bfm(a,b)=\diag(a,b,(a^{-1})^cbb^c), &
&\bfm'(a,t)=\diag(a,t(a^{-1})^c), \\
&\bfn(w,z)=\begin{bmatrix} 1 & -\del^{-1} w^c & z-\frac{w^c w}{2\del} \\ 0 & 1 & w \\ 0 & 0 & 1 \end{bmatrix}, & 
&\bfn'(z)=\begin{bmatrix} 1 & z \\ 0 & 1 \end{bmatrix}. 
\end{align*}
We will write $\bfd(a)=\bfm(a^c,1)$ and $\bfd'(a)=\bfm'(a^c,1)$. \index{$\bfm(a,b)$}\index{$\bfn(w,z)$}\index{$\bfd(a)$}
Put 
\begin{align*}
K&=\calk\cap G(\widehat{\QQ}), & 
K_l&=\calk_l\cap G(\QQ_l), & 
K_i(p^\ell\frkN)&=\calk_i(p^\ell\frkN)\cap G(\widehat{\QQ})  
\end{align*} 
for $i=0,1$. \index{$K_0(p^\ell\frkN)$, $K_1(p^\ell\frkN)$}
Define the maximal compact subgroup $K'=\prod_l^{}K'_l$ of $H(\widehat{\QQ})$ by $K'_l=H(\QQ_l)\cap\GL_2(\frko_{E,l})$. 
We similarly define an open subgroup $K_i'(p^\ell\frkN')$ of $K'$. \index{$K_l^{}$, $K_l'$}


\subsection{Restriction from $\GU(1,1)$ to $\U(1,1)$}\label{ssec:52}

Fix a finite prime $q$ of $\QQ$. 
For $t\in\QQ_q^\times$ we define the character $\addchar_q^t:\QQ_q^{}\to\CC^1$ by $\addchar_q^t(x)=\addchar_q^{}(tx)$. 
Given an admissible representation $\sig$ of $H(\QQ_q)$, we put 
\begin{align*}
\Wh^t(\sig)&=\Hom_{\QQ_q}(\sig\circ\bfn',\addchar_q^t), & 
\Wh(\sig)&=\{t\in\QQ_q^\times\;|\;\Wh^t(\sig)\neq 0\}
\end{align*}
We say that $\sig$ is $\addchar_q$-generic if $1\in\Wh(\sig)$. 

We associate to characters $\mu:E_q^\times\to\CC^\times$ and $\rho:\QQ_q^\times\to\CC^\times$ the space $\til I'(\mu,\rho)$ which consists of functions $f':\calh(\QQ_q)\to\CC$ satisfying 
\[f'(\bfm'(a,t)uh)=\mu(a)\rho(t)|t^{-1}\Nr_{E_q/\QQ_q}(a)|_{\QQ_q}^{1/2}f'(h)\]
for $a\in E_q^\times$, $t\in\QQ_q^\times$, $u\in\caln'(\QQ_q)$ and $h\in\calh(\QQ_q)$, where $\calh$ acts by right translation.
We denote its restriction to $H(\QQ_q)$ by $I'(\mu)$.  

\begin{lemma}\label{lem:51}
Assume that $q$ does not split in $E$ and that $\mu,\rho$ are unitary. 
\begin{enumerate}
\item\label{lem:511} $\til I'(\mu,\rho)$ is irreducible. 
\item\label{lem:512} $I'(\mu)$ is reducible if and only if $\mu|_{\QQ_q^\times}=\eps_{E_q/\QQ_q}$. 
\item\label{lem:513} If $\mu|_{\QQ_q^\times}=\eps_{E_q/\QQ_q}$, then $I'(\mu)=\sig_+\oplus\sig_-$, where 
\begin{align*}
\Wh(\sig_+)&=\Nr_{E_q/\QQ_q}(E_q^\times), & 
\Wh(\sig_-)&=\QQ_q^\times\setminus\Nr_{E_q/\QQ_q}(E_q^\times). 
\end{align*} 
If $\mu$ is unramified, then $\sig_+$ admits a non-zero $K_q'$-invariant vector. 
\end{enumerate}
\end{lemma}

\begin{proof}
Since the restriction of $\til I'(\mu,\rho)$ to $\GL_2(\QQ_q)$ is the principal series associated to a unitary character of $T_2(\QQ_q)$, it is irreducible and so is $\til I'(\mu,\rho)$. 
Part (\ref{lem:512}) is a special case of Theorems 1.1, 1.2 and Lemma 4.4 of \cite{KS}. 
We will prove (\ref{lem:513}). 
Note that $\sig_+$ is a theta lift of the trivial representation of the orthogonal group associated to the norm form of $E_q$. 
If $E_q/\QQ_q$ is unramified, then it admits a non-zero $K_q'$-invariant vector. 
\end{proof}

\begin{remark}\label{rem:51}
\begin{enumerate}
\item\label{rem:511} The group $\calh$ is the quotient of $\GL_2\times\Res_{E/\QQ}\GL_1$ by $\GL_1$ embedded diagonally (cf. \cite[Lemma 4.4]{Shimura97}, \cite[\S 3.1]{H93}). 
An irreducible representation of $\calh(\QQ_v)$ is given by a pair $(\sig^\star,\chi)$, where $\chi$ is a character of $E_v^\times$ and $\sig^\star$ is an irreducible representation of $\GL_2(\QQ_v)$ whose central character is $\chi^{-1}|_{\QQ_v^\times}$. 
The functorial lift of $(\sig^\star,\chi)$ is the base change of $\sig^\star$ to $\GL_2(E_v)$ twisted by $\chi$. 
\item\label{rem:512} Put $\til\sig=\til I'(\mu,\rho)$. 
The restriction of $\til\sig$ to $\GL_2(\QQ_q)$ is the principal series $I(\rho\cdot\mu|_{\QQ_q^\times},\rho)$. 
If $E_q\simeq\QQ_q\oplus\QQ_q$ and if we write $\mu=(\mu_1,\mu_2)$, then $\til\sig\circ\imath_\frkq^{-1}\simeq I(\mu_1^{},\mu_2^{-1})$. 
Observe that $\imath_\frkq^{-1}(h)=(h,J\trs h^{-1}J^{-1})=(h,(\det h)^{-1}h)$ for $h\in\GL_2(\QQ_\frkq)$, where $J=\begin{bmatrix} 0 & 1 \\ -1 & 0 \end{bmatrix}$. \index{$J$}
\end{enumerate}
\end{remark}

\begin{proposition}\label{prop:51}
\begin{enumerate}
\item\label{prop:511} An irreducible representation of $\calg(\QQ_q)$ remains irreducible upon restriction to $G(\QQ_q)$.  
\item\label{prop:512} An irreducible automorphic representation $\til\pi$ of $\calg(\AA)$ remains irreducible upon restriction to $G(\AA)$. 
Moreover, if $\til\vph\in\til\pi$ is not zero, then its restriction to $G(\AA)$ is not zero. 
\end{enumerate}
\end{proposition}

\begin{proof}
Part (\ref{prop:511}) is a special case of Theorem 9(a) of \cite{AP}. 
The result stays same for $\calg(\RR)$, from which the first part of (\ref{prop:512}) follows. 
The restriction gives a non-zero $G(\AA)$-equivariant map $\til\pi\to\scra(G)$, which is injective (cf. Remark \ref{rem:20}). 
\end{proof}


\subsection{Hida families on $\U(1,1)$}\label{ssec:53}

\begin{definition}\label{def:51}
We call a function $g:\frkH\times H(\widehat{\QQ})\to\CC$ a modular form on $H$ of weight $\ulk'=(k_1';k_2')$, level $p^\ell\frkN'$ and nebentypus $\chi'$ if it is holomorphic in $\tau$ and satisfies 
\beq
g(\gam\tau,\gam hu)=\chi'(u)^{-1}\lam(\gam,\tau)^{-k_1'}\mu(\gam,\tau)^{k_2'}g(\tau,h) \label{tag:51}
\eeq
for $\gam\in H(\QQ)$, $h\in H(\widehat{\QQ})$ and $u\in K_0'(p^\ell\frkN')$. 
A modular form $g$ is called a cusp form if it has a Fourier-Jacobi expansion 
\[g(\tau,h)=\sum_{m\in\QQ^\times_+}\FC^m_h(g)e^{2\pi\sqrt{-1}m\tau} \] 
for every $h\in H(\widehat{\QQ})$. 
We write $S_{\ulk'}^H(p^\ell\frkN',\chi',\CC)$ for the space of cusp forms on $H$ of weight $\ulk'$, level $p^\ell\frkN'$ and nebentypus $\chi'$. 
\end{definition}

\begin{remark}\label{rem:52}
By (\ref{tag:24}) the identity (\ref{tag:51}) is equivalent to the following identity
\[g(\gam\tau,\gam hu)=\chi'(u)^{-1}(\det\gam)^{k_1'}\mu(\gam,\tau)^{k_2'-k_1'}g(\tau,h). \]
\end{remark}

Recall the automorphic form $\vPh_{\ulk'}(g)$ on $H$ associated to $g$ defined by 
\[\vPh_{\ulk'}(g)(h)=\lam(h_\infty,\sqrt{-1})^{k_1'}\mu(h_\infty,\sqrt{-1})^{-k_2'}g(h_\infty(\sqrt{-1}),\hat h)\]
for $h=(h_\infty,\hat h)$ with $h_\infty\in H(\RR)$ and $\hat h\in H(\widehat{\QQ})$.  


Let $\Box'$ be the set of prime factors of $pN'$.  

\begin{definition}\label{def:52}
For any subring $\calo$ of $\CC$ the space $S_{\ulk'}^H(p^\ell\frkN',\chi';\calo)$ consists of $g\in S_{\ulk'}^H(p^\ell\frkN',\chi',\CC)$ such that $\FC^m_\alp(g)\in\calo$ for every $\alp\in \bfM'(\widehat{\QQ}^{\Box'})$. 
For a local ring $R$ finite flat over $\calo$ we set 
\[S_{\ulk'}^H(p^\ell\frkN',\chi',R)=S_{\ulk'}^H(p^\ell\frkN',\chi',\calo)\otimes_\calo R. \index{$S_{\ulk'}^H(p^\ell\frkN',\chi',R)$}\]
\end{definition}

We write $\alp=\bfd'(a)\in \bfM'(\widehat{\QQ}^{\Box'})$ and put $\frka=a\frko_E$. 
Since $\FC^m_\alp(g)$ depends only on the fractional ideal $\frka$, we will denote it by $\FC^m_\frka(g)$. \index{$\FC^m_\frka(g)$}
Recall that 
\[[\calu_p(\bet_1)g](\tau,\alp)=p^{-1-k_1'}\sum_{z\in \ZZ_p/p\ZZ_p}f\biggl(\tau,\alp\imath_\frkp^{-1}\biggl(\begin{bmatrix} p & z \\ 0 & 1 \end{bmatrix}\biggl)\biggl)\]
(see \S \ref{ssec:25}). 
Take a positive integer $j$ so that $(\frkp^c)^j$ is a principal ideal. 
Let $\xi$ be a generator of $(\frkp^c)^j$. 
Then we can show that 
\[\FC^m_\frka(\calu_p(\bet_1)^jg)=p^{-jk_1'}(\xi^c)^{k_1'}\xi^{-{k_2'}}\FC^{mp^j}_\frka(g) \]
for $m\in\scrs_\frka^+$ at the cost of replacing $j$ with its suitable multiple if necessary (cf. Lemma \ref{lem:31}). 
It follows that 
\beq
\FC^m_\frka(\bdse'g)=\lim_{j\to\infty}\FC^{mp^{j!}}_\frka(g).  \label{tag:52}
\eeq

\begin{definition}\label{def:53}
An ordinary $\Lambda_2$-adic cusp form $\bdsg$ of tame level $\frkN'$ and nebentypus $\chi'$ is a collection of $A_\frka^m\in\Lam_2$  
indexed by $0<m\in(\frka\frka^c)^{-1}$ and fractional ideals $\frka$ of $\frko_E$ prime to $pN'$ such that for each $\ulQ'\in\frkX_{\calr'}^\cls$ 
\[\bdsg_{\ulQ'}(\tau,\alp):=\sum_m\ulQ'(A^m_\frka)q^m\in\bdse'S_{k_{\ulQ'}}^H(p^{\ell_{\ulQ'}}\frkN',\chi'\eps_{\ulQ'};\calr'(\ulQ')). \]
We will also write $\FC^m_\frka(\bdsg)=A^m_\frka$. 
Let $\bfS^H_{\ord}(\frkN',\chi',\Lambda_2)$ be the $\Lambda_2$-module of those $\Lambda_2$-adic cusp forms. 
\end{definition}

Let $\calr'$ be a normal ring finite flat over $\Lam_2$. 
Put 
\[\bfS^H_{\ord}(\frkN',\chi',\calr')=\bfS^H_{\ord}(\frkN',\chi',\Lambda_3)\otimes\calr'. \index{$\bfS^H_{\ord}(\frkN',\chi',\calr')$}\]
Let $\bdsg\in\bfS^\calh_{\ord}(\frkN',\chi',\calr')$ be an $\calr'$-adic Hida family on $\calh$. 
Since 
\beq
\calh(\AA)=\calh(\QQ)\RR^\times H(\AA)\calk'_0(p^\ell\frkN'), \label{tag:53}
\eeq
the restriction of $\bdsg_{\ulQ'}$ to $\frkH\times H(\widehat{\QQ})$ is not zero and denoted also by $\bdsg_{\ulQ'}$. 
We therefore view $\bdsg$ as an element of $\bfS^H_{\ord}(\frkN',\chi',\calr')$. 


\subsection{Automorphic representations associated to Hida families}\label{ssec:54}

In view of Remark \ref{rem:51}(\ref{rem:511}) we can associate to a Hida family 
\[\bdsg=\{A^m_\frka\}_{\frka\in C_E,\;m\in\scrs_\frka^+}\in\bfS^\calh_{\ord}(\frkN',\chi',\calr')\] 
an $\calr'$-adic Galois representation $\rho_{\bdsg}:\Gamma_E\to \GL_2(\calr')$ unramified outside primes dividing $N'p$.
For $\ulQ'\in\frkX_{\calr'}^\cls$ the specialization $\rho_{\bdsf}\otimes_{\calr',\ulQ'}\overline{\QQ_p}$ is associated with a $p$-ordinary irreducible cuspidal automorphic representation $\til\bdsig_{\ulQ'}\simeq\otimes'_v\til\bdsig_{\ulQ',v}$ of $\calh(\AA)$. 
Let $\bdsig_{\ulQ'}\simeq\otimes'_v\bdsig_{\ulQ',v}$ be the irreducible component of the restriction of $\til\bdsig_{\ulQ'}$ to $H(\AA)$ such that $\bdsig_{\ulQ',\infty}$ is a holomorphic discrete series and such that $\bdsig_{\ulQ',q}$ is $\addchar_q$-generic for all finite primes $q$. \index{$\bdsig_{\ulQ'}$}

The ordinary $\calr'$-adic cuspidal Hecke algebra $\bfT^H_{\ord}(N',\chi',\calr')$ is defined as the $\calr'$-subalgebra of $\End_{\calr'}\bfS_{\ord}^H(\frkN',\chi',\calr')$ generated over $\calr'$ by the Hecke operators in $C_c^\infty(K'_l\backslash H(\QQ_l)/K'_l,\calo)$ with $l\nmid pN'$ and $\calu_p'$. 

\begin{proposition}\label{prop:52}
Let $\bdsg\in\bfS^\calh_{\ord}(\frkN',\chi',\calr')$ be an $\calr'$-adic Hida family on $\calh$. 
Then $\bdsg$ is an eigenform for all operators in $\bfT^H_{\ord}(N',\chi',\calr')$. 
If $\ulQ'\in\frkX_{\bdsg}'$, then $\bdsig_{\ulQ'}$ is generated by $\bdsg_{\ulQ'}$. 
\end{proposition}

\begin{proof}
Fix a non-split prime $q$. 
Let $T_q\in C_c^\infty(K'_q\backslash H(\QQ_q^{})/K'_q,\calo)$. 
In view of Remark \ref{rem:51} and the rigidity stated in \cite[\S 3.2(1)]{MH} there are only finitely many points $\ulQ'\in\frkX^\cls_{\calr'}$ such that $\til\bdsig_{\ulQ',q}$ is reducible. 
Hence there is an element $\bdalp_q\in\calr^{\prime\times}$ such that $\bdsg_{\ulQ'}|T_q=\ulQ'(\bdalp_q)\bdsg_{\ulQ'}$ for all but finitely many points $\ulQ'\in\frkX^\cls_{\calr'}$, which shows that $\bdsg|T_q=\bdalp_q\bdsg$. 

Let $\ulQ'\in\frkX_{\bdsg}'$. 
Then $\bdsg_{\ulQ'}\in \bdse'S_{k_{\ulQ'}}^H(p^{\ell_{\ulQ'}}\frkN',\chi'\eps_{\ulQ'};\calr'(\ulQ'))$ is not zero. 
Put 
\[\til\bdsig_{\ulQ',q}^{K'_q}=\{v\in\til\bdsig_{\ulQ',q}\;|\;\til\bdsig_{\ulQ',q}(k)v=v\text{ for }k\in K_q'\}. \]
Since $\til\bdsig_{\ulQ',q}^{K'_q}\subset\bdsig_{\ulQ',q}^{}$ for every non-split prime $q$ by Lemma \ref{lem:51}(\ref{lem:513}), we conclude that $\bdsg_{\ulQ'}\in\bdsig_{\ulQ'}$. 
\end{proof}

We define the global $\addchar^a$-Whittaker functional on $\scra^0(H)$ or $\scra^0(\calh)$ by 
\[W_{\addchar^a}(\vph)=\int_{\QQ\bsl\AA}\vph\left(\begin{bmatrix} 1 & x \\ 0 & 1 \end{bmatrix}\right)\overline{\addchar(ax)}\,\d x \]
for $a\in\QQ^\times$. 
Let $\sig$ be an irreducible automorphic representation of $H(\AA)$. 
We say that $\sig$ is $\addchar^a$-generic if there is $\vph\in\sig$ such that $W_{\addchar^a}(\vph)\neq 0$. 
By definition $\bdsig_{\ulQ'}$ is $\addchar$-generic. 

\begin{proposition}\label{prop:53}
If $\ulQ'\in\frkX_{\bdsg}''$, then $\FC^1_{\frko_E}(\bdsg_{\ulQ'})\neq 0$. 
\end{proposition}

\begin{proof}
Put $W(\til h)=W_{\addchar}(\bdsig_{\ulQ'}(\til h)\vPh_{k_{\ulQ'}}(\bdsg_{\ulQ'}))$ for $\til h=(\til h_v)\in\calh(\widehat{\AA})$. 
By assumption we have a factorization $W(\til h)=\prod_vW_v(\til h_v)$ by the uniqueness of $\addchar$-Whittaker functionals. 
 
For split primes $l$ we view $\bdsig_{\ulQ',l}$ as a representation of $\GL_2(\QQ_l)$ via $\imath_\frkl$. 
Since $\caln'(\QQ_l)=\imath_\frkl^{-1}(N_2(\QQ_l))$ (cf. Remark \ref{rem:51}), the restriction of $W_l$ to $H(\QQ_l)$ is a local Whittaker function associated to an essential vector in $\bdsig_{\ulQ',l}$ for split primes $l\neq p$, and hence $W_l(\ono_2)\neq 0$ (cf. Definition \ref{def:81}).
Since $\bdsg_{\ulQ'}$ is a stabilized vector at $p$ (see Remark \ref{rem:83}), we have $W_p(\ono_2)\neq 0$. 
By (splt) and the formula for spherical Whittaker functions $W_q(\ono_2)\neq 0$ for non-split primes $q$. 
We conclude that $W_{\addchar}(\vPh_{k_{\ulQ'}}(\bdsg_{\ulQ'}))\neq 0$ as claimed. 
\end{proof}

For a split prime $l$ we denote by $c_l=c(\sig_l)$ the exponent of the conductor of $\sig_l$ (see Definition \ref{def:81}). 

\begin{definition}\label{def:54}
Put $\frkN_\sig=\prod_{\frkl|\frkN'}\frkl^{c_l}$. 
Let $g^\circ_\sig\in S_{\ulk'}^H(\frkN_\sig,\chi';\CC)$ be the newform associated to $\sig$, normalized so that $\FC^1_{\frko_E}(g_\sig^\circ)=1$. \index{$f^\circ$}
We call the prime-to-part of $\frkN_\sig$ the tame conductor of $\sig$. \index{$g^\circ_\sig$, $\vph_\sig^{}$}

Put 
\begin{align*}
N_\sig&=\Nr(\frkN_\sig), & 
\vph_\sig^{}&=\vPh_{\ulk'}(g_\sig^\circ),  
\end{align*}
where $\Nr(\frkN_\sig)$ denotes the absolute norm of $\frkN_\sig$. \index{$N_\pi,\frkN_\pi$}
\end{definition}

\begin{definition}\label{def:55}
Let $\ulQ'\in\frkX_{\calr'}^\cls$. 
We call $\bdsg_{\ulQ'}^{\ord}$ the normalized $p$-stabilized newform associated to $\bdsig_{\ulQ'}$ if it generates $\bdsig_{\ulQ'}$, is new outside $p$ and is a $\calu_p'$-eigenform with unit eigenvalue, and satisfies $\FC^1_{\frko_E}(\bdsg_{\ulQ'}^{\ord})=1$. \index{$\bdsg_{\ulQ'}^{\ord},\vph_{\bdsig_{\ulQ'}}^{\ord}$}
Put 
\[\vph_{\bdsig_{\ulQ'}}^{\ord}:=\vPh_{k_{\ulQ'}}(\bdsg_{\ulQ'}^{\ord})\in\bdsig_{\ulQ'}. \]
\end{definition}

\begin{remark}\label{rem:53}
\begin{enumerate}
\item\label{rem:531} We see from Lemma \ref{lem:51}(\ref{lem:513}) and Proposition \ref{prop:52} that if $\til\bdsig_{\ulQ',q}$ is reducible and $m\notin\Nr_{E_q/\QQ_q}(E_q^\times)$, then $\FC^m_\frka(\bdsg_{\ulQ'})=0$. 
\item\label{rem:532} If $\ulQ'\in\frkX_{\bdsg}''$, then $\bdsg_{\ulQ'}^{\ord}:=\FC^1_{\frko_E}(\bdsg_{\ulQ'})^{-1}\bdsg_{\ulQ'}^{}$. 
\item\label{rem:533} Let $\bdsig_{\ulQ'}^\star$ be the restriction of $\bdsig_{\ulQ'}^{}$ to $\GL_2(\AA)$. 
The restriction $\bdsg^\star$ of $\bdsg$ to $\frkH\times\GL_2(\widehat{\QQ})$ is a Hida family on $\GL_2$ of tame level $\frkN'$. 
However, $\bdsg_{\ulQ'}^\star$ may not be primitive even if $\ulQ'\in\frkX_{\bdsg}''$ (cf. Remark \ref{rem:51}(\ref{rem:512})). 
Nevertheless, Proposition \ref{prop:53} holds.  
\end{enumerate}
\end{remark}


\subsection{Line bundles and Fourier-Jacobi coefficients}

Let $\Box$ be the set of prime factors of $pN$. 

\begin{definition}\label{def:56}
Let $U$ be an open compact subgroup of $G(\widehat{\QQ})$ and $\chi$ a character of $U$. 
We call a function $f:\frkD\times G(\widehat{\QQ})\to\call_{\ulk}(\CC)$ a modular form on $G$ of weight $\ulk$, level $U$ and nebentypus $\chi$ if it is holomorphic in $Z$ and satisfies 
\[f(\gam Z,\gam g u)=\chi(u)^{-1}\rho_{\ulk}(J(\gam,Z))f(Z,g) \]
for $\gam\in G(\QQ)$, $g\in G(\widehat{\QQ})$ and $u\in U$. 
A modular form $f$ is called a cusp form if it has a Fourier-Jacobi expansion 
\[f\biggl(\begin{bmatrix} \tau \\ w \end{bmatrix},\alp\biggl)=\sum_{m\in\QQ^\times_+}\overrightarrow{\FJ}^m_\alp(w,f)e^{2\pi\sqrt{-1}m\tau} \] 
for $\alp\in \bfM(\widehat{\QQ}^\Box)$. 
Let $S_{\ulk}^G(U,\chi,\CC)$ be the space of cusp forms on $G$ of weight $\ulk$, level $U$ and nebentypus $\chi$. 
\end{definition}


Let $a\in\widehat{E}^\times$ be a finite id\`{e}le which satisfies $a_\frkq=1$ for prime ideals $\frkq$ dividing $pN$. 
Since $\overrightarrow{\FJ}^m_{\bfd(a)}(w,f)$ depends only on the fractional ideal $\frka=a\frko_E$, we will denote it by $\overrightarrow{\FJ}^m_\frka(w,f)$. 

For $g\in G(\widehat{\QQ})$ and a function $\calf$ on $\frkD\times G(\widehat{\QQ})$ we set 
\[[r(g)\calf](Z,\bet)=\calf(Z,\bet g). \]
Namely, $r(g)\calf$ is the right translation of $\calf$ by $g$. 
Taking Proposition \ref{prop:32} into account, we will rewrite an $\call_{\ulk}(\CC)$-valued function $\calf$ on $\frkD\times G(\widehat{\QQ})$ as 
\[\calf(Z,\bet)=\sum_{i=0}^\kap\calf_i(Z,\bet)(-1)^i\binom{\kap}{i}X^iY^{\kap-i} \index{$f_0,\vPh_{\ulk}(f)_0$}\]
(cf. (\ref{tag:41})). 
We will also rewrite 
\beq
\overrightarrow{\FJ}_\frka^m(w,f)=\sum_{i=0}^\kap\FJ_\frka^m(f)_i^{}(w)(-1)^i\binom{\kap}{i}X^iY^{\kap-i}. \index{$\FJ^m_\frka(f)_0^{}$}\label{tag:54}
\eeq

We define the different of $E$ by $\frkD_E$. 
Fix a prime-to-$\frkp$ fractional ideal $\frka$ of $E$ such that  
\[\frkf(\frka,m)=\QQ\cap m\del^{-1}\frkD_E^{-1}\Nr(\frka)\]
is integral, where $\Nr(\frka)$ denotes the absolute norm of $\frka$. 
Let 
\[mH_0(w,w')=2m\sqrt{-1}\del^{-1} ww^{\prime c} \index{$E_0$, $H_0$}\]
be a positive Hermitian form on $\CC$ whose imaginary part is a Riemann form $mE_0=m\mathrm{Im}H_0$ on an elliptic curve $E_\frka(\CC)=\CC/\iot_\infty(\frka)$ and given by the alternating form $m\ll\;, \;\gg$ defined in \S \ref{ssec:28} with $\gam_0=\del$. 
We define a semi-character $\alp_\frka:\frka\to\{\pm1\}$ as in \S 7.2.3 of \cite{MH2} and a cocycle $l\mapsto e_l$ of $\frka$ with values in the group of invertible holomorphic functions on $\CC$ by
\[e_l(w)=\alp_\frka(l)e^{\pi mH_0\bigl(w+\frac{l}{2},l\bigl)} \]
to which we associate a line bundle $\call^m_\frka=L(mH_0,\alp_\frka)$ on $E_\frka(\CC)$ as the quotient of the trivial bundle $\CC\times\CC$ over $\frka$ by the action of $\frka$ given by $l\cdot(w,z)=(w+l,e_l(w)z)$. \index{$E_\frka$}

\begin{remark}\label{rem:31}
Put $n=m\Nr(\frka)$. 
The line bundle $\call^m_\frka$ is associated to the divisor $n[0]$ of $E_\frka$ (see Example 1.9 of \cite{BK10Duke}). 
Hence $\call^m_\frka$ is defined over an algebraic number field, symmetric and $\iot_\infty$-admissible in the sense of \cite{BK10Duke}. 
\end{remark}

The space of global sections of this line bundle $\call^m_\frka$ coincides with the space $\bfT^m_\frka(\CC)$ that consists of reduced theta functions $\tht$ on $\CC$ satisfying the equation 
\[\tht(w+l)=e_l(w)\tht(w)\]
for $l\in\frka$. 
The line bundle $\call^m_\frka$ is ample and $\dim\bfT^m_\frka(\CC)=f(\frka,m)D_E$, where $f(\frka,m)$ is the positive generator of the ideal $\frkf(\frka,m)$. 

By the theory of complex multiplication the pair $(E_\frka^{},\call^m_\frka)$ descends to $(\scre_\frka^{},\scrl^m_\frka)$ over a discrete valuation ring $R_F=\overline{\ZZ}_p\cap F$ for some number field $F$ so that $(\scre_\frka^{},\scrl^m_\frka)\otimes_{R_F}\CC\simeq(E_\frka^{},\call^m_\frka)$, where we use the fixed isomorphism $\iot_p:\CC\simeq\CC_p$. 
We similarly define $\bfT^m_\frka(\calo)$ for any $R_F$-algebra $\calo$. \index{$\bfT^m_\frka(\calo)$}

Let $L_0'$ be an ideal of $\frko_E$. \index{$L_0'$}
Therefore hereafter we will rewrite $\frkb_0^{}:=L_0'$. \index{$\frkb_0$}
The following criterion is due to Finis (see Lemma 3.1 of \cite{Finis}). 

\begin{lemma}[Finis]\label{lem:Finis}
Let $\frka$ be an ideal of $\frko_E$ and $\tht\in\bfT^m_\frka(\CC)$. 
Then $\tht\in\bfT^m_\frka(R_F)$ if and only if $[A_m(l)\tht](0)\in R_F$ for all $l\in E$. 
\end{lemma}

\begin{proposition}
If $f\in M^\calg_{\ulk}(p^\ell\frkN,R_F)$, then for $\frka\in C_E$ 
\[\biggl(\frac{\Ome_\infty}{2\pi\sqrt{-1}}\biggl)^{k_2}\FJ^m_\frka(f)_0^{}\in\bfT^m_{\frka\frkb_0}(R_F). \]
\end{proposition}

\begin{proof}
Recall the line bundle $\frkL^m_\frka$ on the elliptic curve $\calz_\frka^\circ$ associated to $\frka\in C_E$ introduced in Sections \ref{ssec:39} and \ref{ssec:310}. 
Since $\frka\frkb_0\subset L_\frka$, there is a natural isogeny $\calz_\frka^\circ\to \scre_{\frka\frkb_0}$. 
Put $\tht=\Big(\frac{\Ome_\infty}{2\pi\sqrt{-1}}\Big)^{k_2}\FJ^m_\frka(f)_0^{}$. 
Proposition \ref{prop:FJ} states that $\tht\in\bfT^m_{\frka\frkb_0}(\CC)$. 
Since $\tht\in\bfT^m(L_\frka,R_F)$ by Proposition \ref{prop:31}, Lemma \ref{lem:Finis} shows that $\tht\in\bfT^m_{\frka\frkb_0}(R_F)$.                        
\end{proof}


\subsection{Hida families on $\U(2,1)$}\label{ssec:55}

We define the notion of arithmeticity of cusp forms on $G$ by means of the arithmetic property of the Fourier-Jacobi coefficients.  See \cite{Lan13} for the general arithmetic theory of Fourier-Jacobi expansions on PEL type Shimura varieties.

\begin{definition}\label{def:58}
Let $U$ be an open compact subgroup of $G(\widehat{\QQ})$ which contains $\caln(\frkb_0)$. 
For any $\scro$-subalgebra $\calo$ of $\CC$ the space $S_{\ulk}^G(U,\chi,\calo)$ consists of $f\in S_{\ulk}^G(U,\chi,\CC)$ such that for every $\frka\in C_E$
\[\biggl(\frac{\Ome_\infty}{2\pi\sqrt{-1}}\biggl)^{k_2}\FJ^m_\frka(f)_0^{}\in\bfT^m_{\frka\frkb_0}(\calo). \] 
For a $p$-adic ring $R$ finite flat over $\calo$ we set 
\[S_{\ulk}^G(U,\chi;R):=\Ome_\frkp^{k_2} S_{\ulk}^G(U,\chi;\calo)\otimes R. \index{$S_{\ulk}^G(U,\chi;R)$}\]
Put 
\begin{align*}
S_{\ulk}^G(U;R)&=S_{\ulk}^G(U,1;R), & 
S_{\ulk}^G(p^\ell\frkN,\chi;R)&=S_{\ulk}^G(K_0(p^\ell\frkN),\chi;R). 
\end{align*}
\end{definition} 

Let $\calr$ be a normal ring finite flat over $\Lam_3$. 
Let 
\[\bdsf=\{\Tht^m_\frka\}_{\frka\in C_E,\;m\in\scrs_\frka^+}^{}\in\bfS^\calg_{\ord}(\frkN,\chi,\calr)\] 
be an $\calr$-adic Hida family on $\calg$ (cf. Remark \ref{rem:42}). 
Denote the restriction of $\bdsf_{\ulQ}$ to $\frkD\times G(\widehat{\QQ})$ also by $\bdsf_{\ulQ}\in\bdse S_{k_{\ulQ}}^G(p^{\ell_{\ulQ}}\frkN,\chi\eps_{\ulQ};\calr(\ulQ))$ for $\ulQ\in\frkX_\calr^\cls$. 
The reader is reminded that when $\alp=\bfd(a)\in \bfM(\widehat{\QQ}^\Box)$ and $\frka=a\frko_E$, 
\beq
\ulQ(\bdsf):=\sum_{m\in\scrs_\frka^+}\ulQ(\Tht^m_\frka)\,q^m=\biggl(\frac{\Ome_\infty}{2\pi\sqrt{-1}\Ome_\frkp}\biggl)^{k_{Q_2}}\bigl(\bdsf_{\ulQ}\bigl)_0^{}(Z,\alp). \label{tag:55}
\eeq

For $\ulQ\in\frkX_{\bdsf}'$ Lemma \ref{lem:41} associated to $\bdsf_{\ulQ}$ an irreducible cuspidal automorphic representation $\til\bdpi_{\ulQ}$ of $\calg(\AA)$. 
We denote by $\bdpi_{\ulQ}$ its restriction to $G(\AA)$, which is irreducible as guaranteed by Proposition \ref{prop:51}(\ref{prop:512}). \index{$\bdpi_{\ulQ}$}

Let $\pi\simeq\otimes_v'\pi_v^{}$ be an irreducible cuspidal tempered automorphic representation of $G(\AA)$ whose archimedean part $\pi_\infty$ is a holomorphic discrete series with minimal weight $-\ulk$ and such that $\pi_q$ admits a non-zero $K_q$-invariant vector for every non-split prime $q$. 
For each split prime $l$ we denote the conductor of $\pi_l$ by $c(\pi_l)$ in the sense of (\ref{tag:82}). 

\begin{definition}\label{def:57}
Put $N_\pi=\prod_ll^{c(\pi_l)}$. 
We take ideals $\frkN_\pi$ of $\frko_E$ such that $\frko_E/\frkN_\pi\simeq\ZZ/N_\pi\ZZ$. \index{$N_\pi,\frkN_\pi$}
Then there exists $f^\circ\in S_{\ulk}^G(\frkN_\pi^{},\chi,\overline{\QQ})$ such that $\bfv\mapsto\vPh_{\ulk}(f^\circ)_\bfv$ gives a $\calk_\infty$-equivariant embedding $\call_{\ulk}^\vee(\CC)\hookrightarrow\pi$. 
We call $f^\circ$ a $\overline{\QQ}$-rational newform associated to $\pi$. \index{$f^\circ$}
We call $f^\circ_0$ a highest weight newform associated to $\pi$. 
Note that $f^\circ$ is defined up to multiplication by $\overline{\QQ}^\times$. 
\end{definition} 



\section{Theta operators on Picard modular forms}\label{sec:6}


\subsection{The $\calu(\frkp^c)$-operator}

Recall that 
\begin{align*}
\bfn(w,z)&=\begin{bmatrix} 1 & -\del^{-1} w^c & z-\frac{w^c w}{2\del} \\ 0 & 1 & w \\ 0 & 0 & 1 \end{bmatrix}, & 
\bfd(\xi)&=\begin{bmatrix} \xi^c & 0 & 0 \\ 0 & 1 & 0 \\ 0 & 0 &  \xi^{-1}\end{bmatrix}, \\
\bet_1&=\begin{bmatrix} p^{-1} & 0 & 0 \\ 0 & 1 & 0 \\ 0 & 0 & 1 \end{bmatrix}, & 
\alp_1&=\begin{bmatrix} 1 & 0 & 0 \\ 0 & 1 & 0 \\ 0 & 0 & p \end{bmatrix}. 
\end{align*}
for $w\in\EE$, $z\in\AA$ and $\xi\in\EE^\times$. 
We will write
\begin{align*}
\calu(\frkp)&=\calu_p(\alp_1), &
\calu(\frkp^c)&=\calu_p(\bet_1). \index{$\calu(\frkp^c)$}
\end{align*} 
We see from \S \ref{ssec:25} that 
\[[\calu(\frkp^c)f](Z,\bet)=p^{-1-k_1}\sum_{y,z\in\ZZ_p/p\ZZ_p}f\left(Z,\bet\imath_\frkp^{-1}\left(\begin{bmatrix} p & y & z \\ 0 & 1 & 0 \\ 0 & 0 & 1 \end{bmatrix}\right)\right). \]

Let $f\in S_{\ulk}^G(p^\ell\frkN,\chi;\calo)$ be a Hecke eigenform.  
We write $\pi$ for the cuspidal automorphic representation associated to $f$. 

\begin{definition}\label{def:31}
We say that $\pi$ is $\frkp^c$-ordinary with respect to $\iot_p$ if the local representation $\pi_p$, viewed as a representation of $\GL_3(\QQ_p)$, is a subquotient of a principal series $I(\nu_p,\rho_p,\mu_p)$ of $\GL_3(\QQ_p)$ (see \S \ref{ssec:84} for notation), and  
\[k_1=\ord_p\iot_p(\mu_p(p))\leq\min\{\ord_p\iot_p(\rho_p(p)), \ord_p\iot_p(\nu_p(p))\}. \]
\end{definition} 

\begin{remark}\label{rem:35}
Put $\gam_p=\mu_p(p)$. 
There exists a cusp form $f\in\pi$ such that 
\[\calu(\frkp^c)f=p^{-k_1}\gam_pf \]
(see Proposition \ref{prop:81}). 
\end{remark}


\subsection{The $\bfU_{\frkp^c}$-operator}\label{ssec:66}

We denote by $\calc(\ZZ_p^\times,\calo)$ the space of $\calo$-valued locally constant functions on $\ZZ_p^\times$. We write $\mathrm{Dist}(\ZZ_p^\times,\calo)$ for the space of all $\calo$-linear functionals on $\calc(\ZZ_p^\times,\calo)$. 
As is well-known, there is the canonical ring isomorphism between $\mathrm{Dist}(\ZZ_p^\times,\calo)$ and $\Lam=\calo\powerseries{\ZZ_p^\times}$. 

Let $\frkb_0$ be a sufficiently small prime-to-$\frkp$ integral ideal of $\frko_E$ with respect to the fixed set $C_E$ of representatives of the ideal class group of $E$.   
Take a set $\calb_j$ of representatives for $(\frko_E/(\frkp^c)^j)^\times$ from $\frkb_0$. 
Let $\breve\phi\in\calc(\ZZ_p^\times,\calo)$. 
Let $j$ be sufficiently large so that $\breve\phi$ factors through the quotient $\ZZ_p^\times\to(\ZZ_p/p^j\ZZ_p)^\times$. 
Then we view it as a function on $(\frko_E/\frkp^j)^\times$. 
Given $b\in\EE$ and a place $v$ of $E$, we here denote by $b_v\in E_v$ the $v$-component of $b$ and write $b_p^{}=(b_\frkp,b_{\frkp^c})\in E_p$.  

\begin{definition}
Let $f\in S_{\ulk}^G(p^\ell\frkN,\chi;\calo)$ be such that $\calu(\frkp^c)f=p^{-k_1}\gam_pf$. 
When $\ord_p\iot_p(\gam_p)=k_1$, we define the $\frkp^c$-depletion of $f$ with respect to $\breve\phi$ by 
\[[\bfU_{\frkp^c}^{\breve\phi} f](Z,g)=\frac{1}{p^j\gam_p^j}\sum_{b\in \calb_j}\sum_{z\in \ZZ_p/p^j\ZZ_p}\breve\phi(b^c)f\left(Z,g\imath_\frkp^{-1}\left(\begin{bmatrix} p^j & b_{\frkp^c} & z \\ 0 & 1 & 0 \\ 0 & 0 & 1 \end{bmatrix}\right)\right). \index{$\bfU_{\frkp^c}^{\breve\phi}$}\]
The definition of $\bfU_{\frkp^c}^{\breve\phi}f$ is independent of the choice of $j$ by Lemma \ref{lem:83}. 
\end{definition}

\begin{proposition}\label{prop:64}
Let $\calo$ be a $\scro$-subalgebra of $\CC$ and $f\in S_{\ulk}^G(p^\ell\frkN,\chi,\calo)$ satisfies $\calu(\frkp^c)f=p^{-k_1}\gam_pf$ with $\ord_p\iot_p(\gam_p)=k_1$. 
Let $\frka\in C_E$. 
Take a sufficiently small open compact subgroup $U$ of $\widehat{E}^\times$ and a natural number $j$ so that we can write $\imath_\frkp^{-1}(\bet_1^{-j})=\bfd(\xi_\bff u)$ with $\xi\in E^\times$ and $u\in U$. 
Then we have $\FJ^m_\frka(w,\bfU_{\frkp^c}^{\breve\phi}f)=0$ unless $m\in\scrs_\frka^+$, in which case
\[\FJ^m_\frka(\bfU_{\frkp^c}^{\breve\phi}f)_0^{}(w)=\frac{(\xi^c)^{k_1}}{\gam_p^j\xi^{k_3}}\sum_{b\in\calb_j}\breve\phi(b^c)\biggl[A_{mp^j}\left(\del\frac{b}{\xi}\right)\FJ^{mp^j}_\frka(f)_0^{}\biggl]\left(\frac{w}{\xi}\right). \]
In particular, $\Big(\frac{\Ome_\infty}{2\pi\sqrt{-1}}\Big)^{k_2}\FJ^m_\frka(\bfU_{\frkp^c}^{\breve\phi}f)_0^{}\in\bfT^m_{\frka(\frkp^c)^j\frkb_0}(\calo)$.  
\end{proposition}


\subsection{Proof of Proposition \ref{prop:64}}

When the class number of $E$ is greater than one, it is complicated to describe the action of $\calu(\frkp^c)$ on Fourier-Jacobi expansions. 
One can use the following result to compute the action of a suitable power $\calu(\frkp^c)^j$ in terms of Fourier-Jacobi coefficients. 

\begin{lemma}\label{lem:31}
Let $f\in S_{\ulk}^G(p^\ell\frkN,\chi,\CC)$ and $a\in\widehat{E}^\times$. Let $\frka=a\frko_E\in C_E$. 
Take a sufficiently small open compact subgroup $U$ of $\widehat{E}^\times$ and a natural number $j$ so that we can write $\imath_\frkp^{-1}(\bet_1^{-j})=\bfd(\xi_\bff u)$ with $\xi\in E^\times$ and $u\in U$. 
Let $z\in\QQ_p$ and $l\in\frkb_0$. 
Then for a positive rational number $m$ 
\begin{multline*}
\vFJ^m_{\bfd(a)\bfn(l_{\frkp^c},z)\imath_\frkp^{-1}(\bet_1^{-j})}(w,f)\\
=\frac{\addchar^m(z)}{\xi^{k_3}}e^{\pi\sqrt{-1}m\del^{-1} l^c(2w-l)}\rho_{\ulk}\left(\begin{bmatrix} (\xi^{-1})^c & 0 \\ \del^{-1}(\xi^{-1}l)^c & 1 \end{bmatrix},1\right)\vFJ^{mp^j}_\frka(\xi^{-1}(w-l),f). 
\end{multline*}
\end{lemma}

\begin{proof}
Let $\xi=g_0=1$ and $\bet=\bfd(a)\bfn(l_{\frkp^c},z)\imath_\frkp^{-1}(\bet_1^{-j})$ in Lemma \ref{lem:21}. 
Then 
\begin{multline*}
\vPh(f)^m_{\bfd(a)\bfn(l_{\frkp^c},z)\imath_\frkp^{-1}(\bet_1^{-j})}(\bfn(w,0))\\
=e^{-\pi m(2+\sqrt{-1}w^c\del^{-1} w)}\rho_{\ulk}\left(\begin{bmatrix} 1 & 0 \\ -\del^{-1} w^c & 1 \end{bmatrix}\right)\FJ^m_{\bfd(a)\bfn(l_{\frkp^c},z)\imath_\frkp^{-1}(\bet_1^{-j})}(w,f). 
\end{multline*} 
Since $\frka$ and $p$ are coprime, the left hand side equals 
\begin{align*}
&\int_{\QQ\bsl\AA}\varPhi(f)(\bfn(w,x)\bfd(a)\bfn(l_{\frkp^c},z)\imath_\frkp^{-1}(\bet_1^{-j}))\overline{\addchar^m(x)}\,\d x\\
=&\int_{\QQ\bsl\AA}\varPhi(f)(\bfn(w+l_{\frkp^c},x+z)\imath_\frkp^{-1}(\bet_1^{-j})\bfd(a))\overline{\addchar^m(x)}\,\d x\\
=&\int_{\QQ\bsl\AA}\varPhi(f)\biggl(\bfn\biggl(w-l_\infty,x+z-\frac{\ll l_\infty,w\gg}{2}\biggl)\bfd(\xi_\bff)\bfd(a)\biggl)\overline{\addchar^m(x)}\,\d x. 
\end{align*} 
Since the ideal of $\frko_E$ generated by $\xi$ is $(\frkp^c)^j$, we have $\Nr_{E/\QQ}(\xi)=p^j$. 
The integral above equals  
\begin{align*}
&\addchar^m\biggl(z-\frac{\ll l_\infty,w\gg}{2}\biggl)\int_{\QQ\bsl\AA}\varPhi(f)(\bfn(\xi_\infty^{-1}(w-l_\infty),p^{-j}x)\bfd(a)\bfd(\xi_\infty^{-1}))\overline{\addchar^m(x)}\,\d x\\
=&\addchar^m(z)e^{\pi m\sqrt{-1}\Tr_{E/\QQ}(l^c\del^{-1} w)}\vPh(f)^{mp^j}_{\bfd(\xi_\infty^{-1})\bfd(a)}(\bfn(\xi_\infty^{-1}(w-l_\infty),0)) \\
=&\addchar^m(z)e^{\pi m\sqrt{-1}\del^{-1}(l^cw-lw^c)}e^{-\pi m(2+\sqrt{-1}(w-l)^c \del^{-1} (w-l))}\\
&\times\rho_{\ulk}\left(\begin{bmatrix} (\xi^{-1})^c & 0 \\ -\del^{-1}(\xi^{-1} (w-l))^c & 1 \end{bmatrix},\xi^{-1}\right)\FJ^{mp^j}_\frka(\xi^{-1}(w-l),f)
\end{align*}
again by Lemma \ref{lem:21}. 
\end{proof}

We put $\Ome=\frac{\Ome_\infty}{2\pi\sqrt{-1}}$ to simplify notation. 
Granted (\ref{tag:35}), we normalize the arithmetic $m$th Fourier-Jacobi coefficient of $f$ at $\bet\in G(\widehat{\QQ})$ by 
\[\vec{\tht}^m_\bet(w,f)=\rho_{\ulk}\biggl(\begin{bmatrix} 1 & \\ & \Ome^{-1}\end{bmatrix},1\biggl)\overrightarrow{\FJ}^m_\bet(w,f).\]
When $\bet=\bfd(a)$ and $\frka=a\frko_E$, we write 
\[\vec{\tht}^m_\bet(w,f)=\vec{\tht}^m_\frka(w,f)=\sum_{i=0}^\kap\tht_\frka^m(f)_i^{}(w)(-1)^i\binom{\kap}{i}X^iY^{\kap-i}. \]
Then $\tht_\frka^m(f)_i^{}(w)=\Ome^{k_2-i}\FJ_\frka^m(f)_i^{}(w)$ and $\tht_\frka^m(f)_i^{}(0)\in\calo$ by Proposition \ref{prop:32}. 
Proposition \ref{prop:FJ} shows that for every $l\in\frka\frkb_0$
\[e^{-\pi m\sqrt{-1}l^c\del^{-1}(2w+l)}\rho_{\ulk}\left(\begin{bmatrix} 1 & 0 \\ -\del^{-1}\frac{l^c}{\Ome} & 1 \end{bmatrix},1\right)\vec{\tht}^m_\frka(w+l)=\vec{\tht}^m_\frka(w). \]

Recall the operator $A_m(l)$ defined for functions $\Tht$ on $\CC$ by 
\[[A_m(l)\Tht](w)=e^{-\pi m\sqrt{-1}l^c\del^{-1}(2w+l)}\Tht(w+l)\index{$A_m(l)$}\]
(see \S \ref{ssec:28}). 
It follows from Lemma \ref{lem:31} that 
\begin{align*}
&\vec{\tht}^m_{\bfd(a)\bfn(-l_{\frkp^c},z)\imath_\frkp^{-1}(\bet_1^{-j})}(w,f)\\
=&\frac{\addchar^m(z)}{\xi^{k_3}}e^{-\pi\sqrt{-1}m\del^{-1} l^c(2w+l)}\rho_{\ulk}\left(\begin{bmatrix} (\xi^{-1})^c & 0 \\ -\frac{(\xi^{-1}l)^c}{\del\Ome} & 1 \end{bmatrix},1\right)\vec{\tht}^{mp^j}_\frka(\xi^{-1}(w+l),f)\\
=&\addchar^m(z)\frac{(\xi^{k_1})^c}{\xi^{k_3}}\sum_{i=0}^\kap\left[A_{mp^j}\left(\frac{l}{\xi}\right)\tht^{mp^j}_\frka(f)_i\right]\left(\frac{w}{\xi}\right)\binom{\kap}{i}(-\xi^cX)^i\biggl(Y+\frac{l^cX}{\del \Ome}\biggl)^{\kap-i}. 
\end{align*}

We are now ready to prove Proposition \ref{prop:64}. 
For our choice of $\calb_j$ we can apply this formula to $l=\del b$ with $b\in\calb_j$ to get 
\begin{align*}
&\vec{\tht}^m_{\bfd(a)\bfn(-\del b_{\frkp^c},z)\imath_\frkp^{-1}(\bet_1^{-j})}(w,f)\\
=&\addchar^m(z)\frac{(\xi^{k_1})^c}{\xi^{k_3}}\sum_{i=0}^\kap\left[A_{mp^j}\left(\frac{\del b}{\xi}\right)\tht^{mp^j}_\frka(f)_i\right]\left(\frac{w}{\xi}\right)\binom{\kap}{i}(-\xi^cX)^i\biggl(Y-\frac{b^cX}{\Ome}\biggl)^{\kap-i}. 
\end{align*}
Since 
\[\bfn(-\del b_{\frkp^c},z)\imath_\frkp^{-1}(\bet_1^{-j})=\imath_\frkp^{-1}\left(\begin{bmatrix} p^j & b_{\frkp^c} & z \\ 0 & 1 & 0 \\ 0 & 0 & 1 \end{bmatrix}\right), \]
we have  
\begin{multline}
\frac{\gam_p^j\xi^{k_3}}{(\xi^{k_1})^c}\vec{\tht}^m_\frka(w,\bfU_{\frkp^c}^{\breve\phi}f) \label{tag:61}\\
=\sum_{i=0}^\kap\sum_{b\in\calb_j}\breve\phi(b^c)\biggl[A_{mp^j}\left(\del\frac{b}{\xi}\right)\tht^{mp^j}_\frka(f)_i\biggl]\left(\frac{w}{\xi}\right)\binom{\kap}{i}(-\xi^c X)^i\biggl(Y-\frac{b^c}{\Ome}X\biggl)^{\kap-i}
\end{multline}
if $m$ is $p$-integral.   
Clearly, $\vec{\tht}^m_\frka(w,\bfU_{\frkp^c}^{\breve\phi}f)=0$ if $m$ is not $p$-integral. 
We obtain 
\[\frac{\gam_p^j\xi^{k_3}}{(\xi^{k_1})^c}\tht^m_\frka(\bfU_{\frkp^c}^{\breve\phi}f)_0^{}(w)\\
=\sum_{b\in\calb_j}\breve\phi(b^c)\biggl[A_{mp^j}\left(\del\frac{b}{\xi}\right)\tht^{mp^j}_\frka(f)_0^{}\biggl]\left(\frac{w}{\xi}\right) \]
by looking at the coefficient of $Y^\kap$. 
Since $A_{mp^j}(l)$ ($l\in E$) preserves the module of $p$-integral theta functions by Lemma \ref{lem:Finis}, we conclude that $\tht^m_\frka(\bfU_{\frkp^c}^{\breve\phi}f)_0^{}\in\bfT^m_{\frka(\frkp^c)^j\frkb_0}(\calo)$. \qed


\subsection{The $\bfU^{\chi'}_{\frkN'}$-operator}\label{ssec:63}

Let $\chi'$ be a Dirichlet character of conductor $N'$ whose prime factors are split in $E$ and distinct from $p$. 
Take an ideal $\frkN'$ of $\frko_E$ such that $\frko_E/\frkN'\simeq\ZZ/N'\ZZ$. 
To construct $p$-adic families having tame level $\frkN'$ and nebentypus $\chi'$, we apply the $\bfU^{\chi'}_{\frkN'}$-operator which we introduce in \cite{HY} by following the construction \cite{Schmidt} of $p$-adic $L$-functions for $\GL_3\times\GL_2$. 
We define 
\[\bfU^{\chi'}_{\frkN'}f(g)=\sum_{i,j\in(\widehat{\ZZ}/N'\widehat{\ZZ})^\times}\sum_{y\in\widehat{\ZZ}/N^{\prime 2}\widehat{\ZZ}}
\chi'(ij)f\left(g\cdot\vsi_{NN'}\begin{bmatrix} 1 & \frac{i}{N'} & \frac{y}{N^{\prime2}} \\ 0 & 1 & \frac{j}{N'} \\ 0 & 0 & 1 \end{bmatrix}\right), \index{$\bfU^{\chi'}_{\frkN'}$}\]
where $\vsi_{NN'}=(\vsi_{NN',l})\in G(\widehat{\QQ})$ is defined by 
\[\vsi_{NN',l}=\begin{cases}
 \ono_3 &\text{if $l$ does not divide $NN'$, }\\
\imath_\frkl^{-1}(\vsi) &\text{if $l|NN'$. }
\end{cases}\]
This operator $\bfU^{\chi'}_{\frkN'}$ has the following properties: 
\begin{itemize}
\item the restriction of $\bfU^{\chi'}_{\frkN'}f$ to $H(\widehat{\QQ})$ has an appropriate $K$-type; 
\item the local integrals at prime factors of $N'$ attached to $\bfU^{\chi'}_{\frkN'}f$ and newforms of level $\frkN'$ have a simple formula (see \S \ref{ssec:78}). 
\end{itemize}


\subsection{Shimura's differential operators}\label{ssec:61}
Put 
\begin{align*}
\frkT&=\Mat_{2,1}(\CC), & 
\xi(Z)&=\sqrt{-1}\begin{bmatrix} \tau^c-\tau & -w \\ w^c & -\del\end{bmatrix}, &
\vXi(Z)&=(\xi(Z),\eta(Z)) \index{$\vXi(Z),\xi(Z),\eta(Z)$}
\end{align*} 
for $Z=\begin{bmatrix} \tau \\ w \end{bmatrix}\in\frkD$. 
We write $Ml_n(\frkT,V)$ for the vector space of all $\CC$-multilinear maps of $\frkT\times\cdots\times\frkT$ ($n$ copies) into $V$. 
We define the representation $(\rho\otimes\upsilon^n,Ml_n(\frkT,V))$ of $\GL_2(\CC)\times\CC^\times$ by 
\[[(\rho\otimes\upsilon^n)(a,b)h](u_1,\dots,u_n)=\rho(a,b)h(\trs au_1b,\dots,\trs au_nb)\]
for $h\in Ml_n(\frkT,V)$ and $u_1,\dots,u_n\in\frkT$. 

Given $f\in S^G_\rho(p^\ell\frkN,\chi,\CC)$, we define a $Ml_1(\frkT,V)$-valued function $Df$ on $\frkD\times G(\widehat{\QQ})$ by 
\begin{align*}
[Df(u)](Z,\bet)&=u_{11}\frac{\partial f}{\partial\tau}(Z,\bet)+u_{21}\frac{\partial f}{\partial w}(Z,\bet), \\ 
[Cf(u)](Z,\bet)&=[Df(\trs\xi(Z)u\eta(Z))](Z,\bet)
\end{align*}
for $\bet\in G(\widehat{\QQ})$ and $u=\begin{bmatrix} u_{11} \\ u_{21} \end{bmatrix}\in\frkT$. 
Given a non-negative positive integer $n$, we define $D^nf$ and $C^nf$ by
\begin{align*}
D^nf&=D(D^{n-1}f), & D^0f&=f, &
C^nf&=C(C^{n-1}f), & C^0f&=f. 
\end{align*}
Define $Ml_n(\frkT,V)$-valued function $D_\rho^nf$ on $\frkD\times G(\widehat{\QQ})$ by 
\[D_\rho^nf=(\rho\otimes\upsilon^n)(\vXi)^{-1}C^n(\rho(\vXi)f). \index{$D_\rho$}\]
Then $D_\rho^nf$ satisfies
\[D_\rho^nf(\gam Z,\gam\bet k)=(\rho\otimes\upsilon^n)(J(\gam,Z))D_\rho^nf(Z,\bet) \]
for $\gam\in G(\QQ)$, $\bet\in G(\widehat{\QQ})$ and $k\in \calk_1(p^\ell\frkN)$, and 
\begin{align*}
D_\rho^{n+1}&=D_{\rho\otimes\upsilon^n}^{}D_\rho^n=D^n_{\rho\otimes\tau}D_\rho^{} 
\end{align*}
by Proposition 12.10 of \cite{Shimura00}. 
Moreover, $D_\rho^nf$ is nearly holomorphic in the sense that $\varPhi_{\rho\otimes\upsilon^n}(D_\rho^nf)$ is is annihilated by some power of every generator of the antiholomorphic Lie algebra $\frkp_-$. 
Put
\begin{align*}
D_{\ulk}&=D_{\rho_{\ulk}}, & 
\del_{\ulk}^n&=(2\pi\sqrt{-1})^{-n}D_{\ulk}^n. \index{$D_{\ulk},\del_{\ulk}$}
\end{align*}

\begin{remark}\label{rem:23}
When $r=s=1$, the operator $\del_{(k_1;k_2)}^n$ coincides with the Maass-Shimura differential operator $\del_{k_2-k_1}^n$ on $\SL_2$ (cf. (\ref{tag:24})). 
\end{remark}

Strictly speaking, Shimura deals with only the tube domain $\frkD_{r,r}$  in \cite{Shimura00}. 
When $r\neq s$, he considers bounded forms $\frkB_{r,s}$ of the symmetric spaces instead of the unbounded domains $\frkD_{r,s}$. 
The proof is axiomatic and works generally. 
We will prove the analogous results for $\frkD_{r,s}$ in \S \ref{ssec:a4} by using the isomorphism $\frkt:\frkB_{r,s}\stackrel{\sim}{\to}\frkD_{r,s}$, which will be used for the calculation of the archimedean integral in Appendix \ref{sec:b}. 


\subsection{Differential operators on $\U(2,1)$}\label{ssec:64}

\begin{definition}
Put $q^m=e^{2\pi m\sqrt{-1}}$. 
For $\ulk'=(k_1';k_2')$ with $k_1'<k_2'$ and a non-negative integer $t$ we denote by $\scrs^{[t]}_{\ulk'}(p^\ell\frkN',\chi')$ the
space of all $\CC$-valued functions $f$ on $\frkH\times H(\widehat{\QQ})$ satisfying the following two conditions:
\begin{itemize}
\item for $\gam\in H(\QQ)$, $\bet\in H(\widehat{\QQ})$ and $k\in K_1'(p^\ell\frkN')$ 
\[f(\gam \tau,\gam\bet k)=\chi'(k)^{-1}\rho_{\ulk'}(J(\gam,\tau))f(\tau,\bet)\]  
\item for every $\bet\in H(\widehat{\QQ})$ there exists $N_\bet\in\NN$ and $\FC^m_\nu(\bet)\in\CC$ such that
\[f(\tau,\bet)=\sum_{\nu=0}^t\sum_{m=1}^\infty (\pi\eta(\tau))^{-\nu}\FC^m_{\nu,\bet}(f)q^{m/N_\bet}. \] 
\end{itemize}
We call $f\in \scrs^{[t]}_{\ulk'}(p^\ell\frkN',\chi')$ a nearly holomorphic cusp form on $H$ of weight $\ulk'$, level $p^\ell\frkN'$, degree $t$ and nebentypus $\chi'$. 
For a subfield $\Ome$ of $\CC$ we say that $f$ is $\Ome$-rational if $\FC^m_{\nu,\bet}(f)\in\Ome$ for all $m,\nu,\bet$.
The space $\scrs^{[t]}_{\ulk'}(p^\ell\frkN',\chi',\Ome)$ consists of $\Ome$-rational elements of $\scrs^{[t]}_{\ulk'}(p^\ell\frkN',\chi')$. \index{$\scrs^{[t]}_{\ulk'}(p^\ell\frkN',\chi',\Ome)$}
\end{definition}

The $\overline{\QQ}$-rationality is characterized in terms of values at CM-points. 
The following proposition is a special case of Theorem 14.12(1) of \cite{Shimura00}. 

\begin{proposition}[Shimura]\label{prop:61}
The space $\scrs^{[t]}_{\ulk'}(p^\ell\frkN',\chi',\overline{\QQ})$ consists of all elements $f\in \scrs^{[t]}_{\ulk'}(p^\ell\frkN',\chi')$ such that $\frac{f(\tau)}{(\Ome_\infty/\pi)^{k_2'-k_1'}}\in\overline{\QQ}$ for every $\tau\in\frkH\cap E$. 
\end{proposition}

Put 
\begin{align*}
v_0&=\begin{bmatrix} 0 \\ 1 \end{bmatrix}, & 
\ulv_0^n&=(v_0,v_0,\dots,v_0)\in\frkT^n.  
\end{align*}
Proposition \ref{prop:a3} shows that 
\beq
[D_{\ulk}^nf(\ulv_0^n)]_i(\jmath(\tau))\\
=\sum_{j=0}^{\min\{n,\kap-i\}}\frac{(i+j)!}{i!(-\sqrt{-1}\eta(\tau))^j}\binom{n}{j}\frac{\partial^{n-j}f_{i+j}}{\partial w^{n-j}}(\jmath(\tau)). \label{tag:60}
\eeq
The computation of $\del_{\ulk}^nf(\ulv_0^n)\circ\jmath$ is ad hoc. 
So as not to interrupt the flow of of the body of this section, we will do it in Appendix \ref{sec:a}.

We now generalize Proposition \ref{prop:32}. 

\begin{proposition}\label{prop:62}
If $f\in S^G_{\ulk}(p^\ell\frkN,\chi,\overline{\QQ})$ with $\ulk=(k_1,k_2;k_3)$, then  
\[(2\pi\sqrt{-1}/\Ome_\infty)^{n+i-k_2}[\del_{\ulk}^nf(\ulv_0^n)]_i\circ\jmath\in \scrs^{[\min\{n,k_2-k_1-i\}]}_{(k_1+i;k_3+n)}(p^\ell\frkN,\chi|_{K_0'(p^\ell\frkN)},\overline{\QQ}) \]
for $i=0,1,2,\dots,k_2-k_1$ and a non-negative integer $n$. 
\end{proposition}

\begin{remark}
This is a special case of Remark 7.10.3 of \cite{Harris86}. 
\end{remark}

\begin{proof}
Proposition \ref{prop:a2} says that $[\del_{\ulk}^nf(\ulv_0^n)]_i\circ\jmath$ is a nearly holomorphic cusp form on $H$ of weight $(k_1+i;k_3+n)$. 
It has degree at most $\min\{n,k_2-k_1-i\}$ by (\ref{tag:60}). 

To prove the $\overline{\QQ}$-rationality, we use the notation of Sections 11 and 26.3 of Shimura's book \cite{Shimura00}. 
Section 11.2 of \cite{Shimura00} defines the $\overline{\QQ}$-rationality of modular forms on unitary groups over CM fields in terms of values at CM points. 
Since  CM points are related to $\overline{\QQ}$-quadruples, one can show that elements of $S^G_{\ulk}(p^\ell\frkN,\chi,\overline{\QQ})$ are $\overline{\QQ}$-rational in the sense of Shimura. 
Moreover, Theorem 14.7 of \cite{Shimura00} shows that the value of $\del_{\ulk}^nf(\ulv_0^n)$ at any CM point is algebraic. 

Let $Y=E\oplus E\oplus E$ be a (split) CM-algebras. 
Put $Y^u=\{a\in Y\;|\;aa^c=1\}$. 
Fix $\tau_0\in\frkH\cap E$. 
Put $Z_0=\begin{bmatrix} \tau_0 \\ 0 \end{bmatrix}\in\frkD$. 
Define an $E$-linear ring injection $h:Y\to\Mat_3(E)$ by 
\[h(a_1,a_2,a_3)=B(Z_0)\begin{bmatrix} a_1 & 0 & 0 \\ 0 & a_2 & 0 \\ 0 & 0 & a_3 \end{bmatrix}B(Z_0)^{-1}. \]
Since $S_\del^{}\trs h(a_1^{},a_2^{},a_3^{})^cS_\del^{-1}
=h(a_1^c,a_2^c,a_3^c)$ by (6.1.7) and (6.1.8) of \cite{Shimura97}, we have $h(Y^u)\subset G(\QQ)$.  
Moreover, $Z_0$ is the common fixed point of $h(Y^u)$, and 
\begin{align*}
\lam(h(a_1,a_2,a_3),Z_0)&=\begin{bmatrix} \iot_\infty^c(a_1) & 0 \\ 0 & \iot_\infty^c(a_2) \end{bmatrix}, &
\mu(h(a_1,a_2,a_3),Z_0)&=\iot_\infty^{}(a_3)
\end{align*}
for $(a_1,a_2,a_3)\in Y^u$. 

Let $J_E=\{\iot_\infty^{},\iot_\infty^c\}$ be the set of all isomorphisms of $E$ into $\CC$ and $I_E$ the free abelian group generated by $J_E$. 
Shimura introduced the period symbol $p_E:I_E\times I_E\to\CC^\times/\overline{\QQ}^\times$ which is linear with respect to both of the arguments. 
Actually, $p_E(\iot_\infty^{},\iot_\infty^{})$ coincides with the geometric period of a CM elliptic curve.

Let $\vPh=(\iot_\infty^c,\iot_\infty^c,\iot_\infty^{})$ be a ring homomorphism of $Y$ into $\CC$. 
Then 
\beq 
p_Y(\iot_\infty^c,\vPh)=p_Y(\iot_\infty^{},\vPh)=p_E(\iot_\infty^{},\iot_\infty^{})=\Ome_\infty/\pi \label{tag:65}
\eeq
by (11.23) of \cite{Shimura00}. 
We conclude that 
\begin{align*}
\frkp(Z_0)&=\biggl(\begin{bmatrix} p_Y(\iot_\infty^c,\vPh) & 0 \\ 0 & p_Y(\iot_\infty^c,\vPh) \end{bmatrix},p_Y(\iot_\infty^{},\vPh)\biggl)=\biggl(\begin{bmatrix} \frac{\Ome_\infty}{\pi} & 0 \\ 0 & \frac{\Ome_\infty}{\pi} \end{bmatrix},\frac{\Ome_\infty}{\pi}\biggl). 
\end{align*}
Since $\upsilon(\frkp(Z_0))\ulv_0=\Ome_\infty^2/\pi^2$ (see \S \ref{ssec:a2} for the definition of the action $\upsilon$ of $\GL_2(\CC)\times\CC^\times$ on $\frkT=\Mat_{2,1}(\CC)$), we have 
\[(\Ome_\infty/\pi)^{-2n}\rho_{\ulk}(\frkp(Z_0)^{-1})\del_{\ulk}^nf(\ulv_0^n)(Z_0)\in\call_{\ulk}(\overline{\QQ}). \]
Equivalently, 
\[(\Ome_\infty/\pi)^{k_1+k_2-k_3-2n}[\del_{\ulk}^nf(\ulv_0^n)]_i(\jmath(\tau_0))\in\overline{\QQ} \]
for $i=0,1,2,\dots,k_2-k_1$. 
Our proof is now complete by Proposition \ref{prop:61}. 
\end{proof}

We write $\Hol:\scrs^{[t]}_{\ulk'}(p^\ell\frkN',\chi')\to S_{\ulk'}^H(p^\ell\frkN',\chi')$ for the holomorphic projection. 
Recall that if $2t<k_2'-k_1'$ and $f\in\scrs^{[t]}_{\ulk'}(p^\ell\frkN',\chi')$, then $\Hol(f)\in S_{\ulk'}^H(p^\ell\frkN',\chi')$ is the unique cusp form such that 
\[f-\Hol(f)=\sum_{\nu=1}^t\del_{(k_1'+\nu;k_2'-\nu)}^\nu h_\nu\] 
with $h_\nu\in S_{(k_1'+\nu;k_2'-\nu)}^H(p^\ell\frkN',\chi')$ (cf. Remark \ref{rem:23}). 

\begin{proposition}\label{prop:63}
Notations and assumptions being as in Proposition \ref{prop:62}, we let $\bet=\bfd(a)$ and put $\frka=a\frko_E$. 
Then
\[\bdse'\frac{\Hol([\del_{\ulk}^nf(\ulv_0^n)]_i\circ\jmath)(\tau,\bet)}{(2\pi\sqrt{-1}/\Ome_\infty)^{k_2-i-n}}=\lim_{j\to\infty}\biggl(\frac{2\pi\sqrt{-1}}{\Ome_\infty}\biggl)^{i-k_2}\sum_{m\in\QQ^\times_+}\frac{\partial^n\FJ^{mp^{j!}}_\frka(f)_i}{\Ome_\infty^n\;\partial w^n}(0)q^m. \]
\end{proposition}

\begin{remark}\label{rem:32} 
Proposition 4.8 of \cite{MHarris} says that the left hand side coincides with the $p$-adic differential operator $\Theta^n$.
\end{remark}

\begin{proof}
Put $\calf=\Big(\frac{2\pi\sqrt{-1}}{\Ome_\infty}\Big)^{n+i-k_2}[\del_{\ulk}^nf(\ulv_0^n)]_i\circ\jmath$. 
Observe that 
\[2\min\{n,k_2-k_1-i\}<k_3+n-k_1-i. \]
We can write 
\beq
\calf=\Hol(\calf)+\sum_{\nu=1}^{\min\{n,k_2-k_1-i\}}\del_{(k_1'+\nu;k_2'-\nu)}^\nu h_\nu \label{tag:62}
\eeq 
with $h_\nu\in S_{(k_1+i+\nu;k_3+n-\nu)}^H(p^\ell\frkN',\chi|_{K_0'(p^\ell\frkN)})$. 
Recall the relation 
\[\del^\nu_k=\sum_{a=0}^\nu{\nu\choose a}\frac{\Gam(\nu+k)}{\Gam(a+k)}(-4\pi y)^{a-\nu}\left(\frac{1}{2\pi\sqrt{-1}}\frac{\partial}{\partial z}\right)^a \]
(see \cite[(3), p.~311]{Hid93} and Remark \ref{rem:23}). 
Since 
\[[D_{\ulk}^nf(\ulv_0^n)]_i(\jmath(\tau))=\frac{\partial^nf_i}{\partial w^n}(\jmath(\tau))\pmod{\eta(\tau)^{-1}}\]
by (\ref{tag:60}), we see that $\Hol(\calf)$ equals
\[\biggl(\frac{2\pi\sqrt{-1}}{\Ome_\infty}\biggl)^{i-k_2}\sum_{m\in\QQ^\times_+}\frac{\partial^n\FJ^m_\frka(f)_i}{\Ome_\infty^n\;\partial w^n}(0)q^m-\sum_{\nu=1}^{\min\{n,k_2-k_1-i\}}\tht^\nu h_\nu\]
by equating the constant terms of the identity (\ref{tag:62}) as a polynomial in $\frac{1}{\eta(\tau)}=\frac{1}{2y}$. 
Here $\tht$ stands for Serre's operator $\tht(\sum_ma_mq^m)=\sum_mma_mq^m$. 
Since $\bdse' \circ\tht=0$, we get the stated formula in view of (\ref{tag:52}). 
\end{proof}

\begin{remark}\label{rem:33}
Define a representation $\rho'_{\ulk}$ of $\GL_2(\CC)\times\CC^\times$ by 
\begin{align*}
\rho'_{\ulk}(g,t)&=\rho_{\ulk}^{}(\sqrt{T'}g\sqrt{T'}^{-1},t), &
\sqrt{T'}&=\begin{bmatrix} 1 & 0 \\ 0 & \sqrt{\frac{|\del|}{2}} \end{bmatrix}. 
\end{align*}
Here $|\del|=\sqrt{D_E}$. 
Note that $\rho'_{\ulk}(\overline{\lam(\alp,\bfi)},1)=\rho'_{\ulk}(\trs \lam(\alp,\bfi)^{-1},1)$ for $\alp\in\calk_\infty$ by the isomorphism (\ref{tag:23}).  
Set 
\[f'=\rho_{\ulk}(\sqrt{T'},1)f. \]
Define the differential operator $D_{\rho_{\ulk}'}$ in the same way as in \S \ref{ssec:a2} but with 
\begin{align*}
\xi'(Z)&=\sqrt{T'}^{-1}\xi(Z)\sqrt{T'}^{-1}, & 
u'&=\sqrt{T'}u 
\end{align*}
for $u\in\frkT$. 
Observe that 
\[[\del_{\rho_{\ulk}'}^nf'(\ulv_0^{\prime n})]_i=(\rho_{\ulk}\otimes\upsilon^n)(\sqrt{T'},1)[\del_{\ulk}^nf(\ulv_0^n)]_i=\biggl(\frac{2}{|\del|}\biggl)^\frac{k_2-i-n}{2}[\del_{\ulk}^nf(\ulv_0^n)]_i. \]
\end{remark}


\subsection{The work of Bannai and Kobayashi}\label{ssec:67}

We fix a number field $F$ containing $E$ and that for every $\alp\in E$ the morphism $\iot(\alp)$ is also defined over $F$, where $\iot:E\hookrightarrow\End(E_\frka)\otimes\QQ$ is an embedding. 
Fix a nonzero invariant differential one form $\ome_{E_\frka}$ defined over $F$ such that $\iot(\alp)^*\ome_{E_\frka}=\iot_\infty(\alp)\ome_{E_\frka}$ for $\alp\in E$. 
Fix the complex period $\Ome_\infty\in\CC^\times$ so that the pullback of $\ome_{E_\frka}$ via $\CC/\Gam_\frka\to E_\frka(\CC)$ is $\d w$, where $\Gam_\frka:=\iot_\infty(\frka)\Ome_\infty$ is called a period lattice of $(E_\frka,\ome_{E_\frka})$ and $w$ is the canonical basis of $\CC$. \index{$\Ome_\infty$}
We associate to $\tht\in\bfT^m_\frka(\CC)$ a reduced theta function $\tht^{\Ome_\infty}$ on $\CC/\Gam_\frka$ by \[\tht^{\Ome_\infty}(w):=\tht(w/\Ome_\infty). \]
Bannai and Kobayashi have proved that it has algebraic Taylor coefficients. 

\begin{theorem}[{\cite[Theorem 2.9]{BK10Duke}}]\label{thm:31}
If $\tht\in \bfT^m_\frka(\overline{\QQ})$, then for every $l\in E$
\[\Ome_\infty^{-n}\frac{\partial^nA_m(l)\tht}{\partial w^n}(0)\in\overline{\QQ}. \] 
\end{theorem}

\begin{remark}\label{rem:36}
We define the translation $A^{\Ome_\infty}_m(l)$ by $l\in E$ so that 
\[A^{\Ome_\infty}_m(l)\tht^{\Ome_\infty}=[A_m^{}(l)\tht]^{\Ome_\infty}. \index{$A^{\Ome_\infty}_m(l)$}\]
It is explained in \cite{BK10Duke,MH} that the translation $A^{\Ome_\infty}_m(l)$ by a torsion point $l$ preserves the algebraicity. 
Consequently, Definition \ref{def:25} coincides with the notion of algebraicity of theta functions in \cite{BK10Duke}. 
\end{remark}

Take a Weierstrass model
\[\scre_\frka: y^2=4x^3-g_2x-g_3\]
which is good for the prime ideal $\frkP$ of $R_F$. 
Let $F_\frkP$ be the completion of $F$ relative to $\frkP$. 
Denote by $\calo_{F_\frkP}$ the maximal compact subring of $F_\frkP$. 
Let $\widehat{\scre_\frka}$ be the formal group associated to $\scre_\frka$ with respect to the parameter $t=-\frac{2x}{y}$. 
We denote by $\lam(t)$ the formal logarithm of $\widehat{\scre_\frka}$, which is a power series giving a homomorphism of
formal groups $\widehat{\scre_\frka}\to\widehat{\GG}_a$, $z =\lam(t)$ and is normalized so that $\lam'(0) = 1$. 
Since the coefficients of the Taylor expansion of $\tht^{\Ome_\infty}(w)$ are algebraic, we can consider its formal composition $\widehat{\tht^{\Ome_\infty}}(t)$ with $w=\lam(t)$. 

Bannai and Kobayashi proved the $\frkp$-integrality of the Taylor expansion of reduced theta functions associated to meromorphic sections of line bundles of CM abelian varieties with respect to formal group parameters. 
Moreover, they computed the $p$-adic translation by $\frkp$-power torsion points. 
We denote the ring of integers of $\CC_p$ by $\calo_{\CC_p}$ and the maximal ideal of $\CC_p$ by $\frkm_{\CC_p}$. 
Let 
\[\wp_p:\widehat{\scre_\frka}(\frkm_{\CC_p})_\mathrm{tor}\to E_\frka(\CC_p)_\mathrm{tor}=E_\frka(\overline{\QQ})_\mathrm{tor}=E_\frka(\CC)_\mathrm{tor}\stackrel{\sim}{\to}E/\frka\]
be the morphism obtained by the complex uniformization $\wp:E_\frka(\CC)\simeq \CC/\frka$. 
We denote the set of $\frkp^n$-power torsion point of the formal group $\widehat{\scre_\frka}$ by  $\widehat{\scre_\frka}[\frkp^n]$. 
When $\calg$ is a formal group, we write $\oplus_\calg$ for the formal addition. 

\begin{theorem}[{\cite[Propositions 2.15, 2.16, 2.20]{BK10Duke}}]\label{thm:32}
Let $\tht\in \bfT^m_\frka(\calo)$. 
\begin{enumerate}
\item\label{thm:321} $\widehat{\tht^{\Ome_\infty}}(t)=\tht^{\Ome_\infty}(\lambda(t))\in\calo_\frkP\powerseries{t}$. 
\item\label{thm:322} Let $\vep$ be an element of $\frko_E$ such that $\vep\equiv1\pmod{\frkp^n}$ and $\vep\in 2(\frkp^c)^n$. 
Then $\widehat{\tht^{\Ome_\infty}}(t\oplus_{\widehat{\scre}_\frka} t_{\frkp^n})=A^{\Ome_\infty}_m\yhwidehat{(\vep\wp_p(t_{\frkp^n}))\tht^{\Ome_\infty}}(t)$ for every $t_{\frkp^n}\in\widehat{\scre_\frka}[\frkp^n]$. 
\end{enumerate}
\end{theorem}

\begin{remark}\label{rem:37}
More generally, Bannai and Kobayashi study reduced theta functions associated to meromorphic sections of $\frkL_\frka^m$. 
Since we consider only holomorphic sections, we can let $f=1$ in Propositions 2.15, 2.16 and Theorem 2.9 of \cite{BK10Duke}. 
We here apply Proposition 2.20 of \cite{BK10Duke} with $v_0=0$. 
\end{remark}

Let $\WW=\WW(\overline{\FF}_p)$ be the ring of Witt vectors with coefficients in $\overline{\FF}_p$. 
It is known that there exists an isomorphism $\Xi_p:\widehat{\scre_\frka}\simeq\widehat{\GG}_m$ of formal groups over $\WW$. 
We fix such an isomorphism $\Xi_p$, and then $\Xi_p$ is given by a power series 
\[\Xi_p(t)=\exp(\Ome_\frkp^{-1}\lam(t))-1, \]
where $\Ome_\frkp\in\WW^\times$ is a suitable $p$-adic period. \index{$\Ome_\frkp$}
We denote by $\Ups(T)$ the inverse of $\Xi_p$. 
Since $\Xi_p$ is an isomorphism, the coefficients of $\Ups(T)$ are also in $\WW$. 
We associate to $g(t)\in\WW\powerseries{t}$ another power series $g(\Ups(T))\in\WW\powerseries{T}$ on $\widehat{\GG}_m$. 

We briefly recall the bijection between the space $\mathrm{Dist}(\ZZ_p,\calo)$ of $p$-adic measures on $\ZZ_p$ and the ring of the power series $\calo\powerseries{T}$ given by the $p$-adic Mellin transform. See Section 4 of \cite{GTM121} for details. 
The bijection $g\mapsto\d\mu(g)$ is given by 
\[\int_{\ZZ_p}(1+T)^x\d\mu(g)=\sum_{k=0}^\infty T^k\int_{\ZZ_p}\binom{x}{k}\d\mu(g)=g(T).\]
The power series $g(T)$ is called the $p$-adic Mellin transform of $\d\mu(g)$. For every non-negative integer $n$ we have 
\beq
\int_{\ZZ_p}x^n\d\mu(g)=D_T^n g(0), \label{tag:63}
\eeq
where $D_T=(1+T)\frac{\d}{\d T}$. For an $\calo_{\CC_p}$-valued function $\phi$ on $\ZZ_p$ we define $\underline{\bdTht}_\frkp^\phi g\in \calo\powerseries{T}$ by
\[\underline{\bdTht}_\frkp^\phi g(T)=\int_{\ZZ_p}\phi(x)(1+T)^x\d\mu(g)=\sum_{k=0}^\infty T^k\int_{\ZZ_p}\phi(x)\binom{x}{k}\d\mu(g). \index{$\underline{\bdTht}_\frkp^\phi$}\]
If $\phi$ factors through $\ZZ_p\to\ZZ_p/p^n\ZZ_p$, then 
\beq
\underline{\bdTht}_\frkp^\phi g(T)=\frac{1}{p^n}\sum_{u\in\ZZ/p^n\ZZ}\sum_{\zet\in\mu_{p^n}}\zet^{-u}\phi(u)g(\zet(T+1)-1),\label{tag:64}
\eeq
where $\mu_{p^n}$ denotes the group of $p^n$th roots of unity in $\overline{\QQ}_p$. 

\subsection{The $\Theta_{\frkp}$-operator}

Let $f\in S_{\ulk}^G(p^\ell\frkN,\chi;\calo)$. 
Fix a locally constant function $\breve\phi\in\calc(\ZZ_p^\times,\calo)$. 
To simplify notation, we put \index{$\mho$}
\begin{align*}
f'&=\bfU_{\frkp^c}^{\breve\phi}\bfU^{\chi'}_{\frkN'}f, & 
\mho&=\frac{2\pi\sqrt{-1}\Ome_\frkp}{\Ome_\infty}. 
\end{align*} 

\begin{remark}
The period ratio $\frac{2\pi\sqrt{-1}\Ome_\frkp}{\Ome_\infty}$ is not an element of any ring. 
The symbol $\mho$ has to be interpreted as an operation rather than a number.
\end{remark}

Let $j$ be sufficiently large such that $\breve\phi$ factors through $(\ZZ_p/p^j\ZZ_p)^\times$. 
Put $\frkc=(\frkp^c)^j\frkb_0$. 
Proposition \ref{prop:64} shows that 
\[\tht^m_\frka(f'):=\frac{\FJ^m_\frka(f')_0^{}}{\mho^{k_2}}\in \bfT_{\frka\frkc}^m(\calo).\]
Therefore the reduced theta function $\tht^m_\frka(f')^{\Ome_\infty}$ on an elliptic curve $E_{\frka\frkc}(\CC)=\CC/\iot_\infty(\frka\frkc)\Ome_\infty$ has algebraic Taylor coefficients by Theorem \ref{thm:31}. 
Moreover, we define the power series
\beq
\widetilde{\FJ}^m_\frka(f')(T):=\yhwidehat{\tht^m_\frka(f')^{\Ome_\infty}}(\Ups(T)). \label{tag:67}
\eeq
Then $\widetilde{\FJ}^m_\frka(f')(T)$ is a power series with $p$-integral coefficients by Theorem \ref{thm:32}(\ref{thm:321}). 
Hence we can associate the $p$-adic measure $\d\mu(\widetilde{\FJ}^m_\frka(f'))$. 



\begin{definition}\label{D:Thop}
Let $U$ be an open compact subgroup of $G(\widehat{\QQ})$ which contains $\caln(\frkc)$. 
We define the theta operator $\bdTht_\frkp^\phi\calf$ for $\calf\in S_{\ulk}^G(U;\calo)$ with respect to a locally constant function $\phi\in\calc(\ZZ_p^\times,\calo)$ by 
\[[\bdTht_\frkp^\phi\calf](g)=\frac{1}{p^n}\sum_{x\in p^{-n}\ZZ_p/\ZZ_p}\sum_{z\in (\ZZ_p/p^n\ZZ_p)^\times}\phi(z)\addchar_p(zx)\calf\left(g\imath_\frkp^{-1}\left(\begin{bmatrix} 1 & 0 & 0 \\ 0 & 1 & x \\ 0 & 0 & 1 \end{bmatrix}\right)\right), \index{$\bdTht_\frkp^\phi$}\]
provided that $\phi$ factors through the quotient $\ZZ_p^\times\to(\ZZ_p/p^n\ZZ_p)^\times$. 
It is easy to check that this definition is independent of the choice of $n$ (cf. (\ref{tag:64})). 
\end{definition}
\begin{proposition}\label{prop:65}
If $\phi\in \calc(\ZZ_p^\times,\calo)$ is locally constant, then for $\calf\in S_{\ulk}^G(U;\calo)$ 
\[\widetilde{\FJ}^m_\frka(\bdTht_\frkp^\phi\calf)=\underline{\bdTht}_\frkp^\phi\widetilde{\FJ}^m_\frka(\calf). \]
\end{proposition}

\begin{proof}
Take a set $\calx_n$ of representatives for $\frkp^{-n}\frkc/\frkc$. 
Given $x\in\EE$, we write its $p$-component by $x_p\in E_p$. 
Put 
\[\bfn'(x,0)=\imath_\frkp^{-1}\left(\begin{bmatrix} 1 & 0 & 0 \\ 0 & 1 & x \\ 0 & 0 & 1 \end{bmatrix}\right)\]
for $x\in\QQ_p$. 
By definition, we have\begin{align*}
&\FJ_\frka^m(\bdTht_\frkp^\phi\calf)_0^{}(w)\\
=&\frac{1}{p^n}\sum_{x\in p^{-n}\ZZ_p/\ZZ_p}\sum_{z\in (\ZZ_p/p^n\ZZ_p)^\times}\phi(z)\addchar_p(zx)\FJ_{\bfn'(x,0)\bfd(a)}^m(\calf)_0^{}(w) \\
=&\frac{1}{p^n}\sum_{x\in\calx_n}\sum_{z\in (\ZZ_p/p^n\ZZ_p)^\times}\phi(z)\addchar_p(z\Tr_{E_p/\QQ_p}(x_p))[A_m(-x)\FJ_\frka^m(\calf)_0^{}](w).  
\end{align*} 
Define an isomorphism $\iot_n:\widehat\scre_{\frka\frkc}[\frkp^n]\simeq\mu_{p^n}$ by $\iot_n(t_{\frkp^n})=e^{-2\pi\sqrt{-1}\Tr_{E/\QQ}(\wp_p(t_{\frkp^n}))}$.
Then Theorem \ref{thm:32}(\ref{thm:322}) shows that 
\[\yhwidehat{\tht_\frka^m(\bdTht_\frkp^\phi\calf)^{\Ome_\infty}}(t)\\
=\frac{1}{p^n}\sum_{t_{\frkp^n}\in\scre_{\frka\frkc}[\frkp^n]}\sum_{z\in (\ZZ_p/p^n\ZZ_p)^\times}\phi(z)\iot_n(t_{\frkp^n})^z\yhwidehat{\tht_\frka^m(\calf)^{\Ome_\infty}}(t\oplus_{\widehat{\scre}_{\frka\frkc}} t_{\frkp^n}). \] 
We conclude by (\ref{tag:64}) that 
\begin{align*}\widetilde{\FJ}_\frka^m(\bdTht_\frkp^\phi\calf)(T)
=&\frac{1}{p^n}\sum_{\zet\in\mu_{p^n}}\sum_{z\in (\ZZ_p/p^n\ZZ_p)^\times}\phi(z)\zet^z\cdot \widetilde{\FJ}_\frka^m(\calf)(T\oplus_{\widehat{\GG_m}}(\zet-1))\\=&\underline{\bdTht}_\frkp^\phi\widetilde{\FJ}_\frka^m(\calf), 
\end{align*}
as claimed. 
\end{proof}


\subsection{The $p$-adic interpolation of $\bdTht_\frkp^\phi$}

Let $J_\frka^{\chi',m}(f)\in\Lam_2$ correspond to the $\calo$-valued distribution 
\[(\phi,\breve\phi)\mapsto\int_{\ZZ_p^\times}\phi(x)\,\d\mu(\widetilde{\FJ}^m_\frka(\bfU_{\frkp^c}^{\breve\phi}\bfU^{\chi'}_{\frkN'}f)). \]
Namely, for $\ulQ=(Q_1,Q_2)\in\frkX_{\Lam_2}$ with $k_{\ulQ}=(0,0)$ and $\eps_{\ulQ}=(\eps_{Q_1},\eps_{Q_2})$, we have  
\[J_\frka^{\chi',m}(f)_{\ulQ}=\int_{\ZZ_p^\times}\eps_{Q_2}(x)\,\d\mu(\widetilde{\FJ}^m_\frka(\bfU_{\frkp^c}^{\eps_{Q_1}}\bfU^{\chi'}_{\frkN'}f)). \]

\begin{proposition}\label{prop:66}
If $k_{Q_1}=0$ and $k_{Q_2}\geq 0$, then  
\[J_\frka^{\chi',m}(f)_{\ulQ}=\mho^{k_{Q_2}-k_2}\frac{\partial^{k_{Q_2}}\FJ_\frka^m(\bdTht_\frkp^{\eps_{Q_2}}\bfU_{\frkp^c}^{\eps_{Q_1}}\bfU^{\chi'}_{\frkN'}f)_0^{}}{(2\pi\sqrt{-1})^{k_{Q_2}}\partial w^{k_{Q_2}}}(0). \]
\end{proposition}

\begin{proof}
By the definition of $\underline{\bdTht}_\frkp^{\eps_{Q_2}}$ in \S \ref{ssec:67} we have 
\begin{align*}
J_\frka^{\chi',m}(f)_{\ulQ}
&=\int_{\ZZ_p^\times}x^{k_{Q_2}}\eps_{Q_2}(x)\,\d\mu(\widetilde{\FJ}^m_\frka(\bfU_{\frkp^c}^{\eps_{Q_1}}\bfU^{\chi'}_{\frkN'}f))\\
&=\int_{\ZZ_p}x^{k_{Q_2}}\,\d\mu(\underline{\bdTht}_\frkp^{\eps_{Q_2}}\widetilde{\FJ}^m_\frka(\bfU_{\frkp^c}^{\eps_{Q_1}}\bfU^{\chi'}_{\frkN'}f)). 
\end{align*}
Proposition \ref{prop:65} and (\ref{tag:63}) show that 
\[J_\frka^{\chi',m}(f)_{\ulQ}
=D_T^{k_{Q_2}}\widetilde{\FJ}^m_\frka(\bdTht_\frkp^{\eps_{Q_2}}\bfU_{\frkp^c}^{\eps_{Q_1}}\bfU^{\chi'}_{\frkN'}f)(0). \]
Since $\lam(\Ups(T))=\Ome_\frkp\log(1+T)$, we conclude that  
\[J_\frka^{\chi',m}(f)_{\ulQ}
=\mho^{-k_2}\Ome_\frkp^{k_{Q_2}}\frac{\partial^{k_{Q_2}}\FJ^m_\frka(\bdTht_\frkp^{\eps_{Q_2}}\bfU_{\frkp^c}^{\eps_{Q_1}}\bfU^{\chi'}_{\frkN'}f)_0^{}}{\Ome_\infty^{k_{Q_2}}\partial w^{k_{Q_2}}}(0) \]
as claimed. 
\end{proof}


\subsection{Construction of $p$-adic families}\label{ssec:68}
Put 
\[\frkY_{\Lam_2}^{\ulk}=\{(Q_1',Q_2')\in\frkX_{\Lam_2}^{}\;|\;k_1=k_{Q_1'}\leq k_2,\;k_3\leq k_{Q_2'}\}. \]
We write $\chi_p|_{K_0'(p^\ell)}=\eps_1^\uparrow\eps_3^\downarrow$, where 
\begin{align*}
\eps_1^\uparrow(u)&=\eps_1^{}(a), &
\eps_3^\downarrow(u)&=\eps_3^{}(d) 
\end{align*}
for $u=\imath_\frkp^{-1}\biggl(\begin{bmatrix} a & b \\ c & d\end{bmatrix}\biggl)\in K_0'(p^\ell)$. \index{$\ome^\downarrow$}

Proposition \ref{prop:66} gives an element $J^{\chi',\eps_1,\eps_3,m}_{\frka,\ulk}(f)\in\Lam_2$ such that 
\[\ulQ'(J_{\frka,\ulk}^{\chi',\eps_1,\eps_3,m}(f))
=\mho^{k_{Q_2'}-k_2-k_3}\frac{\partial^{k_{Q_2'}-k_3}\FJ_\frka^m(\bdTht_\frkp^{\eps_{Q_2'}\eps_3^{-1}}\bfU_{\frkp^c}^{\eps_{Q_1'}\eps_1^{-1}}\bfU^{\chi'}_{\frkN'}f)_0^{}}{(2\pi\sqrt{-1})^{k_{Q_2'}-k_3}\partial w^{k_{Q_2'}-k_3}}(0) \]
for $\ulQ'\in\frkY_{\Lam_2}^{\ulk}$. 
Now we define the formal power series $J^{\chi',\eps_1,\eps_3}_{\frka,\ulk}(f)\in\Lam_2\powerseries{q}$ by 
\[J^{\chi',\eps_1,\eps_3}_{\frka,\ulk}(f)=\lim_{j\to\infty}\sum_{m\in\scrs_\frka^+}J_{\frka,\ulk}^{\chi',\eps_1,\eps_3,mp^{j!}}(f)\,q^m. \] 
We will see that the limit makes sense in the proof of the following corollary. 


\begin{corollary}\label{cor:32}
Notation being as above, if $\ulQ'=(Q_1',Q_2')\in\frkY_{\Lam_2}^{\ulk}$, then $J^{\chi',\eps_1,\eps_3}_{\frka,\ulk}(f)_{\ulQ'}$ is the Fourier expansion at $\frka$ of 
\[\mho^{k_{Q_2'}-k_2-k_3}\bdse'\Hol\biggl(\biggl[\del_{\ulk}^{k_{Q_2'}-k_3}\bdTht_\frkp^{\eps_{Q_2'}\eps_3^{-1}}\bfU_{\frkp^c}^{\eps_{Q_1'}\eps_1^{-1}}\bfU^{\chi'}_{\frkN'}f(v_0^{k_{Q_2'}-k_3})\biggl]_0\circ\jmath\biggl). \] 
In particular, we have 
\[J^{\chi',\eps_1,\eps_2}_{\ulk}(f)\in \bfS^H_{\ord}(\frkN',\chi',\Lam_2). \]
\end{corollary}

\begin{proof}
Proposition \ref{prop:63} shows that 
\[\calf:=\Hol\biggl(\biggl[\del_{\ulk}^{k_{Q_2'}-k_3}\bdTht_\frkp^{\eps_{Q_2'}\eps_3^{-1}}\bfU_{\frkp^c}^{\eps_{Q_1'}\eps_1^{-1}}\bfU^{\chi'}_{\frkN'}f(v_0^{k_{Q_2'}-k_3})\biggl]_{k_{Q_1'}-k_1}\circ\jmath\biggl) \]
has weight $(k_{Q_1'};k_{Q_2'})$. 
Moreover, Proposition \ref{prop:66} says that 
\[\lim_{j\to\infty}\ulQ'(J_{\frka,\ulk}^{\chi',\eps_1,\eps_3,mp^{j!}}(f))=\mho^{k_{Q_2'}-k_2-k_3}\FC^m_\frka(\bdse'\calf) \]
by (\ref{tag:52}). 
The open compact subgroup $K_0'(p^\ell\frkN')$ acts on $\calf$ by the character $(\chi'\eps_{\ulQ})^{-1}$ in view of Corollary \ref{cor:81} and \cite[Proposition 6.7(i)]{HY}, which confirms that $J^{\chi',\eps_1,\eps_2}_{\ulk}(f)_{\ulQ'}\in\bdse'S^H_{k_{\ulQ'}}(p^\ell\frkN',\chi'\eps_{\ulQ'},\Lam_2(\ulQ'))$.   
\end{proof}


\subsection{Construction of theta elements}\label{ssec:69}

We normalize the integrals by 
\begin{align*}
\Pet_{K'}(\Phi)&=\int_{H(\QQ)\bsl H(\AA)}\Phi((h,h))\,\d_{K'}h, \\
\scrp_{K'}(\Phi')&=\int_{H(\QQ)\bsl H(\AA)}\Phi'((h,h))\,\d_{K'}h
\end{align*}
for $\Phi\in\scra^0(H\times H)$ and $\Phi'\in\scra^0(G\times H)$, where $\d_{K'}h$ is defined in \S \ref{ssec:73}. 
Let $\bfI_n$ be a normal ring finite flat over $\Lam_n=\calo\powerseries{T_n(\ZZ_p)}$. 
Let 
\begin{align*}
\bdsf&\in\bfS^\calg_{\ord}(\frkN,\chi,\bfI_3), & 
\bdsg&\in\bfS^\calh_{\ord}(\frkN',\chi',\bfI_2)
\end{align*} 
be Hida families. 
We define $J_\frka^{\chi'}(\bdsf)\in\bfI_3\widehat{\otimes}\bfI_2\powerseries{q}$ by
\[J_\frka^{\chi'}(\bdsf)_{\calq}=J_{\frka,k_{\ulQ}}^{\chi',\eps_{Q_1},\eps_{Q_3}}(\bdsf_{\ulQ})_{\ulQ'} \]
for $\calq=(\ulQ,\ulQ')\in\frkY^{\rm crit}_\bfV$. 
By (\ref{tag:55}) and (\ref{tag:67}) we have 
\[J_{\frka,k_{\ulQ}}^{\chi',\eps_{Q_1},\eps_{Q_3},m}(\bdsf_{\ulQ})_{\ulQ'}=\int_{\ZZ_p^\times}(Q_2'Q_3^{-1})(x)\,\d\mu(\yhwidehat{\FJ^m_\frka(\bfU_{\frkp^c}^{Q_1'Q_1^{-1}}\bfU^{\chi'}_{\frkN'}\ulQ(\bdsf))_0^{\Ome_\infty}}\circ\Upsilon). \]
Corollary \ref{cor:32} says that $\{J_\frka^{\chi'}(\bdsf)\}_{\frka\in C_E}$ gives an $\bfI_3\widehat{\otimes}\bfI_2$-adic cusp form
\[J^{\chi'}(\bdsf)\in\bfS^H_{\ord}(\frkN',\chi',\bfI_3\widehat{\otimes}\bfI_2). \]

We may assume that $\frko_E\in C_E$. 
Writing $\1_{\bdsg}\in\bfT^H_{\ord}(N',\chi',,\bfI_2)\otimes_{\bfI_2}\mathrm{Frac}\bfI_2$ for the idempotent corresponding to $\bdsg$ (cf. Proposition \ref{prop:52}), we define the theta element $\Tht_{\bdsf,\bdsg}\in\bfI_3\widehat{\otimes}\mathrm{Frac}\bfI_2$ attached to $\bdsf$ and $\bdsg$ by   
\[\Tht_{\bdsf,\bdsg}=\FC^1_{\frko_E}(\1_{\bdsg}J^{\chi'}(\bdsf)). \index{$\Tht_{\bdsf,\bdsg}$}\]
See Definition \ref{def:55} for the symbol $\vph_{\bdsig_{\ulQ'}}^\ord\in\bdsig_{\ulQ'}$. 
Let $\vph_{\bdsig_{\ulQ'}^\vee}^\flat\in\bdsig_{\ulQ'}^\vee$ be a $\calu'_p$-eigenform with $p$-unit eigenvalue which is anti-holomorphic and $K_1'(p^\ell\frkN')$-invariant. 

\begin{proposition}\label{prop:67}
If $\ell$ is sufficiently large, then 
\begin{align*}
&\calq(\Tht_{\bdsf,\bdsg})\Pet_{K'}\Big(\vph_{\bdsig_{\ulQ'}}^\ord\otimes\bdsig_{\ulQ'}^\vee(t_\ell)^{-1}\vph_{\bdsig_{\ulQ'}^\vee}^\flat\Big)\\
=&\mho^{m_\calq}\scrp_{K'}\biggl(\del_{k_{\ulQ}}^{k_{Q_2'}-k_{Q_3^{}}}\bdTht_\frkp^{\eps_{Q_2'}^{}\eps_{Q_3^{}}^{-1}}\bfU_{\frkp^c}^{\eps_{Q_1'}^{}\eps_{Q_1^{}}^{-1}}\bfU^{\chi'}_{\frkN'}\vPh(\bdsf_{\ulQ})_0\otimes\bdsig_{\ulQ'}^\vee(t_\ell)^{-1}\vph_{\bdsig_{\ulQ'}^\vee}^\flat\biggl)  
\end{align*}
for $\calq=(\ulQ,\ulQ')\in\frkY^{\rm crit}_\bfV\cap(\frkX_\calr^\cls\times\frkX_{\bdsg}'')$, where 
\[m_{\calq}=k_{Q_2'}-(k_{Q_2}^{}+k_{Q_3}^{}). \]
\end{proposition}

\begin{proof}
The proof is similar to that of Proposition 3.7 of \cite{MH}. 
Since 
\[\Pet_{K'}(\calu_p'\vph,\bdsig_{\ulQ'}^\vee(t_\ell)^{-1}\vph')=\Pet_{K'}(\vph,\bdsig_{\ulQ'}^\vee(t_\ell)^{-1}\calu_p'\vph'), \]
we have 
\[\Pet_{K'}(\bdse'\vph,\bdsig_{\ulQ'}^\vee(t_\ell)^{-1}\vph_{\bdsig_{\ulQ'}^\vee}^\flat)=\Pet_{K'}(\vph,\bdsig_{\ulQ'}^\vee(t_\ell)^{-1}\vph_{\bdsig_{\ulQ'}^\vee}^\flat). \]
Since $\vPh(\bdsg_{\ulQ'})=\FC^1_{\frko_E}(\bdsg_{\ulQ'})\vph_{\bdsig_{\ulQ'}}^\ord$ is a $p$-stabilized ordinary newform, we have 
\[\calq(\Tht_{\bdsf,\bdsg})\vph_{\bdsig_{\ulQ'}}^\ord=\1_{\bdsg_{\ulQ'}}\vPh(J^{\chi'}(\bdsf)_{\calq})\]
by multiplicity one for $\U(1,1)$, and for new and ordinary forms (see Remark \ref{rem:41}, \S \ref{ssec:82} and \cite[Proposition 2.2]{MH}). 
Note that $\FC^1_{\frko_E}(\bdsg_{\ulQ'})\neq 0$ by Proposition \ref{prop:53}. 
Taking the pairing with $\bdsig_{\ulQ'}^\vee(t_\ell)^{-1}\vph_{\bdsig_{\ulQ'}^\vee}^\flat\in\bdsig_{\ulQ'}^\vee$, we get 
\begin{align*}
&\calq(\Tht_{\bdsf,\bdsg})\Pet_{K'}\Big(\vph_{\bdsig_{\ulQ'}}^\ord\otimes\bdsig_{\ulQ'}^\vee(t_\ell)^{-1}\vph_{\bdsig_{\ulQ'}^\vee}^\flat\Big)\\
=&\Pet_{K'}\Big(\vPh(J^{\chi'}(\bdsf)_\calq)\otimes\bdsig_{\ulQ'}^\vee(t_\ell)^{-1}\vph_{\bdsig_{\ulQ'}^\vee}^\flat\Big) \\
=&\Pet_{K'}\Big(\vPh\bigl(J_{k_{\ulQ}}^{\chi',\eps_{Q_1},\eps_{Q_3}}(\bdsf_{\ulQ})_{\ulQ'}\bigl)\otimes\bdsig_{\ulQ'}^\vee(t_\ell)^{-1}\vph_{\bdsig_{\ulQ'}^\vee}^\flat\Big) \\
=&\mho^{m_\calq}\scrp_{K'}\biggl(\del_{k_{\ulQ}}^{k_{Q_2'}-k_{Q_3^{}}}\bdTht_\frkp^{\eps_{Q_2'}^{}\eps_{Q_3^{}}^{-1}}\bfU_{\frkp^c}^{\eps_{Q_1'}^{}\eps_{Q_1^{}}^{-1}}\bfU^{\chi'}_{\frkN'}\vPh(\bdsf_{\ulQ})_0^{}\otimes\bdsig_{\ulQ'}^\vee(t_\ell)^{-1}\vph_{\bdsig_{\ulQ'}^\vee}^\flat\biggl) 
\end{align*}
by Corollary \ref{cor:32}. 
\end{proof}

The theta element $\Theta_{\bdsf,\bdsg}$ is the five variable $p$-adic $L$-function of the title. 
Proposition \ref{prop:67} is the basis for the comparison of the square of the specialization of $\Theta_{\bdsf,\bdsg}$ at classical points with normalized central values of tensor product $L$-functions. 
This comparison is carried out in Section \ref{sec:7}, using the Ichino-Ikeda formula. 
The calculation of the non-archimedean Euler factors is carried out in Section \ref{sec:8}, and that of the archimedean Euler factor is completed in the appendices.


\section{The central value formulas}\label{sec:7}


\subsection{An outer involution}\label{ssec:71}

Put $I_{r,s}=\begin{bmatrix} \ono_r & \\ & -\ono_s \end{bmatrix}$. 
Suppose that $\gam_0^c=-\gam_0$, that is, $\trs\gam_0=\gam_0$.  
Since $S_{\gam_0}^c=-I_{r,s}S_{\gam_0}^{} I_{r,s}$, we can define an outer automorphism of $G$ by
\beq
g^\natural=I_{r,s}g^c I_{r,s}=S_{\gam_0}^{}I_{r,s}\trs g^{-1}I_{r,s}S_{\gam_0}^{-1}. \label{tag:71} \index{$g^\natural$}
\eeq
Observe that for $\alp\in G(\RR)$ and $Z\in\frkD_{r,s}$. 
\begin{align*}
&\alp^\natural(Z)=-\alp(-Z^c)^c, & 
&\lam(\alp^\natural,-Z^c)=\lam(\alp,Z)^c, &
&\mu(\alp^\natural,-Z^c)=\mu(\alp,Z)^c.  
\end{align*} 
It follows that $\calk_\infty^\natural=\calk_\infty^{}$, and for $k\in\calk_\infty$
\begin{align*}
\lam(k^\natural,\bfi)&=\trs(T'\lam(k,\bfi)T^{\prime-1})^{-1}, &
\mu(k^\natural,\bfi)&=\trs\mu(k,\bfi)^{-1}. 
\end{align*} 

Let $\rho=\rho_{\ulk}$. 
Put $T=(T',\ono_s)$. 
Define representations $\rho^c$ and $\rho^T$ by 
\begin{align*}
\rho^c(h)&=\rho(h^c), & 
\rho^T(h)&=\rho(\trs(T hT^{-1})^{-1})
\end{align*}
for $h\in\scrh(R)= \GL_r(R)\times\GL_s(R)$.  
Let $^\natural:L_{\ulk}(R)\stackrel{\sim}{\to}L_{\ulk^\vee}(R)$ such that \index{$P^\natural$}
\begin{align*}
(\rho_{\ulk}^T(h)P)^\natural&=\rho_{\ulk^\vee}(h)P^\natural, &
\ell_{\ulk}(P^\natural\otimes Q)=\ell_{\ulk}(Q^\natural\otimes P). 
\end{align*} 

Let $\chi$ be a character of $K_0(\frkN)$ whose restriction to $K_1(\frkN)$ is trivial.  
Given $f\in S_\rho^G(\frkN,\chi,\CC)$, we define a smooth function $f^\natural:\frkD_{r,s}\times G(\widehat{\QQ})\to L_{\ulk^\vee}(\CC)$ by 
\[f^\natural(Z,\bet)=f(-Z^c,\bet^\natural)^\natural. \index{$f^\natural$}\]

We define $\tau_\frkN^{}=(\tau_{\frkN,l}^{})\in G(\widehat{\QQ})$ by 
\begin{align*}
T_{\gam_0}^{}&=\begin{bmatrix} 0 & 0 & \ono_s \\ 0 & -\gam_0 & 0 \\ \ono_s & 0 & 0 \end{bmatrix}, &
\tau_{\frkN,l}&=\begin{cases}
\imath_\frkl^{-1}\Biggl(T_{\gam_0}\begin{bmatrix} N\ono_{r+s-1}  &  \\ & 1 \end{bmatrix}\Biggl) &\text{if $l|N$, }\\
\ono_{r+s} &\text{otherwise. }
\end{cases}
\end{align*}
Note that $(\tau_\frkN^{}K_0^{}(\frkN)\tau_\frkN^{-1})^\natural=K_0^{}(\frkN)$. 
Put 
\[\tau_\frkN f=(r(\tau_\frkN)f)^\natural. \]
Define the function $\vPh_{\rho^c}(\tau_\frkN f)$ as in \S \ref{ssec:34}.  
Then   
\[\vPh_{\rho^c}(\tau_\frkN f)(\gam g\alp k)=\rho_{\ulk^\vee}(J(\alp,\bfi))^{-1}\vPh_{\rho^c}(\tau_\frkN f)(g)\chi(k)^{-1} \]
for $\gam\in G(\QQ)$, $\alp\in\calk_\infty$ and $k\in K_0(\frkN)$. \index{$\tau_\frkN f$}


\subsection{A factorization of the dual representation}\label{ssec:72}

For a function on $\vph$ on $G(\AA)$ we define a function $\vph^\natural:G(\AA)\to\CC$ by $\vph^\natural(g)=\vph(g^\natural)$.
When $\vph$ is a cusp form on $G$, so is $\vph^\natural$. 
Let $\pi^\vee\simeq\otimes_v'\pi_v^\vee$ denote the contragredient representation of $\pi$. 
We define representations $\pi^\natural$ of $G(\AA)$ and $\pi_v^\natural$ of $G(\QQ_v)$ as the twists $\pi^\natural(g)=\pi(g^\natural)$ and $\pi_v^\natural(g_v^{})=\pi_v(g_v^\natural)$ for $g\in G(\AA)$ and $g_v\in G(\QQ_v)$. \index{$\pi^\natural$}
It is well-known that $\pi_v^\vee\simeq\pi_v^\natural$ (see \cite{MVW}). 
Since 
\[\vph^\natural(gh)=\vph((gh)^\natural)=(\pi^\natural(h)\vph)^\natural(g), \index{$\vph^\natural$}\]
we have $\{\overline{\vph}\;|\;\vph\in\pi\}=\{\vph^\natural\;|\;\vph\in\pi\}$ by the global multiplicity one for unitary groups (see Remark \ref{rem:41}), where the automorphic form $\overline{\vph}$ is defined by $\overline{\vph}(g)=\overline{\vph(g)}$. 

For every non-split prime $q$ we require that $\pi_q$ admits a non-zero $K_q$-invariant vector. 
For each split prime $l$ we denote the conductor of $\pi_l$ by $c(\pi_l)$ in the sense of (\ref{tag:82}). 
Let $N_\pi=\prod_ll^{c(\pi_l)}$ be the conductor of $\pi$. 
We take ideals $\frkN_\pi$ of $\frko_E$ such that $\frko_E/\frkN_\pi\simeq\ZZ/N_\pi\ZZ$. 

Let $\vph_\pi\in\pi$ be an essential vector with respect to $K_1(\frkN_\pi)$ which has highest weight in the minimal $\calk_\infty$-type (see Definition \ref{def:81}). 
Define the longest Weyl element $w_n\in\GL_n(F)$ by 
\begin{align*}
w_1&=1, & 
w_n&=\begin{bmatrix} 0 & 1 \\ w_{n-1} & 0 \end{bmatrix} &
(n&\geq 2). 
\end{align*}
When $l$ is split in $E$, we identify $G(\QQ_l)$ with $\GL_n(E_\frkl)$ via $\imath_\frkl$, realize $\pi_l$ in the Whittaker model $\scrw_{\addchar_l}(\pi_l)$ with respect to $\addchar_l$ and denote by $W_{\pi_l}\in\pi_l^{}$ the normalized essential Whittaker vector with respect to $\addchar_l^{}$, and by $W_{\pi_l^\vee}\in\pi_l^\vee$ the normalized essential Whittaker vector with respect to $\addchar_l^{-1}$. 
For $W\in\pi_l$ we define $W^\natural\in\pi_l^\vee$ by 
\beq
W^\natural(g)=W(w_n\trs g^{-1}T_{\gam_0}^{-1})=\pi_l^\vee(T_{\gam_0})\widetilde{W_l}(g) \index{$W^\natural$}\label{tag:72}
\eeq
(cf. (\ref{tag:71})) for  $g\in\GL_n(E_\frkl)$, where  
\begin{align*}
\widetilde{W_l}(g)&=W_l^{}(w_n\trs g^{-1}), & 
T_{\gam_0}^{}&=-S_{\gam_0}^{} I_{r,s}^{}\in\GL_n(E_\frkl). 
\end{align*}
It is important to note that 
\[\pi^\vee(h) W^\natural(g)=W^\natural(gh)=W((gh)^\natural)=(\pi^\natural(h)W)^\natural(g). \]

When $q$ remains a prime in $\frko_E$, we fix $\calk_q$-invariant vectors $W_{\pi_q}\in\pi_q$ and $W_{\pi_q^\vee}\in\pi_q^\vee$. 
Fix a highest weight vector $W_{\pi_\infty}$ in the minimal $\calk_\infty$-type of $\pi_\infty$ and a lowest weight vector $W_{\pi_\infty^\vee}$ in the minimal $\calk_\infty^{}$-type of $\pi_\infty^\vee$. 
Let $^\natural:\pi_v^\natural\simeq\pi_v^\vee$ be the $G(\QQ_v)$-equivariant isomorphism determined by $W_{\pi_v}^\natural=W_{\pi_v^\vee}$ for $v=q,\infty$. 

To apply the Ichino-Ikeda formula to $\vPh$ and $\vPh^\natural$, we need explicate the factorization of $\vPh^\natural$. 
Fix isomorphisms $\pi\simeq\otimes_v'\pi_v^{}$ and $\pi^\vee\simeq\otimes_v'\pi_v^\vee$ so that 
\begin{align*}
\vph_\pi^{}&=\otimes_v'W_{\pi_v}^{}, &
\vph_\pi^\natural&=\otimes_v'W_{\pi_v}^\natural. 
\end{align*}
Using this factorization $\pi^\vee\simeq\otimes_v'\pi_v^\vee$, we define a cusp form $\vph_{\pi^\vee}\in\pi^\vee$ by 
\[\vph_{\pi^\vee}=\otimes_v'W_{\pi_v^\vee}. \] 

The following result is nothing but Lemma 4.6 of \cite{HY}. 

\begin{lemma}\label{lem:71}
If $\vph=\otimes_v'W_v^{}\in\pi$ is factorizable and if $W_v=W_{\pi_v}$ for all non-split places $v$, then $\vph^\natural=\otimes_v'W_v^\natural$. 
\end{lemma}


\subsection{Measures}\label{ssec:73}

We define a $G(\RR)$-invariant measure $\bfd Z$ on $\frkD_{r,s}$ by 
\beq
\bfd Z=\frac{2^{s(r+s)}}{\det\eta(Z)^{r+s}}\prod_{i=1}^r\prod_{j=1}^s\left[\frac{\sqrt{-1}}{2}\d Z_{ij}\land\d\overline{Z_{ij}}\right]. \index{$\bfd Z$}\label{tag:73}
\eeq 
We extend $\bfd Z$ to a Haar measure $\d g_\infty$ on $G(\RR)$ so that 
\beq
\int_G f(g_\infty(\bfi))\,\d g_\infty=\int_{\frkD_{r,s}}f(Z)\,\bfd Z. \index{$\d g_\infty$}\label{tag:74}
\eeq

From now on we let $r=2$, $s=1$ and $\gam_0=\del$. 
We view $H=\U(1,1)$ as a subgroup of $G$ via the embedding 
\[\iot\biggl(\begin{bmatrix} a & b \\ c & d \end{bmatrix}\biggl)=\begin{bmatrix} a & 0 & b \\ 0 & 1 & 0 \\ c & 0 & d \end{bmatrix}. \]

For each finite prime $l$ the local Haar measures $\d h_l$ and $\d g_l$ are defined so that maximal compact subgroups $K_l'$ and $K_l^{}$ have volume $1$. 
For arithmetic applications it is suitable to use the Haar measures $\d_{K'} h=\prod_v\d h_v$ and $\d_K g=\prod_v\d g_v$. \index{$\d_K g$}
We normalize the Petersson pairing by 
\[\Pet_K(\Phi)=\int_{G(\QQ)\bsl G(\AA)}\Phi((g,g))\,\d_K g \index{$\Pet_K(\Phi)$}\]
for $\Phi\in\scra^0(G\times G)$. 
We define the Petersson pairing $\Pet_{K'}$ on  $\scra^0(H\times H)$ relative to $\d_{K'}h$. 

Let $F$ be an algebraic number field. 
For each place $v$ of $F$ we normalize the Haar measure $\d a_v$ of the complete field $F_v$ in the following way:
if $F_v=\RR$, then $\d a_v$ is the Lebesgue measure; 
if $F_v=\CC$, then $\d a_v=\sqrt{-1}\d a_v^{}\d a_v^c$; 
if $v <\infty$, then $\int_{\frko_{F_v}}\d a_v=1$. 
Put $\d_F a=D_F^{-1/2}\prod_v\d a_v$, where $D_F$ is the absolute value of the discriminant of $F$. 
We define the Haar measure $\d^\times a_v$ of $F_v^\times$ by $\d^\times a_v=\zet_{F_v}(1)\frac{\d a_v}{|a_v|_{F_v}}$. 
Put $\zet_F(s)=\prod_v\zet_{F_v}(s)$ and $\xi_F(s)=D_F^{s/2}\zet_F(s)$. 
Denote the residue of $\xi_F(s)$ at $s=1$ by $\rho_F$.  
Then $\d^\times_F a=\rho_F^{-1}\prod_v\d^\times a_v$ is the Tamagawa measure. 

Let $\eps_{E/\QQ}$ be the quadratic Dirichlet character corresponding to $E/\QQ$. 
We denote the complete Dirichlet $L$-function associated to $\eps_{E/\QQ}$ by $L(s,\eps_{E/\QQ})$, and its finite part by $L_\bff(s,\eps_{E/\QQ})$. 
We define the Haar measure $\d t_v$ of $\U(1)(\QQ_v)$ as the quotient measure of the Haar measures of $E_v^\times$ and $\QQ_v^\times$. 
The Tamagawa measure of $\U(1)(\AA)$ is given by $\d^\tau\! t=D_E^{-1/2}L(1,\eps_{E/\QQ})^{-1}\prod_v\d t_v$. 

Let $\d^\tau\! g$ and $\d^\tau\! h$ denote the Tamagawa measures on $G(\AA)$ and $H(\AA)$. \index{$\d^\tau g$}
We choose the constants $C_H$ and $C_G$ so that 
\begin{align*}
\d^\tau\! h&=C_H\d_{K'} h, & 
\d^\tau\! g&=C_G\d_K g. 
\end{align*}

\begin{lemma}[\cite{Ichino07}]\label{lem:70}
We have 
\begin{align*}
C_H&=2^2D_E^{-1}L(1,\eps_{E/\QQ})^{-1}\zet_\QQ(2)^{-1}, \\
C_G&=2^3D_E^{-7/2}L(1,\eps_{E/\QQ})^{-1}\zet_\QQ(2)^{-1}L(3,\eps_{E/\QQ})^{-1}\lam_K, 
\end{align*}
where $\lam_K$ is a rational constant depending on $K$.  
\end{lemma}

\begin{proof}
Define the Tamagawa measures of $\bfM'(\AA)$, $\caln'(\AA)$, $\bfM(\AA)$, $\caln(\AA)$ by 
\begin{align*}
\d^\tau\!\bfm'(a)&:=\d_E^\times a, & 
\d^\tau\!\bfn'(z)&:=\d_\QQ z, \\
\d^\tau\!\bfm(a,t)&:=\d_E^\times a\d^\tau\! t, &
\d^\tau\!\bfn(w,z)&:=\d_Ew\d_\QQ z. 
\end{align*}
Take the Haar measure $\d k_v'$  (resp. $\d k_v^{}$) of $K_v'$ (resp. $K_v^{}$) of total volume $1$.  
Let $\d k'=\prod_v^{}\d k_v'$ and  $\d k=\prod_v^{}\d k_v^{}$.  
We define the Haar measures 
\begin{align*}
\d^\tau_{K'} h&=\d^\tau\!\bfm'\d^\tau\!\bfn'\,\d k', & 
\d^\tau_K g&=\d^\tau\!\bfm\,\d^\tau\!\bfn\,\d k. 
\end{align*}
We choose the constants $C_H'$ and $C_G'$ so that 
\begin{align*}
\d^\tau\! h&=C_H'\d_{K'}^\tau h, & 
\d^\tau\! g&=C_G'\d_K^\tau g. 
\end{align*}

Ichino computes the constant $C_G'$ for a good maximal compact subgroup of $G(\AA)$. 
When we replace it by our open compact subgroup $K$, we have only to change by a rational constant.
Theorems 9.5 and 9.6 of \cite{Ichino07} applied with $F=\QQ$, $m'=r-s$, $r_0=s$ give  
\begin{align*}
C_H'&=2^{-1}D_E^{-1/2}\zet_\QQ(2)^{-1}, &
C_G'&=D_E^{-1}\rho_E\zet_\QQ(2)^{-1}\xi(3,\eps_E)^{-1}\lam_K, 
\end{align*}
where we put  $\xi(s,\eps_{E/\QQ})=D_E^{s/2}L(s,\eps_{E/\QQ})$. 

We define the Haar measures of $\bfM'(\QQ_v)$, $\caln'(\QQ_v)$, $\bfM(\QQ_v)$, $\caln(\QQ_v)$ by 
\begin{align*}
\d\bfm'(a_v)&:=\d^\times a_v, & 
\d\bfn'(z_v)&:=\d z_v, \\
\d\bfm(a_v,t_v)&:=\d^\times a_v\d t_v, &
\d\bfn(w_v,z_v)&:=\d w_v\d z_v. 
\end{align*}
Recall the measure 
\[\bfd Z=\frac{2^3}{\eta(Z)^3}\biggl(\frac{\sqrt{-1}}{2}\d\tau\wedge\d\bar\tau\biggl)\biggl(\frac{\sqrt{-1}}{2}\d w\wedge\d\bar w\biggl)\index{$\bfd Z$}\]
on $\frkD$ defined in (\ref{tag:73}). 
Given $Z=\begin{bmatrix} \tau \\ w \end{bmatrix}\in\frkD$, we put $a=\sqrt{\eta(Z)/2}$. 
Since 
\[Z=\bfn(w,\Re\tau)\bfm(a,1)\bfi\] 
and the measure $\d t_\infty$ gives $\CC^\times/\RR^\times$ the volume 2, we have 
\[\d g_\infty=a^{-4}\d^\times a(\sqrt{-1}\d w\d w^c)\d\Re\tau\d k_\infty=2^{-3}\d\bfm\d\bfn\d k_\infty. \index{$\d g_\infty$}\]
Given $\tau\in\frkH$, we put $a=\sqrt{\Im\tau}$. 
Then
\[\d h_\infty=a^{-2}\d^\times a\d\Re\tau\d k_\infty'=2^{-3}\d\bfm'\d\bfn'\d k_\infty'. \]
We have seen that 
\begin{align*}
\d^\tau_{K'} h&=2^3\rho_E^{-1}\d_{K'}^{} h, & 
\d^\tau_K g&=2^3\rho_E^{-1}D_E^{-1}L(1,\eps_{E/\QQ})^{-1}\d_K^{} g. 
\end{align*}
It follows that $C_H^{}=2^3\rho_E^{-1}C_H'$ and $C_G^{}=2^3\rho_E^{-1}D_E^{-1}L(1,\eps_{E/\QQ})^{-1}C_G'$. 
Since $\rho_E^{}=D_E^{1/2}L(1,\eps_{E/\QQ})$, we get the formula for $C_H$. 
\end{proof}


\subsection{The Ichino-Ikeda formula}\label{ssec:74}

Let $\pi\simeq\otimes_v'\pi_v^{}$ be an irreducible cuspidal tempered automorphic representation of $G(\AA)$ and $\sig\simeq\otimes_v'\sig_v^{}$ an irreducible cuspidal tempered automorphic representation of $H(\AA)$. 
We define the product $L$-function associated to $\pi$ and $\sig^\vee$ by 
\[L(s,\pi\times\sig^\vee)=L^\GL(s,\BC(\pi)\times\BC(\sig^\vee)), \]
where $\BC(\pi)$ (resp. $\BC(\sig^\vee)$) is the functorial lift of $\pi$ (resp. $\sig^\vee$) to an automorphic representation of $\GL_3(\EE)$ (resp. $\GL_2(\EE)$). 
The right hand side is the {\it complete} $L$-function defined by Jacquet, Piatetski-Shapiro and Shalika in \cite{JPSS2}, and by Shahidi \cite{Shahidi83}. 
Assume that both $\pi$ and $\sig$ are tempered. 
Put 
\begin{align*}
\scrl(\pi\times\sig^\vee)&=\frac{L\left(\frac{1}{2},\pi\times\sig^\vee\right)}{L(1,\pi,\mathrm{Ad})L(1,\sig^\vee,\mathrm{Ad})}\prod_{i=1}^3L(i,\eps_{E/\QQ}^i), \\
\scrl(\pi_v^{}\times\sig_v^\vee)&=\frac{L\left(\frac{1}{2},\pi_v^{}\times\sig_v^\vee\right)}{L(1,\pi_v^{},\mathrm{Ad})L(1,\sig_v^\vee,\mathrm{Ad})}\prod_{i=1}^3L(i,\eps_{E_v/\QQ_v}^i), 
\end{align*}
where $L(s,\sig^\vee,\mathrm{Ad})$ denotes the adjoint $L$-series for $\sig^\vee$. 

\begin{remark}\label{rem:71} 
Let $\sig$ be an irreducible cuspidal automorphic representation of a unitary group in $n$ variables. 
If $n$ is even, then $L(s,\sig,\mathrm{Ad})=L(s,\BC(\sig),\As)$ is the Asai $L$-series. 
If $n$ is odd, then $L(s,\sig,\mathrm{Ad})=L(s,\BC(\sig),\As^-)$ is the twisted Asai $L$-series by Proposition 7.4 of \cite{GGP}. 
Hence if $l$ splits in $E$, then $L(s,\sig_l,\mathrm{Ad})=L^\GL(s,\sig_l^{}\times\sig_l^\vee)$, where $\sig_l$ is viewed as a representation of $\GL_n(\QQ_l)$ in the right hand side. 
\end{remark}
  
Put $\bfG=G\times H$. 
We introduce the two Petersson pairings  by 
\begin{align*}
\Pet(\vPh\otimes\vPh')&=\int_{\bfG(\QQ)\bsl\bfG(\AA)}\vPh(g,h)\vPh'(g,h)\,\d^\tau\! g\d^\tau\! h, \\
\Pet_{K,K'}(\vPh\otimes\vPh')&=\int_{\bfG(\QQ)\bsl\bfG(\AA)}\vPh(g,h)\vPh'(g,h)\,\d_K g\d_{K'}h
\end{align*}
for cusp forms $\vPh$ and $\vPh'$ on $\bfG$. 
We consider the integrals \index{$\scrp$, $\scrp_{K'}$}
\begin{align*}
\scrp(\vPh)&=\int_{H(\QQ)\bsl H(\AA)}\vPh((h,h))\,\d^\tau\! h, &
\scrp_{K'}(\vPh)=\int_{H(\QQ)\bsl H(\AA)}\vPh((h,h))\,\d_{K'} h. 
\end{align*}

Set $\vPi_v^{}=\pi_v^{}\otimes\sig_v^\vee$. 
Fix a $\bfG(\QQ_v)$-invariant local perfect pairing 
\[\La\!\La\quad\Ra\!\Ra_v^{}:\vPi_v^{}\otimes\vPi_v^\vee\to\CC.  \]
If $\pi_v$ and $\sig_v$ are tempered, then the integral
\[I(\bfW_1\otimes\bfW_2)=\int_{H(\QQ_v)}\La\!\La\vPi_v((h_v,h_v))\bfW_1\otimes\bfW_2\Ra\!\Ra_v^{}\,\d h_v  \]
is convergent for $\bfW_1\in\vPi_v^{}$ and $\bfW_2\in\vPi_v^\vee$. 

Ichino and Ikeda \cite{II} have refined the global Gross-Prasad conjecture and predicted an explicit relation between the central value and the period for orthogonal groups. The analogue of the Ichino-Ikeda conjecture for unitary groups was formulated in \cite{NHarris} and proved generally in \cite{BLZZ,BCZ}.

\begin{theorem}[\cite{Z,BLZZ,BCZ}]\label{thm:71}
Let $\pi$ (resp. $\sig$) be an irreducible tempered cuspidal automorphic representation of $G(\AA)$ (resp. $H(\AA)$). 
Put $\vPi=\pi\otimes\sig^\vee$. 
If $\vPh=\otimes_v'\bfW_v^{}\in\vPi$ and $\vPh'=\otimes_v'\bfW_v'\in\vPi^\vee$ are factorizable, then 
\[\frac{\scrp(\vPh)\scrp(\vPh')}{\Pet(\vPh\otimes\vPh')}
=C_H\frac{\scrl(\pi\times\sig^\vee)}{2^{\vka_\pi+\vka_\sig}}\prod_v\frac{I(\bfW_v^{}\otimes\bfW_v')}{\scrl(\pi_v\times\sig_v^\vee)\La\!\La \bfW_v^{}\otimes\bfW_v'\Ra\!\Ra_v^{}}, \]
where $2^{\vka_\pi}$ (resp. $2^{\vka_\sig}$) is the order of the component group associated to the $L$-parameter of $\pi$ (resp. $\sig$). 
\end{theorem}


\subsection{The Petersson norm on $\U(1,1)$}\label{ssec:75}

We will rewrite Theorem \ref{thm:71} in a form which is more convenient for our later application. 
We use the measures $\d_{K'}h$ and $\d_Kg$ to rewrite the formula as 
\[\frac{\scrp_{K'}(\vPh)\scrp_{K'}(\vPh')}{\Pet_{K,K'}(\vPh\otimes\vPh')}
=C_G\frac{\scrl(\pi\times\sig^\vee)}{2^{\vka_\pi+\vka_\sig}}\prod_v\frac{I(\bfW_v^{}\otimes\bfW_v')}{\scrl(\pi_v^{}\times\sig_v^\vee)\La\!\La \bfW_v^{}\otimes\bfW_v'\Ra\!\Ra_v^{}}. \]
We substitute the formula for $C_G$ to get 
\begin{multline}
\frac{\scrp_{K'}(\vPh)\scrp_{K'}(\vPh')}{\Pet_{K,K'}(\vPh\otimes\vPh')}\\
=2^{3-\vka_\pi-\vka_\sig}D_E^{-7/2}\lam_K\frac{L\left(\frac{1}{2},\pi\times\sig^\vee\right)}{L(1,\pi,\mathrm{Ad})L(1,\sig^\vee,\mathrm{Ad})} 
\prod_v\frac{I(\bfW_v^{}\otimes\bfW_v')}{\scrl(\pi_v\times\sig_v^\vee)\La\!\La \bfW_v^{}\otimes\bfW_v'\Ra\!\Ra_v^{}}. \label{tag:75}
\end{multline}

Let $\vPh'=\vPh^\natural$. 
Then $\scrp_{K'}(\vPh')=\scrp_{K'}(\vPh)$ and $\bfW_v'=\bfW_v^\natural$. 
Roughly speaking, this formula tells us that the square of the period integral $\scrp_{K'}(\vPh)$ is a product of the central value $L\bigl(\frac{1}{2},\pi\times\sig^\vee\bigl)$  and the local integrals $I(\bfW_v\otimes\bfW_v^\natural)$. 
Thanks to Proposition \ref{prop:67}, our task of proving the interpolation formula of $\Tht_{\bdsf,\bdsg}$ boils down to (i) choices of test vectors $\bfW_v$ and (ii) the explicit evaluation of the local integrals $I(\bfW_v\otimes\bfW_v^\natural)$. 
The purpose of this section is to carry out the step (i) and give the explicit formula of $I(\bfW_v\otimes\bfW_v^\natural)$. 
The computations of $I(\bfW_v\otimes\bfW_v^\natural)$ have been carried out in \cite{HY} except at the real and $p$-adic places. 
The details of the step (ii) are left to Section \ref{sec:8} for $v=p$ and to Appendix \ref{sec:b} for $v=\infty$. 

Let $\ulk=(k_1,k_2;k_3)$ and $\ulk'=(k_1';k_2')$ be sequences of integers such that $k_1^{}\leq k_1'\leq k_2^{}$ and $k_3^{}\leq k_2'$. 
We write $\tau_{\ulk}^\vee$ for the irreducible representation of $\calk_\infty$ whose highest weight is $-\ulk$. 
Let $\sig$ be an irreducible tempered cuspidal automorphic representation of $H(\AA)$ such that $\sig_\infty$ is a holomorphic discrete series with minimal $K$-type $-\ulk'$.   
Let $\pi$ be an irreducible tempered cuspidal automorphic representation of $G(\AA)$ such that $\pi_\infty$ is a holomorphic discrete series or limit of holomorphic discrete series with minimal $K$-type $\call_{\ulk}(\CC)^\vee$. 
For every non-split prime $q$ we assume that $\pi_q$ admits a non-zero $K_q$-invariant vector and $\sig_q$ admits a non-zero $K_q'$-invariant vector. 
Put $N=\prod_{l\neq p}l^{c(\pi_l)}$ and $N'=\prod_{l\neq p}l^{c(\sig_l)}$. 
We take ideals $\frkN$ and $\frkN'$ of $\frko_E$ such that $\frko_E/\frkN\simeq\ZZ/N\ZZ$ and $\frko_E/\frkN'\simeq\ZZ/N'\ZZ$. 
We define open subgroups $K_0'(p^\ell\frkN')$ of $K'$ as in Section \ref{ssec:51}. 

Lemma \ref{lem:51} implies that $\sig$ is $\addchar$-generic.
Thus $\sig^\vee$ is $\addchar^{-1}$-generic. 
We realize the local representation $\sig_v$ (resp. $\sig_v^\vee$) in its Whittaker model $\scrw_{\addchar_v}(\sig_v)$ (resp. $\scrw_{\addchar_v^{-1}}(\sig_v^\vee)$) with respect to $\addchar_v^{}$ (resp. $\addchar_v^{-1}$). 
For each place $v$ of $\QQ$ we normalize the vector $W_{\sig_v}\in\scrw_{\addchar_v}(\sig_v)$ in the following way:
if $v$ splits in $E$, then $W_{\sig_v}$ is the normalized essential vector (see Definition \ref{def:81});  
if $v$ is finite and does not split in $E$, then $W_{\sig_v}\in\scrw_{\addchar_v}(\sig_v)$ is the spherical vector normalized by $W_{\sig_v}(\ono_2)=1$; 
let $W_{\sig_\infty}\in\scrw_{\addchar_\infty}(\sig_\infty)$ be the lowest weight vector normalizedby $W_{\sig_\infty}(\ono_2)=e^{-2\pi}$. 
  
Let $\sig\simeq\otimes_v'\sig_v^{}$ and $\sig^\vee\simeq\otimes_v'\sig_v^\vee$ be the isomorphisms that are compatible with the factorization of the Whittaker functionals $W_{\addchar}$ and $W_{\addchar^{-1}}$ respectively. 
Namely, if $\vph=\otimes_vW_v\in\sig$ is factorizable, then $W_{\addchar}(\sig(h)\vph)=\prod_vW_v(h_v)$ for $h=(h_v)\in H(\AA)$. 
For $h\in H$ we have 
\[h^\natural=w_2\trs h^{-1}w_2=(\det h)^{-1}\begin{bmatrix} -1 & 0 \\ 0 & 1 \end{bmatrix}h\begin{bmatrix} -1 & 0 \\ 0 & 1 \end{bmatrix}. \index{$g^\natural$}\]

We will show that the factorization $\sig^\vee\simeq\otimes_v'\sig_v^\vee$ is compatible with $\natural$. 
Observe that 
\begin{align*}
W_{\addchar}(\sig(h)\vph^\natural)
=\int_{\QQ\bsl\AA}\vph\left(\begin{bmatrix} 1 & -z \\ 0 & 1 \end{bmatrix}h^\natural\right) 
\overline{\addchar(z)}\,\d_\QQ z\notag\\
=W_{\addchar^{-1}}(\sig^\vee(h^\natural)\vph)=\prod_vW_v(h^\natural_v). 
\end{align*} 
Let $\vph_\sig\in\sig$ be the newform associated to $\sig$ (see Definition \ref{def:54}). 
By uniquness of essential vectors the newform $\vph_\sig$ is factorizable, and by our normalization $\vph_\sig=\otimes_v'W_{\sig_v}$. 
If $l$ is split in $E$, then $W_l^\natural(h_l^{})=W_l^{}(h_l^\natural)$ for all $W_l\in\scrw_{\addchar_v}(\sig_v)$ by the definition (\ref{tag:72}) of $W_l^\natural$. 
If $v$ does not split in $E$, then $W_{\sig_v}^\natural(h_v^{})=W_{\sig_v}^{}(h_v^\natural)$ for our choice of $W_{\sig_v}$.   
Thus $\vph_\sig^\natural=\otimes_v'W_{\sig_v}^\natural$. 
It follows that $\vph_{\sig^\vee}=\otimes_v'W_{\sig_v^\vee}$ is the $K_0'(\frkN_\sig)$-fixed vector on which $K_\infty'$ acts by the character $\ulk'$, normalized by $W_{\addchar^{-1}}(\vph_{\sig^\vee})=e^{-2\pi}$. 


\begin{lemma}\label{lem:72} 
$\vph_{\sig^\vee}=\overline{\vph_\sig^{}}$. 
\end{lemma}

\begin{proof}
Since $\vph_{\sig^\vee}$ and $\overline{\vph_\sig^{}}$ are proportional by multiplicity one for the unitary group $H$ (see \S 6.2 of \cite{LM15}) and essential vectors for generic representations of $H(\QQ_l)\simeq\GL_2(\QQ_l)$, Lemma \ref{lem:72} follows from the observation that $W_{\addchar^{-1}}(\vph_{\sig^\vee})=e^{-2\pi}=\overline{W_{\addchar}(\vph_\sig^{})}$. 
\end{proof}

Taking Lemma \ref{lem:72} into account, we will write 
\[\|\vph_\sig\|_{K'}^2=\Pet_{K'}^{}(\vph_\sig^{}\otimes\vph_{\sig^\vee}). \index{$\|\vph_\sig\|_{K'}$}\]
For each place $v$ we define the $\bfG(\QQ_v)$-equivariant isomorphism $^\natural:\vPi_v^\natural\simeq\vPi_v^\vee$ by 
\[(W\otimes W')^\natural=W^\natural\otimes W^{\prime\natural}. \index{$W^\natural$}\]


\subsection{Algebraicity of central values}\label{ssec:76}

We denote the complexified Lie algebra of $G$ by $\hat\frkg_\CC$.  
Let $f$ be a Picard cusp form of weight $\ulk$ associated to $\pi$ defined over $\overline{\QQ}$. 
Put 
\begin{align*}
n_1^*&=k'_2-k^{}_3, & 
\bfv_i&=X^{k_2-k_1-i}Y^i
\end{align*}
for $i=0,1,2,\dots,k_2-k_1$. 
Let $\vph=\vPh_{\ulk}(f)_{\bfv_{k_1'-k_1^{}}}$. 
In Appendix \ref{sec:a} we define a differential operator $\del_{\ulk}$ associated to an element $\caly_+:=\frac{Y^\calc_+}{2\pi\sqrt{-1}}\in\hat\frkg_\CC$ such that the restriction of $\caly_+^{n_1^*}\vph$ to $H(\AA)$ has weight $\ulk'$ (cf. Proposition 2.6 of \cite{MHarris}). 

Assuming that $\vph$ is factorizable, we write $\vph\otimes\vph_\sig=\otimes_v'\bfW_v^{}$. 
Fix a non-zero $\calk_\infty$-invariant vector $W_\infty^{\calk_\infty}\in\pi_\infty^{}\otimes\pi_\infty^\vee$ and a non-zero $\calk_\infty'$-invariant vector $W_\infty^{\calk_\infty'}\in\sig_\infty^\vee\otimes\sig_\infty^{}$. 
Let 
\begin{align*} 
\vPh_{\pi\otimes\pi^\vee}^{\calk_\infty}&=W_\infty^{\calk_\infty}\otimes_l(W_{\pi_l}^{}\otimes W_{\pi_l^\vee})\in\pi\otimes\pi^\vee, \\
\vph_{\sig^\vee}\otimes\vph_\sig&=W_\infty^{\calk_\infty'}\otimes_l(W_{\sig_l^\vee}\otimes W_{\sig_l}^{})\in\sig^\vee\otimes\sig
\end{align*}
be $\calk_\infty$ and $\calk_\infty'$ invariant algebraic cusp forms. \index{$\vPh_{\pi\otimes\pi^\vee}^{\calk_\infty}$}

For each place $v$ we put 
\begin{align*}
\bfW_{\vPi_v}&=W_{\pi_v^{}}\otimes W_{\sig_v^\vee}\in\vPi_v, & 
\bfW_{\vPi_v^\vee}&=W_{\pi_v^\vee}\otimes W_{\sig_v^{}}\in\vPi_v^\vee. 
\end{align*}
Put 
\[I_\infty=\frac{I((\caly_+^{n_1^*}\otimes\mathrm{Id})\bfW_\infty\otimes((\caly_+^{n_1^*}\otimes\mathrm{Id})\bfW_\infty)^\natural)}{\scrl(\pi_\infty^{}\times\sig_\infty^\vee)\La\!\La W_\infty^{\calk_\infty^{}}\otimes W_\infty^{\calk_\infty'}\Ra\!\Ra_\infty^{}}.  \]
Proposition \ref{prop:b1} and Remark \ref{rem:a1} give
\beq
I_\infty=2^{2k_3^{}-2k_2'}\biggl(-\frac{2}{|\del|}\biggl)^{k_1'+k_2'-k_1^{}-k_2^{}-k_3^{}}. \label{tag:76}
\eeq

Since $\scrp_{K'}(\Phi)=\scrp_{K'}(\Phi^\natural)$ for $\Phi\in\scra^0(\bfG)$, one can deduce the following formula from Lemma \ref{lem:71} and (\ref{tag:75}):  
\begin{multline}
\frac{\scrp_{K'}(\caly_+^{n_1^*}\vph\otimes\vph_{\sig^\vee})^2}{\Pet_K\Big(\vPh_{\pi\otimes\pi^\vee}^{\calk_\infty}\Big)\|\vph_\sig\|^2_{K'}}\\
=\frac{D_E^{-7/2}\lam_K}{2^{\vka_\pi+\vka_\sig-3}}I_\infty\frac{L\left(\frac{1}{2},\pi\times\sig^\vee\right)}{L(1,\pi,\mathrm{Ad})L(1,\sig^\vee,\mathrm{Ad})}\prod_l\frac{I(\bfW_l^{}\otimes\bfW_l^\natural)}{\scrl(\pi_l^{}\times\sig_l^\vee)\La\!\La \bfW_{\vPi_l^{}}\otimes\bfW_{\vPi_l^\vee}\Ra\!\Ra_l^{}}. \label{tag:77}
\end{multline}

Let $\Ome_\infty$ be a complex period of a CM elliptic curve (see \S \ref{ssec:67}). 
Applying Proposition \ref{prop:62} with $i=k_1'-k_1^{}$ and $n=n_1^*$, we have the following lemma: 

\begin{lemma}\label{lem:73}
The holomorphic projection of the restriction of 
\[\biggl(\frac{2\pi\sqrt{-1}}{\Ome_\infty}\biggl)^{k_1'+k_2'-k_1^{}-k_2^{}-k_3^{}}\caly_+^{n_1^*}\vph\]
to $H(\AA)$ is associated to a $\overline{\QQ}$-rational form of weight $\ulk'$. 
\end{lemma}

The symbol $\Hol$ stands for the holomorphic projection from nearly holomorphic functions on $\frkH$. 
Observe that 
\[\scrp_{K'}(\caly_+^{n_1^*}\vph\otimes\vph_{\sig^\vee})=\scrp_{K'}(\1_\sig\Hol\;\caly_+^{n_1^*}\vph\otimes\vph_{\sig^\vee}), \]
where $\1_\sig$ denotes the projection from the space of holomorphic cusp forms on $H$ to the space of $\sig$-isotypic modular forms. 
Then $\caly_+^{n_1^*}\vph$ is an algebraic nearly holomorphic Picard modular form by Theorem 14.7 of \cite{Shimura00}. 
If the restriction of $\caly_+^{n_1^*}\vph$ to $H(\AA)$ has level $\frkN_\sig$, then Lemma \ref{lem:73} gives an algebraic number $C_\vph$ such that 
\[\biggl(\frac{2\pi\sqrt{-1}}{\Ome_\infty}\biggl)^{k_1'+k_2'-k_1^{}-k_2^{}-k_3^{}}\1_\sig\Hol\;\caly_+^{n_1^*}\vph=C_\vph^{}\vph_\sig. \] 
It follows that 
\begin{multline*}
\biggl(\frac{\Ome_\infty}{2\pi\sqrt{-1}}\biggl)^{2(k_1'+k_2'-k_1^{}-k_2^{}-k_3^{})}C_\vph^2\frac{\|\vph_\sig\|_{K'}^2}{\Pet_K\Big(\vPh_{\pi\otimes\pi^\vee}^{\calk_\infty}\Big)}\\
=\frac{D_E^{-7/2}\lam_K}{2^{\vka_\pi+\vka_\sig-3}}I_\infty\frac{L\left(\frac{1}{2},\pi\times\sig^\vee\right)}{L(1,\pi,\mathrm{Ad})L(1,\sig^\vee,\mathrm{Ad})}\prod_l\frac{I(\bfW_l^{}\otimes\bfW_l^\natural)}{\scrl(\pi_l^{}\times\sig_l^\vee)\La\!\La \bfW_{\vPi_l^{}}\otimes\bfW_{\vPi_l^\vee}\Ra\!\Ra_l^{}}. 
\end{multline*} 

Note that $L(1,\sig^\vee,\mathrm{Ad})=L(1,\sig,\mathrm{Ad})$. 
Since $\frac{L(1,\sig,\mathrm{Ad})}{\|\vph_\sig\|_{K'}^2}$ is an algebraic number by Corollary \ref{cor:c1}, we get the following result:

\begin{theorem}\label{cor:71} 
Notations and assumptions being as above, we have 
\[\biggl(\frac{2\pi\sqrt{-1}}{\Ome_\infty}\biggl)^{2(k_1'+k_2'-k_1^{}-k_2^{}-k_3^{})}\frac{\Pet_K\Big(\vPh_{\pi\otimes\pi^\vee}^{\calk_\infty}\Big)L\left(\frac{1}{2},\pi\times\sig^\vee\right)}{L(1,\pi,\mathrm{Ad})\|\vph_\sig\|_{K'}^4}\in\overline{\QQ} \]
\end{theorem}


\subsection{An application of the splitting lemma}\label{ssec:77}

If $q$ and $pNN'$ are coprime, then 
\beq
\frac{I((W_{\pi_q}\otimes W_{\sig_q^\vee})\otimes(W_{\pi_q}\otimes W_{\sig_q^\vee})^\natural)}{\La W_{\pi_q},W_{\pi_q^\vee}\Ra_q^{}\La W_{\sig_q},W_{\sig_q^\vee}\Ra_q'}=\scrl(\pi_q^{}\times\sig_q^\vee) \label{tag:78}
\eeq
by \cite[Theorem 2.12]{NHarris} and \cite[Proposition C.1]{HY}. 
We put 
\begin{align}
B_{\pi_l}&=\La W_{\pi_l}^{},W_{\pi_l^\vee}^{}\Ra^{}_l, &
B_{\sig_l}&=\La W_{\sig_l}^{},W_{\sig_l^\vee}^{}\Ra_l', \notag\\
\calb_{\pi_l}&=\frac{\zet_l(3)}{L^\GL(1,\pi_l^{}\times\pi_l^\vee)}B_{\pi_l}, &
\calb_{\sig_l}&=\frac{\zet_l(2)}{L^\GL(1,\sig_l^{}\times\sig_l^\vee)}B_{\sig_l} \label{tag:79}
\end{align} 
for each split prime $l$, where the local pairings $\La\;,\;\Ra_l^{}$ and $\La\;,\;\Ra_l'$ are constructed by (\ref{tag:83}). 
It is well-known that $\calb_{\pi_l}=\calb_{\sig_l}=1$ if $\pi_l$ and $\sig_l$ are unramified. 
To avoid possible confusion, we recall that 
\begin{align*}
L(s,\pi_l^{},\mathrm{Ad})&=L^\GL(s,\pi_l^{}\times\pi_l^\vee), & 
L(s,\sig_l^\vee,\mathrm{Ad})&=L^\GL(s,\sig_l^{}\times\sig_l^\vee). 
\end{align*}
We regard $\pi_l$ and $\sig_l$ as representations of unitary groups in the left hand side and representations of general linear groups in the right hand side. 

Let $\vph=\otimes_v^{}W_v^{}\in\pi$ and $\vph'=\otimes_v^{}W_v'\in\sig^\vee$. 
If $W_l^{}=W_{\pi_l}$ and $W_l'=W_{\sig_l^\vee}$ unless $pNN'$ is divisible by $l$, then by (\ref{tag:77}) and (\ref{tag:78})
\begin{multline}
\frac{\scrp_{K'}(\caly_+^{n_1^*}\vph\otimes\vph')^2}{\Pet_K\Big(\vPh_{\pi\otimes\pi^\vee}^{\calk_\infty}\Big)} \label{tag:70}\\
=\frac{D_E^{-7/2}\lam_K}{2^{\vka_\pi+\vka_\sig-3}}I_\infty\frac{\|\vph_\sig\|^2_{K'}L\left(\frac{1}{2},\pi\times\sig^\vee\right)}{L(1,\pi,\mathrm{Ad})L(1,\sig,\mathrm{Ad})}\prod_{l|pNN'}\frac{I(W_l^{}\otimes W_l'\otimes(W_l^{}\otimes W_l')^\natural)}{\scrl(\pi_l^{}\times\sig_l^\vee)B_{\pi_l}B_{\sig_l}}. 
\end{multline}

Fix a split rational prime $l$. 
We regard $\pi_l$ and $\sig_l$ as representations of general linear groups via $\imath_\frkl^{}$ and $\imath_\frkl'$. 
Put $\zet_l(s)=(1-l^{-s})^{-1}$. 
Using the Whittaker models with respect to $\addchar_l$, we define the local pairings $\La\;,\;\Ra_l^{}$ and $\La\;,\;\Ra_l'$ by (\ref{tag:83}). 
Lemma \ref{lem:81} gives  
\begin{align*}
I(W_l^{}\otimes W_l'\otimes(W_l^{}\otimes W_l')^\natural)
&=J(\pi_l^{}(\vsi)W_l^{},\pi^\vee_l(\vsi)W_l^\natural,W_l^{\prime\natural},W_l')\\
&=\zet_l(1)Z\biggl(\frac{1}{2},\pi_l^{}(\vsi)W_l^{},W_l'\biggl)Z\biggl(\frac{1}{2},\pi^\vee_l(\vsi)W_l^\natural,W_l^{\prime\natural}\biggl)
\end{align*}
where
\[\vsi=\begin{bmatrix}
1 & 0 & 0 \\ 
0 & 0 & 1 \\
0 & 1 & 0   
\end{bmatrix}. \]
Since $\vsi T_\del=-\del\begin{bmatrix} -\del^{-1} w_2 & \\ & 1 \end{bmatrix}\vsi$, we have 
\begin{align*}
Z\biggl(\frac{1}{2},\pi^\vee_l(\vsi)W_l^\natural,W_l^{\prime\natural}\biggl)
&=Z\biggl(\frac{1}{2},\pi^\vee_l(\vsi T_\del)\widetilde{W_l^{}},\sig_l^{}(w_2)\widetilde{W_l'}\biggl)\\
&=\frac{\ome_{\sig_l}(-\del)}{\ome_{\pi_l}(-\del)}Z\biggl(\frac{1}{2},\pi^\vee_l(\vsi)\widetilde{W_l^{}},\widetilde{W_l'}\biggl)\\
&=\frac{\ome_{\sig_l}(-\del)}{\ome_{\pi_l}(-\del)}\gam^\GL\biggl(\frac{1}{2},\pi_l^{}\times\sig_l^\vee,\addchar_l^{}\biggl)Z\biggl(\frac{1}{2},\pi_l^{}(\vsi)W_l^{},W_l'\biggl) 
\end{align*}
by (\ref{tag:72}), 
and the invariance and the functional equation (\ref{tag:84}) of the JPSS integrals. 
We can rewrite the identity above as  
\[\frac{I(W_l^{}\otimes W_l'\otimes(W_l^{}\otimes W_l')^\natural)}{\scrl(\pi_l^{}\times\sig_l^\vee)B_{\pi_l}^{}B_{\sig_l}^{}}=\frac{\scri(W_l^{}\otimes W_l')}{(\ome_{\pi_l}^{}\ome_{\sig_l}^{-1})(-\del)\calb_{\pi_l} \calb_{\sig_l}}, \]
where 
\[\scri(W_l^{}\otimes W_l')=\gam^\GL\biggl(\frac{1}{2},\pi_l^{}\times\sig_l^\vee,\addchar_l^{}\biggl)\frac{Z\bigl(\frac{1}{2},\pi_l^{}(\vsi)W_l^{},W_l'\bigl)^2}{L\bigl(\frac{1}{2},\pi_l^{}\times\sig_l^\vee\bigl)}. \]

We use Corollary \ref{cor:c1} to rewrite (\ref{tag:70}) as  
\[\frac{\scrp_{K'}(\caly_+^{n_1^*}\vph\otimes\vph')^2}{\Pet_K\Big(\vPh_{\pi\otimes\pi^\vee}^{\calk_\infty}\Big)} \\
=\frac{D_E^{-5/2}\lam_K}{2^{\vka_\pi-1+k_2'-k_1'}}I_\infty\frac{L\left(\frac{1}{2},\pi\times\sig^\vee\right)}{L(1,\pi,\mathrm{Ad})}\prod_{l|pNN'}\frac{\scri(W_l^{}\otimes W_l')}{(\ome_{\pi_l}^{}\ome_{\sig_l}^{-1})(-\del)\calb_{\pi_l}}. \]

\begin{definition}[Periods on $\U(2,1)$]\label{def:71}
Put \index{$\Pet(\pi)$}
\begin{align*}
\Pet(\pi)&=2^{\vka_\pi}L(1,\pi,\Ad)\prod_{l|N_\pi}\calb_{\pi_l}, \\
\Xi_{\vPh_{\pi\otimes\pi^\vee}^{\calk_\infty}}&=(-2)^{k_1+k_2-k_3}\frac{\Pet(\pi)}{\Pet_K\Big(\vPh_{\pi\otimes\pi^\vee}^{\calk_\infty}\Big)}. 
\end{align*}
\end{definition}

We now get the following formula from (\ref{tag:76}):

\begin{corollary}\label{cor:72}
Let $\pi\simeq\otimes_v'\pi_v^{}$ be an irreducible cuspidal tempered automorphic representation of $G(\AA)$ whose archimedean part $\pi_\infty$ is a holomorphic discrete series or limit of holomorphic discrete series and such that $\pi_q$ admits a non-zero $K_q$-invariant vector for every non-split prime $q$.
Let $\sig\simeq\otimes_v'\sig_v^{}$ be an irreducible cuspidal tempered automorphic representation of $H(\AA)$ whose archimedean part $\sig_\infty$ is a holomorphic discrete series and such that $\sig_q$ admits a non-zero $K_q'$-invariant for every non-split prime $q$. 
Let 
\begin{align*}
\vph^\dagger&=\otimes_v^{}W_v^\dagger\in\pi, & 
\vph^\flat&=\otimes_v^{}W_v^\flat\in\sig^\vee. 
\end{align*}
If $W_l^\dagger=W_{\pi_l}$ and $W_l^\flat=W_{\sig_l}$ unless $pNN'$ is divisible by $l$, then  
\begin{align*}
&\scrp_{K'}(\caly_+^{n_1^*}\vph^\dagger\otimes\vph^\flat)^2\Xi_{\vPh_{\pi\otimes\pi^\vee}^{\calk_\infty}} \\
=&\frac{2^{2k_1'-2k_2'}(-1)^{k_1'+k_2'}}{|\del|^{k_1'+k_2'-k_1^{}-k_2^{}-k_3^{}}}2D_E^{-5/2}\lam_KL\left(\frac{1}{2},\pi\times\sig^\vee\right)\prod_{l|pNN'}\frac{\scri(W_l^\dagger\otimes W_l^\flat)}{(\ome_{\pi_l}^{}\ome_{\sig_l}^{-1})(-\del)}. 
\end{align*}
\end{corollary}


\subsection{The local integral at a ramified split prime $l$}\label{ssec:78}

Let $l\neq p$ be a rational prime that is split in $E$. 
We write $\ome_{\sig_l}$ for the central character of $\sig_l$, which is viewed as a character of $\QQ_l^\times$ via $\imath_\frkl$. 
Observe that 
\[\scri(\pi_l^{}(\vsi)W_l^{}\otimes W_l')=\vep^\GL\biggl(\frac{1}{2},\pi_l^{}\times\sig_l^\vee,\addchar_l^{}\biggl)\Biggl(\frac{Z\bigl(\frac{1}{2},W_l^{},W_l'\bigl)}{L^\GL\bigl(\frac{1}{2},\pi_l^{}\times\sig_l^\vee\bigl)}\Biggl)^2. \]

If $\sig_l$ is unramified, then the essential vector $W_{\sig_l^\vee}\in \scrw_{\addchar_l^{-1}}(\sig_l^\vee)$ is the normalized spherical vector and so by Th\'{e}orem\`{e} on p.~208 of \cite{JPSS}, we have 
\[Z(s,W_{\pi_l^{}}^{},W_{\sig_l^\vee})=L^\GL(s,\pi_l^{}\times\sig_l^\vee). \]
We see that  
\[\scri(\pi_l(\vsi)W_{\pi_l^{}}^{}\otimes W_{\sig_l^\vee})=\vep^\GL\left(\frac{1}{2},\pi_l^{}\times\sig_l^\vee,\addchar_l^{}\right)=:\frkf_{\pi_l^{},\sig_l^\vee}. \]

We consider the case when $\sig_l$ is ramified. 
Put $f_l=c(\ome_{\sig_l})$. 
We define $\bfU^{\ome_{\sig_l}^{-1}} W_{\pi_l}\in\scrw_{\addchar_l}(\pi_l)$ by 
\[\bfU^{\ome_{\sig_l}^{-1}} W_{\pi_l}=\sum_{i,j\in(\ZZ_l/l^{f_l}\ZZ_l)^\times}\sum_{y\in\ZZ_l/l^{2f_l}\ZZ_l}
\ome_{\sig_l}(ij)^{-1}\pi_l\left(\begin{bmatrix} 1 & \frac{i}{l^{f_l}} & \frac{y}{l^{2f_l}} \\ 0 & 1 & \frac{j}{l^{f_l}} \\ 0 & 0 & 1 \end{bmatrix}\right)W_{\pi_l}. \]
Put $\tau_{N'}'=\begin{bmatrix} 0 & 1 \\ N' & 0 \end{bmatrix}$. 
We consider the embedding 
\begin{align*}
\iot'&:\GL_2(F)\hookrightarrow\GL_3(F), & 
\iot'(g)&=\begin{bmatrix} g & \\ & 1 \end{bmatrix}. 
\end{align*}
Proposition 6.8 of \cite{HY} tells us that 
\[Z(s,\bfU^{\ome_{\sig_l}^{-1}} W_{\pi_l}^{},\sig_l^\vee(\tau_{N'}')W_{\sig_l^\vee})=\frac{l^{3c(\ome_{\sig_l})}\ome_{\sig_l}(l)^{2c(\ome_{\sig_l})}\vep\bigl(\frac{1}{2},\sig_l^\vee,\addchar_l^{}\bigl)}{\vep\bigl(\frac{1}{2},\ome_{\sig_l}^{-1},\addchar_l^{}\bigl)^2[\GL_2(\ZZ_l):\calk_0^{(2)}(l^{2c(\ome_{\sig_l})}\ZZ_l)]} \]
if $c(\sig_l)\leq 2c(\ome_{\sig_l})$. 
Put 
\begin{align*}
\bfU_{\frkN'}^{\ome_{\sig_l}^{-1}}W_{\pi_l}^{}&=\pi_l^{}(\iot'(\tau'_{N'})^{-1}\vsi)\U^{\ome_{\sig_l}^{-1}}W_{\pi_l}^{}, \\ 
\frkf_{\pi_l^{},\sig_l^\vee}&=\vep^\GL\left(\frac{1}{2},\pi_l^{}\times\sig_l^\vee,\addchar_l^{}\right)\frac{l^{6c(\ome_{\sig_l})}\ome_{\sig_l}(l)^{4c(\ome_{\sig_l})}\vep\bigl(\frac{1}{2},\sig_l^\vee,\addchar_l^{}\bigl)^2}{\vep\bigl(\frac{1}{2},\ome_{\sig_l}^{-1},\addchar_l^{}\bigl)^4L^\GL\bigl(\frac{1}{2},\pi_l^{}\times\sig_l^\vee\bigl)^2}. 
\end{align*}
It follows that 
\[\scri(\pi_l^{}(\vsi)\bfU_{\frkN'}^{\ome_{\sig_l}^{-1}}W_{\pi_l}^{},W_{\sig_l^\vee})=\frac{\frkf_{\pi_l^{},\sig_l^\vee}}{[\GL_2(\ZZ_l):\calk_0^{(2)}(l^{2c(\ome_{\sig_l})}\ZZ_l)]^2}. \]

Put $\frkM=\prod_{l|N'}\frkl^{c(\ome_{\sig_l})}$. 

\begin{corollary}\label{cor:73}
Notations and assumptions being as in Corollary \ref{cor:72}, if $c(\sig_l)\leq 2c(\ome_{\sig_l})$, $W_l^\flat=W_{\sig_l^\vee}$ and $W_l^\dagger=\pi_l^{}(\vsi)\bfU_{\frkN'}^{\ome_{\sig_l}^{-1}}W_{\pi_l}^{}$ for every $l\neq p$, then 
\begin{align*}
&|\del|^{k_1'+k_2'-k_1^{}-k_2^{}-k_3^{}}(\ome_{\pi_p}^{}\ome_{\sig_p}^{-1})(-\del)\scrp_{K'}(\caly_+^{n_1^*}\vph^\dagger\otimes\vph^\flat)^2\\
=&\frac{2^{1+2k_1'-2k_2'}(-1)^{k_1'+k_2'}}{D_E^{5/2}[K':K_0'(\frkM^2)]^2}\lam_K\frac{L\bigl(\frac{1}{2},\pi\times\sig^\vee\bigl)}{\Xi_{\vPh_{\pi\otimes\pi^\vee}^{\calk_\infty}}}\scri(W_p^\dagger\otimes W_p^\flat)\prod_{l|NN'}\frac{\frkf_{\pi_l^{},\sig_l^\vee}}{(\ome_{\pi_l}^{}\ome_{\sig_l}^{-1})(-\del)}. 
\end{align*}
\end{corollary}


\subsection{The local integral at $p$ and the modified $p$-factor}\label{ssec:79}

Assume that $\pi_p$ is the irreducible generic constituent of $I(\nu_p,\rho_p,\mu_p)$ and that $\sig_p$ is the irreducible generic constituent of $I(\mu'_p,\nu'_p)$, where $\nu_p^{}$, $\rho_p^{}$, $\mu_p^{}$, $\mu_p'$, $\nu_p'$ are characters of $\QQ_p^\times$ (see \S \ref{ssec:84} for notations). 
Put 
\begin{align*}
\alp_p&=\nu_p^{}(p), & 
\bet_p&=\rho_p^{}(p), & 
\gam_p&=\mu_p^{}(p), & 
\alp'_p&=\mu'_p(p), & 
\bet'_p&=\nu'_p(p). 
\end{align*}
Assume that $\pi$ is $\frkp^c$-ordinary and $\sig$ is $p$-ordinary with respect to $\iot_p$, that is,  
\begin{align*}
&\iot_p(p^{-k_1}\gam_p); &
&\iot_p\bigl(p^{-k_2'+\frac{1}{2}}\alp_p'\bigl), &
&\iot_p\bigl(p^{-k_1'-\frac{1}{2}}\bet_p'\bigl)
\end{align*} 
are $p$-units, where $\ord_p\gam_p\leq\min\{\ord_p\alp_p,\ord_p\bet_p\}$. 
Let $W^\ord_{\pi_p}\in\pi_p$ and $W^\ord_{\sig_p}\in\sig_p$ be the ordinary vectors defined in \S \ref{ssec:84}. 
Define $W^\flat_{\sig_p^\vee}\in\sig_p^\vee$ by $W^\flat_{\sig_p^\vee}=(W^\ord_{\sig_p})^\natural$. 

\begin{definition}\label{def:72}
Define the modified $p$-factor $\cale(\pi_p^{},\sig_p^\vee)$ by 
\begin{multline*}
\frac{1}{\cale(\pi_p^{},\sig_p^\vee)}=L\biggl(\frac{1}{2},\pi_p^{}\times\sig_p^\vee\biggl)\\
\times\gam\biggl(\frac{1}{2},\pi_p\otimes\mu_p^{\prime-1},\addchar_p^{}\biggl)\gam\biggl(\frac{1}{2},\pi^\vee_p\otimes\nu'_p,\addchar^{-1}_p\biggl)\gam\biggl(\frac{1}{2},\mu_p^{}\nu_p^{\prime-1},\addchar_p^{}\biggl)^2. 
\end{multline*}
\end{definition}

Define the $\frkp^c$-depletion $\bfU_{\frkp^c}^{\mu_p^{}\nu_p^{\prime-1}}W_{\pi_p}^\ord$ of $W_{\pi_p}^{\ord}$ by 
\[\bfU_{\frkp^c}^{\mu_p^{}\nu_p^{\prime-1}}W_{\pi_p}^\ord=\sum_{y\in (\ZZ_p/p^j\ZZ_p)^\times}\sum_{z\in\ZZ_p/p^j\ZZ_p}\frac{(\mu_p^{}\nu_p^{\prime-1})(y)}{p^j\gam_p^j}\pi_p\left(\begin{bmatrix} p^j & y & z \\ 0 & 1 & 0 \\ 0 & 0 & 1 \end{bmatrix}\right)W_{\pi_p}^{\ord}, \]
and the $\frkp$-depletion $W_p^\dagger=\bdTht_\frkp^{\nu_p^{}\mu_p^{\prime-1}}\bfU_{\frkp^c}^{\mu_p^{}\nu_p^{\prime-1}}W_{\pi_p}^\ord$ by 
\[W_p^\dagger=\sum_{x\in\ZZ_p/p^n\ZZ_p}\sum_{a\in (\ZZ_p/p^n\ZZ_p)^\times}\frac{\addchar_p\bigl(\frac{ax}{p^n}\bigl)}{p^n(\mu_p^{-1}\nu_p')(a)}\pi_p\left(\begin{bmatrix} 1 & 0 & 0 \\ 0 & 1 & \frac{x}{p^n}\\ 0 & 0 & 1 \end{bmatrix}\right)\bfU_{\frkp^c}^{\mu_p^{}\nu_p^{\prime-1}}W_{\pi_p}^\ord, \]
where $n\in\NN$ is an arbitrarily fixed large natural number (cf. Lemma \ref{lem:83}). 

Put $t_\ell=\imath_\frkp^{\prime-1}\biggl(\begin{bmatrix} 0 & 1 \\ p^\ell & 0 \end{bmatrix}\biggl)\in H(\QQ_p)$. 
If $\ell$ is sufficiently large, then Theorem \ref{thm:81} below gives 
\begin{align}
&\scri(W^\dagger_p,\sig_p^\vee(t_\ell)^{-1}W^\flat_{\sig_p^\vee}) \label{tag:710}\\
=&\frac{\gam^\GL\bigl(\frac{1}{2},\pi_p^{}\times\sig_p^\vee,\addchar_p^{}\bigl)}{L\bigl(\frac{1}{2},\pi_p^{}\times\sig_p^\vee\bigl)}
\left(\frac{\zet_p(2)}{p^{\ell/2}\zet_p(1)}\nu_p'(p)^\ell\frac{\gam\bigl(\frac{1}{2},\mu_p^{-1}\nu_p',\addchar^{-1}_p\bigl)}{\gam\bigl(\frac{1}{2},\pi_p^{}\otimes\mu_p^{\prime-1},\addchar_p^{}\bigl)}\right)^2 \notag\\
=&\frac{\zet_p(2)^2}{p^\ell\zet_p(1)^2}\nu_p'(p)^{2\ell}\cale(\pi_p^{},\sig_p^\vee)
=[K':K_0'(\frkp^\ell)]^{-2}p^\ell\nu_p'(p)^{2\ell}\cale(\pi_p^{},\sig_p^\vee) \notag
\end{align}
in view of the multiplicativity and functional equation of the gamma factor.


\subsection{An explicit central value formula}\label{ssec:710}

We say that an irreducible representation of $G(\QQ_q)$ (resp. $H(\QQ_q)$) is unramified if it admits a non-zero $K_q$ (resp. $K_q'$) invariant vector.
Let $\pi\simeq\otimes'_v\pi_v^{}$ be an irreducible tempered cuspidal automorphic representation of $G(\AA)$ and $\sig\simeq\otimes'_v\sig_v^{}$ that of $H(\AA)$ satisfying the following conditions: 
\begin{enumerate}
\item[($H_1$)] $\pi_q$ and $\sig_q$ are unramified for every non-split rational prime $q$; 
\item[($H_2$)] $\sig_\infty$ is a holomorphic discrete series with minimal $K$-type $-(k_1';k_2')$ and $\pi_\infty$ is a holomorphic discrete series or limit of holomorphic discrete series with minimal $K$-type $-(k_1^{},k_2^{};k_3^{})$ such that 
\begin{align*}
k_1^{}&\leq k_1'\leq k_2^{}, & k_3^{}&\leq k_2'; 
\end{align*}
\item[($H_3$)] $c(\sig_l)\leq 2c(\ome_{\sig_l})$ for every split rational prime $l\neq p$;  
\item[($H_4$)] $\pi_p$ is a generic constituent of a principal series $I(\nu_p,\rho_p,\mu_p)$; \\
$\sig_p$ is a generic constituent of a principal series $I(\mu_p',\nu_p')$.  
\end{enumerate}

Recall that $N=\prod_{l\neq p}l^{c(\pi_l)}$ and $N'=\prod_{l\neq p}l^{c(\sig_l)}$, where $c(\pi_l)=c(\sig_l)=0$ if $l$ does not split in $E$. 
Put $M=\prod_{l|N'}l^{c(\ome_{\sig_l})}$. 
Take an ideal $\frkM$ of $\frko_E$ such that $\frko_E/\frkM\simeq\ZZ/M\ZZ$. 
We define $\vph^\dagger=\otimes_vW_v^\dagger\in\pi$ in the following way: 
Set 
\begin{align*}
W_\infty^\dagger&=W_\infty^{\vep_\sig}, & 
W_p^\dagger&=\bdTht_\frkp^{\nu_p^{}\mu_p^{\prime-1}}\bfU_{\frkp^c}^{\mu_p^{}\nu_p^{\prime-1}}W_{\pi_p}^\ord. 
\end{align*}
Let $W_q^\dagger=W_{\pi_q}^{}$ be a spherical or an essential vector unless $pNN'$ is divisible by $q$. 
For each prime factor $l$ of $NN'$ we set $W_l^\dagger=\pi_l^{}(\iot(\xi'_l)\vsi)\vTh^{\ome_{\sig_l}^{-1}}W_{\pi_l}^{}$.

Let $\vph_\sig^\ord=\otimes_v^{}W_v^\ord\in\sig$ be such that $W_v^\ord=W_{\sig_v}$ for $v\neq p$ and $W_p^\ord=W_{\sig_p}^\ord\in\sig_p^{}$ is the normalized $p$-stabilized new vector with respect to $\bet'=\nu_p'(p)$ (see Remark \ref{rem:83}). 
Let $\vph_{\sig^\vee}^\flat=\otimes_v^{}W_v^\flat\in\sig$ be such that $W_v^\flat=W_{\sig_v^\vee}^{}$ for $v\neq p$ and $W_p^\flat=W_{\sig_p^\vee}^\flat\in\sig_p^\vee$.  

\begin{definition}[The $p$-modified period on $\U(1,1)$]\label{def:73}
Put
\[\Ome(\sig)=2^{\vka_\sig}(\sqrt{-1})^{k_1'-k_2'}[K':K_0'(\frkN_\sig)]L(1,\sig,\Ad)\cale(\sig_p,\Ad,\addchar_p)\prod_{l|N_\sig}\calb_{\sig_l}, \]
where the modified adjoint $p$-factor $\cale(\sig_p,\Ad,\addchar_p)$ is as in Definition \ref{def:82}. 
\end{definition}

\begin{remark}\label{rem:72}
We renormalize the Petersson norm by setting 
\[\|\vph_\sig\|^2_{K'_0(\frkN_\sig^{})}=[K':K_0'(\frkN_\sig)]\|\vph_\sig\|^2_{K'}. \]
Corollary \ref{cor:c1} gives 
\[\Ome(\sig)=D_E^{-1}2^2(-2\sqrt{-1})^{k_2'-k_1'}\|\vph_\sig\|^2_{K'_0(\frkN_\sig^{})}\cale(\sig_p,\Ad,\addchar_p)\]
(cf. Definition 3.12 of \cite{MH}). 
\end{remark}

\begin{lemma}\label{lem:74}
If $\sig_p$ is $p$-ordinary, then 
\[\frac{\Pet_{K'}(\vph_\sig^\ord\otimes\sig^\vee(t_\ell^{-1})\vph_{\sig^\vee}^\flat)}{\mu_p'(-1)p^{\ell/2}\nu_p'(p)^\ell}
=\frac{\|\vph_\sig\|^2_{K'_0(\frkN_\sig^{})}\cale(\sig_p,\Ad,\addchar_p)}{[K':K_0'(\frkp^\ell\frkN')]}. \]
\end{lemma}

\begin{proof}
Since $\vph_\sig^\ord$ is obtained by replacing $W_{\sig_p}^{}$ with $W_{\sig_p}^\ord$ in the definition of $\vph_\sig$ and since $\vph_{\sig^\vee}^\flat$ is obtained by replacing $W_{\sig_p^\vee}^{}$ with $W_{\sig_p^\vee}^\flat$ in the definition of $\vph_{\sig^\vee}$, we have
\[\frac{\Pet_{K'}(\vph_\sig^\ord\otimes\sig^\vee(t_\ell^{-1})\vph_{\sig^\vee}^\flat)}{\Pet_{K'}(\vph_\sig\otimes\vph_{\sig^\vee})}
=\frac{\calb_{\sig_p^\ord}^{[\ell]}}{\calb_{\sig_p}}
=\frac{\mu_p'(-1)p^{\ell/2}\nu_p'(p)^\ell}{[K_0'(\frkp^{c(\sig_p)}):K_0'(\frkp^\ell)]}\cale(\sig_p,\Ad,\addchar_p)\] 
by Lemma \ref{lem:82} if $\ell$ is sufficiently large. 
\end{proof}

We will use the following formula to prove Theorem \ref{thm:11}. 

\begin{proposition}\label{prop:71}
Notations and assumptions being as in Corollary \ref{cor:73}, we set $W_p^\dagger=\bdTht_\frkp^{\nu_p^{}\mu_p^{\prime-1}}\bfU_{\frkp^c}^{\mu_p^{}\nu_p^{\prime-1}}W_{\pi_p}^\ord$. 
If $\ell$ is sufficiently large, then 
\begin{align*}
&|\del|^{k_1'+k_2'-k_1^{}-k_2^{}-k_3^{}}(\ome_{\pi_p}^{}\ome_{\sig_p}^{-1})(-\del)\frac{\scrp_{K'}(\del^{\ulk,\ulk'}\vph^\dagger\otimes\sig^\vee(t_\ell^{-1})\vph_{\sig^\vee}^\flat)^2}{\Pet_{K'}(\vph_\sig^\ord\otimes\sig^\vee(t_\ell^{-1})\vph_{\sig^\vee}^\flat)^2}\\
=&\frac{2^5}{D_E^{9/2}[K'_0(\frkN'):K_0'(\frkM^2)]^2}\lam_K\frac{L\bigl(\frac{1}{2},\pi\times\sig^\vee\bigl)}{\Xi_{\vPh_{\pi\otimes\pi^\vee}^{\calk_\infty}}\Ome(\sig)^2}\cale(\pi_p^{},\sig_p^\vee)\prod_{l|NN'}\frac{\frkf_{\pi_l^{},\sig_l^\vee}}{(\ome_{\pi_l}^{}\ome_{\sig_l}^{-1})(-\del)}.
\end{align*}
\end{proposition}

\begin{proof}
Corollary \ref{cor:73} and (\ref{tag:710}) give 
\begin{align*}
&|\del|^{k_1'+k_2'-k_1^{}-k_2^{}-k_3^{}}(\ome_{\pi_p}^{}\ome_{\sig_p}^{-1})(-\del)\scrp_{K'}(\caly_+^{n_1^*}\vph^\dagger\otimes\sig^\vee(t_\ell^{-1})\vph_{\sig^\vee}^\flat)^2\\
=&\frac{2^{1+2k_1'-2k_2'}(-1)^{k_1'+k_2'}}{D_E^{5/2}[K':K_0'(\frkp^\ell\frkM^2)]^2}\lam_K\frac{L\bigl(\frac{1}{2},\pi\times\sig^\vee\bigl)}{\Xi_{\vPh_{\pi\otimes\pi^\vee}^{\calk_\infty}}}p^\ell\nu_p'(p)^{2\ell}\cale(\pi_p^{},\sig_p^\vee)\prod_{l|NN'}\frac{\frkf_{\pi_l^{},\sig_l^\vee}}{(\ome_{\pi_l}^{}\ome_{\sig_l}^{-1})(-\del)}. 
\end{align*}
By multiplying $\Pet_{K'}\bigl(\vph_\sig^\ord\otimes\sig^\vee(t_\ell^{-1})\vph_{\sig^\vee}^\flat\bigl)^{-2}
=\frac{[K':K_0'(\frkp^\ell\frkN')]^2p^{-\ell}\nu_p'(p)^{-2\ell}}{\|\vph_\sig\|^4_{K'_0(\frkN_\sig^{})}\cale(\sig_p,\Ad,\addchar_p)^2}$ we see that the left hand side of the stated identity is equal to 
\begin{align*}
&\frac{2^{1+2k_1'-2k_2'}(-1)^{k_1'+k_2'}}{D_E^{5/2}[K'_0(\frkN'):K_0'(\frkM^2)]^2}\lam_K\frac{L\bigl(\frac{1}{2},\pi\times\sig^\vee\bigl)}{\Xi_{\vPh_{\pi\otimes\pi^\vee}^{\calk_\infty}}\|\vph_\sig\|^4_{K'_0(\frkN_\sig^{})}\cale(\sig_p,\Ad,\addchar_p)^2}\\
&\times\cale(\pi_p^{},\sig_p^\vee)\prod_{l|NN'}\frac{\frkf_{\pi_l^{},\sig_l^\vee}}{(\ome_{\pi_l}^{}\ome_{\sig_l}^{-1})(-\del)}\\
=&\frac{2^5}{D_E^{9/2}[K'_0(\frkN'):K_0'(\frkM^2)]^2}\lam_K\frac{L\bigl(\frac{1}{2},\pi\times\sig^\vee\bigl)}{\Xi_{\vPh_{\pi\otimes\pi^\vee}^{\calk_\infty}}\Ome(\sig)^2}\cale(\pi_p^{},\sig_p^\vee)\prod_{l|NN'}\frac{\frkf_{\pi_l^{},\sig_l^\vee}}{(\ome_{\pi_l}^{}\ome_{\sig_l}^{-1})(-\del)} 
\end{align*}
by Remark \ref{rem:72}.
\end{proof}


\subsection{The Petersson norm on $\U(2,1)$}\label{ssec:711}

We shall translate the periods in Definition \ref{def:71} into more classical terminology. 
We retain the notations of \ref{ssec:71} and Remark \ref{rem:33}. 
Given $f\in S_{\ulk}^\calg(\frkN,\chi,\CC)$, we set $f'=\rho_{\ulk}(\sqrt{T'},1)f$. 
Put $\kap=k_2-k_1$. 
We write  
\[f'(Z,\bet)=\sum_{i=0}^\kap f'_i(Z,\bet)(-1)^i\binom{\kap}{i}X_1^iY_1^{\kap-i}, \]
where $(\rho_{\ulk},\call_{\ulk}(\CC))$ is realized on the space of homogeneous polynomials of degree $\kap$ in variables $X_1,Y_1$. 
We define $^\natural:\call_{\ulk}^{}(\CC)\stackrel{\sim}{\to}\call_{\ulk}^\vee(\CC)$ by \index{$P^\natural$}
\begin{align*}
P^\natural(X_2',Y_2')&=\rho_\kap(J)P(X_2',Y_2')=P(-Y_2',X_2'), & 
J&=\begin{bmatrix} 0 & 1 \\ -1 & 0 \end{bmatrix}, 
\end{align*}
where $(\rho_{\ulk}^\vee,\call_{\ulk}^\vee(\CC))$ is realized on the space of homogeneous polynomials of degree $\kap$ in variables $X_2',Y_2'$. 
Recall the involution $\tau_\frkN f=(r(\tau_\frkN)f)^\natural:\frkD\times G(\widehat{\QQ})\to\call_{\ulk}^\vee(\CC)$ defined in \S \ref{ssec:71}.
Observe that 
\[\tau_\frkN f'(Z,\bet)=\sum_{i=0}^\kap f_i'(-Z^c,(\bet\tau_\frkN)^\natural)\binom{\kap}{i}Y_2^{\prime i}X_2^{\prime\kap-i}. \]
Put $\tau_\frkN f=\rho_{\ulk}^\vee(\sqrt{T'},1)^{-1}\tau_\frkN f'$. 

We write $Z_G$ for the center of $G$. 
We define a perfect pairing
\[(\;,\;)_\frkN:S_{\ulk}^G(\frkN,\chi,\CC)\times S_{\ulk}^G(\frkN,\chi,\CC)\to\CC\]
by 
\begin{align*}
(h,f)_\frkN
&=\sum_{[\bet]\in G(\QQ)\bsl G(\widehat{\QQ})/ K_0(\frkN)}\int_{\Gam_{\frkN,\bet}\bsl\frkD}\frac{\ell_{\ulk}(\tau_\frkN h(Z,\bet)\otimes f(Z,\bet))}{\sharp(Z_G(\QQ)\cap\Gam_{\frkN,\bet})}\bfd Z\\
&=\sum_{[\bet]}\sum_{i=0}^\kap\binom{\kap}{i}\int_{\Gam_{\frkN,\bet}\bsl\frkD}\frac{h_i'(-Z^c,(\bet\tau_\frkN)^\natural)f_i'(Z,\bet)}{\sharp(Z_G(\QQ)\cap\Gam_{\frkN,\bet})}\bfd Z, 
\end{align*}
where $[\bet]$ means the double coset $G(\QQ)\bet K_0(\frkN)$ and 
\[\Gam_{\frkN,\bet}=G(\QQ)\cap \bet K_0(\frkN)\bet^{-1}. \]

Let $f^\circ\in S_{\ulk}^G(\frkN_\pi^{},\chi,\overline{\QQ})$ be a newform associated to $\pi$ in the sense of Definition \ref{def:57}. 
We identify $\call_{\ulk}^\vee(\CC)$ with the minimal $K$-type of $\pi_\infty$. 
Take a factorization $\pi\simeq\otimes'_v\pi_v^{}$ so that $\vPh_{\ulk}(f^\circ)_{\bfv_i}=\bfv_i\otimes_l W_{\pi_l}$. 
Then Lemma \ref{lem:71} gives 
\[\vPh_{\ulk^\vee}(f^{\circ\natural})_{\bfv_i}=\bfv_{\kap-i}\otimes_l W_{\pi_l}^\natural. \]
Moreover, we have 
\[\vPh_{\ulk^\vee}(\tau_{\frkN_\pi}f^\circ)_{\bfv_i}=\bfv_{\kap-i}\otimes_lW_{\pi_l^\vee}\prod_{l|N_\pi}\vep^\GL\biggl(\frac{1}{2},\pi_l,\addchar_l\biggl)^2\]
by (\ref{tag:72}) and Proposition \ref{prop:80}. 

Let 
\[\ell_{\ulk}\otimes\ell_{\ulk}:(\call_{\ulk}^{}(\CC)\otimes\call_{\ulk}^\vee(\CC))^\vee\otimes\call_{\ulk}^{}(\CC)\otimes\call_{\ulk}^\vee(\CC)\to\CC\]
be the perfect invariant pairing, where polynomials of $X_i^{},Y_i^{}$ are paired with those of $X_i',Y_i'$ for $i=1,2$. 
The linear map
\[\bfv_2^{}\otimes\bfv_1'\mapsto(\ell_{\ulk}\otimes\ell_{\ulk})(\bfv_2^{}\otimes\bfv_1'\otimes(\vPh_{\ulk}^{}(f^\circ)\otimes\vPh_{\ulk^\vee}(\tau_{\frkN_\pi}f^\circ)))\]
gives an embedding of $(\call_{\ulk}^{}(\CC)\otimes\call_{\ulk}^\vee(\CC))^\vee$ into $\pi\otimes\pi^\vee\subset\scra^0(G\times G)$. 
Now we construct the $\calk_\infty$-invariant cusp form $\vPh_{\pi\otimes\pi^\vee}^{\calk_\infty}\in\pi\otimes\pi^\vee$ in \S \ref{ssec:76}, using the polynomial
\[\bfP_{\ulk}=(X_2^{}\otimes Y_1'-Y_2^{}\otimes X_1')^\kap\in(\call_{\ulk}^{}(\CC)\otimes\call_{\ulk}^\vee(\CC))^\vee, \]
which is fixed by the diagonal action of $\GL_2(\CC)$. 
Define 
\[\vPh_{\pi\otimes\pi^\vee}^{\calk_\infty}=(\ell_{\ulk}\otimes\ell_{\ulk})(\bfP_{\ulk}^{}\otimes(\vPh_{\ulk}^{}(f^\circ)\otimes\vPh_{\ulk^\vee}(\tau_{\frkN_\pi}f^\circ)))\prod_{l|N_\pi}\vep^\GL\biggl(\frac{1}{2},\pi_l,\addchar_l\biggl)^{-2}.  \index{$\vPh_{\pi\otimes\pi^\vee}^{\calk_\infty}$}\]
Note that 
\[(\ell_{\ulk}\otimes\ell_{\ulk})(\bfP_{\ulk}^{}\otimes(\vPh_{\ulk}^{}(f^{\circ\prime})\otimes\vPh_{\ulk^\vee}(\tau_{\frkN_\pi}f^{\circ\prime})))=(\ell_{\ulk}\otimes\ell_{\ulk})(\bfP_{\ulk}^{}\otimes(\vPh_{\ulk}^{}(f^\circ)\otimes\vPh_{\ulk^\vee}(\tau_{\frkN_\pi}f^\circ)))\]
by the invariance of $\bfP_{\ulk}$. 
The Schur orthogonality relation gives 
\beq
\Pet_K\bigl(\vPh_{\pi\otimes\pi^\vee}^{\calk_\infty}\bigl)=\frac{(f^\circ,f^\circ)_{\frkN_\pi}}{[K:K_0(\frkN_\pi)]}\prod_{l|N_\pi}\vep^\GL\biggl(\frac{1}{2},\pi_l,\addchar_l\biggl)^{-2} \label{tag:711}
\eeq
by (\ref{tag:34}). 


\subsection{Proof of Theorem \ref{thm:11}}\label{ssec:713}

Let 
\begin{align*}
\bdsf&\in\bfS^\calg_{\ord}(\frkN,\chi,\calr), & 
\bdsg&\in\bfS^\calh_{\ord}(\frkN',\chi',\calr')
\end{align*} 
be Hida families. 
For $\ulQ\in\frkX_{\bdsf}'$ we write $\bdpi_{\ulQ}$ for the cuspidal automorphic representation of $G(\AA)$ associated with $\bdsf_{\ulQ}$. 
For $\ulQ'\in\frkX_{\calr'}^\cls$ we obtain the cuspidal automorphic representation $\bdsig_{\ulQ'}$ of $H(\AA)$ (see \S \ref{ssec:54}). 

\begin{definition}[modified Euler factors]
Choose $\mu_p^{}$, $\mu_p'$, $\nu_p'$ for $\bdpi_{\ulQ,p}$ and $\bdsig_{\ulQ',p}$ as in \S \ref{ssec:79}. 
For $\calq=(\ulQ,\ulQ')\in\frkX_\bfV^\mathrm{crit}$ we define the modified $p$-adic and archimedean factors by  
\begin{align*}
\cale_p(\mathrm{Fil}^+\bfV_\calq)&=\cale(\bdpi_{\ulQ,p}^{},\bdsig_{\ulQ',p}^\vee), \\
\cale_\infty(\bfV_\calq)&=(-\sqrt{-1})^{k_{Q_1'}+k_{Q_2'}-k_{Q_1}^{}-k_{Q_2}^{}-k_{Q_3}^{}}. 
\end{align*}
\end{definition}

\begin{definition}[A period on $\U(2,1)$]
Let $\bdsf_{\ulQ}^\circ$ be the vector-valued newform of weight $k_{\ulQ}$ associated to $\bdsf_{\ulQ}$. 
Following (\ref{tag:711}), for $\ulQ\in\frkX_{\bdsf}''$ we put 
\[\Xi_{\bdsf_{\ulQ}}=\frac{\Pet(\bdpi_{\ulQ})[K:K_0(\frkN_{\bdpi_{\ulQ}})]}{(-2)^{k_{Q_3}-k_{Q_1}-k_{Q_2}}(\bdsf_{\ulQ}^\circ,\bdsf_{\ulQ}^\circ)_{\frkN_{\bdpi_{\ulQ}}}}\prod_{l|N_{\bdpi_{\ulQ}}}\vep^\GL\biggl(\frac{1}{2},\bdpi_{\ulQ,l},\addchar_l\biggl)^2. \]
\end{definition}

\begin{proposition}\label{prop:73}
Notations and assumptions being as in Theorem \ref{thm:11}, we have  
\begin{align*}
\frac{\calq(\Tht_{\bdsf,\bdsg}^2)}{\Ome_\frkp^{2m_\calq}}
=&\biggl(\frac{2\pi\sqrt{-1}}{\Ome_\infty}\biggl)^{2m_\calq}\frac{L\bigl(\frac{1}{2},\bdpi_{\ulQ}^{}\times\bdsig_{\ulQ'}^\vee\bigl)}{\Xi_{\bdsf_{\ulQ}^{}}^{}\Ome(\bdsig_{\ulQ'})^2}\cale_\infty(\bfV_\calq)\cale_p(\mathrm{Fil}^+\bfV_\calq)\\
&\times\frac{2^5}{D_E^{9/2}}\lam_K\frac{\prod_{l|NN'}(\ome_{\bdpi_{\ulQ,l}}^{-1}\ome_{\bdsig_{\ulQ',l}}^{})(-\del)\frkf_{\bdpi_{\ulQ,l}^{},\bdsig_{\ulQ',l}^\vee}}{\calq([-\del\ono_3]_{\bfI_3}^{-1}[-\del\ono_2]_{\bfI_2}^{})[K'_0(\frkN'):K'_0(\frkM^2)]^2}
\end{align*}
for $\calq=(\ulQ,\ulQ')\in\frkY^{\rm crit}_\bfV\cap(\frkX_{\bdsf}''\times\frkX_{\bfI_2}^\cls)$. 
\end{proposition}

\begin{proof}
Fix $\calq_0=(\ulQ,\ulQ')\in\frkY^{\rm crit}_\bfV\cap(\frkX_{\bdsf}''\times\frkX_{\bfI_2}^\cls)$. 
Put $\sig=\bdsig_{\ulQ'}$. 
Let $g$ be the normalized $p$-stabilized newform associated to $\sig$. 
Put $\vph_\sig^{\ord}=\vPh_{k_{\ulQ'}}(g)$. 
Theorem \ref{thm:41} allows us to lift $g$ to a Hida family $\bdsg$ on $\calh$. 
Notice that the interpolation formula is independent of the choice of $\bdsg$. 

Hypothesis ($H_1$) holds by (splt), ($H_2)$ holds by the definition of $\frkY^{\rm crit}_\bfV$, and ($H_4^{}$) holds by \cite[Theorem 5.3]{Hid04}. 
Thanks to Remark \ref{rem:12}(\ref{rem:122}), we can apply Proposition \ref{prop:71} to $\bdsf_{\ulQ}$ and $g$. 

If $\ell$ is sufficiently large, then 
\begin{align*}
&\left(\frac{\Ome_\infty}{2\pi\sqrt{-1}\Ome_\frkp}\right)^{2m_{\calq_0}}\calq_0(\Tht_{\bdsf,\bdsg}^2)\\
=&\frac{\scrp_{K'}\biggl(\del_{k_{\ulQ}}^{k_{Q_2'}-k_{Q_3^{}}}\bdTht_\frkp^{\eps_{Q_2'}^{}\eps_{Q_3^{}}^{-1}}\bfU_{\frkp^c}^{\eps_{Q_1'}^{}\eps_{Q_1^{}}^{-1}}\bfU^{\chi'}_{\frkN'}\vPh((\bdsf_{\ulQ})_0)\otimes\sig^\vee(t_\ell)^{-1}\vph_{\sig^\vee}^\flat\biggl)^2}{\Pet_{K'}\Big(\vph_{\sig}^\ord\otimes\sig^\vee(t_\ell)^{-1}\vph_{\sig^\vee}^\flat\Big)^2}\\
=&\frac{L\bigl(\frac{1}{2},\bdpi_{\ulQ}^{}\times\sig^\vee\bigl)}{\Xi_{\bdpi_{\ulQ}^{}}^{}\Ome(\sig)^2}(-\sqrt{-1}\del)^{k_{Q_1^{}}^{}+k_{Q_2^{}}^{}+k_{Q_3^{}}^{}-k_{Q_1'}^{}-k_{Q_2'}^{}}\bigl(\ome_{\bdpi_{\ulQ,p}}^{-1}\ome_{\sig_p}^{}\bigl)(-\del)\\
&\times\cale_p(\mathrm{Fil}^+\bfV_{\calq_0})\frac{2^5}{D_E^{9/2}}\lam_K\frac{\prod_{l|NN'}(\ome_{\bdpi_{\ulQ,l}}^{-1}\ome_{\sig_l}^{})(-\del)\frkf_{\bdpi_{\ulQ,l}^{},\sig_l^\vee}}{[K_0'(\frkN'):K_0'(\frkM^2)]^2}
\end{align*}
by Propositions \ref{prop:67} and \ref{prop:71}. 
\end{proof}

For each prime factor $l$ of $NN'$ Lemma 4.17 of \cite{HY} gives an element $\sqrt{\frkf_{\bdpi_l^{},\bdsig_l^\vee}}\in\bfI_3\widehat{\otimes}\bfI_2$ such that for $\calq=(\ulQ,\ulQ')\in\frkX^{\rm crit}_\bfV\cap(\frkX_{\bdsf}'\times\frkX_{\bfI_2}^\cls)$
\[\calq(\sqrt{\frkf_{\bdpi_l^{},\bdsig_l^\vee}})^2=\frkf_{\bdpi_{\ulQ,l}^{},\bdsig_{\ulQ',l}^\vee}. \]
We enlarge $\bfI_n$ so that it contains a square-root of $[-\del\ono_n]_{\bfI_n}$. 
Assume that $\bfI_3$ contains an element $\sqrt{\ome_{\bdpi_l^{}}(-\del)}$ which satisfies $\ulQ(\sqrt{\ome_{\bdpi_l^{}}(-\del)})=\sqrt{\ome_{\bdpi_{\ulQ,l}}(-\del)}$ and that  $\bfI_3$ contains a similar element $\sqrt{\ome_{\bdsig_l}(-\del)}$. 
Define the fudge factor $\sqrt{(\ome_{\bdpi}^{-1}\ome_{\bdsig}^{})(-\del)\frkf_{\bdpi,\bdsig^\vee}}\in\bfI_3\widehat{\otimes}\bfI_2$ by 
\[\sqrt{(\ome_{\bdpi}^{-1}\ome_{\bdsig}^{})(-\del)\frkf_{\bdpi,\bdsig^\vee}}=\prod_{l|NN'}\sqrt{(\ome_{\bdpi_l}^{-1}\ome_{\bdsig_l}^{})(-\del)\frkf_{\bdpi_l^{},\bdsig_l^\vee}}. \]

We normalize the theta element $\Tht_{\bdsf,\bdsg}$ by   
\[\scrl_{\bdsf,\bdsg}=\sqrt{\frac{D_E^{9/2}}{2^5\lam_K}}[-\del\ono_3]_{\bfI_3}^{-1/2}[-\del\ono_2]_{\bfI_2}^{1/2}\frac{[K_0'(\frkN'):K_0'(\frkM^2)]}{\sqrt{(\ome_{\bdpi}^{-1}\ome_{\bdsig}^{})(-\del)\frkf_{\bdpi,\bdsig}}}\Tht_{\bdsf,\bdsg}. \]
Theorem \ref{thm:11} now follows from Proposition \ref{prop:73}. 


\subsection{Proof of Theorem \ref{thm:12}}\label{ssec:714} 

\begin{definition}\label{def:74}
Let $\sig$ be an irreducible cuspidal automorphic representation of $H(\AA)$ whose archimedean part is a discrete series or limit of discrete series with Blattner parameter $-(k_1';k_2')$. 
For a positive integer $N'$ we put
\[\Ome^{(N')}(\sig)=2^{\vka_\sig}(\sqrt{-1})^{k_1'-k_2'}[K':K_0'(\frkN_\sig)]L^{(N')}(1,\sig,\Ad)\cale(\sig_p,\Ad,\addchar_p)\calb_{\sig_p}, \]
where $L^{(N')}(1,\sig,\Ad)$ is the partial adjoint $L$-series of $\sig$ with the archimdean factor but without Euler factors at primes dividing $N'$.
\end{definition}

We retain the notation from \S \ref{ssec:17}. 
The proof is sketched in the introduction. 
Denoting the characteristic function of $\ZZ_p^\times$ by $\phi_0$, we wrote $\bfU_{\frkp^c}^0=\bfU_{\frkp^c}^{\phi_0}$ there. 
We can assume that 
\[\tht^m_\frka(\bfU_{\frkp^c}^0f_\pi^{\ord})\in \bfT_{\frka\frkp^c\frkb_0}^m(\calo)\]
for every $m\in\scrs_\frka^+$ at the cost of replacing $f_\pi^{\ord}$ with its $p$-power multiple. 
We can define $\calj^{\chi'}(f_\pi^{\ord})\in \bfS^H_{\ord}(\frkN',\chi',\Lam)$ by
\[J^{\chi'}_\frka(f_\pi^{\ord})=\lim_{j\to\infty}\sum_{m\in\scrs_\frka^+}\calj_\frka^{\chi',mp^{j!}}(f_\pi^{\ord})\,q^m, \]
where $\calj_\frka^{\chi',m}(f_\pi^{\ord})\in\Lam$ corresponds to the $\calo$-valued distribution 
\[\phi\mapsto\int_{\ZZ_p^\times}\phi(x)x^{-k_3}\,\d\mu(\widetilde{\FJ}^m_\frka(\bfU_{\frkp^c}^0\bfU^{\chi'}_{\frkN'}f_\pi^{\ord})) \]
(see Corollary \ref{cor:32}). 

Let $\bdsg'\in\bfS_\ord^\calh(\frkN',\chi',\bfI)$ be a one-variable cuspidal Hida family on $\calh$ such that $\bdsg'_Q\in\bdse' S_{(k_1;k_Q)}^\calh(p^\ell\frkN',\chi'\eps_Q^\downarrow,\bfI(Q))$ for $Q\in\frkX_\bfI$ with $k_Q-k_1\geq 2$. 
In the first step we assume ($H_3$) and define the theta element $\Tht_{f_\pi^{\ord},\bdsg'}\in\mathrm{Frac}\bfI$ attached to $f_\pi^{\ord}$ and $\bdsg'$ by   
\[\Tht_{f_\pi^{\ord},\bdsg'}=\FC^1_{\frko_E}(\1_{\bdsg'}\calj^{\chi'}(f_\pi^{\ord})). \]
We normalize this theta element by   
\[\scrl_{f_\pi^{\ord},\bdsg'}=\sqrt{\frac{D_E^{9/2}}{2^5\lam_K}}[-\del]_\bfI^{1/2}\frac{[K_0'(\frkN'):K_0'(\frkM^2)]}{\sqrt{(-\del)^{k_2+k_3}(\ome_\pi^{-1}\ome_{\bdsig}^{})(-\del)\frkf_{\pi,\bdsig}}}\Tht_{f_\pi^{\ord},\bdsg'}, \]
where 
\[\sqrt{(\ome_\pi^{-1}\ome_{\bdsig}^{})(-\del)\frkf_{\pi,\bdsig^\vee}}=\prod_{l|NN'}\sqrt{(\ome_{\pi_l}^{-1}\ome_{\bdsig_l}^{})(-\del)\frkf_{\pi_l^{},\bdsig_l^\vee}}. \]
The same computation as in \S \ref{ssec:713} shows that for $Q\in\frkX_\bfI$ with $k_Q\geq k_3$
\[\frac{Q(\scrl_{f_\pi^{\ord},\bdsg'})^2}{\Ome_\frkp^{2m_Q'}}=\biggl(\frac{2\pi\sqrt{-1}}{\Ome_\infty}\biggl)^{2m_Q'}\frac{\Gam(\bfV_Q',0)L(\bfV_Q',0)}{\Xi_{f_\pi^{\ord}}\Ome(\bdsig_Q)^2}\cale_\infty(\bfV_Q')\cale_p(\mathrm{Fil}^+\bfV_Q'). \]

Next we remove the assumption ($H_3$). 
Let $M$ be a prime-to-$p$ {\it odd} integer divisibly only by prime factors of $N'$. 
Take an ideal $\frkM$ of $\frkr$ such that $\frkr/\frkM\simeq\ZZ/M\ZZ$. 
Fix a primitive character $\vrh$ of $(\frkr/\frkM)^\times$ such that $\chi'\vrh^{-2}$ has conductor $M$. 
We extend $\vrh$ to an automorphic character $\vrh_\AA=\prod_v\vrh_v$ of $\U(1)(\AA)$. 

Since $\bdsg_Q'$ has tame level $\frkN'$ for $Q\in \frkX^\cls_\bfI$, there exists a positive integer $A$ such that if $m\geq A$ and $M$ is divisible by $N^{\prime m}$, then $\bdsig_Q\otimes\vrh_\AA$ has tame conductor $M^2$ and $(\bdsig_{Q,l}\otimes\vrh_l^{})_u^{}$ is trivial for all $Q\in \frkX^\cls_\bfI$ and prime factors $l$ of $N'$ and $Q\in \frkX^\cls_\bfI$ by (\ref{tag:82}) and stability of the epsilon factor (cf. \cite{JS}). 
We enlarge $M$ so that $(\pi_l^{}\otimes\vrh_l)_u^{}$ is trivial for prime factors $l$ of $N'$.  
Then $B_{\bdsig_{Q,l}\otimes\vrh_l^{}}=B_{\pi_l\otimes\vrh_l}=1$ for all prime factors $l$ of $N'$ by Proposition 6.4 of \cite{HY}. 
It follows that 
\begin{align*} 
\Ome(\bdsig_Q\otimes\vrh_\AA)&=[K_0'(\frkN'):K_0'(\frkM^2)]\Ome^{(N')}(\bdsig_Q)\prod_{l|N'}\zet_l(2), \\ 
\Pet(\pi\otimes\vrh_\AA)&=2^{\vka_\pi}L(1,\pi,\Ad)\prod_{q|pN,q\nmid N'}\calb_{\pi_q}\prod_{l|N'}\frac{\zet_l(3)}{L^\GL(1,\pi_l^{}\times\pi_l^\vee)}. 
\end{align*}
Since 
\[\frac{\Pet_K\Big(\vPh^{\calk_\infty}_{(\pi\otimes\vrh_\AA)\otimes(\pi\otimes\vrh_\AA)^\vee}\Big)}{\Pet_K\Big(\vPh^{\calk_\infty}_{\pi\otimes\pi^\vee}\Big)}
=\prod_{l|N'}\frac{B_{\pi_l\otimes\vrh_l}}{B_{\pi_l^{}}}
=\prod_{l|N'}\frac{1}{B_{\pi_l^{}}}, \]
we have $\Xi_{\vPh^{\calk_\infty}_{(\pi\otimes\vrh_\AA)\otimes(\pi\otimes\vrh_\AA)^\vee}}=\Xi_{\vPh^{\calk_\infty}_{\pi\otimes\pi^\vee}}$. 

Fix $Q_0\in\frkX_\bfI$ with $k_{Q_0}\geq k_3$. 
Let 
\begin{align*}
f_{\pi\otimes\vrh}^{\ord}&\in\bdse S^G_{\ulk}(p\frkN_\pi\frkM^3,\chi\vrh^{-3},\calo), &
g_\vrh'&\in\bdse' S^H_{(k_1;k_{Q_0})}(p^\ell\frkM^2,\chi'\vrh^{-2},\calo)
\end{align*}
be $p$-stabilized newforms associated to $\pi\otimes\vrh_\AA^{}$ and $\bdsig_{Q_0}\otimes\vrh_\AA^{}$. 
We extend $g_\vrh'$ to a cusp form on $\calh$ by (\ref{tag:53}) and lift it to a Hida family $\bdsg_\vrh'\in\bfS^\calh_\ord(\frkM^2,\chi'\vrh^{-2},\bfI)$ by Theorem \ref{thm:41}. 
For our choice of $\vrh$ we see that $\bdsg_\vrh'$ satisfies ($H_3$). 
We have seen that 
\begin{align*}
\frac{Q(\scrl_{f_{\pi\otimes\vrh}^{\ord},\bdsg_\vrh'})^2}{\Ome_\frkp^{2m_Q'}}
&=\biggl(\frac{2\pi\sqrt{-1}}{\Ome_\infty}\biggl)^{2m_Q'}\frac{\Gam(0,\bfV_Q')L(0,\bfV_Q')}{\Xi_{f_{\pi\otimes\vrh}^{\ord}}\Ome(\bdsig_Q\otimes\vrh_\AA)^2}\cale_\infty(\bfV_Q')\cale_p(\mathrm{Fil}^+\bfV_Q')\\
&=\biggl(\frac{2\pi\sqrt{-1}}{\Ome_\infty}\biggl)^{2m_Q'}\frac{\Gam(0,\bfV_Q')L(0,\bfV_Q')\cale_\infty(\bfV_Q')\cale_p(\mathrm{Fil}^+\bfV_Q')}{\Xi_{f_\pi^{\ord}}\Ome^{(N')}(\bdsig_Q)^2[K_0'(\frkN'):K_0'(\frkM^2)]^2\prod_{l|N'}\zet_l(2)^2}
\end{align*}
for $Q\in\frkX_\bfI$ with $k_Q\geq k_3$. 
Therefore 
\[\call_{f_\pi^{\ord},\bdsg'}^2=[\calk_0'(\frkN'):\calk_0'(\frkM^2)]\scrl_{f_{\pi\otimes\vrh}^{\ord},\bdsg_\vrh'}\prod_{l|N'}\zet_l(2)\in\mathrm{Frac}\bfI_2 \]
satisfies the interpolation formula in Theorem \ref{thm:12} for $Q\in\frkX_\bfI$ with $k_Q\geq k_3$. 


\section{The non-archimedean computation}\label{sec:8}


\subsection{The Gauss sum}\label{ssec:81}

Let $F$ be a finite extension of $\QQ_p$ which contains the integer ring $\frko$ having a single prime ideal $\frkp$. 
We denote the order of the residue field $\frko/\frkp$ by $q$. 
The absolute value $\Abs_F=|\cdot|$ on $F$ is normalized via $|\vpi|=q^{-1}$ for any generator $\vpi$ of $\frkp$, where $q$ denotes the order of the residue field $\frko/\frkp$. 
We denote by $B_n$ the group of upper triangular matrices in $\GL_n(F)$, and by $N_n$ its unipotent radical. 
Let $\d g$ be the Haar measure on $\GL_n(F)$ giving $\GL_n(\frko)$ volume $1$. 

We choose a non-trivial additive character $\addchar$ of $F$ so that the maximal fractional ideal on which it is trivial is $\frko$. \index{$\addchar$}
We write $\cals(F)$ for the space of locally constant compactly supported functions on $F$. 
The Fourier transform of $\phi\in\cals(F)$ is defined by 
\[\widehat{\phi}(y)=\int_F\phi(x)\addchar(-xy)\,\d x. \index{$\widehat{\phi}$}\]
The measure $\d x$ is chosen so that $\widehat{\widehat{\phi}}(x)=\phi(-x)$. 
Put \index{$\scrn_n$}
\begin{align*}
\scrn_n&=N_n\cap\GL_n(\frko), & 
J&=\begin{bmatrix} 0 & 1 \\ -1 & 0 \end{bmatrix}. 
\end{align*}

Let $\chi$ be a character of $F^\times$. 
When $\chi$ is ramified, the conductor $c(\chi)$ of $\chi$ is defined as the smallest positive integer $n$ such that $\chi$ is trivial on $1+\frkp^n$. 
When $\chi$ is unramified, we set $c(\chi)=0$. 

If $c(\chi)\geq 1$, then the Gauss sum is defined by 
\[\frkg(\chi,\addchar)=\sum_{a\in(\frko/\frkp^{c(\chi)})^\times}\chi(a)^{-1}\addchar\biggl(-\frac{a}{\vpi^{c(\chi)}}\biggl). \index{$\frkg(\chi,\addchar)$}\]
When $c(\chi)=0$, we formally set $\frkg(\chi,\addchar)=1$. 

\begin{remark}\label{rem:81}
The Gauss sum is related to the root number in the following way: 
\[\vep\left(\frac{1}{2},\chi,\addchar^{-1}\right)
=q^{-c(\chi)/2}\chi(\vpi)^{c(\chi)}\frkg(\chi,\addchar). \]
\end{remark}


\subsection{Essential vectors}\label{ssec:82}

Let $\pi$ be an irreducible admissible generic representation of $\GL_{m+1}(F)$. 
We write $\scrw_{\addchar}(\pi)$ for the Whittaker model of $\pi$ with respect to the additive character $\addchar$ of $F$. 
Given an open compact subgroup $\Gam$ of $\GL_{m+1}(F)$ and its character $\calx:\Gam\to\CC^\times$, we put 
\[\scrw_{\addchar}(\pi,\Gam,\calx)=\{W\in \scrw_{\addchar}(\pi)\;|\;\pi(\gam)W=\calx(\gam)W\text{ for }\gam\in\Gam\}. \]

Assume that $n\geq 1$. 
For a positive integer $\ell$ the subgroup $\calk_0^{(m+1)}(\frkp^\ell)$ consists of matrices of the form \index{$\calk_0^{(n)}(\frkp^\ell)$}
\begin{align*}
&\begin{bmatrix} A & B \\ C & d\end{bmatrix} & 
(A&\in\GL_m(\frko),\;B\in\frko^m,\;\trs C\in(\frkp^\ell)^m,\; d\in\frko^\times). 
\end{align*}
When $\ell=0$, we set $\calk_0^{(m+1)}(\frkp^\ell)=\GL_{m+1}(\frko)$. 
Given a character $\ome$ of $\frko^\times$, we define the characters $\ome^\downarrow:\calk_0^{(m+1)}(\frkp^\ell)\to\CC^\times$ and $\ome^\uparrow:\calk_0^{(2)}(\frkp^\ell)\to\CC^\times$ by \index{$\ome^\downarrow$}
\begin{align}
\ome^\downarrow\left(\begin{bmatrix} A & B \\ C & d\end{bmatrix}\right)&=\ome(d), &
\ome^\uparrow\left(\begin{bmatrix} a & b \\ c & d\end{bmatrix}\right)&=\ome(a). \label{tag:81}
\end{align}

We write $\ome_\pi$ for the central character of $\pi$. \index{$c(\pi)$}
Let $c(\pi)$ denote the exponent of the conductor of $\pi$, i.e., the epsilon factor of $\pi$ is of the form  
\beq
\vep^\GL\left(s+\frac{1}{2},\pi,\addchar\right)=q^{-c(\pi)s}\vep^\GL\left(\frac{1}{2},\pi,\addchar\right). \label{tag:82}
\eeq
Th\'{e}orem\`{e} on p.~211 of \cite{JPSS} says that 
\[\dim \scrw_{\addchar}\bigl(\pi,\calk_0^{(m+1)}\bigl(\frkp^{c(\pi)}\bigl),\ome^\downarrow_\pi\bigl)=1. \] 
Theorem 3.1 of \cite{NM} enables us to normalize a basis vector of this one-dimensional space in the following way:  

\begin{definition}[essential vectors]\label{def:81}
There exists a unique vector $W_\pi\in\scrw_{\addchar}\bigl(\pi,\calk_0^{(m+1)}\bigl(\frkp^{c(\pi)}\bigl),\ome^\downarrow_\pi\bigl)$ which satisfies $W_\pi(\ono_{m+1})=1$. 
This vector $W_\pi$ is called a normalized essential Whittaker vector of $\pi$ with respect to $\addchar$. \index{$W_\pi$}
\end{definition}

Given $W\in\scrw_{\addchar}(\pi)$, we define $\widetilde{W}\in\scrw_{\addchar^{-1}}(\pi^\vee)$ by $\widetilde{W}(g)=W(w_n\trs g^{-1})$. \index{$\widetilde{W}$}
Put $\xi_{m,\ell}=\begin{bmatrix} \vpi^{-\ell}\ono_m & 0 \\ 0 & 1 \end{bmatrix}$. 
Then $\pi^\vee(\xi_{m,c(\pi)})\widetilde{W}_\pi$ is an essential vector of $\pi^\vee$. 
The following result is Proposition 6.2 of \cite{HY}. 

\begin{proposition}\label{prop:80}
Let $\pi$ be an irreducible admissible generic representation of $\GL_{m+1}(F)$. 
Let $W_{\pi^\vee}$ be the essential vector of $\pi^\vee$ with respect to $\addchar^{-1}$. 
Then $\pi^\vee(\xi_{m,c(\pi)})\widetilde{W}_\pi=\vep^\GL\bigl(\frac{1}{2},\pi,\addchar\bigl)^mW_{\pi^\vee}$. 
\end{proposition}


\subsection{The JPSS integrals}\label{ssec:83}

Let $\pi$ be an irreducible admissible generic representation of $\GL_n(F)$. 
One can define an invariant perfect pairing 
\[\La\;,\;\Ra:\scrw_{\addchar}(\pi)\otimes\scrw_{\addchar^{-1}}(\pi^\vee)\to\CC\] 
by 
\beq
\La W_1,W_2\Ra=\int_{N_{n-1}\bsl\GL_{n-1}(F)}W_1\biggl(\begin{bmatrix} h & \\ & 1 \end{bmatrix}\biggl)W_2\biggl(\begin{bmatrix} h & \\ & 1 \end{bmatrix}\biggl)\,\d h. \index{$\La W_1,W_2\Ra$}\label{tag:83}
\eeq 

Let $m$ be a positive integer which is less than $n$. 
Put $l=n-m-1$. 
Let $\sig$ be an irreducible admissible generic representation of $\GL_m(F)$ whose central character is $\ome_\sig$. 
We associate to Whittaker functions $W\in \scrw_{\addchar}(\pi)$ and $W'\in\scrw_{\addchar^{-1}}(\sig^\vee)$ the local zeta integrals
\begin{align*}
Z(s,W,W')&=\int_{N_m\bsl G_m}W\biggl(\begin{bmatrix} h & \\ & \ono_{l+1} \end{bmatrix}\biggl)W'(h)|\det h|^{s-\frac{l+1}{2}}\,\d h, \\
\widetilde{Z}(s,\widetilde{W},\widetilde{W'})&=\int_{N_m\bsl G_m}\int_{\Mat_{l,m}(F)} \widetilde{W}\left(\begin{bmatrix} h & & \\ x & \ono_l & \\ & & 1 \end{bmatrix}\right)\widetilde{W'}(h)|\det h|^{s-\frac{l+1}{2}}\,\d x\d h, 
\end{align*}
which converge absolutely for $\Re s\gg 0$. 

We write $L^\GL(s,\pi\times\sig^\vee)$, $\vep^\GL(s,\pi\times\sig^\vee,\addchar)$ and $\gam^\GL(s,\pi\times\sig^\vee,\addchar)$ for the $L$, epsilon and gamma factors associated to $\pi$ and $\sig^\vee$. 
These local factors are studied extensively in \cite{JPSS2}. 
The gamma factor is defined as the proportionality constant of the functional equation 
\beq
Z(1-s,\pi^\vee(w_{n,m})\widetilde{W},\widetilde{W'})=\ome_\sig(-1)^{n-1}\gam^\GL(s,\pi\times\sig^\vee,\addchar)\widetilde{Z}(s,W,W'), \label{tag:84}
\eeq
where 
\[w_{n,m}=\begin{bmatrix} 
 \ono_r & \\ & w_{n-m} 
 \end{bmatrix}. \]
 
\begin{remark}\label{rem:82}
When we view $\pi$ and $\sig$ are representations of unitary groups over the split quadratic algebra $F\oplus F$, 
\[L(s,\pi\times\sig^\vee)=L^\GL(s,\pi\times\sig^\vee)L^\GL(s,\pi^\vee\times\sig). \]
When $m=1$ and $\chi$ is a character of $F^\times$, we will write 
\begin{align*}
L(s,\pi\otimes\chi)&=L^\GL(s,\pi\times\chi), \\   
\vep(s,\pi\otimes\chi,\addchar)&=\vep^\GL(s,\pi\times\chi,\addchar), \\
\gam(s,\pi\otimes\chi,\addchar)&=\gam^\GL(s,\pi\times\chi,\addchar). 
\end{align*}
\end{remark}
 
Let $\pi$ be an irreducible admissible tempered representation of $\GL_n(F)$ and $\sig$ that of $\GL_{n-1}(F)$. 
We consider the integral
\[J(W_1^{},W_2^{},W_1',W_2')=\int_{\GL_{n-1}(F)}\biggl\La\pi\biggl(\begin{bmatrix} h & \\ & 1 \end{bmatrix}\biggl)W_1^{},W_2^{}\biggl\Ra\La W_1',\sig^\vee(h)W_2'\Ra'\,\d h,  \]
which is convergent for 
\begin{align*}
W_1&\in\scrw_{\addchar}(\pi^{}), &
W_2&\in\scrw_{\addchar^{-1}}(\pi^\vee), &
W_1'&\in\scrw_{\addchar}(\sig^{}), & 
W_2'&\in\scrw_{\addchar^{-1}}(\sig^\vee). 
\end{align*}

The following result is called a splitting lemma and was proved by Wei Zhang in Proposition 4.10 of \cite{Z} by using the work of Lapid and Mao \cite{LM}. 
It is worth noting that Proposition 4.10 of \cite{Z} uses unnormalized local Haar measures (cf. \S 2.1 of \cite{Z}) while we here use normalized ones. 

\begin{lemma}\label{lem:81}
Notation being as above, we have 
\[J(W_1^{},W_2^{},W_1',W_2')=Z\biggl(\frac{1}{2},W_1^{},W_2'\biggl)Z\biggl(\frac{1}{2},W_2^{},W_1'\biggl)\prod_{i=1}^{m-2}\zet_F(i). \]
\end{lemma}


\subsection{Ordinary vectors of representations of $\GL_n(F)$}\label{ssec:84}

For a subset $S$ we write $\II_S$ for the characteristic function of $S$. 
For a compact subgroup $\Gam$ of $\GL_n(F)$ and a representation $(\pi,V)$ of $\GL_n(F)$ we let 
\[V^\Gam=\{v\in V\;|\;\pi(\gam)v=v\text{ for }\gam\in\Gam\}\] 
be the space of $\Gam$-invariant vectors in $V$. 
Let $\mu_1,\mu_2,\dots,\mu_n$ be quasi-characters of $F^\times$. 
The space $V$ of $\pi=I(\mu_1,\mu_2,\dots,\mu_n)$ consists of functions $h:\GL_n(F)\to\CC$ which satisfy
\begin{align*}
h(tug)&=h(g)\wp_n(t)^{1/2}\prod_{i=1}^n\mu_i(t_i), & 
\wp_n(t)&=\prod_{i=1}^n|t_i|^{n+1-2i}
\end{align*}
for $t=\diag[t_1,t_2,\dots,t_n]\in T_n$, $u\in N_n$ and $g\in\GL_n(F)$. 

We define a function $h^\mathrm{ord}$ on $B_nw_nN_n$ by 
\[h^{\ord}_\pi(tuw_nv)=\II_{\scrn_n}(v)\wp_n(t)^{1/2}\prod_{i=1}^n\mu_i(t_i) \index{$h^{\ord}_\pi$}\]
for $t=\diag[t_1,t_2,\dots,t_n]\in T_n$ and $u,v\in N_n$, where $\II_{\scrn_n}$ denotes the characteristic function of $\scrn_n=N_n\cap\GL_n(\frko)$. 
Since $B_nw_nN_n$ is the cell of the longest Weyl element $w_n$ in the Bruhat decomposition of $\GL_n(F)$, we can extend $h^{\ord}_\pi$ by zero to an element of $V$ 
(cf. \cite[(B), p.~138]{Cartier79Corvallis}). 
We call $h_\pi^\mathrm{ord}\in V^{\scrn_n}$ the normalized ordinary vector with respect to $(\mu_1,\mu_2,\dots,\mu_n)$.

Put $\alp_i=\begin{bmatrix} \ono_{n-i} & \\ & \vpi\ono_i \end{bmatrix}\in T_n$ for $i=1,2,\dots,n$. 
Define the operator $U_\frkp(\alp_i)$ on $V^{\scrn_n}$ by  
\[U_\frkp(\alp_i^{}) h=\sum_{u\in\scrn_n/\alp_i^{-1}\scrn_n\alp_i^{}}\pi(u\alp_i^{-1})h. \index{$U_\frkp(\bet_1)$}\]
Put $\bet_j=\frac{\alp_{n-j}}{\vpi}$ for $1\leq j\leq n-1$. 
We similarly define the operator $U_\frkp(\bet_j)$. 

\begin{proposition}\label{prop:81}
\begin{align*}
U_\frkp(\alp_i)h^{\ord}_\pi&=\frac{q^{\frac{i(n-i)}{2}}}{\prod_{l=1}^i\mu_l(\vpi)}h^{\ord}_\pi , &
U_\frkp(\bet_j)h^{\ord}_\pi&=\frac{q^{\frac{j(n-j)}{2}}}{\prod_{l=n-j+1}^n\mu_l(\vpi)^{-1}}h^{\ord}_\pi.  
\end{align*}
\end{proposition}

\begin{proof}
Let $g\in\GL_n(F)$ be such that $[U_\frkp(\alp_i) h^{\ord}_\pi](g)\neq 0$. 
There exists $u\in\scrn_n$ such that $h^{\ord}_\pi(gu\alp_i^{-1})\neq 0$. 
We have $gu\alp_i^{-1}\in B_nw_n\scrn_n$, and hence $g\in B_nw_n\scrn_n$. 
By the characterization of $h^{\ord}_\pi$ we see that $h^{\ord}_\pi$ is an eigenvector of $U_\frkp(\alp_i)$ with eigenvalue 
\[[U_\frkp(\alp_i)h^{\ord}_\pi](w_n)=\sum_{u\in\scrn_n/\alp_i^{-1}\scrn_n\alp_i^{}}h^{\ord}_\pi(w_n^{}u\alp_i^{-1})=h^{\ord}_\pi(w_n^{}\alp_i^{-1}), \]
from which we obtain the first formula. 
Since $U_\frkp(\bet_j)=U_\frkp(\alp_{n-j})\ome_\pi(\vpi)$, where $\ome_\pi=\prod_{l=1}^n\mu_l$ is the central character of $\pi$, the second formula follows from the first.  
\end{proof}

\begin{definition}\label{def:82}
Let $\sig$ be the irreducible generic quotient of $I(\mu,\nu)$. 
\[\frac{1}{\cale(\sig,\Ad,\addchar)}=L^\GL(1,\sig_u^{}\times\sig_u^\vee)\gam(1,\mu^{-1}\nu,\addchar)
\times\begin{cases}
\frac{1}{\zet_F(1)^2} &\text{if $c(\sig)=0$, }\\
\frac{q^{c(\sig)}}{\zet_F(1)} &\text{if $c(\sig)>0$ }
\end{cases}\]
(cf. Definition 6.5 of \cite{HY}). 
Here we write $\sig_u$ for the unramified component of the first non-zero spherical Bernstein-Zelevinsky derivative of $\pi$ (see Definition 1.3 of \cite{NM} for the precise definition).
\end{definition}

Define the additive character of $N_n$ by 
\[\addchar(u)=\addchar(u_{1,2}+u_{2,3}+\cdots+u_{n-1,n})\]
for $u=(u_{i,j})\in N_n$.  
We redefine $\pi$ as the irreducible generic constituent of $I(\mu_1,\mu_2,\dots,\mu_n)$. 
For $h\in I(\mu_1,\mu_2,\dots,\mu_n)$ we define $W_{\addchar}(h)\in\scrw_{\addchar}(\pi)$ by 
\[W_{\addchar}(g,h)=\int_{N_n}h\left(w_nug\right)\overline{\addchar(u)}\,\d u. \]
Put $W^\ord_\sig=W_{\addchar}(h^\ord_\sig)$. 
For positive integers $\ell$ we define $t_\ell\in\GL_2(F)$ by 
\beq
t_\ell=\begin{bmatrix} 0 & 1 \\ \vpi^\ell & 0 \end{bmatrix}. \label{tag:85}
\eeq
By analogy with the formulas in (\ref{tag:79}), we set
\[\calb_{\sig^\ord}^{[\ell]}=\frac{\zet_F(2)}{L^\GL(1,\sig\times\sig^\vee)}\La W^\ord_\sig,\sig^\vee(t_\ell)^{-1}W^{\ord \natural}_\sig\Ra. \]

The following result is Corollary 6.6 of \cite{HY}. 

\begin{lemma}\label{lem:82}
Let $\sig$ be the irreducible generic quotient of $I(\mu,\nu)$. 
We denote the exponent of the conductor of $\sig$ by $c(\sig)$. 
If $\ell$ is sufficiently large, then 
\[\frac{\calb_{\sig^\ord}^{[\ell]}}{\calb_\sig}=\frac{q^{\ell/2}\nu(\vpi)^\ell}{[\calk_0^{(2)}(\frkp^{c(\sig)}):\cali_0^{(2)}(\frkp^\ell)]}\cale(\sig,\Ad,\addchar)\mu(-1). \]
\end{lemma}


\subsection{Depleted vectors of representations of $\GL_3(F)$}\label{ssec:85}

Let $\mu,\nu,\rho$ be characters of $F^\times$. 
The space $V$ of $\pi=I(\nu,\rho,\mu)$ consists of functions $h:\GL_3(F)\to\CC$ which satisfy
\[h\left(\begin{bmatrix} a & y & z \\ 0 & b & x \\ 0 & 0 & c \end{bmatrix}g\right)=\nu(a)\rho(b)\mu(c)|ac^{-1}|h(g)\]
for $a,b,c\in F^\times$, $x,y,z\in F$ and $g\in\GL_3(F)$. 

Let $\calp_{1,2}=\{(g_{ij})\in\GL_3(F)\;|\;g_{21}=g_{31}=0\}$ be a maximal parabolic subgroup of $\GL_3(F)$. \index{$\calp_{1,2}$}
Given $h''\in I(\rho,\mu)$ and $\phi_1,\phi_3\in\cals(F)$, we define $h_{\phi_1,h'',\phi_3}^{\nu,\rho,\mu}\in I(\nu,\rho,\mu)$ by 
\[h_{\phi_1,h'',\phi_3}^{\nu,\rho,\mu}\left(\begin{bmatrix} a & \\ & g \end{bmatrix}\begin{bmatrix} & 1 \\ \ono_2 & \end{bmatrix}\begin{bmatrix} 1 & 0 & z \\ 0 & 1 & x \\ 0 & 0 & 1 \end{bmatrix}\right)
=\frac{|a|\nu(a)}{|\det g|^{1/2}}h''(g)\phi_1(x)\phi_3(z) \]
for $g\in\GL_2(F)$, $a\in F^\times$, $x,z\in F$, and by setting $h_{\phi_1,h'',\phi_3}^{\nu,\rho,\mu}(g)=0$ unless $g\in\calp_{1,2}w_3N_3$. 
Given $\phi_2\in\cals(F)$, we define $h_{\phi_2}^{\rho,\mu}\in I(\rho,\mu)$ by
\[h_{\phi_2}^{\rho,\mu}\left(\begin{bmatrix} a & x \\ 0 & d \end{bmatrix}\begin{bmatrix} 0 & 1 \\ 1 & 0 \end{bmatrix}\begin{bmatrix} 1 & z \\ 0 & 1 \end{bmatrix}\right)
=|ad^{-1}|^{1/2}\rho(a)\mu(d)\phi_2(z)\]
for $d\in F^\times$ and by setting $h_{\phi_2}^{\rho,\mu}(g)=0$ unless $g\in B_2w_2N_2$. 
Put  
\[h_{\phi_1,\phi_2,\phi_3}^{\nu,\rho,\mu}=h_{\phi_1,h_{\phi_2}^{\rho,\mu},\phi_3}^{\nu,\rho,\mu}. \]
Bear in mind that the ordinary vector $h^\mathrm{ord}_\pi\in I(\nu,\rho,\mu)$ with respect to $(\nu,\rho,\mu)$ is given by
\[h^\mathrm{ord}_\pi=h_{\II_\frko,\II_\frko,\II_\frko}^{\nu,\rho,\mu}. \]

\begin{remark}\label{rem:83}
Given $\phi\in\cals(F)$ and $t\in F^\times$, we define $\phi^t\in\cals(F)$ by $\phi^t(x)=\phi(tx)$ for $x\in F$. 
One can show that for $a,b,c\in F^\times$ 
\[\pi\left(\begin{bmatrix} a & 0 & 0 \\ 0 & b & 0 \\ 0 & 0 & c \end{bmatrix}\right)h_{\phi_1,\phi_2,\phi_3}^{\nu,\rho,\mu}=\nu(c)\rho(b)\mu(a)|ac^{-1}|h_{\phi_1^{c/b},\phi_2^{b/a},\phi_3^{c/a}}^{\nu,\rho,\mu}\]
(cf. Remark 2.3 of \cite{HY2}). 
In particular, $h^\mathrm{ord}_\pi$ and $h_{\II_\frko}^{\rho,\mu}$ are eigenvectors of $U_p$-operators (cf. Proposition 5.4 of \cite{HY}). 
The vector $h_{\II_\frko}^{\rho,\mu}$ is called the normalized $p$-stabilized new vector with respect to $\mu(\vpi)$. 
\end{remark}


\subsection{$p$-depletion}\label{ssec:86}

Define subgroups of $\scrn_3$ by 
\begin{align*}
\scrn_{1,2}&=\left\{\left.\begin{bmatrix} 1 & y & z \\ 0 & 1 & 0 \\ 0 & 0 & 1 \end{bmatrix}\right|\;y,z\in\frko\right\}, & 
\scrn_{2,1}&=\left\{\left.\begin{bmatrix} 1 & 0 & z \\ 0 & 1 & x \\ 0 & 0 & 1 \end{bmatrix}\right|\;x,z\in\frko\right\}.  
\end{align*}
Fix a prime element $\vpi$ of $\frko$. 
The space $\calc(\frko^\times,\CC)$ consists of $\CC$-valued locally constant functions on $\frko^\times$. 
Given $\breve\phi\in\calc(\frko^\times,\CC)$, we define $\bfU_{\frkp^c}^{\breve\phi} h^{\ord}_\pi\in V^{\scrn_{2,1}}$ by 
\[\bfU_{\frkp^c}^{\breve\phi} h^{\ord}_\pi=\frac{1}{q^j\mu(\vpi)^j}\sum_{y\in (\frko/\frkp^j)^\times}\sum_{z\in\frko/\frkp^j}\breve\phi(y)\pi\left(\begin{bmatrix} \vpi^j & y & z \\ 0 & 1 & 0 \\ 0 & 0 & 1 \end{bmatrix}\right)h^{\ord}_\pi, \index{$\bfU_{\frkp^c}^{\breve\phi}$}\] 
where $j\in\NN$ is such that $\breve\phi$ factors through the quotient $\frko^\times\to(\frko/\frkp^j)^\times$. 

\begin{lemma}\label{lem:83}
The definition of $\bfU_{\frkp^c}^{\breve\phi} h^{\ord}_\pi$ is independent of the choice of $j$. 
\end{lemma}

\begin{proof}
If $\breve\phi\in\calc(\frko^\times,\CC)$ factors through $\frko^\times\to(\frko/\frkp^j)^\times$, then 
\begin{align*}
&\sum_{y\in (\frko/\frkp^{j+1})^\times}\sum_{z\in\frko/\frkp^{j+1}}\breve\phi(y)\pi\left(\begin{bmatrix} \vpi^{j+1} & y & z \\ 0 & 1 & 0 \\ 0 & 0 & 1 \end{bmatrix}\right)h^{\ord}_\pi\\
=&\sum_{y\in (\frko/\frkp^j)^\times}\sum_{z\in \frko/\frkp^j}\breve\phi(y)\sum_{b,c\in \frko/\frkp}\pi\left(\begin{bmatrix} \vpi^j & y & z \\ 0 & 1 & 0 \\ 0 & 0 & 1 \end{bmatrix}\begin{bmatrix} \vpi & b & c \\ 0 & 1 & 0 \\ 0 & 0 & 1 \end{bmatrix}\right)h^{\ord}_\pi. 
\end{align*}
The inner sum is $q\mu(\vpi)\pi\left(\begin{bmatrix} \vpi^j & y & z \\ 0 & 1 & 0 \\ 0 & 0 & 1 \end{bmatrix}\right)h^{\ord}_\pi$ (cf. Remark \ref{rem:83}).  
\end{proof}

Let $\chi$ and $\chi'$ be characters of $\frko^\times$. 
We define an operator on $V^{\scrn_{1,2}}$ by 
\[\tht_{\frkp^c}^{\chi'} h=\sum_{y\in(\frko/\frkp)^\times}
\pi\left(\begin{bmatrix} 1 & \frac{y}{\vpi} & 0 \\ 0 & 1 & 0 \\ 0 & 0 & 1 \end{bmatrix}\right)h \]
if $c(\chi')=0$, and 
\[\tht_{\frkp^c}^{\chi'} h=\frkg(\chi^{\prime-1},\addchar)^{-1}\sum_{y\in(\frko/\frkp^{c(\chi')})^\times}
\chi'(y)\pi\left(\begin{bmatrix} 1 & \frac{y}{\vpi^{c(\chi')}} & 0 \\ 0 & 1 & 0 \\ 0 & 0 & 1 \end{bmatrix}\right)h \]
otherwise. 
For positive integers $\ell$ we define 
\[\calw_\ell h=q^{-\ell}\sum_{z\in\frko/\frkp^\ell}\pi\left(\begin{bmatrix} 1 & 0 & \frac{z}{\vpi^\ell} \\ 0 & 1 & 0 \\ 0 & 0 & 1 \end{bmatrix}\right)h. \]
Put $n'=\{1,c(\chi')\}$ and 
\[\vka_{\frkp^c}=\begin{bmatrix} \vpi & 0 & 0 \\ 0 & 1 & 0 \\ 0 & 0 & 1 \end{bmatrix}. \] 
Then we can write $\bfU_{\frkp^c}^{\chi'} h^{\ord}_\pi$ as 
\[\bfU_{\frkp^c}^{\chi'} h^{\ord}_\pi=\mu(\vpi)^{-n'}\frkg(\chi^{\prime-1},\addchar)\pi(\vka_{\frkp^c}^{n'})\calw_{n'}^{}\tht_{\frkp^c}^{\chi'} h^{\ord}_\pi. \]

Given $\phi\in\calc(\frko^\times,\CC)$, we define $\bdTht_\frkp^\phi:V^{\scrn_{2,1}}\to V$ by 
\[\bdTht_\frkp^\phi h=q^{-n}\sum_{x\in\frko/\frkp^n}\sum_{a\in (\frko/\frkp^n)^\times}\phi(a)\addchar\left(\frac{ax}{\vpi^n}\right)\pi\left(\begin{bmatrix} 1 & 0 & 0 \\ 0 & 1 & \frac{x}{\vpi^n}\\ 0 & 0 & 1 \end{bmatrix}\right)h, \index{$\bdTht_\frkp^\phi$}\]
provided that $\phi$ factors through the quotient $\frko^\times\to(\frko/\frkp^n)^\times$. 
Since 
\[\sum_{a\in (\frko/\frkp^{n+1})^\times}\phi(a)\addchar\left(\frac{ax}{\vpi^{n+1}}\right)=q\II_\frkp(x)\sum_{a\in (\frko/\frkp^n)^\times}\phi(a)\addchar\left(\frac{ax}{\vpi^{n+1}}\right), \]
the definition of $\bdTht_\frkp^\phi$ is independent of the choice of $n$. 
If $c(\chi)\geq 1$, then we define an operator on $\tht_\frkp^\chi:V^{\scrn_{2,1}}\to V$ by 
\[\tht_\frkp^\chi h=\frkg(\chi,\addchar)^{-1}\sum_{x\in(\frko/\frkp^{c(\chi)})^\times}
\chi(x)^{-1}\pi\left(\begin{bmatrix} 1 & 0 & 0 \\ 0 & 1 & \frac{x}{\vpi^{c(\chi)}} \\ 0 & 0 & 1 \end{bmatrix}\right)h,\]
and if $\chi$ is trivial, then 
\[\tht_\frkp^\chi h=h-q^{-1}\sum_{x\in\frko/\frkp}
\pi\left(\begin{bmatrix} 1 & 0 & 0 \\ 0 & 1 & \frac{x}{\vpi} \\ 0 & 0 & 1 \end{bmatrix}\right)h. \]

\begin{lemma}\label{lem:84}
$\bdTht_\frkp^\chi=\tht_\frkp^\chi$. 
\end{lemma} 

\begin{proof}
If $c(\chi)\geq 1$, then the relation follows from the equalities  
\begin{align*}
\frkg(\chi,\addchar)\frkg(\chi^{-1},\addchar)&=q^{c(\chi)}\chi(-1), \\ 
\frkg(\chi^{-1},\addchar^{-y})&=\II_{\frko^\times}(y)\chi(-y)^{-1}\frkg(\chi^{-1},\addchar)
\end{align*} 
for $y\in\frko$. 
One can easily prove the case when $c(\chi)=0$. 
\end{proof}

\begin{definition}[$p$-depletion]\label{def:83}
Put 
\begin{align*}
n&=\max\{1,c(\chi)\}, & 
n'&=\max\{1,c(\chi')\}, & 
c&=n+n'.  
\end{align*}
For $\ell\geq c$ we define the $p$-depletion of $h^{\ord}_\pi$ by 
\begin{align*}
\bfU_{p,\ell}^{\chi,\chi'}h^{\ord}_\pi
&=\calw_{\ell-n'}^{}\bdTht_\frkp^\chi\bfU_{\frkp^c}^{\chi'} h^\mathrm{ord} \\
&=\mu(\vpi)^{-n'}\frkg(\chi^{\prime-1},\addchar)\cdot \pi(\vka_{\frkp^c})^{n'}\calw_\ell\bdTht_\frkp^\chi\tht_{\frkp^c}^{\chi'} h^\mathrm{ord}. 
\end{align*}
\end{definition}


\subsection{A description of the $p$-depleted vector}\label{ssec:87}

Define $\vph_\chi^{},\vph'_{\chi'}\in\cals(F)$ by 
\begin{align*}
\vph_\chi(x)&=\chi(x)\II_{\frko^\times}(x), &
\vph'_{\chi'}&=\begin{cases}
\vph_{\chi^{\prime-1}} &\text{if $c(\chi')\geq 1$, }\\
q\II_\frkp-\II_\frko &\text{if $c(\chi')=0$ }
\end{cases}
\end{align*}
for characters $\chi$ and $\chi'$ of $\frko^\times$. 

\begin{proposition}\label{prop:82}
If $\ell\geq n+n'$, then 
\begin{align*}
\bfU_{p,\ell}^{\chi,\chi'}h^{\ord}_\pi&=\mu(\vpi)^{-n'}\frkg(\chi^{\prime-1},\addchar)\cdot \pi(\vka_{\frkp^c}^{n'})h^\ddagger, & 
h^\ddagger&=q^{-\ell}h_{\widehat{\varphi_\chi},\widehat{\varphi_{\chi'}'},\II_{\frkp^{-\ell}}}^{\nu,\rho,\mu}. \index{$h^\ddagger$}
\end{align*}
\end{proposition}

Define subspaces $\cals(\frko)$ and $\cals(F/\frko)$ of $\cals(F)$ by 
\begin{align*}
\cals(\frko)&=\{\phi\in\cals(F)\;|\;\phi(x)=0\text{ for }x\notin\frko\}, \\
\cals(F/\frko)&=\{\phi\in\cals(F)\;|\;\phi(x+c)=0\text{ for }x\in F\text{ and }c\in\frko\}. 
\end{align*} 
The operator $\tht_\frkp^\chi:\cals(F/\frko)\to\cals(F/\frko)$ is defined by 
\[\tht_\frkp^\chi\phi(x)=\phi(x)-q^{-1}\sum_{a\in \frko/\frkp}\phi(x+a\vpi^{-1}) \]
if $c(\chi)=0$, and by 
\[\tht_\frkp^\chi\phi(x)=\frkg(\chi,\addchar)^{-1}\sum_{a\in(\frko/\frkp^{c(\chi)})^\times}\chi(a)^{-1}\phi(x+a\vpi^{-c(\chi)}) \]
if $c(\chi)\geq 1$. 
We define the operator $\tht_{\frkp^c}^{\chi'}:\cals(F/\frko)\to\cals(F/\frko)$ by 
\[\tht_{\frkp^c}^{\chi'}\phi(x)=\sum_{a\in (\frko/\frkp)^\times}\phi(x+a\vpi^{-1}) \]
if $c(\chi')=0$, and by $\tht_{\frkp^c}^{\chi'}=\tht^{\chi^{\prime-1}}_\frkp$ if $c(\chi')\geq 1$. 
We shall prove the following lemma to prove Proposition \ref{prop:82}. 

\begin{lemma}\label{lem:85}
Let $\chi,\chi'$ be characters of $\frko^\times$. 
Let $\phi\in\cals(F/\frko)$. 
Then for $y\in F$ 
\begin{align*}
\widehat{\tht_\frkp^\chi\phi}(y)&=\varphi_\chi(-y)\cdot\widehat{\phi}(y), &
\widehat{\tht_{\frkp^c}^{\chi'}\phi}(y)&=\vph'_{\chi'}(-y)\cdot\widehat{\phi}(y). 
\end{align*}
\end{lemma}

\begin{proof}
Note that $\phi\in\cals(F/\frko)\Leftrightarrow\widehat{\phi}\in\cals(\frko)$. 
Observe that if $c(\chi)=0$, then 
\[\widehat{\tht_\frkp^\chi\phi}(y)=\widehat{\phi}(y)-q^{-1}\sum_{a\in \frko/\frkp}\addchar(ay\vpi^{-1})\widehat{\phi}(y)=\II_{\frko^\times}(y)\widehat{\phi}(y), \]
while if $c(\chi)\geq 1$, then 
\[\widehat{\tht_\frkp^\chi\phi}(y)
=\frkg(\chi,\addchar)^{-1}\sum_{a\in (\frko/\frkp^{c(\chi)})^\times}\chi(a)^{-1}\addchar(ay\vpi^{-c(\chi)})\widehat{\phi}(y)
=\chi(-y)\II_{\frko^\times}(y)\widehat{\phi}(y). \]
Thus we obtain the first identity. 
If $c(\chi')=0$, then 
\[\widehat{\tht_{\frkp^c}^{\chi'}\phi}(y)
=\sum_{a\in (\frko/\frkp)^\times}\addchar(ay\vpi^{-1})\widehat{\phi}(y)
=\vph_{\chi'}'(-y)\widehat{\phi}(y), \]
which proves the second identity. 
\end{proof} 

\begin{lemma}\label{lem:86}
Let $\phi_1\in\cals(F/\frko)$, $\phi_3\in\cals(F)$ and $h''\in I(\rho,\mu)$. 
\begin{enumerate}
\item\label{lem:861} Take $m$ so that $\phi_1(x)=0$ unless $x\in\frkp^{-m}$. 
If $\ell\geq n'+m$, then 
\[\calw_\ell\tht_{\frkp^c}^{\chi'} h^{\nu,\rho,\mu}_{\phi_1,h'',\II_\frko}=q^{-\ell}h^{\nu,\rho,\mu}_{\phi_1,\tht_{\frkp^c}^{\chi'}h'',\II_{\frkp^{-\ell}}}, \]
where $\tht_{\frkp^c}^{\chi'} h''(g)=\displaystyle\sum_{y\in(\frko/\frkp)^\times}
h''\left(g\begin{bmatrix} 1 & \frac{y}{\vpi} \\ 0 & 1 \end{bmatrix}\right)$ if $c(\chi')=0$, and 
\[\tht_{\frkp^c}^{\chi'} h=\frkg(\chi^{\prime-1},\addchar)^{-1}\sum_{y\in(\frko/\frkp^{c(\chi')})^\times}
\chi'(y)h''\left(g\begin{bmatrix} 1 & \frac{y}{\vpi^{c(\chi')}} \\ 0 & 1 \end{bmatrix}\right) \]
otherwise. 
\item\label{lem:862} $\bdTht_\frkp^\chi h^{\nu,\rho,\mu}_{\phi_1,h'',\phi_3}=h^{\nu,\rho,\mu}_{\tht_\frkp^\chi\phi_1,h'',\phi_3}$. 
\end{enumerate}
\end{lemma}

\begin{remark}\label{rem:84} 
If $c(\chi')\geq 1$, then $h''\mapsto\tht_{\frkp^c}^{\chi'} h''$ coincides with the twisting operator associated to $\chi^{\prime-1}$ (cf. (2.12) of \cite{MH}). 
Clearly, $\tht_{\frkp^c}^{\chi'} h_{\phi_2}^{\rho,\mu}=h_{\tht_{\frkp^c}^{\chi'}\phi_2}^{\rho,\mu}$. 
\end{remark}

\begin{proof}
It suffices to show the equality as a function on the big cell because it is an open dense subset of $G$.   
Observe that 
\[\begin{bmatrix} 1 & 0 & z \\ 0 & 1 & x \\ 0 & 0 & 1 \end{bmatrix}\begin{bmatrix} 1 & \frac{b}{\vpi^{n'}} & \frac{c}{\vpi^\ell} \\ 0 & 1 & 0 \\ 0 & 0 & 1 \end{bmatrix}
=\begin{bmatrix} 1 & \frac{b}{\vpi^{n'}} & 0 \\ 0 & 1 & 0 \\ 0 & 0 & 1 \end{bmatrix}
\begin{bmatrix} 1 & 0 & z-\frac{bx}{\vpi^{n'}}+\frac{c}{\vpi^\ell} \\ 0 & 1 & x \\ 0 & 0 & 1 \end{bmatrix}. \]
Therefore, if $\ell\geq n'+m$ and $b\in\frko$, then we have
\begin{align*}
&\sum_{c\in\frko/\frkp^\ell}\pi\left(\begin{bmatrix} 1 & \frac{b}{\vpi^{n'}} & \frac{c}{\vpi^\ell} \\ 0 & 1 & 0 \\ 0 & 0 & 1 \end{bmatrix}\right)h^{\nu,\rho,\mu}_{\phi_1,h'',\II_\frko}\left(\begin{bmatrix} 1 & \\ & g \end{bmatrix}\begin{bmatrix} & 1 \\ \ono_2 & \end{bmatrix}\begin{bmatrix} 1 & 0 & z \\ 0 & 1 & x \\ 0 & 0 & 1 \end{bmatrix}\right)\\
=&\sum_{c\in\frko/\frkp^\ell}\frac{\phi_1(x)}{|\det g|^{1/2}}h''\biggl(g\begin{bmatrix} 1 & \frac{b}{\vpi^{n'}} \\ 0 & 1 \end{bmatrix}\biggl)\II_\frko\biggl(z-\frac{bx}{\vpi^{n'}}+\frac{c}{\vpi^\ell}\biggl)\\
=&\frac{\phi_1(x)}{|\det g|^{1/2}}h''\biggl(g\begin{bmatrix} 1 & \frac{b}{\vpi^{n'}} \\ 0 & 1 \end{bmatrix}\biggl)\II_{\frkp^{-\ell}}(z), 
\end{align*}
from which (\ref{lem:861}) follows. 
The proof for (\ref{lem:862}) is easier by Lemma \ref{lem:84}. 
\end{proof}

We are now ready to prove Proposition \ref{prop:82}. 
By the definition of $h_{\phi_1,\phi_2,\II_\frko}^{\nu,\rho,\mu}$ Lemma \ref{lem:86} gives 
\begin{align*}
\calw_\ell\bdTht_\frkp^\chi\tht_{\frkp^c}^{\chi'} h^{\nu,\rho,\mu}_{\phi_1,\phi_2,\II_\frko}
&=\bdTht_\frkp^\chi\calw_\ell\tht_{\frkp^c}^{\chi'} h^{\nu,\rho,\mu}_{\phi_1,h_{\phi_2}^{\rho,\mu},\II_\frko}\\
&=q^{-\ell}\bdTht_\frkp^\chi h^{\nu,\rho,\mu}_{\phi_1,h_{\tht_{\frkp^c}^{\chi'}\phi_2}^{\rho,\mu},\II_{\frkp^{-\ell}}}
=q^{-\ell}h^{\nu,\rho,\mu}_{\tht_\frkp^\chi\phi_1,\tht_{\frkp^c}^{\chi'}\phi_2,\II_{\frkp^{-\ell}}}.  
\end{align*}  
Since $\tht_\frkp^\chi\II_\frko=\widehat{\varphi_\chi}$ and $\tht_{\frkp^c}^{\chi'}\II_\frko=\widehat{\varphi'_{\chi'}}$ by Lemma \ref{lem:85}, we get the stated relation. 


\subsection{The $K$-type of $\bfU_{p,\ell}^{\chi,\chi'}h^{\ord}_\pi$}\label{ssec:88}

For a positive integer $\ell$ we define an open subgroup of $\GL_3(\frko)$ by 
\[\cali^{(3)}_0(\frkp^\ell)=\begin{bmatrix} \frko^\times & \frko & \frko \\ \frkp^\ell & \frko^\times & \frko \\ \frkp^\ell & \frkp^\ell & \frko^\times \end{bmatrix}. \index{$\cali_0^{(n)}(\frkp^\ell)$}\]
Define the characters $(\mu,\rho,\nu)_0$ of $\cali^{(3)}_0(\frkp^\ell)$ and $(\mu,\rho)_0$ of $\calk^{(2)}_0(\frkp^\ell)$ by 
\begin{align*}
(\mu,\rho,\nu)_0\left(\begin{bmatrix} a & y & z \\ v & b & x \\ w & u & c \end{bmatrix}\right)&=\mu(a)\rho(b)\nu(c), &
(\mu,\rho)_0\left(\begin{bmatrix} a & y \\ v & b \end{bmatrix}\right)&=\mu(a)\rho(b).  
\end{align*}

\begin{proposition}\label{prop:83}
Put $\phi_3=q^{-\ell}\II_{\frkp^{-\ell}}$. 
If $\ell\geq\max\{n,c(\nu)\}$ and $h''\in I(\rho,\mu)$ satisfies $\pi'(k')h''=(\chi_1,\chi_2)_0(k')h''$ for $k'\in\calk_0^{(2)}(\frkp^{2\ell})$, then for $k\in\cali_0^{(3)}(\frkp^{3\ell})$
\[\pi(k)h_{\widehat{\varphi_\chi},h'',\phi_3}^{\nu,\rho,\mu}=(\chi_1,\chi_2\chi,\nu\chi^{-1})_0(k)h_{\widehat{\varphi_\chi},h'',\phi_3}^{\nu,\rho,\mu}. \]
\end{proposition}

\begin{proof}
The proof is similar to that of \cite[Lemma 2.5]{HY}. 
We write 
\begin{align*}
k&=\frku^-(y)\begin{bmatrix} k' & \\ & d \end{bmatrix}\frku(b), &
\frku(b)&=\begin{bmatrix} \ono_2 & b \\ & 1 \end{bmatrix}, &
\frku^-(y)&=\begin{bmatrix} \ono_2 & \\ \trs y & 1 \end{bmatrix}
\end{align*}
for $b\in\frko^{\oplus 2}$, $\trs y\in(\frkp^{3\ell})^{\oplus 2}$, $d\in\frko^\times$ and $k'\in\calk_0^{(2)}(\frkp^{3\ell})$. 
Define $\Psi\in\cals(F^2)$ by $\Psi\Big(\begin{matrix} x_1 \\ x_2 \end{matrix}\Big)=\phi_3(x_1)\widehat{\vph_\chi}(x_2)$. 
Since $\Psi$ is invariant under the translation by elements of $\frko^{\oplus 2}$, we have 
\[\pi(\frku(b))h_{\widehat{\varphi_\chi},h'',\phi_3}^{\nu,\rho,\mu}=h_{\widehat{\varphi_\chi},h'',\phi_3}^{\nu,\rho,\mu}. \]
Put $\pi'=I(\rho,\mu)$. 
Since $\Psi(dk^{\prime-1}x)=\chi(d^{-1}k'_{2,2})\Psi(x)$ by (2.11) of \cite{HY},   
\[\pi\biggl(\begin{bmatrix} k' & \\ & d \end{bmatrix}\biggl)h_{\widehat{\varphi_\chi},h'',\phi_3}^{\nu,\rho,\mu}\biggl(\begin{bmatrix} 1 & \\ & g \end{bmatrix}\begin{bmatrix} & 1 \\ \ono_2 & \end{bmatrix}\frku(x)\biggl)
=\frac{\nu(d)\chi(k'_{2,2})\Psi(x)}{\chi(d)|\det g|^{1/2}}\pi'(k')h''(g) \]
by definition. 
We conclude that 
\[\pi\biggl(\begin{bmatrix} k' & \\ & d \end{bmatrix}\biggl)h_{\widehat{\varphi_\chi},h'',\phi_3}^{\nu,\rho,\mu}=(\chi_1,\chi_2\chi,\nu\chi^{-1})_0\biggl(\begin{bmatrix} k' & \\ & d \end{bmatrix}\biggl)h_{\widehat{\varphi_\chi},h'',\phi_3}^{\nu,\rho,\mu}. \]
If $h_{\widehat{\varphi_\chi},h'',\phi_3}^{\nu,\rho,\mu}(g\frku^-(y))\neq 0$, then there is an element $x\in(\frkp^{-\ell})^{\oplus 2}$ such that $g\frku^-(y)\in \calp_{1,2}\begin{bmatrix} & 1 \\ \ono_2 & \end{bmatrix}\frku(x)$. 
Since 
\begin{align*}
\frku(x)\frku^-(-y)
&=\begin{bmatrix} \ono_2-x\trs y & x \\ -\trs y & 1 \end{bmatrix}\\
&=\begin{bmatrix} \ono_2-x\trs y & 0 \\ -\trs y & 1+\trs y(\ono_2-x\trs y)^{-1}x \end{bmatrix}
\begin{bmatrix} \ono_2 & (\ono_2-x\trs y)^{-1}x \\ 0 & 1 \end{bmatrix}, 
\end{align*} 
we can write $g=\begin{bmatrix} 1 & \\ & g' \end{bmatrix}\begin{bmatrix} & 1 \\ \ono_2 & \end{bmatrix}\frku(x)$ with $x\in(\frkp^{-\ell})^{\oplus 2}$ and $g'\in\GL_2(F)$. 
Since $\ono_2+x\trs y\in\calk^{(2)}_0(\frkp^{2\ell})$, we see that 
\begin{align*}
&h_{\widehat{\varphi_\chi},h'',\phi_3}^{\nu,\rho,\mu}\biggl(\begin{bmatrix} 1 & \\ & g' \end{bmatrix}\begin{bmatrix} & 1 \\ \ono_2 & \end{bmatrix}\frku(x)\frku^-(y)\biggl)\\
=&\frac{\nu(1-\trs y(\ono_2+x\trs y)^{-1}x)}{|\det g'|^{1/2}}\pi'(\ono_2+x\trs y)h''(g')\Psi((\ono_2+x\trs y)^{-1}x)\\
=&|\det g'|^{-1/2}h''(g')\Psi(x)
=h_{\widehat{\varphi_\chi},h'',\phi_3}^{\nu,\rho,\mu}\biggl(\begin{bmatrix} 1 & \\ & g' \end{bmatrix}\begin{bmatrix} & 1 \\ \ono_2 & \end{bmatrix}\frku(x)\biggl)
\end{align*} 
by the assumption and definition above. 
\end{proof}

\begin{corollary}\label{cor:81}
Put $\ell=\max\{n,n',c(\nu),c(\rho),c(\mu)\}$. 
If $m\geq 3\ell$, then 
\[\pi(k)\bfU_{p,\ell}^{\chi,\chi'}h^{\ord}_\pi=(\mu\chi^{\prime-1},\rho\chi\chi',\nu\chi^{-1})_0(k)\bfU_{p,\ell}^{\chi,\chi'}h^{\ord}_\pi\] 
for every $k\in\cali^{(3)}_0(\frkp^m)$. 
\end{corollary}

\begin{proof}
One can easily show that for every $k'\in\calk_0^{(2)}(\frkp^{2\ell})$
\[\pi'(k')h_{\widehat{\vph'_{\chi'}}}^{\rho,\mu}
=(\mu\chi^{\prime-1},\rho\chi')_0(k')h_{\widehat{\vph'_{\chi'}}}^{\rho,\mu}\]
 (see Remark \ref{rem:83}).  
Proposition \ref{prop:83} applied to $h''=h_{\widehat{\vph'_{\chi'}}}^{\rho,\mu}$ gives 
\[\pi(k)h^\ddagger=(\mu\chi^{\prime-1},\rho\chi\chi',\nu\chi^{-1})_0(k)h^\ddagger\]
for $k\in\cali_0^{(3)}(\frkp^{3\ell})$, which is equivalent to Corollary \ref{cor:81} in view of Proposition \ref{prop:82}. 
\end{proof}


\subsection{The zeta integral of $\bfU_{p,\ell}^{\chi,\chi'}h^{\ord}_\pi$}\label{ssec:89}

Let $\pi$ be the irreducible generic constituent of the principal series $I(\nu,\rho,\mu)$ and $\sig$ that of $I(\mu',\nu')$. 
Put $\ome_\sig=\mu'\nu'$. 
The element $t_\ell\in\GL_2(F)$ is defined in (\ref{tag:85}). 
Recall that $W_{\addchar}(h)\in\scrw_{\addchar}(\pi)$ is associated to $h\in I(\nu,\rho,\mu)$ by 
\[W_{\addchar}(g,h)=\int_{F^3}h\left(w_3\begin{bmatrix} 1 & y & z \\ 0 & 1 & x \\ 0 & 0 & 1 \end{bmatrix}g\right)\overline{\addchar(x+y)}\,\d x\d y\d z. \]
Similarly, $W_{\addchar}(h')\in\scrw_{\addchar}(\sig)$ is defined for $h'\in I(\mu',\nu')$. 

\begin{theorem}\label{thm:81}
Let $\chi=(\nu\mu^{\prime-1})|_{\frko^\times}$ and $\chi'=(\mu\nu^{\prime-1})|_{\frko^\times}$. 
Put 
\begin{align*}
&W^\dagger=W_{\addchar}(\bdTht_\frkp^\chi\bfU_{\frkp^c}^{\chi'} h^{\ord}_\pi), & 
&h^\ord_\sig=h_{\II_\frko}^{\mu',\nu'}, & 
&W^\ord_\sig=W_{\addchar}(h^\ord_\sig), & 
&W^\flat_{\sig^\vee}=(W^\ord_\sig)^\natural. 
\end{align*}
If $\ell$ is sufficiently large, then  
\[Z\left(\frac{1}{2},\pi(\vsi)W^\dagger,\sig^\vee(t_\ell)^{-1}W^\flat_{\sig^\vee}\right)^2
=\Biggl(\frac{\zet_F(2)}{q^{\ell/2}\zet_F(1)}\nu'(\vpi)^\ell\frac{\gam\bigl(\frac{1}{2},\mu^{-1}\nu',\addchar^{-1}\bigl)}{\gam\bigl(\frac{1}{2},\pi^{}\otimes\mu^{\prime-1},\addchar\bigl)}\Biggl)^2. \]
\end{theorem}


\subsection{Proof of Theorem \ref{thm:81}}\label{ssec:810}

Since $t_\ell^{}\bfu(x)t_\ell^{-1}=\begin{bmatrix} 1 & 0 \\ \vpi^\ell x & 1 \end{bmatrix}$, we have 
\[Z(s,\pi(\vsi)W,\sig^\vee(t_\ell)^{-1}W')=Z(s,\pi(\vsi)\calw_mW,\sig^\vee(t_\ell)^{-1}W')\]
for $m\ll\ell$ and $W\in\scrw_{\addchar}(\pi)$. 
Put $W^\ddagger=W_{\addchar}(h^\ddagger)$. 
Then 
\[\calw_{\ell-n'}W^\dagger=\mu(\vpi)^{-n'}\frkg(\chi^{\prime-1},\addchar)\cdot \pi(\vka_{\frkp^c}^{n'})W^\ddagger\]
by Proposition \ref{prop:82}, and hence  
\[Z(s,\pi(\vsi)W^\dagger,\sig^\vee(t_\ell)^{-1}W')=\mu(\vpi)^{-n'}\frkg(\chi^{\prime-1},\addchar)Z(s,\pi(\vsi\vka_{\frkp^c}^{n'})W^\ddagger,\sig^\vee(t_\ell)^{-1}W'). \]

Observe that 
\[\left[\begin{array}{c|c} t_\ell & \\ \hline & 1 \end{array}\right]\vsi\vka_{\frkp^c}^{n'}=\left[\begin{array}{c|c} \vpi^{\ell+n'}\ono_2 & \\ \hline & 1 \end{array}\right]\calj_{\ell+n'}, \]
where 
\[\calj_\ell
=\left[\begin{array}{cc|c} 0 & 1 & 0 \\ \vpi^\ell & 0 & 0 \\ \hline 0 & 0 & 1 \end{array}\right]^{-1}\vsi
=\begin{bmatrix} 0 & 0 & \frac{1}{\vpi^{\ell}} \\ 1 & 0 & 0 \\ 0 & 1 & 0 \end{bmatrix}. \]
The invariance of the JPSS integral gives 
\[Z(s,\pi(\vsi\vka_{\frkp^c}^{n'})W^\ddagger,\sig^\vee(t_\ell)^{-1}W')=\ome_\sig(\vpi^{\ell+n'})Z(s,\pi(\calj_{\ell+n'})W^\ddagger,W'). \]
We conclude that 
\[Z(s,\pi(\vsi)W^\dagger,\sig^\vee(t_\ell)^{-1}W')=\frac{\ome_\sig(\vpi^{\ell+n'})}{\mu(\vpi)^{n'}}\frkg(\chi^{\prime-1},\addchar)Z(s,\pi(\calj_{\ell+n'})W^\ddagger,W'). \]

We deduce Theorem \ref{thm:81} from Remark \ref{rem:81} and the following lemma:  

\begin{lemma}\label{lem:87}
Notation being as in Theorem \ref{thm:81}, we have 
\[\frac{Z\bigl(\frac{1}{2},\pi(\calj_{\ell+n'})W^\ddagger,W^\flat_{\sig^\vee}\bigl)^2}{(\mu\ome_\sig^{-1})(\vpi)^{2n'}\frac{\zet_F(2)^2}{\zet_F(1)^2}}
=\frac{q^{-\ell-c(\chi')}\mu'(\vpi^{-2\ell})L\bigl(\frac{1}{2},\mu\nu^{\prime-1}\bigl)^2}{(\mu\nu^{\prime-1})(\vpi)^{2c(\chi')}\gam\bigl(\frac{1}{2},\pi^{}\otimes\mu^{\prime-1},\addchar\bigl)^2L\bigl(\frac{1}{2},\mu^{-1}\nu'\bigl)^2}. \]
\end{lemma}


\subsection{Proof of Lemma \ref{lem:87}}\label{ssec:811}
Recall the local zeta integrals associated to Whittaker functions $W\in \scrw_{\addchar}(\pi)$ and a character $\mu'$ of $F^\times$ 
\begin{align*}
Z(s,W,\mu')&=\int_{F^\times}W\biggl(\begin{bmatrix} a & \\ & \ono_2 \end{bmatrix}\biggl)\mu'(a)|a|^{s-1}\,\d a, \\
\widetilde{Z}(s,W,\mu^{\prime-1})&=\int_{F^\times}\int_{F} W\left(\begin{bmatrix} a & & \\ x & 1 & \\ & & 1 \end{bmatrix}\right)\mu'(a)^{-1}|a|^{s-1}\,\d x\d a, 
\end{align*}
which converge absolutely for $\Re s\gg 0$. 
Fix $h''\in I(\rho,\mu)$. 
Put 
\begin{align*}
h&=h_{\phi_1,h'',\phi_3}^{\nu,\rho,\mu}, & 
h_\sig^\ord&=h^{\mu',\nu'}_{\II_\frko}, &
W''&=W_{\addchar}(h''), &
W&=\pi(\calj_\ell)W_{\addchar}(h).  
\end{align*}

Substituting the integral expression of $W^\flat_{\sig^\vee}=(W_{\addchar}(h^\ord_\sig))^\natural$, we get
\[Z(s,W,W^\flat_{\sig^\vee})=\int_{\GL_2(F)}W\biggl(\begin{bmatrix} g & \\ & 1 \end{bmatrix}\biggl)h^\ord_\sig\biggl(\begin{bmatrix} 0 & 1 \\ 1 & 0 \end{bmatrix}g^\natural\biggl)|\det g|^{s-\frac{1}{2}}\,\d g. \]
Recall the integration formula 
\[\int_{\GL_2(F)}\calf(g)\,\d g=\frac{\zet_F(2)}{\zet_F(1)}\int_{F^{\times 2}\times F^2}\calf\biggl(\begin{bmatrix} a & \\ x & b \end{bmatrix}\bfu(y)\biggl)\,\d y\frac{\d x\d a\d b}{|b|} \]
for an integrable function $\calf$ on $\GL_2(F)$ (cf. \cite[3.1.6, p. 206]{MV}). 
Since $g^\natural=w_2\trs g^{-1}w_2$, we see that $\frac{\zet_F(1)}{\zet_F(2)}Z(s,W,W^\flat_{\sig^\vee})$ equals 
\begin{align*}
&\int_{F^{\times 2}\times F^2}
W\left(\begin{bmatrix} \begin{bmatrix} a & \\ x & b \end{bmatrix}\bfu(y) & \\ & 1 \end{bmatrix}\right)\biggl|\frac{b}{a}\biggl|^{1/2}\II_\frko(y)\frac{|ab|^{s-\frac{1}{2}}}{\mu'(a)\nu'(b)}\,\d y\frac{\d x\d a\d b}{|b|}. 
\end{align*}
Since $\pi\biggl(\begin{bmatrix} \bfu(y) & \\ & 1\end{bmatrix}\biggl)W=W$ for $y\in\frko$ for sufficiently large $\ell$, we get 
\begin{align*}
\frac{\zet_F(1)}{\zet_F(2)}Z(s,W,W^\flat_{\sig^\vee})
=&\int_{F^{\times 2}\times F}
W\left(\left[\begin{array}{c|cc}  a & & \\ \hline  x & b & \\ & & 1 \end{array}\right]\right)\frac{|ab|^{s-1}}{\nu'(b)\mu'(a)}\,\d x\d a\d b\\
=&\int_{F^\times}\widetilde{Z}(s,W_b,\mu^{\prime-1})\nu'(b)^{-1}|b|^{s-1}\d b, 
\end{align*}
where 
\[W_b=\pi\left(\left[\begin{array}{c|cc} 1 & & \\ \hline  & b & \\ & & 1 \end{array}\right]\right)W. \]
By the functional equation (\ref{tag:84}) we get 
\[\widetilde{Z}(s,W_b,\mu^{\prime-1})
=\gam(s,\pi^{}\otimes\mu^{\prime-1},\addchar)^{-1}Z(1-s,\pi^\vee(\vsi)\widetilde{W}_b,\mu'). \]
In summary we have  
\[Z(s,W,W^\flat_{\sig^\vee})
=\frac{\frac{\zet_F(2)}{\zet_F(1)}}{\gam(s,\pi^{}\otimes\mu^{\prime-1},\addchar)}\int_{F^\times}Z(1-s,\pi^\vee(\vsi)\widetilde{W}_b,\mu')\frac{|b|^{s-1}}{\nu'(b)}\,\d b. \]

Observe that 
\begin{align*}
\pi^\vee(\vsi)\widetilde{W}_b(g)
=&\widetilde{W}_b(g\vsi)\\
=&W_b(w_3\trs g^{-1}\vsi)\\
=&W_{\addchar}\left(w_3\trs g^{-1}\vsi\begin{bmatrix}  1 & 0 & 0 \\  0 & b & 0 \\ 0 & 0 & 1 \end{bmatrix}\calj_\ell,h\right). 
\end{align*}
We have 
\[\pi^\vee(\vsi)\widetilde{W}_b\left(\begin{bmatrix} a & 0 & 0 \\ 0 & 1 & 0 \\ 0 & 0 & 1 \end{bmatrix}\right)
=W_{\addchar}\left(\left[\begin{array}{cc|c} b & 0 & 0 \\ 0 & 1 & 0 \\ \hline 0 & 0 & -(\vpi^\ell a)^{-1} \end{array}\right],h\right). \]
Now we need the following formula: 

\begin{lemma}\label{lem:88}
Let $a,b\in F^\times$. 
Then  
\[W_{\addchar}\left(\left[\begin{array}{cc|c} b & 0 & 0 \\ 0 & 1 & 0 \\\hline 0 & 0 & a^{-1} \end{array}\right],h\right)\\
=\widehat{\phi_3}(0)|b|^{1/2}W''\biggl(\begin{bmatrix} b & 0 \\ 0 & 1 \end{bmatrix}\biggl)\cdot |a|\nu(a)^{-1}\widehat{\phi_1}(a). \]
\end{lemma}

\begin{proof}
Since  
\[w_3\begin{bmatrix} 1 & y & z+xy \\ 0 & 1 & x \\ 0 & 0 & 1 \end{bmatrix}\begin{bmatrix} b & 0 & 0 \\ 0 & 1 & 0 \\ 0 & 0 & a \end{bmatrix}
=\begin{bmatrix} a & \\ & w_2\bfu(y)\begin{bmatrix} b & 0 \\ 0 & 1 \end{bmatrix} \end{bmatrix}\begin{bmatrix} & 1 \\ \ono_2 & \end{bmatrix}\begin{bmatrix} 1 & 0 & \frac{az}{b} \\ 0 & 1 & ax \\ 0 & 0 & 1 \end{bmatrix} \]
for $x,y,z\in F$, we get 
\begin{align*}
&W_{\addchar}\left(\begin{bmatrix} b & 0 & 0 \\ 0 & 1 & 0 \\ 0 & 0 & a \end{bmatrix},h\right)\\
=&\int_{F^3}h\left(\begin{bmatrix} a & \\ & w_2\bfu(y)\begin{bmatrix} b & 0 \\ 0 & 1 \end{bmatrix} \end{bmatrix}\begin{bmatrix} & 1 \\ \ono_2 & \end{bmatrix}\begin{bmatrix} 1 & 0 & \frac{a}{b}z \\ 0 & 1 & ax \\ 0 & 0 & 1 \end{bmatrix}\right)\overline{\addchar(x+y)}\,\d x\d y\d z\\
=&\int_{F^3}\frac{|a|\nu(a)}{|b|^{1/2}}h''\biggl(w_2\bfu(y)\begin{bmatrix} b & 0 \\ 0 & 1 \end{bmatrix}\biggl)\phi_1(ax)\phi_3\biggl(\frac{a}{b}z\biggl)\overline{\addchar(x+y)}\,\d x\d y\d z
\\
=&\int_{F^3}\frac{|b|^{1/2}\nu(a)}{|a|}h''\biggl(w_2\bfu(y)\begin{bmatrix} b & 0 \\ 0 & 1 \end{bmatrix}\biggl)\phi_1(x)\phi_3(z)\overline{\addchar\biggl(\frac{x}{a}+y\biggl)}\,\d x\d y\d z\\
=&\frac{|b|^{1/2}\nu(a)}{|a|}\widehat{\phi_1}(a^{-1})\widehat{\phi_3}(0)W''\biggl(\begin{bmatrix} b & 0 \\ 0 & 1 \end{bmatrix}\biggl). 
\end{align*} 
We obtain the stated formula by replacing $a$ by $a^{-1}$. 
\end{proof}

For $\phi\in\cals(F)$ and a character $\chi$ of $F^\times$ we define Tate's local integral by 
\[Z(s,\phi,\chi)=\int_{F^\times}\phi(a)\chi(a)|a|^s\,\d a. \] 

\begin{proposition}\label{prop:84}
$\frac{\zet_F(1)}{\zet_F(2)}Z(s,\pi(\calj_\ell)W_{\addchar}(h),W^\flat_{\sig^\vee})$ equals
\[\frac{\widehat{\phi_3}(0)q^{-\ell s}\mu'(-\vpi^{-\ell})}{\gam(s,\pi\otimes\mu^{\prime-1},\addchar)}Z(1-s,\widehat{\phi_1},\nu^{-1}\mu')Z(s,W'',\nu^{\prime-1}). \]
\end{proposition}

\begin{proof}
By Lemma \ref{lem:88} we get 
\begin{align*}
&Z(1-s,\pi^\vee(\vsi)\widetilde{W}_b,\mu')\\
=&\widehat{\phi_3}(0)|b|^{1/2}W''\biggl(\begin{bmatrix} b & 0 \\ 0 & 1 \end{bmatrix}\biggl)\int_{F^\times}|-\vpi^\ell a|\nu(-\vpi^\ell a)^{-1}\widehat{\phi_1}(-\vpi^\ell a)\mu'(a)\,\frac{\d a}{|a|^s}\\
=&\widehat{\phi_3}(0)|b|^{1/2}W''\biggl(\begin{bmatrix} b & 0 \\ 0 & 1 \end{bmatrix}\biggl)\int_{F^\times}|a|\nu(a)^{-1}\widehat{\phi_1}(a)\frac{\mu'(-\vpi^{-\ell}a)}{q^{\ell s}|a|^s}\,\d a\\
=&q^{-\ell s}\mu'(-\vpi^{-\ell})\widehat{\phi_3}(0)|b|^{1/2}W''\biggl(\begin{bmatrix} b & 0 \\ 0 & 1 \end{bmatrix}\biggl)Z(1-s,\widehat{\phi_1},\nu^{-1}\mu'). 
\end{align*}
It follows that
\begin{align*}
&\int_{F^\times}Z(1-s,\pi^\vee(\vsi)\widetilde{W}_b,\mu')\frac{|b|^{s-1}}{\nu'(b)}\,\d b\\
=&q^{-\ell s}\mu'(-\vpi^{-\ell})\widehat{\phi_3}(0)Z(1-s,\widehat{\phi_1},\nu^{-1}\mu')Z(s,W'',\nu^{\prime-1})
\end{align*}
as claimed. 
\end{proof}

We are now ready to prove Lemma \ref{lem:87}.  
Set 
\begin{align*}
h''&=h_{\phi_2}^{\rho,\mu}, &
\phi_1&=\widehat{\varphi_\chi}, &
\phi_2&=\widehat{\varphi_{\chi'}'}, &
\chi&=(\nu\mu^{\prime-1})|_{\frko^\times}, & 
\chi'&=(\mu\nu^{\prime-1})|_{\frko^\times}. 
\end{align*}
By definition we see that 
\[W''\biggl(\begin{bmatrix} b & 0 \\ 0 & 1 \end{bmatrix}\biggl)=\frac{\mu(b)}{|b|^{1/2}}\int_Fh''(w_2\bfu(b^{-1}x))\overline{\addchar(x)}\,\d x=\mu(-b)|b|^{1/2}\widehat{\phi_2}(b)\]
for $b\in F^\times$. 
It follows that 
\[Z(s,W'',\nu^{\prime-1})=\mu(-1)Z(s,\widehat{\phi_2},\mu\nu^{\prime-1}). \]
Then 
\begin{align*}
Z(s,\widehat{\phi_1},\nu^{-1}\mu')&=(\nu\mu')(-1), \\ 
Z(s,\widehat{\phi_2},\mu\nu^{\prime-1})&=\begin{cases}
(\mu\nu')(-1) &\text{if $c(\chi')\geq 1$, }\\
\frac{(\mu\nu^{\prime-1})(\vpi)L(s,\mu\nu^{\prime-1})}{q^{s-1}L(1-s,\mu^{-1}\nu')} &\text{otherwise. }
\end{cases}
\end{align*}
Letting $\phi_3=q^{-\ell}\II_{\frkp^{-\ell}}$, we obtain 
\[Z(s,\pi(\calj_\ell)W^\ddagger,W^\flat_{\sig^\vee})
=\pm\frac{q^{-\ell s}\mu'(\vpi^{-\ell})\frac{\zet_F(2)}{\zet_F(1)}}{\gam(s,\pi^{}\otimes\mu^{\prime-1},\addchar)}
\times\begin{cases}
1 &\text{if $c(\chi')\geq 1$, }\\
\frac{(\mu\nu^{\prime-1})(\vpi)L(s,\mu\nu^{\prime-1})}{q^{s-1}L(1-s,\mu^{-1}\nu')} &\text{otherwise. }
\end{cases} \]
Lemma \ref{lem:87} follows by replacing $\ell$ by $\ell+n'$ in view of $n'=\max\{1,c(\chi')\}$.  


\part{Appendices}
\renewcommand{\thesection}{\Alph{section}}
\setcounter{section}{0}

\section{Differential operators on symmetric spaces}\label{sec:a}


\subsection{Bounded symmetric domains and factors of automorphy}\label{ssec:a1}

For a Hermitian matrix $X$ we write $X>0$ to indicate that $X$ is positive definite. 
For non-negative integers $r,s$ we define the Lie group $\U(r,s)$ by 
\begin{align*}
\U(r,s)&=\{g\in\GL_{r+s}(\CC)\;|\;gI_{r,s}\trs g^c=I_{r,s}\}, &
I_{r,s}&=\begin{bmatrix} \ono_r & \\ & -\ono_s \end{bmatrix}. 
\end{align*}
The group $\U(r,s)$ acts transitively on the bounded symmetric domain  
\[\frkB_{r,s}=\{z\in\Mat_{r,s}(\CC)\;|\;1_s-\trs z^c z>0\}  \index{$\frkB_{r,s}$}\] 
by 
\[gz=(Az+B)(Cz+D)^{-1}\]
where $z\in\frkB_{r,s}$ and $g=\begin{bmatrix} A & B \\ C & D \end{bmatrix}\in\U(r,s)$ with $D$ of size $r$. 

Let $\bfo\in\frkB_{r,s}$ be the zero matrix. 
Then 
\[K=\{g\in\U(r,s)\;|\;g\bfo=\bfo\}\]
is the maximal compact subgroup of $\U(r,s)$. 
We identify $\frkB_{r,s}$ with $\U(r,s)/K$ via the map $g\mapsto g\bfo$. 
Put 
\begin{align*}
&\lam(g,z)=B^c\trs z+A^c, & 
&\mu(g,z)=Cz+D, &
&J(g,z)=(\lam(g,z),\mu(g,z)), \\
&\xi(z',z)=\ono_r-z^{\prime c}\trs z, &
&\eta(z',z)=\ono_s-\trs z^{\prime c} z, &
&B(z)=\begin{bmatrix} \ono_r & z \\ \trs z^c & \ono_s \end{bmatrix}\\
&\xi(z)=\xi(z,z), &
&\eta(z)=\eta(z,z), &
&\vXi(z)=(\xi(z),\eta(z))
\end{align*}
for $z,w\in\frkB_{r,s}$ and $g\in\U(r,s)$. 
Recall the equalities  
\begin{align}
&gB(z)=B(gz)
\begin{bmatrix} \lam(g,z)^c & \\ & \mu(g,z)\end{bmatrix}, \label{tag:a1}\\
&\trs B(z')^c I_{r,s} B(z)=\begin{bmatrix} \xi(z',z)^c & z-z' \\ \trs z^{\prime c}-\trs z^c & -\eta(z',z)\end{bmatrix} \notag
\end{align}
from which we have for $g,g'\in\U(r,s)$ and $z\in\frkB_{r,s}$ 
\begin{align*}
&\lam(gg',z)=\lam(g,g' z)\lam(g',z), &
&\mu(gg',z)=\mu(g,g' z)\mu(g',z), \\
&\trs\lam(g,z)^c\xi(g z)\lam(g,z)=\xi(z), &
&\trs\mu(g,z)^c\eta(g z)\mu(g,z)=\eta(z). 
\end{align*} 
Therefore the isomorphism $K\simeq\U(r)\times\U(s)$ is given by $k\mapsto J(k,\bfo)$. 
Its inverse map is given by 
\beq
(A,D)\mapsto\begin{bmatrix} \trs A^{-1} & \\ & D \end{bmatrix}. \label{tag:a2}
\eeq 
The group $\GL_r(\CC)\times\GL_s(\CC)$ acts on $\frkT=\Mat_{r,s}(\CC)$ by 
\begin{align*}
(A,D)u&=Au\trs D &
(A&\in\GL_r(\CC),\; D\in\GL_s(\CC),\; u\in\Mat_{r,s}(\CC)). 
\end{align*}
By a representation $(\rho,V)$ of $\GL_r(\CC)\times\GL_s(\CC)$ we mean a pair formed by a finite dimensional complex vector space $V$ and a complex analytic homomorphism $\rho$ of $\GL_r(\CC)\times\GL_s(\CC)$ into $\GL(V)$. 
Given a $C^\infty$-manifold $M$, we write $C^\infty(M,V)$ for the space of all $V$-valued $C^\infty$ functions on $M$ and $C^\infty(\U(r,s),\rho)$ the space of all $V$-valued $C^\infty$ functions on $\U(r,s)$ such that for $g\in\U(r,s)$ and $k\in K$
\[h(g k)=\rho(J(k,\bfo))^{-1}h(g). \] 

For $g\in\U(r,s)$ and $f\in C^\infty(\frkB_{r,s},V)$ we define $f\|_\rho g\in C^\infty(\frkB_{r,s},V)$ by 
\[f\|_\rho g(z)=\rho(J(g,z))^{-1}f(gz). \]
We associate $f^\rho\in C^\infty(\U(r,s),\rho)$ to $f\in C^\infty(\frkB_{r,s},V)$ by 
\[f^\rho(g)=\rho(J(g,\bfo))^{-1}f(g\bfo)=f\|_\rho g(\bfo). \]
The map $f\mapsto f^\rho$ is a $\CC$-linear bijection of $C^\infty(\frkB_{r,s},V)$ onto $C^\infty(\U(r,s),\rho)$. 


\subsection{Differential operators on $\frkB_{r,s}$}\label{ssec:a2}

Denote the Lie algebra of $\U(r,s)$ by $\frkg$ and the subalgebra of $\frkg$ corresponding to $K$ by $\frkk$. 
The Lie algebra $\frkg$ acts on the space of $C^\infty$ functions on $\U(r,s)$ by 
\[[Yf](g)=\frac{\d}{\d t}f(g\exp(tY))\biggl|_{t=0} \]
for $Y\in\frkg$. 
This action is extended to that of the universal enveloping algebra $\frkU$ of the complexified Lie algebra $\frkg_\CC$. 
Let $\Ad$ denote the adjoint representation of $K$ on $\frkg_\CC$. 
For $k\in K$ the action of $\Ad(k)$ can be extended to $\frkU$, which we denote also by $\Ad(k)$. 

Since we have fixed complex structure of $\frkB_{r,s}=\U(r,s)/K$, we have the Harish-Chandra decomposition 
\[\frkg_\CC=\frkk_\CC\oplus\frkp_+\oplus\frkp_-, \]
where 
\begin{align*}
\frkg_\CC&=\Mat_{r+s}(\CC), & 
\frkk_\CC&=\biggl\{\begin{bmatrix} a & 0 \\ 0 & d \end{bmatrix}\biggl|\;a\in\Mat_r(\CC),\;d\in\Mat_s(\CC)\biggl\} 
\end{align*}
and the $\CC$-linear bijections $\iot_\pm$ of $\frkT=\Mat_{r,s}(\CC)$ onto $\frkp_\pm$ are defined by 
\begin{align*}
\iot_+(u)&=\begin{bmatrix} 0 & u \\ 0 & 0 \end{bmatrix}, &
\iot_-(u)&=\begin{bmatrix} 0 & 0 \\ \trs u & 0 \end{bmatrix}
\end{align*}
for $u\in\frkT$. 
Note that $Xh=-\d\rho(X)h$ for every $X\in\frkk$ and $h\in C^\infty(\rho)$. 
Since $[\frkp_\pm,\frkp_\pm]=\{0\}$, the symmetric algebra $\frkS(\frkp_\pm)$ can be embedded in $\frkU$ and is stable under $\Ad(k)$ for $k\in K$. 

We write $Ml_n(\frkT,V)$ for the vector space of all $\CC$-multilinear maps of $\frkT\times\cdots\times\frkT$ ($n$ copies) into $V$. 
We define the representation $(\rho\otimes\upsilon^n,Ml_n(\frkT,V))$ 
of $\GL_r(\CC)\times\GL_s(\CC)$ by 
\begin{align*}
[(\rho\otimes\upsilon^n)(a,b)h](u_1,\dots,u_n)&=\rho(a,b)h(\trs au_1b,\dots,\trs au_nb) \index{$\rho\otimes\upsilon^n$}
\end{align*}
for $(a,b)\in \GL_r(\CC)\times\GL_s(\CC)$, $h\in Ml_n(\frkT,V)$ and $u_1,\dots,u_n\in\frkT$. 

Given $f\in C^\infty(\frkB_{r,s},V)$, we define $Df\in C^\infty(\frkB_{r,s},Ml_1(\frkT,V))$ by 
\begin{align*}
[Df(u)](z)&=\sum_{i=1}^r\sum_{j=1}^su_{ij}\frac{\partial f}{\partial z_{ij}}(z), & 
[Cf(u)](z)&=[Df(\trs\xi(z)u\eta(z))](z)
\end{align*}
for $z=(z_{ij})\in\frkB_{r,s}$ and $u=(u_{ij})\in\frkT$. 
Given a non-negative positive integer $n$, we define $D^nf$ and $C^nf$ by
\begin{align*}
D^nf&=D(D^{n-1}f), & D^0f&=f, &
C^nf&=C(C^{n-1}f), & C^0f&=f. 
\end{align*}
These have values in $Ml_n(\frkT,V)$ in the sense that 
\[C^nf(u_1,\dots,u_n)=C[C^{n-1}f(u_1,\dots,u_{n-1})](u_n). \]
Define $D_\rho^nf\in C^\infty(\frkB_{r,s},Ml_n(\frkT,V))$ by 
\[D_\rho^nf=(\rho\otimes\upsilon^n)(\vXi)^{-1}C^n(\rho(\vXi)f). \]
Proposition 12.10 and (A8.6a) of \cite{Shimura00} say that 
\begin{align}
&D_\rho^{n+1}=D_{\rho\otimes\upsilon^n}^{}D_\rho^n=D^n_{\rho\otimes\tau}D_\rho^{}, \notag \\  
&D_\rho^n(f\|_\rho g)=(D_\rho^nf)\|_{\rho\otimes\upsilon^n}g, \notag \\
&\iot_+(u_1)\cdots\iot_+(u_n)f^\rho=[D_\rho^nf]^{\rho\otimes\upsilon^n}(u_1,\dots,u_n) \label{tag:a3}
\end{align}
for $f\in C^\infty(\frkB_{r,s},V)$, $g\in\U(r,s)$ and $u_1,\dots,u_n\in\frkT$. 


\subsection{Hermitian symmetric spaces}\label{ssec:a3}

We retain the notation in Sections \ref{ssec:21}--\ref{ssec:26}. 
Put 
\begin{align*}
\vXi(Z)&=(\xi(Z),\eta(Z)), &
B(Z)&=\begin{bmatrix} \trs\tau^c & \trs w^c & \tau \\ 0 & \gam_0 & w  \\ \ono_s & 0 & \ono_s \end{bmatrix}, \index{$B(Z)$} 
\index{$\vXi(Z),\xi(Z),\eta(Z)$}
\end{align*} 
where 
\begin{align*}
\xi(Z)&=\sqrt{-1}\begin{bmatrix} \tau^c-\trs\tau & -\trs w \\ w^c & -\trs\gam_0\end{bmatrix}, & 
\eta(Z)&=\sqrt{-1}(\trs\tau^c-\tau-\trs w^c\gam_0^{-1}w). 
\end{align*}
It is well-known that 
\begin{align}
&\alp B(Z)=B(\alp Z)
\begin{bmatrix} \lam(\alp,Z)^c & \\ & \mu(\alp,Z)\end{bmatrix}, \index{$\lam(\alp,Z)$}\index{$\mu(\alp,Z)$}\label{tag:a4} \\
&\trs B(Z')^c S_{\gam_0}^{-1} B(Z)=\begin{bmatrix} -\sqrt{-1}\xi(Z',Z)^c & Z-Z' \\ \trs(Z-Z')^c & \sqrt{-1}\eta(Z',Z)\end{bmatrix} \notag
\end{align}
(see (6.3.1) and (6.6.1) of \cite{Shimura97}) from which we have 
\begin{align}
&\lam(\alp\bet,Z)=\lam(\alp,\bet Z)\lam(\bet,Z), &
&\mu(\alp\bet,Z)=\mu(\alp,\bet Z)\mu(\bet,Z), \notag\\
&\trs\lam(\alp,Z)^c\xi(\alp Z)\lam(\alp,Z)=\xi(Z), &
&\trs\mu(\alp,Z)^c\eta(\alp Z)\mu(\alp,Z)=\eta(Z) \label{tag:a5}
\end{align}
for $\alp,\bet\in G$ and $Z\in\frkD_{r,s}$. 
Recall the maximal compact subgroup $\calk$ of $G$  
\begin{align*}
\calk&=\{\alp\in G\;|\;\alp\bfi=\bfi\}, &
\bfi=\begin{bmatrix} \sqrt{-1}\ono_s \\ \oo_{t,s} \end{bmatrix}\in\frkD_{r,s}.  
\end{align*}

Let $C^\infty(G,\rho)$ the space of all $V$-valued $C^\infty$ functions on $G$ that satisfy
\[h(g k)=\rho(J(k,\bfi))^{-1}h(g)\]
for $g\in G$ and $k\in\calk$. 
The map $f\mapsto f^\rho$ is a $\CC$-linear bijection of $C^\infty(\frkD_{r,s},V)$ onto $C^\infty(G,\rho)$. 
Given $g\in G$ and $f\in C^\infty(\frkD_{r,s},V)$, we define $f\|_\rho g\in C^\infty(\frkD_{r,s},V)$ by 
\[f\|_\rho g(Z)=\rho(J(g,Z))^{-1}f(gZ). \]
We define $f^\rho\in C^\infty(G,\rho)$ by 
\[f^\rho(g)=\rho(J(g,\bfi))^{-1}f(g\bfi)=f\|_\rho g(\bfi). \]

Given $f\in C^\infty(\frkD_{r,s},V)$, we define $Df\in C^\infty(\frkD_{r,s},Ml_1(\frkT,V))$ by 
\begin{align*}
[Df(u)](Z)&=\sum_{i,j=1}^su_{ij}\frac{\partial f}{\partial\tau_{ij}}(Z)+\sum_{i=1}^t\sum_{j=1}^su_{s+i,j}\frac{\partial f}{\partial w_{ij}}(Z), \\ 
[Cf(u)](Z)&=[Df(\trs\xi(Z)u\eta(Z))](Z)
\end{align*}
for $Z=\begin{bmatrix} \tau \\ w \end{bmatrix}\in\frkD_{r,s}$ and $u=(u_{ij})\in\frkT$. 
Given a non-negative positive integer $n$, we define $D^nf$ and $C^nf$ by
\begin{align*}
D^nf&=D(D^{n-1}f), & D^0f&=f, &
C^nf&=C(C^{n-1}f), & C^0f&=f. 
\end{align*}
Define $D_\rho^nf\in C^\infty(\frkD_{r,s},Ml_n(\frkT,V))$ by 
\[D_\rho^nf=(\rho\otimes\upsilon^n)(\vXi)^{-1}C^n(\rho(\vXi)f). \index{$D_\rho$}\]
Proposition 12.10 of \cite{Shimura00} shows that 
\begin{align}
D_\rho^{n+1}&=D_{\rho\otimes\upsilon^n}^{}D_\rho^n=D^n_{\rho\otimes\tau}D_\rho^{}, &  
D_\rho^n(f\|_\rho g)&=(D_\rho^nf)\|_{\rho\otimes\upsilon^n}g. \label{tag:a6}
\end{align}


\subsection{Translation}\label{ssec:a4}

We take $\bet\in\GL_t(\CC)$ such that $\sqrt{-1}\bet\trs\bet^c=2\gam_0$ and put 
\begin{align*} 
\calc&=\begin{bmatrix} \sqrt{-1}\ono_s & 0 & \sqrt{-1}\ono_s \\ 0 & \bet & 0 \\ -\ono_s & 0 & \ono_s \end{bmatrix},  
\end{align*}
following \cite[A2.2]{Shimura97}. 
Since $\calc I_{r,s}\trs\calc^c=-2\sqrt{-1}S_{\gam_0}$, we define the isomorphism 
\begin{align*}
\vph&:\U(r,s)\stackrel{\sim}{\to}G, & 
\vph(g)&=\calc g\calc^{-1}. 
\end{align*} 

Lemma A2.3 of \cite{Shimura97} gives a holomorphic bijection  
\begin{align*}
\frkt&:\frkB_{r,s}\stackrel{\sim}{\to}\frkD_{r,s}, & 
\frkt\left(\begin{bmatrix} z_1 \\ z_2 \end{bmatrix}\right)&=\begin{bmatrix} \sqrt{-1}(\ono_s+z_1)(\ono_s-z_1)^{-1} \\ \bet z_2(\ono_s-z_1)^{-1} \end{bmatrix}
\end{align*}
such that $\frkt(gz)=\vph(g)\frkt z$. 
Define holomorphic maps $\kap:\frkB_{r,s}\to\GL_r(\CC)$ and $\nu:\frkB_{r,s}\to\GL_s(\CC)$ by 
\[\calc B(z)=B(\frkt z)\begin{bmatrix} \kap(z)^c & \\ & \nu(z) \end{bmatrix} \]
(see (A2.3.4) of \cite{Shimura97}). 
We see from (\ref{tag:a1}) and (\ref{tag:a4}) that 
\begin{align*}  
&\lam(\vph(g),\frkt z)=\kap(gz)\lam(g,z)\kap(z)^{-1}, &
&\mu(\vph(g),\frkt z)=\nu(gz)\mu(g,z)\nu(z)^{-1}. 
\end{align*}
Moreover, the proof of Lemma A2.3 shows that
\begin{align*}  
2\xi(z',z)&=\trs\kap(z')^c\xi(\frkt z',\frkt z)\kap(z), &
2\eta(z',z)&=\trs\nu(z')^c\eta(\frkt z',\frkt z)\nu(z), 
\end{align*}
Put $\vka(z)=(\kap(z),\nu(z))$. 
We rewrite the formulas as 
\begin{align}
J(\vph(g),\frkt z)&=\vka(gz)J(g,z)\vka(z)^{-1}, & 
2\vXi(z)&=\trs\vka(z)^c\vXi(\frkt z)\vka(z). \label{tag:a7}
\end{align}

Let $f\in C^\infty(\frkD_{r,s},V)$. 
Define $\til f\in C^\infty(\frkD_{r,s},V)$ by 
\[\til f_\rho(Z)=\rho(\vka(\frkt^{-1}Z))f(Z). \]
Let $\hat\frkg$ denote the Lie algebras of $G$. 
Define the isomorphism $\frkg_\CC\stackrel{\sim}{\to}\hat\frkg_\CC$ by $X\mapsto\calc X\calc^{-1}$. 
For $u\in\frkT$ we define $\iot_+^\calc(u)\in\hat\frkg_\CC$ by $\iot_+^\calc(u)=\calc\iot_+^{}(u)\calc^{-1}$. 

\begin{proposition}\label{prop:a1}
Given $g\in\U(r,s)$ and $f\in C^\infty(\frkD_{r,s},V)$, we have 
\begin{align*}
\til f_\rho\|_\rho\vph(g)&=\widetilde{(f\circ\frkt\|_\rho g\circ\frkt^{-1})_\rho}, \\
(-2\sqrt{-1})^nD_\rho^n\til f_\rho&=\widetilde{(D_\rho^n(f\circ\frkt)\circ\frkt^{-1})_{\rho\otimes\upsilon^n}}. 
\end{align*}
\end{proposition}

\begin{proof}
Let $g\in\U(r,s)$ and $z\in\frkB_{r,s}$. 
It follows from (\ref{tag:a7}) that 
\begin{align*}
\til f_\rho\|_\rho\vph(g)(\frkt z)
&=\rho(J(\vph(g),\frkt z))^{-1}\til f_\rho(\vph(g)\frkt z)\\
&=\rho(\vka(gz)J(g,z)\vka(z)^{-1})^{-1}\til f_\rho(\frkt(gz))
=\rho(\vka(z))(f\circ\frkt\|_\rho g(z)). 
\end{align*} 

Shimura has proved that
\[D(f\circ\frkt)(\trs\kap(z)u_1\nu(z))=-2\sqrt{-1}[Df(u_1)]\circ\frkt\]
for $u_1\in\frkT$ in \S 26.2 of \cite{Shimura00}. 
It follows from (\ref{tag:a7}) that 
\begin{align*}
4[C(f\circ\frkt)(u_1)](z)
&=4[D(f\circ\frkt)(\trs\xi(z)u_1\eta(z))](z)\\
&=[D(f\circ\frkt)(\trs\kap(z)\trs\xi(\frkt z)\kap(z)^c u_1\trs \nu(z)^c\eta(\frkt z)\nu(z))](\frkt z)\\ 
&=-2\sqrt{-1}[Cf(\kap(z)^c u_1\trs\nu(z)^c)](\frkt z). 
\end{align*}
Since $\kap(z)^c$ and $\nu(z)^c$ are anti-holomorphic, we have 
\[4^n[C^n(f\circ\frkt)(\ulu)](z)\\
=(-2\sqrt{-1})^n[C^nf(\kap(z)^c\ulu\trs\nu(z)^c)](\frkt z) \]
for $\ulu=(u_1,\dots,u_n)\in\frkT^n$. 
We get  
\begin{align*}
[D_\rho^n(f\circ\frkt)(\ulu)](z)
=&\rho(\Xi(\frkt z)\vka(z))^{-1}[C^n(\rho(\vXi)\til f_\rho)\circ\frkt(\trs\xi(z)^{-1}\ulu\eta(z)^{-1})](z)\\
=&(-2\sqrt{-1})^n\rho(\vka(z))^{-1}[D_\rho^n\til f_\rho(\trs\kap(z)^{-1}\ulu\nu(z)^{-1})](\frkt z)\\
=&(-2\sqrt{-1})^n(\rho\otimes\upsilon^n)(\vka(z))^{-1}[D_\rho^n\til f_\rho(\ulu)](\frkt z) 
\end{align*}
again by (\ref{tag:a7}) as claimed. 
\end{proof}

\begin{corollary}\label{cor:a1}
\begin{enumerate}
\item\label{cor:a11} $(f\circ\frkt)^\rho=\rho(\vka(\bfo))^{-1}(\til f_\rho)^\rho\circ\vph$.  
\item\label{cor:a12} $[D_\rho^n(f\circ\frkt)]^{\rho\otimes\upsilon^n}=(-2\sqrt{-1})^n(\rho\otimes\upsilon^n)(\vka(\bfo))^{-1}[D_\rho^n\til f_\rho]^{\rho\otimes\upsilon^n}\circ\vph$. 
\item\label{cor:a13} For $u_1,\dots,u_n\in\frkT$ 
\[\iot_+^\calc(u_1)\cdots\iot_+^\calc(u_n)(\til f_\rho)^\rho
=(-2\sqrt{-1})^n\upsilon^n(\vka(\bfo))^{-1}D_\rho^n\til f_\rho(u_1,\dots,u_n)^{\rho\otimes\upsilon^n}. \]
\end{enumerate} 
\end{corollary}

\begin{proof}
Since $\bfi=\frkt\bfo$, we obtain (\ref{cor:a11}) by specializing the first identify of Proposition \ref{prop:a1}. 
Proposition \ref{prop:a1} combined with (\ref{tag:a6}) gives 
\[(-2\sqrt{-1})^n(D_\rho^n\til f_\rho)\|_{\rho\otimes\upsilon^n}\vph(g)=\widetilde{(D_\rho^n(f\circ\frkt)\|_{\rho\otimes\upsilon^n}g\circ\frkt^{-1})}_{\rho\otimes\upsilon^n} \]
from which we obtain (\ref{cor:a12}). 
By definition  we have 
\[(\iot_+^\calc(u_1)\cdots\iot_+^\calc(u_n)(\til f_\rho)^\rho)\circ\vph=\iot_+^{}(u_1)\cdots\iot_+^{}(u_n)((\til f_\rho)^\rho\circ\vph). \]
The right hand side is 
\[\iot_+^{}(u_1)\cdots\iot_+^{}(u_n)(\rho(\vka(\bfo))(f\circ\frkt)^\rho)
=\rho(\vka(\bfo))[D_\rho^n(f\circ\frkt)]^{\rho\otimes\upsilon^n}(\ulu)\]
by (\ref{cor:a11}) and (\ref{tag:a3}), where $\ulu=(u_1,\dots,u_n)$. 
Since 
\[\rho(\vka(\bfo))[D_\rho^n(f\circ\frkt)](\ulu)^{\rho\otimes\upsilon^n}=(-2\sqrt{-1})^n\upsilon^n(\vka(\bfo))^{-1}D_\rho^n\til f_\rho(\ulu)^{\rho\otimes\upsilon^n}\circ\vph\]
by (\ref{cor:a12}), we have completed the proof of (\ref{cor:a13}). 
\end{proof}


\subsection{Differential operators on $\frkB_{2,1}$}\label{ssec:a5}

Letting $r=2$, $s=1$ and $\gam_0=\del$, we retain the notation of \S \ref{ssec:31}. 
Put 
\begin{align*}
v_0&=\begin{bmatrix} 0 \\ 1 \end{bmatrix}\in\frkT, &
Y_+&=\iot_+(v_0)=\left[\begin{array}{cc|c} 0 & 0 & 0 \\ 0 & 0 & 1 \\ \hline 0 & 0 & 0 \end{array}\right]\in\frkp_+.  
\end{align*} 
Let $\ulk=(k_1,k_2;k_3)$ and $\bfv_i=X^{\kap-i}Y^i\in \call_{\ulk}^\vee$ be a non-zero vector such that 
\[\rho^\vee_{\ulk}\biggl(\begin{bmatrix} a_1 & 0 \\ 0 & a_2 \end{bmatrix},a_3^{}\biggl)\bfv_i=a_1^{k_1+i}a_2^{k_2-i}a_3^{-k_3}\bfv_i\]
for $i=0,1,2,\dots,\kap$ and $\biggl(\begin{bmatrix} a_1 & 0 \\ 0 & a_2 \end{bmatrix},a_3\biggl)\in\GL_2(\CC)\times\CC^\times$ . 
Given $\bfv\in\call_{\ulk}^\vee$, we associate to $f\in C^\infty(\frkB_{2,1},V)$ and $h\in C^\infty(\U(2,1),\rho)$ functions $f_\bfv:\frkB_{2,1}\to\CC$ and $h_\bfv:\U(2,1)\to\CC$ defined by 
\begin{align*}
f_\bfv(z)&=\ell_\rho(\bfv\otimes f(z)), & 
h_\bfv(g)&=\ell_\rho(\bfv\otimes h(g)). 
\end{align*}
For $g\in \U(2,1)$ and $t=\diag[t_1,t_2,t_3]\in K$ we see that 
\[h_{\bfv_i}(gt)=\ell_\rho\biggl(\rho^\vee\biggl(\begin{bmatrix} t_1^{-1} & 0 \\ 0 & t_2^{-1}\end{bmatrix},t_3^{}\biggl)\bfv_i\otimes h(g)\biggl)=h_{\bfv_i}(g)t_1^{-k_1-i}t_2^{-k_2+i}t_3^{-k_3} \]
(cf. (\ref{tag:a2})). 
Since $\Ad(t)Y_+=t_2^{}t_3^{-1}Y_+$, we see that 
\[Y_+^nf^\rho_{\bfv_i}(gt)=(Y_+^nf^\rho)_{\bfv_i}(gt)=Y_+^nf^\rho_{\bfv_i}(g)t_1^{-k_1-i}t_2^{-k_2+i+n}t_3^{-k_3-n}. \]
It follows from (\ref{tag:a3}) that 
\[Y_+^nf^\rho=D_\rho^nf(v_0,v_0,\dots,v_0)^{\rho\otimes\upsilon^n}. \]


\subsection{The restriction to $\frkD_{1,1}$}\label{ssec:a7}

We retain the notation of \S \ref{ssec:pullback}. 
It follows from (\ref{tag:26}) that for $f\in C^\infty(\frkD_{2,1},V)$  
\begin{align*}
f\|_\rho \iot(h)\circ\jmath
&=f\circ\jmath\|_{\rho\circ\iot'}h, & 
f^\rho\circ\iot&=(f\circ\jmath)^{\rho\circ\iot'}. 
\end{align*}
Recall that for $Z=\begin{bmatrix} \tau \\ w \end{bmatrix}\in\frkD_{2,1}$
\begin{align*}
\eta(Z)&=\sqrt{-1}(\tau^c-\tau-w^c\del^{-1} w), & 
\eta(\tau)&=\sqrt{-1}(\tau^c-\tau)=\eta(\jmath(\tau)), \\
\xi(\jmath(\tau))&=\begin{bmatrix} \eta(\tau) & 0 \\ 0 & -\sqrt{-1}\del\end{bmatrix}, &
\vXi(\tau)&=(\xi(\tau),\eta(\tau)). 
\end{align*}
Put $\ulv_0^n=(v_0,v_0,\dots,v_0)\in\frkT^n$. \index{$\ulv_0^n$}
We see from (\ref{tag:a6}) that 
\begin{align}
[D_\rho^n(f\|_\rho\iot(h))(\ulv_0^n)](\jmath(\tau)) \label{tag:a8}
=&[D_\rho^nf\|_{\rho\otimes\upsilon^n}\iot(h)(\ulv_0^n)](\jmath(\tau)) \\
=&[(D_\rho^nf)\circ\jmath\|_{\rho\otimes\upsilon^n\circ\iot}h(\ulv_0^n)](\tau) \notag\\
=&\mu(h,\tau)^{-n}[D_\rho^nf(\ulv_0^n)]\circ\jmath\|_{\rho\circ\iot}h(\tau). \notag
\end{align}
Put $Y_+^\calc=\calc Y_+^{}\calc^{-1}\in\hat\frkg_\CC$. \index{$Y_+^\calc$}
We will consider the differential operator  
\[(Y_+^\calc)^nf^\rho=[D_\rho^nf]^{\rho\otimes\upsilon^n}(\ulv_0^n) \]
for $f\in C^\infty(\frkD_{2,1},V)$ (cf. Corollary \ref{cor:a1}(\ref{cor:a13})). 

We write $f\in C^\infty(\frkD_{2,1},\call_{\ulk}(\CC))$ as 
\[f(Z)=\sum_{i=0}^\kap f_i(Z)X^iY^{\kap-i}. \]

\begin{remark}\label{rem:a1}
We take $\bet=\sqrt{2|\del|}$. 
Then $\vka(\bfo)=\left(\begin{bmatrix} -1 & 0 \\ 0 & \sqrt{-2/|\del|} \end{bmatrix},1\right)$. 
We can rewrite Corollary \ref{cor:a1}(\ref{cor:a12}) as 
\[[D_{\rho_{\ulk}}^n\til f_\rho(\underline{v}_0^n)]_i^{\rho\otimes\upsilon^n}\circ\vph=\frac{(-\sqrt{-1})^{i+2k_1+k_2}}{(-2)^n}\sqrt{2/|\del|}^{n-k_2+i}[D_{\rho_{\ulk}}^n(f\circ\frkt)(\underline{v}_0^n)]_i^{\rho\otimes\upsilon^n} \]
for $i=0,1,2,\dots,\kap$. 
The dependance on $\del$ is compatible with Remark \ref{rem:33}. 
\end{remark}

Given $k,\ell\in\ZZ$, $g\in C^\infty(\frkD_{1,1},\CC)$ and $h\in H$, we define  
\[g\|_{(\ell,k)}h(\tau)=\lam(h,\tau)^\ell\mu(h,\tau)^{-k}g(h\tau). \]

\begin{proposition}\label{prop:a2}
For $h\in H$ we have 
\[[D_{\rho_{\ulk}}^n(f\|_{\rho_{\ulk}}\iot(h))(\ulv_0^n)]_i\circ\jmath=[D_{\rho_{\ulk}}^nf(\ulv_0^n)]_i\circ\jmath\|_{(k_1+i,k_3+n)}h. \]
\end{proposition}

\begin{proof}
We have 
\begin{align*}
[D_{\rho_{\ulk}}^n(f\|_{\rho_{\ulk}}\iot(h))(\ulv_0^n)]_i(\jmath(\tau))
&=\mu(h,\tau)^{-n}([D_{\rho_{\ulk}}^nf(\ulv_0^n)]\circ\jmath\|_{\rho_{\ulk}\circ\iot}h)_i(\tau)\\
&=\lam(h,\tau)^{k_1+i}\mu(h,\tau)^{-n-k_3}[D_{\rho_{\ulk}}^nf(\ulv_0^n)]_i\circ\jmath(h\tau) 
\end{align*}
by (\ref{tag:a8}).  
\end{proof}

\begin{lemma}\label{lem:a1}
Let $f\in C^\infty(\frkD_{2,1},V)$. 
We have 
\[[D_\rho^nf(\ulv_0^n)](\jmath(\tau))
=\rho(\vXi(\jmath(\tau)))^{-1}\left[\left(\frac{\partial}{\partial w}-\frac{w^c}{\del}\frac{\partial}{\partial\tau}\right)^n(\rho(\vXi)f)\right](\jmath(\tau)). \]
\end{lemma}

\begin{proof}
Observe that  
\begin{align*}
[Cf(v_0)](Z)&=\left[Df\left(\trs\xi(Z)\begin{bmatrix} 0 \\ 1 \end{bmatrix}\eta(Z)\right)\right](Z)\\
&=\sqrt{-1}\eta(Z)\left(w^c\frac{\partial}{\partial\tau}-\del\frac{\partial}{\partial w}\right)f(Z)
\end{align*}
for $Z=\begin{bmatrix} \tau \\ w \end{bmatrix}\in\frkD_{2,1}$. 
Since 
\beq
\left(w^c\frac{\partial}{\partial\tau}-\del\frac{\partial}{\partial w}\right)\eta(Z)=0, \label{tag:a9}
\eeq
we have 
\[[C^nf(\ulv_0^n)](Z)=(\sqrt{-1}\eta(Z))^n\left(w^c\frac{\partial}{\partial\tau}-\del\frac{\partial}{\partial w}\right)^nf(Z). \]

Since $\trs\xi(\jmath(\tau))^{-1}v_0\eta(\tau)^{-1}=\frac{\sqrt{-1}}{\del\eta(\tau)}v_0$, we have 
\[[D_\rho^nf(\ulv_0^n)](\jmath(\tau))
=\left(\frac{\sqrt{-1}}{\del\eta(\tau)}\right)^n\rho(\vXi(\tau))^{-1}[C^n(\rho(\vXi)f)(\ulv_0^n)](\jmath(\tau)), \]
which proves the stated formula. 
\end{proof}

\begin{proposition}\label{prop:a3}
Given $f\in C^\infty(\frkD_{2,1},\call_{\ulk}(\CC))$, we have 
\[[D_{\rho_{\ulk}}^nf(\ulv_0^n)](\jmath(\tau))\\
=\sum_{i=0}^\kap \sum_{j=0}^{\min\{n,i\}}\frac{i!}{(i-j)!}\binom{n}{j}\frac{\partial^{n-j}f_i}{\partial w^{n-j}}(\jmath(\tau))\frac{X^{i-j}Y^{\kap-i+j}}{(-\sqrt{-1}\eta(\tau))^j}. \]
In particular, 
\[[D_{\rho_{\ulk}}^nf(\ulv_0^n)](\jmath(\tau))\\
=\sum_{i=0}^\kap \frac{\partial^nf_i}{\partial w^n}(\jmath(\tau))X^iY^{\kap-i} \pmod{\eta(\tau)^{-1}}.\]
\end{proposition}

\begin{proof}
Recall that $\xi(Z)=\begin{bmatrix} \eta(\tau) & -\sqrt{-1}w \\ \sqrt{-1} w^c & -\sqrt{-1}\del\end{bmatrix}$. 
Since 
\begin{align*}
\det\xi(Z)&=-\sqrt{-1}\del\eta(Z), & 
\trs\xi(Z)^{-1}&=\frac{1}{\del\eta(Z)}\begin{bmatrix} \del & w^c \\ -w & \sqrt{-1}\eta(\tau) \end{bmatrix}, 
\end{align*}
we have 
\[(\rho_{\ulk}(\vXi)f)(Z)=\eta(Z)^{k_3}\sum_{i=0}^\kap f_i(Z)\frac{(\del X-wY)^i(w^c X+\sqrt{-1}\eta(\tau)Y)^{\kap-i}}{(-\sqrt{-1}\del\eta(Z))^{k_1}(\del\eta(Z))^\kap}. \]
It follows again from (\ref{tag:a9}) that 
\begin{align*}
&\frac{(-\sqrt{-1})^{k_1}\del^{k_2}}{\eta(Z)^{k_3-k_2}}\left[\left(\frac{\partial}{\partial w}-\frac{w^c}{\del}\frac{\partial}{\partial\tau}\right)^n(\rho_{\ulk}(\vXi)f)\right](Z)\\
=&\sum_{i=0}^\kap \left(\frac{\partial}{\partial w}-\frac{w^c}{\del}\frac{\partial}{\partial\tau}\right)^nf_i(Z)(\del X-wY)^i(w^c X+\sqrt{-1}\eta(\tau)Y)^{\kap-i}. 
\end{align*}
Now we evaluate at $w=0$. 
Then since $w^c=0$, we get 
\begin{align*}
&\frac{(-\sqrt{-1})^{k_1}\del^{k_2}}{\eta(\tau)^{k_3-k_2}}\left[\left(\frac{\partial}{\partial w}-\frac{w^c}{\del}\frac{\partial}{\partial\tau}\right)^n(\rho_{\ulk}(\vXi)f)\right](\jmath(\tau))\Big|_{w=0}\\
=&\sum_{i=0}^\kap \frac{\partial^n}{\partial w^n}f_i(Z)(\del X-wY)^i(\sqrt{-1}\eta(\tau)Y)^{\kap-i}\Big|_{w=0}\\
=&\sum_{i=0}^\kap(\sqrt{-1}\eta(\tau)Y)^{\kap-i}\sum_{j=0}^n\binom{n}{j} \frac{\partial^{n-j}f_i}{\partial w^{n-j}}(Z)\frac{\partial^j}{\partial w^j}(\del X-wY)^i\Big|_{w=0}\\
=&\sum_{i=0}^\kap(\sqrt{-1}\eta(\tau)Y)^{\kap-i}\sum_{j=0}^{\min\{n,i\}}\binom{n}{j} \frac{\partial^{n-j}f_i}{\partial w^{n-j}}(\jmath(\tau))\frac{i!(-Y)^j}{(i-j)!}(\del X)^{i-j}. 
\end{align*}
Finally, we see from Lemma \ref{lem:a1} that $[D_{\rho_{\ulk}}^nf(\ulv_0^n)](\jmath(\tau))$ equals 
\begin{align*}
&\frac{(-\sqrt{-1}\del\eta(\tau))^{k_1}}{\eta(\tau)^{k_3}}\rho_\kap(\xi(\jmath(\tau)))\left[\left(\frac{\partial}{\partial w}-\frac{w^c}{\del}\frac{\partial}{\partial\tau}\right)^n(\rho_{\ulk}(\vXi)f)\right](\jmath(\tau))\\
=&\sum_{i=0}^\kap\frac{(\eta(\tau)\del Y)^{\kap-i}}{(\del\eta(\tau))^\kap}\sum_{j=0}^{\min\{n,i\}}\binom{n}{j} \frac{\partial^{n-j}f_i}{\partial w^{n-j}}(\jmath(\tau))\frac{i!(\sqrt{-1}\del Y)^j}{(i-j)!}(\del\eta(\tau)X)^{i-j}, 
\end{align*}
which proves our formula. 
\end{proof}


\section{Archimedean computations}\label{sec:b}


\subsection{Discrete series representations of $\U(r,s)$}\label{ssec:b1}

Put 
\[\1_n=(1,1,\dots,1)\in\ZZ^n. \]
Let $r$ and $s$ be non-negative integers such that $r+s=n$. 
Discrete series of $\U(r, s)$ are parametrized by Harish-Chandra parameters.
Let $(\lam,z)$ be a Harish-Chandra parameter for $\U(r, s)$, where $\lam\in\ZZ^n+\frac{n-1}{2}\1_n$ is a sequence of distinct integers or half integers and $z\in\{\pm 1\}^n$ is a sequence of $+$ and $-$ corresponding to each entry in $\lam$. 
Here the total number of $+$'s must be $r$ and the total number of $-$'s must be $s$. 

Let $D(\lam,z)$ denote the discrete series of $\U(r,s)$ with Harich-Chandra parameter $(\lam,z)$. 
We also interpret $z$ as a partition of $\lam$ into $\lam^+\in\bigl(\frac{1}{2}\ZZ\bigl)^r$ and $\lam^-\in\bigl(\frac{1}{2}\ZZ\bigl)^s$, and write $D_{(\lam^+;\lam^-)}=D(\lam,z)$. 


\subsection{Branching laws for $\U(r,s)\times\U(r-1,s)$}\label{ssec:b2}

\begin{definition}\label{definition:b1}
We say that two Harish-Chandra parameters $(\lam,z)$ and $(\mu,t)$ of $\U(r, s)$ and $\U(r-1, s)$ respectively, satisfy the GGP interlacing relation, if one can line up $\lam$ and $\mu$ in the descending ordering such that the corresponding sequence of signs from $z$ and $t$ only has the following eight adjacent pairs
\[(\oplus+),\;(+\oplus),\;(-\ominus),\;(\ominus-),\;(+-),\;(-+),\;(\oplus\ominus),\;(\ominus\oplus). \]
Here $\oplus$ and $\ominus$ represent $+1$ and $-1$ in $t$, and $\pm$ represents $\pm1$ in $z$.
\end{definition}

\begin{theorem}[He \cite{He}]\label{thm:b1}
$D(\mu,t)$ appears as a subrepresentation of the restriction of $D(\lam, z)$ to $\U(r-1,s)$ if and only if $(\lam, z)$ and $(\mu, t)$ satisfy the GGP interlacing relation.
\end{theorem}


\subsection{Archimedean local factors}\label{ssec:b3}
 
Let $\CC^1$ denote the group of complex numbers of absolute value $1$. 
Define the character $\bvep:\CC^\times\to\CC^1$ by $\bvep(x)=\frac{x}{|x|}$. 
We view it as a character of any unitary group via composition with the determinant character. 
Fix tuples $\lam_1>\cdots>\lam_n$ and $\mu_1>\cdots>\mu_{n-1}$ of half integers such that $\lam_i-\frac{n+1}{2}\in\ZZ$ and $\mu_j-\frac{n}{2}\in\ZZ$. 
Let 
\begin{align*}
(\lam,z)&=(\underset{\oplus}{\lam_1},\dots,\underset{\oplus}{\lam_r};\underset{\ominus}{\lam_{r+1}},\dots,\underset{\ominus}{\lam_n}), & 
(\mu,t)&=(\underset{+}{\mu_1},\dots,\underset{+}{\mu_{r-1}};\underset{-}{\mu_r},\dots,\underset{-}{\mu_{n-1}})
\end{align*} 
be Harish-Chandra parameters of $\U(V)$ and $\U(V')$. 
We identify the maximal compact subgroup $K$ of $\U(V)$ with $\U(r)\times\U(s)$ via the map 
\[\begin{bmatrix} A & \\ & D \end{bmatrix}\mapsto(A,D) \]
(cf. (\ref{tag:a2})).  
Let $\pi$ be a discrete series of $\U(r,s)$ with Harish-Chandra parameter $(\lam,z)$ and $\sig$ a discrete series of $\U(r-1,s)$ with Harish-Chandra parameter $(\mu,t)$. 
The minimal $K$-types of $\pi$ and $\sig$ are 
\begin{align*}
-\ulk=&\biggl(\lam_1+\frac{s-(r-1)}{2},\dots,\lam_r+\frac{s+(r-1)}{2};\\
&\quad\lam_{r+1}-\frac{r+(s-1)}{2},\dots,\lam_n-\frac{r-(s-1)}{2}\biggl), \\
-\ulk'=&\biggl(\mu_1+\frac{s-(r-2)}{2},\dots,\mu_{r-1}+\frac{s+(r-2)}{2};\\
&\quad\mu_r-\frac{r-1+(s-1)}{2},\dots,\mu_{n-1}-\frac{r-1-(s-1)}{2}\biggl). 
\end{align*}

The $L$-parameters of $\pi$ and $\sig$ restricted to $W_\CC=\CC^\times\subset W_\RR$ are given by
\begin{align*}
\phi_\pi|_{W_\CC}&=\bvep^{2\lam_1}\oplus\cdots\oplus\bvep^{2\lam_n}, &
\phi_\sig|_{W_\CC}&=\bvep^{2\mu_1}\oplus\cdots\oplus\bvep^{2\mu_{n-1}}.
\end{align*}
One can easily compute the adjoint $L$-factors, combining Remark \ref{rem:71} with \cite[Lemma 7.1]{Prasad92}.
The $L$-factors are defined by 
\begin{align}
L(s,\pi\times\sig^\vee)
=&\prod_{i=1}^n\prod_{j=1}^{n-1}\Gam_\CC(s+|\lam_i-\mu_j|), \notag\\
L(s,\pi,\Ad)
=&
\Gam_\RR(s+1)^n\prod_{i<j}\Gam_\CC(s+\lam_i-\lam_j), \notag\\
L(s,\sig^\vee,\Ad)
=&
\Gam_\RR(s+1)^{n-1}\prod_{i<j}\Gam_\CC(s+\mu_i-\mu_j), \label{tag:b0}
\end{align}
where $\Gam_\RR(s)=\pi^{-s/2}\Gam\bigl(\frac{s}{2}\bigl)$ and $\Gam_\CC(s)=2(2\pi)^{-s}\Gam(s)$. 
Suppose that 
\[\lam_1>\mu_1>\cdots>\mu_{r-1}>\lam_r; \qquad \lam_{r+1}>\mu_r>\cdots>\lam_n>\mu_{n-1}. \]
Then $\sig$ appears as a subrepresentation of $\pi|_{\U(r-1,s)}$, and 
\begin{align}
\frac{L\left(\frac{1}{2},\pi\times\sig^\vee\right)}{L(1,\pi,\mathrm{Ad})L(1,\sig^\vee,\mathrm{Ad})}=&\frac{\prod_{i=1}^n\prod_{j=1}^{n-1}\Gam\bigl(\frac{1}{2}+|\lam_i-\mu_j|\bigl)}{\prod_{i<j}\Gam(1+\lam_i-\lam_j)\prod_{i<j}\Gam(1+\mu_i-\mu_j)} \label{tag:b1}\\
&\times 2^{-n}(2\pi)^{\frac{n(n+1)}{2}-2\sum_{j=1}^s(\lam_{r+j}-\mu_{r-1+j})}. \notag
\end{align}


\subsection{The Ichino-Ikeda integral}\label{ssec:b4}

We let $V=W$ and equip the space $W=\CC^3$ with the Hermitian form on $(\;,\;)=\frac{1}{\sqrt{-1}}\La\;,\;\Ra_W$. 
Fix an orthogonal basis $\vep,\vep_0,\vep_1$ of $V$ so that 
\begin{align*}
(\vep,\vep)&=(\vep_0,\vep_0)=1, & 
(\vep_1,\vep_1)&=-1. 
\end{align*} 
Put 
\begin{align*}
V'&=\CC\vep+\CC\vep_1, & 
L&=\CC\vep, & 
L_0&=\CC\vep_0, & 
L_1&=\CC\vep_1, & 
L_2&=L+L_0. 
\end{align*} 
Define the unitary groups of $V$ and $V'$ by 
\begin{align*}
\U(V)&=\{g\in\End_\CC(V)\;|\;(gx_1,gx_2)=(x_1,x_2)\text{ for }x_1,x_2\in V\}, \\
\U(V')&=\{h\in\End_\CC(V')\;|\;(hy_1,hy_2)=(y_1,y_2)\text{ for }y_1,y_2\in V'\}. 
\end{align*}
We view $\U(V')$ as a subgroup of $\U(V)$. 
Let 
\begin{align*}
K&=\U(L_2)\times\U(L_1), & 
K'&=\U(L)\times\U(L_1)
\end{align*} 
be maximal compact subgroups of $\U(V)$ and $\U(V')$.
Let 
\begin{align*}
\pi&=D_{(\lam_1,\lam_2;\lam_3)}, & 
\sig&=D_{(\mu_1;\mu_2)}. 
\end{align*}
The minimal $K$-type of $\pi$ is $-(k_1,k_2;k_3):=(\lam_1,\lam_2+1;\lam_3-1)$. 
The minimal $K'$-type of $\sig$ is $-(k_1';k_2'):=\bigl(\mu_1+\frac{1}{2};\mu_2-\frac{1}{2}\bigl)$. 

We identify the space of the minimal $K$-type 
\[\frkH_{(\lam_1,\lam_2;\lam_3)}\simeq\rho_{(\lam_1,\lam_2+1)}\boxtimes\bvep^{\lam_3-1}=\rho_{(-k_1,-k_2)}\boxtimes\bvep^{-k_3} \index{$\rho_{\ulk}$}\] 
with $\call_{k_2-k_1}(\CC)$. 
Let $v_\pi^{(i)}\in\frkH_{(\lam_1,\lam_2;\lam_3)}$ and $\bar v_\pi^{(i)}\in\frkH_{(\lam_1,\lam_2;\lam_3)}^\vee$ be the vectors identified with 
\begin{align*}
v_\pi^{(i)}&=X^{k_2-k_1-i}Y^i, &
\bar v_\pi^{(j)}&=X^{\prime j}(-Y')^{k_2-k_1-j}. 
\end{align*}
Note that 
\begin{align*}
\pi(\iot(t_1,t_2))v_\pi^{(i)}&=t_1^{-k_1-i}t_2^{-k_3}v_\pi^{(i)}, &
\pi^\vee(\iot(t_1,t_2))\bar v_\pi^{(i)}&=t_1^{k_1+i}t_2^{k_3}\bar v_\pi^{(i)}
\end{align*} 
for $(t_1,t_2)\in K'$ and $i=0,1,2,\dots,k_2-k_1$. 
Let
\[\bfP_{\ulk}=(XY'-YX')^{k_2-k_1}\in\pi\otimes\pi^\vee \index{$\bfP_{\ulk}$}\]
be a $K$-invariant vector. 

Fix non-zero vectors $v_\sig\in\sig$ and $\bar v_{\sig^\vee}\in\sig^\vee$ such that for $t_1,t_2\in\CC^1$ 
\begin{align*}
\sig(t_1,t_2)v_\sig&=t_1^{-k_1'}t_2^{-k_2'}v_\sig, & 
\sig^\vee(t_1,t_2)\bar v_{\sig^\vee}&=t_1^{k_1'}t_2^{k_2'}\bar v_{\sig^\vee}. 
\end{align*} 
Let $^\natural:\pi^\natural\simeq\pi^\vee$ and $^\natural:\sig^\natural\simeq\sig^\vee$ be the equivariant isomorphisms determined by $v_\pi^{(0)\natural}=\bar v_{\pi^\vee}^{(0)}$ and $v_\sig^\natural=\bar v_{\sig^\vee}$. 
Fix perfect pairings $\ell_\pi:\pi\otimes\pi^\vee\to\CC$ and $\ell_\sig:\sig\otimes\sig^\vee\to\CC$. 
We consider the integral
\[J(W,W')=\int_{\U(V')}\ell_\pi(h,W\otimes W^\natural)\cdot\ell_{\sig^\vee}(h,W^{\prime\natural}\otimes W')\,\d h, \]
which converges absolutely for $W\in\pi$ and $W'\in\sig$. 
Set 
\begin{align*}
\frkv_\pi^{(i)}&=v_\pi^{(i)}\otimes\bar v_\pi^{(i)}, &
\frkv_\sig&=v_\sig\otimes\bar v_{\sig^\vee}, & 
\bar\frkv_{\sig^\vee}&=\bar v_{\sig^\vee}\otimes v_\sig. 
\end{align*} 

\begin{proposition}\label{prop:b1}
If the relation $\lam_1>\mu_1>\lam_2\geq\lam_3>\mu_2$ holds, then  
\[\frac{J(Y_+^{k_2'-k_3^{}}v_\pi^{(k_1'-k_1^{})},v_\sig)}{(2\pi)^{2n_1^*}\ell_{\ulk}^{}(\bfP_{\ulk}^{})\ell_{\sig^\vee}(\bar\frkv_{\sig^\vee}^{})}
=\Gam_\RR(2)^2\Gam_\RR(4)\frac{L\left(\frac{1}{2},\pi\times\sig^\vee\right)}{L(1,\pi,\mathrm{Ad})L(1,\sig^\vee,\mathrm{Ad})}. \]
\end{proposition}

\begin{remark}\label{rem:b1}
\begin{enumerate}
\item\label{rem:b11} If the interlacing relation $\lam_1>\mu_1>\lam_2>\lam_3>\mu_2$ holds, then $\sig$ must satisfy $\mu_1^{}-\mu_2^{}\geq 2$ and hence $k_2'-k_1'\geq 3$. 
Proposition \ref{prop:b1} includes the case $k_2'-k_1'=2$, 
and in this case, $\pi$ is a limit of discrete series with minimal $K$-type $-(k_1^{},k_1';k_2')$ (cf. Remark \ref{rem:b2}). 
\item\label{rem:b12} We consider the additive character $\addchar^\CC_{-2}(x)=e^{2\pi(x^c-x)}$. 
Since 
\[\vep\left(\frac{1}{2},\bvep^{2\kap},\addchar^\CC_{-2}\right)
=\begin{cases}
-1 &\text{if $\kap>0$, }\\
1 &\text{if $\kap<0$ }
\end{cases}\]
for $\kap\in\frac{1}{2}\ZZ\setminus\ZZ$, if $\lam_1>\mu_1>\lam_2\geq\lam_3>\mu_2$, then 
\[\vep\left(\frac{1}{2},\pi\times\sig^\vee,\addchar^\CC_{-2}\right)
=\prod_{i,j}\vep\left(\frac{1}{2},\bvep^{2(\lam_i-\mu_j)},\addchar^\CC_{-2}\right)=1. \]
\end{enumerate}
\end{remark}


\subsection{Weil representations}\label{ssec:b5}

We will take the auxiliary additive character $\addchar':\RR\to\CC^1$ defined by $\addchar'(x)=e^{-\sqrt{-1}x}$ for $x\in\RR$ so that the correspondences in \S \ref{ssec:b8} hold. 
This choice makes our computations simpler. 
Given a Hermitian space $(V,(\;,\;))$ of signature $(r,s)$ and a skew Hermitian space $(\calh,\La\;,\;\Ra)$ of dimension $m$, we let $\WW=V\otimes_\CC \calh$ be the symplectic space over $\RR$ of dimension $2(r+s)m$ equipped with the symplectic form
\[\ll v_1\otimes w_1,v_2\otimes w_2\gg=\Tr_{\CC/\RR}((v_1,v_2)\La w_1,w_2\Ra)\]
for $v_1,v_2\in V$ and $w_1,w_2\in \calh$. 
Let $Sp(\WW)$ be the symplectic group of $\WW$ and $\Mp(\WW)$ the metaplectic $\CC^1$-cover of $Sp(\WW)$. 

Once we choose characters $\chi_V,\chi_\calh:\CC^\times\to\CC^1$ such that $\chi_V|_{\RR^\times}=\sgn^{r+s}$ and $\chi_\calh|_{\RR^\times}=\sgn^m$, the natural homomorphism $\U(V)\times\U(\calh)\to Sp(\WW)$ has a lift $\iot_{V,\calh,\chi_V,\chi_\calh,\addchar'}:\U(V)\times\U(\calh)\to\Mp(\WW)$. 
Composing this with the Weil representation $\Ome_{\addchar'}$ of $\Mp(\WW)$ relative to $\addchar'$, we obtain a representation 
\[\Ome_{V,\calh,\chi_V,\chi_\calh,\addchar'}=\Ome_{\addchar'}\circ\iot_{V,\calh,\chi_V,\chi_\calh,\addchar'} \]
of $\U(V)\times\U(\calh)$. 
When $n$ is even, we will take $\chi_\calh=1$ and write 
\[\Ome_{V,\calh}^{\chi_V,\addchar'}=\Ome_{V,\calh,\chi_V,1,\addchar'}^{}. \]
Let $\ome_{V,\calh}^{\chi_V,\addchar'}\subset\Ome_{V,\calh}^{\chi_V,\addchar'}$ denote the Harish-Chandra module. 


\subsection{Fock model}\label{ssec:b6}

Assume that $n$ is even and that $\U(\calh)$ is compact. 
Let $\calp=\calp_{V,\calh}^{\chi_V,\addchar'}$ be the Fock model of the Weil representation $\ome_{V,\calh}^{\chi_V,\addchar'}$ of $\U(V)\times\U(\calh)$ relative to the data $(\chi_V,1,\addchar')$, where $\calp_{V,\calh}^{\chi_V,\addchar'}$ is the space of polynomials in $(r+s)m$ variables $x_{ik},y_{jk}$ $(1\leq i\leq r,\;1\leq j\leq s,\; 1\leq k\leq m)$. 


Let $\chi_V=\bvep^{n_0}$ with $n_0\equiv r+s\pmod 2$. 
Then 
\[\ome_{V,\calh}^{\bvep^{n_0},\addchar'}(h)\varPhi\begin{pmatrix} x \\ y \end{pmatrix}=(\det h)^{(r-s+n_0)/2}\varPhi\begin{pmatrix} xh \\ y\trs h^{-1} \end{pmatrix} \]
for $h\in\U(\calh)$, $x\in\Mat_{r,m}(\CC)$ and $y\in\Mat_{s,m}(\CC)$. 
Let $\eps=1$ or $\eps=-1$ according to whether the Hermitian form $\frac{1}{\sqrt{-1}}\La\;,\;\Ra$ is positive or negative definite. 
Then for $t=(t_1,t_2)\in\U(r)\times\U(s)$ 
\[\ome_{V,\calh}^{\bvep^{n_0},\addchar'}(t)\varPhi\begin{pmatrix} x \\ y \end{pmatrix}=(\det t_1)^{\eps m/2}(\det t_2)^{-\eps m/2}\varPhi\begin{pmatrix} \trs t_1x \\ t_2^{-1}y \end{pmatrix}. \]
Observe that for $\bar x\in\Mat_{r,m}(\CC)$ and $\bar y\in\Mat_{s,m}(\CC)$
\begin{align*}
(\ome_{V,\calh}^{\bvep^{n_0},\addchar'})^\vee(h)\varPhi'\left(\begin{bmatrix} \bar x \\ \bar y \end{bmatrix}\right)&=(\det h)^{(r-s-n_0)/2}\varPhi'\left(\begin{bmatrix} \bar x\trs h^{-1} \\ \bar yh\end{bmatrix}\right), \\
(\ome_{V,\calh}^{\bvep^{n_0},\addchar'})^\vee(t)\varPhi'\left(\begin{bmatrix} \bar x \\ \bar y \end{bmatrix}\right)&=(\det t_1)^{-\eps m/2}(\det t_2)^{\eps m/2}\varPhi'\left(\begin{bmatrix} t_1^{-1}\bar x \\ \trs t_2\bar y \end{bmatrix} \right)
\end{align*}
(see Lemma 5.2 of \cite{KK}). 
Define an invariant pairing 
\[\ell_{\ome_{V,\calh}^{\bvep^{n_0},\addchar'}}:\calp_{V,\calh}^{\bvep^{n_0},\addchar'}\otimes(\calp_{V,\calh}^{\bvep^{n_0},\addchar'})^\vee\to\CC\] 
by 
\[\ell_{\ome_{V,\calh}^{\bvep^{n_0},\addchar'}}(Q\otimes P)=\pi^{-(r+s)m}\int_{\Mat_{r+s,m}(\CC)}Q(z)P(\bar z)e^{-\tr(z\trs\bar z)}\d z\d\bar z. \] 
One can easily see that 
\[\ell_{\ome_{V,\calh}^{\bvep^{n_0},\addchar'}}(Q\otimes P)=\left[Q\begin{pmatrix} \frac{\partial}{\partial x_{ik}} \\ \frac{\partial}{\partial y_{jk}} \end{pmatrix}P\right]\begin{pmatrix} 0 \\ 0 \end{pmatrix},  \]
using polar coordinates. 
Following \cite[p.~38]{KV}, we put 
\[\Del_{ij}=\sum_{k=1}^m\frac{\partial^2}{\partial x_{ik}\partial y_{jk}}. \] 
The space of pluriharmonic polynomials is defined by 
\[\frkH_{V,\calh}^{\chi_V,\addchar'}=\{\calf\in\calp_{V,\calh}^{\chi_V,\addchar'}\;|\;\Del_{ij}\calf=0\;(1\leq i\leq r,\;1\leq j\leq s)\}. \]


\subsection{Theta correspondence}\label{ssec:b8}

We write $\HH$ for the Hamilton quaternion algebra over $\RR$. 
Recall that 
\begin{align*}
\HH&=\CC\oplus\CC\cdot {\boldsymbol j}, & 
\bj^2&=-1, & 
\bj^{-1}\cdot\sqrt{-1}\cdot\bj&=-\sqrt{-1}. 
\end{align*}
We denote the main involution of $\HH$ also by $\bar\cdot$ and set $\bnu(x)=x\bar x$ for $x\in\HH$. 
Define a skew Hermitian form $\La\;,\;\Ra$ on $\HH$ by letting $\La x,y\Ra$ be the projection of $\sqrt{-1}x\bar y$ onto the $\CC$-factor via the decomposition above.  
Then
\begin{align*}
\U(\HH)&\simeq(\SL(\HH)\times\CC^1)/\Del\mu_2, 
\end{align*}
where $\SL(\HH)=\{h\in\HH\;|\;\bnu(h)=1\}$, and $\Del\mu_2$ is the diagonally embedded copy of $\{\pm 1\}$ in $\HH^\times\times\CC^\times$. 

Let $r=2$, $s=1$ and $n_0=1$. 
Lemma 7.10 of \cite{Ichino} states that 
\begin{align*}
\calp^{\bvep,\addchar'}_{L_0,\HH}=&\bigoplus_{l=0}^\infty\bvep^{l+1}\boxtimes\rho_{(l+1,1)}, \\
\calp^{1,\addchar'}_{V',\HH}=&\bigoplus_{b,c\geq 0} D_{\bigl(b+\frac{1}{2};-c-\frac{1}{2}\bigl)}\boxtimes\rho_{(b,-c)}, \\
\calp^{\bvep,\addchar'}_{V,\HH}
=&\bigoplus_{a\geq a'\geq 1}D_{(a+1,a';0)}\boxtimes\rho_{(a+1,a'+1)}\oplus\bigoplus_{a,d\geq 0}D_{(a+1,0;-d)}\boxtimes\rho_{(a+1,1-d)}.  
\end{align*}

\begin{remark}\label{rem:b2}
In the decomposition above the representation $D_{(a+1,0;0)}$, which is a theta lift of $\rho_{(a+1,1)}$, is a limit of discrete series (cf. \cite{P2}). 
\end{remark}

The space $\calp_{V,\HH}^{\bvep,\addchar'}$ consists of polynomials in $z=\begin{bmatrix} x_{11} & x_{12} \\ x_{21} & x_{22} \\ y_1 & y_2 \end{bmatrix}$. 
Let $\frkg,\frkg',\frkk,\frkk'$ be the Lie algebras of $\U(V),\U(V'),K,K'$, respectively.  
Let 
\begin{align*}
\frkg_\CC&=\frkp_+\oplus\frkk_\CC\oplus\frkp_-, &
\frkg'_\CC&=\frkp'_+\oplus\frkk'_\CC\oplus\frkp'_-
\end{align*}
 be the Harish-Chandra decompositions. 
Recall that $\ome^{\bvep,\addchar'}_{V,\HH}(\frkp_-)$ is spanned by 
\begin{align*}
\Del_{1,1}&=\frac{\partial^2}{\partial x_{11}\partial y_1}+\frac{\partial^2}{\partial x_{12}\partial y_2}, &
\Del_{2,1}&=\frac{\partial^2}{\partial x_{21}\partial y_1}+\frac{\partial^2}{\partial x_{22}\partial y_2}.  
\end{align*}
We write $\frkH_{(\lam^+;\lam^-)}$ for the space of minimal $K$-type of $D_{(\lam^+;\lam^-)}$. 
Then 
\begin{align*}
\frkH^{1,\addchar'}_{V',\HH}=&\bigoplus_{b,c\geq0}\frkH_{\bigl(b+\frac{1}{2};-c-\frac{1}{2}\bigl)}\boxtimes\rho_{(b,-c)}, \\
\frkH_{V,\HH}^{\bvep,\addchar'}
=&\bigoplus_{a\geq a'\geq 1}\frkH_{(a+1,a';0)}\boxtimes\rho_{(a+1,a'+1)}\oplus\bigoplus_{a,d\geq 0}\frkH_{(a+1,0;-d)}\boxtimes\rho_{(a+1,1-d)}. 
\end{align*}

Following (4.14) of \cite{KK}, we put 
\begin{align*}
U_{2,1}&=\left[\begin{array}{cc|c} 0 & 0 & 0 \\ 1 & 0 & 0 \\ \hline 0 & 0 & 0 \end{array}\right]\in\frkk_\CC, & 
Y_+&=Y_{2,1}=\left[\begin{array}{cc|c} 0 & 0 & 0 \\ 0 & 0 & 1 \\ \hline 0 & 0 & 0 \end{array}\right]\in\frkp_+, \\
U_{1,2}&=\left[\begin{array}{cc|c} 0 & -1 & 0 \\ 0 & 0 & 0 \\ \hline 0 & 0 & 0 \end{array}\right]\in\frkk_\CC, & 
X_+&=X_{1,2}=\left[\begin{array}{cc|c} 0 & 0 & 0 \\ 0 & 0 & 0 \\ \hline 0 & 1 & 0 \end{array}\right]\in\frkp_-. 
\end{align*}
Proposition 4.4(iii) of \cite{KK} shows that 
\begin{align*}
\ome^{\bvep,\addchar'}_{V,\HH}(Y_+)&=-(x_{21}y_1+x_{22}y_2), &
\ome^{\bvep,\addchar'}_{V,\HH}(U_{2,1})&=x_{21}\frac{\partial}{\partial x_{11}}+x_{22}\frac{\partial}{\partial x_{12}}, \\
\ome^{\bvep,\addchar'}_{V,-\HH}(X_+)&=-(\bar x_{21}\bar y_1+\bar x_{22}\bar y_2), &
\ome^{\bvep,\addchar'}_{V,-\HH}(U_{1,2})&=\bar x_{21}\frac{\partial}{\partial\bar x_{11}}+\bar x_{22}\frac{\partial}{\partial\bar x_{12}}.    
\end{align*}


\subsection{Formal degree}\label{ssec:b9}

Let $\tau$ be an irreducible square-integrable representation of a unitary group $\U(\calh)$. 
Fix a $\U(\calh)$-invariant pairing $\ell_\tau:\tau\otimes\tau^\vee\to\CC$. 
Given $\calf_1\otimes \calf_2\in\tau\otimes\tau^\vee$, we have a matrix coefficient
\[\ell_\tau(h,\calf_1\otimes\calf_2)=\ell_\tau(\tau(h)\calf_1\otimes\calf_2). \] 
The formal degree $\d(\tau)$ of $\tau$ is a positive real number characterized by 
\[\int_{\U(\calh)}\ell_\tau(h,\calf_1^{}\otimes\calf_2^{})\ell_{\tau_1^\vee}(h,\calf_1'\otimes\calf_2')\;\d h=\d(\tau)^{-1}\ell_\tau(\calf_1^{}\otimes\calf_1')\ell_{\tau_1^\vee}(\calf_2^{}\otimes\calf_2')\]
for $\calf_1^{}\otimes\calf_2^{}\in\tau\otimes\tau^\vee$ and $\calf_1'\otimes\calf_2'\in\tau^\vee\otimes\tau$. 
 
When $\U(\calh)$ is compact and $\d h$ has total volume 1, we have 
\[\d(\tau)=\dim\tau. \]
When $\sig=D_{(\mu_1;\mu_2)}$ and $\d h$ is defined by (\ref{tag:74}), we have  
\[\d(\sig)=(4\pi)^{-1}(\mu_1-\mu_2) \]
by Proposition A.7 of \cite{KL19} and its third remark. 




\subsection{Representations of $\U(2)$}\label{ssec:b10}

Fix a commutative integral domain $A$ of characteristic zero. 
Let $\calp'(A)$ be the set of polynomials in $\begin{bmatrix} x_{11} & x_{12} \\ y_1 & y_2 \end{bmatrix}$. 
The decomposition $\calp'(A)=\oplus_{b,c\geq 0}^{}\calp'_{b,c}(A)$ holds, where the submodule $\calp'_{b,c}(A)$ of $\calp'(A)$ consists of homogeneous polynomials of degree $b$ in $x_{11},x_{12}$ and of degree $c$ in $y_1,y_2$. 
The group $\GL_2(A)$ acts on the module $\calp'_{b,c}(A)$ by  
\[\rho(h)P\left(\begin{matrix} x \\ y \end{matrix}\right)=P\left(\begin{matrix} xh \\ y\trs h^{-1}\end{matrix}\right). \]
Define a submodule $\calt'_{b,c}(A)$ of $\calp'_{b,c}(A)$ by
\[\calt'_{b,c}(A)=(x_{11}y_1+x_{12}y_2)\calp'_{b-1,c-1}(A). \]
Since $\calt'_{b,c}(A)$ is stable under the action of $\GL_2(A)$, the group $\GL_2(A)$ acts on the quotient module $\frkH'_{b,c}(A)=\calp'_{b,c}(A)/\calt'_{b,c}(A)$. 

We define a bilinear form 
\[l_{\calp'}:\calp'(A)\otimes\calp'(A)\to A\] 
by 
\[l_{\calp'}(P\otimes Q)=\left[Q\begin{pmatrix} \frac{\partial}{\partial y_1} & \frac{\partial}{\partial y_2} \\ \frac{\partial}{\partial x_{11}} & \frac{\partial}{\partial x_{12}} \end{pmatrix}P\right]\begin{pmatrix} 0 & 0 \\ 0 & 0 \end{pmatrix}. \]
More explicitly, the bilinear form $l_{\calp'}$ is given by 
\[l_{\calp'}\biggl(\prod_{j=1}^2x_{1j}^{n_j^{}}y_j^{m_j^{}}\otimes x_{1j}^{\prime n_j'}y_j^{\prime m_j'}\biggl)=\begin{cases}
\prod_{j=1}^2n_j!m_j! &\text{if $n_j^{}=m_j'$ and $m_j^{}=n_j'$, }\\
0 &\text{otherwise. }
\end{cases}\]

Put $\frkH_{c,b}^{\prime\vee}(A)=\{P\in\calp_{b,c}'(A)\;|\;\Del' P=0\}$, where 
\[\Del'=\frac{\partial^2}{\partial x_{11}\partial y_1}+\frac{\partial^2}{\partial x_{12}\partial y_2}. \]
If $A$ is a field, then since 
\[\frkH_{c,b}^{\prime\vee}(A)=\{P\in\calp'_{b,c}(A)\;|\;l_{\calp'}(P\otimes Q)=0\text{ for }Q\in\calt'_{c,b}(A)\},  \]
the restriction of $\rho$ to $\frkH_{c,b}^{\prime\vee}(A)$ is the contragredient representation of $\frkH_{c,b}'(A)$. 
The linear form $l_{\calp'}$ induces a perfect pairing 
\[l_{\calp'}:\frkH_{c,b}^{\prime\vee}(A)\otimes\frkH'_{c,b}(A)\to A. \]

\begin{remark}\label{rem:b3}
When we view $\frkH_{c,b}^{\prime\vee}(\CC)$ as a representation of $\U(2)$, it is a theta lift of the discrete series of $\U(1,1)$ with Harish-Chandra parameter $\bigl(b+\frac{1}{2};c-\frac{1}{2}\bigl)$ (see \S \ref{ssec:b8}). 
Hence $\frkH'_{c,b}(\CC)$ is an irreducible representation of $\GL_2(\CC)$ with highest weight $(c,-b)$. 
\end{remark}

One can apply the same computation as in the proofs of Lemmas A.2 and A.4 of \cite{HY} to prove the following lemma. 

\begin{lemma}\label{lem:b1} 
\begin{enumerate}
\item\label{lem:b11} For $P\in\calp'(A)$, $Q\in\calp'(A)$ and $h\in\GL_2(A)$  
\[l_{\calp'}(\rho(h)Q\otimes\rho(h) P)
=l_{\calp'}(Q\otimes P). \]
\item\label{lem:b12} $\calp'_{b,c}(\QQ)$ is a direct sum of $\frkH_{c,b}^{\prime\vee}(\QQ)$ and $\calt'_{b,c}(\QQ)$. 
\end{enumerate}
\end{lemma}

Let $\bfP:\calp'_{b,c}(\QQ)\to\frkH_{c,b}^{\prime\vee}(\QQ)$ be the projection given by Lemma \ref{lem:b1}(\ref{lem:b12}). 
Define a homomorphism $\wp:\calp'_{c,b}(A)\to \call_{b+c}(A)$ by 
\[\wp(P)=P\begin{pmatrix} X & Y \\ -Y & X \end{pmatrix}. \]
Since $\wp(\rho(h)P)=(\det h)^{-b}\rho_{b+c}(h)\wp(P)$, it factors though an equivariant morphism $\frkH'_{c,b}(A)\to\rho_{(c,-b)}(A)$ and its restriction gives an equivariant morphism $\frkH_{b,c}^{\prime\vee}(A)\to\rho_{(c,-b)}(A)$. 

In view of (\ref{tag:41}) we normalize perfect pairings $\ell_\pi,\ell_\sig$ and $\ell_{\sig^\vee}$ by 
\beq
\ell_\pi(\frkv_\pi^{(0)})=\ell_\sig(\frkv_\sig^{})=\ell_{\sig^\vee}(\bar\frkv_{\sig^\vee}^{})=1. \label{tag:b2}
\eeq  
Then 
\beq
\ell_\pi(\frkv_\pi^{(i)})=\binom{k_2-k_1}{i}^{-1}. \label{tag:b3}
\eeq 

\begin{lemma}\label{lem:b2}
For $P\in\calp'_{b,c}(\QQ)$ and $Q\in\calp'_{c,b}(\QQ)$ we have 
\[l_{\calp'}(\bfP(P)\otimes Q)=(-1)^cb!c!\ell_{b+c}(\wp(P)\otimes\wp(Q)). \]  
\end{lemma}

\begin{proof}
Since both sides give invariant pairings on $\rho_{(b,-c)}\otimes\rho_{(c,-b)}$, they are proportional. 
Let $P=x_{11}^by_2^c\in\frkH_{c,b}^{\prime\vee}(\QQ)$ and $Q=x_{12}^cy_1^b\in\frkH_{b,c}^{\prime\vee}(\QQ)$. 
Then 
\begin{align*}
l_{\calp'}(P\otimes Q)&=b!c!, &
\wp(P)&=X^{b+c}, &
\wp(Q)&=Y^{\prime c}(-Y)^{\prime b}. 
\end{align*}
Since $\ell_{b+c}(\wp(P)\otimes\wp(Q))=(-1)^c$, the proof is complete.  
\end{proof}

We define the isomorphism $^\natural:\calp_{b,c}'(A)\to\calp_{c,b}'(A)$ by 
\[P^\natural\begin{pmatrix} x_{11} & x_{12} \\ y_1 & y_2 \end{pmatrix}=P\begin{pmatrix} y_1' & y_2' \\ x_{11}' & x_{12}' \end{pmatrix} \]
and the action $\rho^\natural$ by $\rho^\natural(h)=\rho(\trs h^{-1})$ of $\GL_2(A)$. 
Then 
\begin{align*}
(\rho^\natural(h)P)^\natural&=\rho(h)P^\natural, &
\ell_{\ome_{V',\HH}^{1,\addchar'}}(P\otimes Q)&=l_{\calp'}(P\otimes Q^\natural). 
\end{align*}

\begin{lemma}\label{lem:b3} 
Let $P,Q\in\calp'_{b,c}(\CC)$.  
We view $P$ (resp. $Q$) as an element of $\calp^{1,\addchar'}_{V',\HH}$ (resp. $(\calp^{1,\addchar'}_{V',\HH})^\vee$). 
Put $\sig=D_{\bigl(b+\frac{1}{2};-c-\frac{1}{2}\bigl)}$. 
Then for $h\in\U(\HH)$ 
\begin{align*}
&\int_{\U(V')}\ell_{\ome^{1,\addchar'}_{V',\HH}}((g,h),P\otimes Q)\frac{\ell_{\sig^\vee}(g,\bar\frkv_{\sig^\vee})}{b!c!}\,\d g=\frac{(-1)^c}{\d(\sig)}\ell_{b+c}(h,\wp(P)\otimes\wp(Q^\natural)). 
\end{align*}
\end{lemma}

\begin{proof}
Define the character $\chi$ of $K'$ by $\chi(t_1,t_2)=t_1^{b+1}t_2^{-c-1}$. 
Given a representation $(\vPi,V)$ of $\U(1,1)$, we write 
\[V^\chi=\{v\in V\;|\;\vPi(k)v=\chi(k)v\text{ for }k\in K'\}. \]
Recall that 
\[\calp'_{b,c}(\CC)
=\calp'(\CC)^\chi
=\bigoplus_{i=0}^m \sig_i^\chi\boxtimes\rho_{(b_i,-c_i)}, \]
where $\sig_i=D_{\bigl(b_i+\frac{1}{2};-c_i-\frac{1}{2}\bigl)}$ are discrete series of $\U(V')$ such that $\sig_i^\chi\neq\{0\}$. 
Set $\sig_0=\sig$. 
We write $P=\sum_{i=0}^mP_i$ with $P_i\in\sig_i^\chi\boxtimes\rho_{(b_i,-c_i)}$. 
Since $\chi$ is the minimal $K$-type of $\sig_0$, we have $P_0=\bfP(P)$. 
Lemma \ref{lem:b2} and (\ref{tag:b2}) give 
\begin{align*}
\ell_{\ome^{1,\addchar'}_{V',\HH}}((g,h),P_0\otimes Q)
&=\ell_\sig(g,\frkv_\sig)l_{\calp'}(h,P_0\otimes Q^\natural)\\
&=(-1)^cb!c!\ell_\sig(g,\frkv_\sig)\ell_{b+c}(h,\wp(P)\otimes\wp(Q^\natural)). 
\end{align*}
The proof is complete by the definition of the formal degree. 
\end{proof}


\subsection{Reduction to the trilinear form}\label{ssec:b11}

To prove Proposition \ref{prop:b1}, we may assume that $\lam_2=0$, replacing $\pi$ by $\pi\otimes\bvep^{-\lam_2}$ and $\sig$ by $\sig\otimes\bvep^{-\lam_2}$. 
Recall the interlacing relation: 
\[\lam_1^{}=-k_1^{}>\mu_1^{}=-k_1'-\frac{1}{2}>\lam_2^{}=-k_2^{}-1=0\geq\lam_3^{}=1-k_3^{}>\mu_2^{}=-k_2'+\frac{1}{2}. \]
Equivalently, we rewrite it as 
\begin{align*}
&k_1^{}\leq k_1'\leq k_2^{}=-1, & 
&1\leq k_3^{}\leq k_2'. 
\end{align*}
We include the case $\lam_2=\lam_3$, i.e., $k_3-k_2=2$. 

Define non-negative integers $n_1,n_2,n_3$ by 
\begin{align*}
n_1&=\lam_1-\lam_3-1, \\
n_2&=\mu_1-\mu_2-1, \\
n_3&=\lam_1+\lam_3-\mu_1-\mu_2-1. 
\end{align*}
Observe that 
\begin{align*}
n_1+n_2&\geq n_3, & 
n_2+n_3&\geq n_1, & 
n_3+n_1&\geq n_2.  
\end{align*}
Define non-negative integers by $n_i^*=\frac{n_1+n_2+n_3}{2}-n_i$. 
Note that 
\begin{align*}
n_1^*&=k_2'-k_3^{}, &
n_2^*&=k_1'-k_1^{}. 
\end{align*}
It is worth noting that 
\begin{align*}
n_1^*+n_2^*&=n_3^{}, & 
n_2^*+n_3^*&=n_1^{}, & 
n_3^*+n_1^*&=n_2^{}.  
\end{align*}

Define non-negative integers $b$ and $c$ by
\begin{align*}
b&=-k_1'-1, & c&=k_2'-1. 
\end{align*}
It follows from (\ref{tag:b1}) that 
\begin{multline*}
\frac{L\left(\frac{1}{2},\pi\times\sig^\vee\right)}{L(1,\pi,\mathrm{Ad})L(1,\sig^\vee,\mathrm{Ad})}\\
=4\pi^5\frac{\Gam\bigl(\frac{n_1+n_2+n_3}{2}+2\bigl)\Gam(n_1^*+1)\Gam(n_2^*+1)\Gam(n_3^*+1)\Gam(-k_1')\Gam(k_2')}{(2\pi)^{2n_1^*}\Gam(n_1+2)\Gam(n_2+2)\Gam(1-k_1)\Gam(k_3)}. 
\end{multline*}
Define the representations $\tau_1,\tau_2,\tau_3$ of $\U(\HH)$ by 
\begin{align*}
\tau_1&=\rho_{(\lam_1,\lam_3+1)}, & 
\tau_2&=\rho_{(-k_1'-1,-k_2'+1)}, & 
\tau_3&=\rho_{(n_3+1,1)}.  
\end{align*}

Given an irreducible representation $\tau$ of dimension $n+1$, we put 
\begin{align*}
v_\tau^{(i)}&=X^{n-i}Y^i, &
\bar v_\tau^{(j)}&=X^{\prime j}(-Y')^{n-j}, &
\frkv_\tau^{(i,j)}&=v_\tau^{(i)}\otimes \bar v_{\tau^\vee}^{(j)}. 
\end{align*}
Recall that  
\beq
\ell_\tau^{}(\frkv_\tau^{(i,j)})=\binom{n}{i}^{-1}\del_{i,j} \label{tag:b4}
\eeq
by (\ref{tag:41}), where $\del_{i,j}$ denotes the Kronecker's delta. 

We now consider the following seesaw: 

\begin{picture}(300,70)
\put(191,10){$\U(V')\times \U(L_0)$}
\put(210,55){$\U(V)$}
\put(270,55){$\pi$}
\put(270,10){$\sig^\vee$}
\put(110,10){$\U(\HH)$}
\put(122,50){\line(0,-1){30}}
\put(90,55){$\U(\HH)\times\U(\HH)$}
\put(45,55){$\tau_2\otimes\tau_3$}
\put(60,10){$\tau^\vee_1$}
\put(141,20){\line(2,1){67}}
\put(142,50){\line(2,-1){58}}
\put(225,20){\line(0,1){30}}
\end{picture}

The Ichino-Ikeda integral is related to the trilinear form.  

\begin{lemma}\label{lem:b4}
Notations and assumptions being as above, we have 
\begin{align*}
J(Y_+^{n_1^*}v_\pi^{(n_2^*)},\bar v_{\sig^\vee}^{})
=&\frac{b!c!n_3!(n_1+1)}{\d(\sig)(k_3-1)!(-k_1-1)!}\sum_{i,j=0}^{n_1^*}(-1)^{i+j}\binom{n_1^*}{i}\binom{n_1^*}{j}\\
&\times\int_{\SL(\HH)}\ell_{\tau_1^\vee}(h,\frkv_{\tau_1^\vee}^{(n_1,n_1)})\ell_{\tau_2}(h,\frkv_{\tau_2}^{(i,j)})\ell_{\tau_3}(h,\frkv_{\tau_3}^{(n_1^*-i,n_1^*-j)})\,\d h. 
\end{align*}
\end{lemma}

\begin{proof}
We write $-\HH$ for the space $\HH$ equipped with the skew Hermitian form $-\La\;,\;\Ra$. 
We employ the idea of \cite{HX}, which uses the theta correspondence. 
Let 
\begin{align*}
w^{}_0&=x_{11}^{\lam_1-1}y_2^{-\lam_3}\in \frkH_{(\lam_1,0;\lam_3)}\boxtimes\tau_1\subset\frkH^{\bvep,\addchar}_{V,\HH}, \\
\bar w_0&=\bar x_{11}^{\lam_1-1}\bar y_2^{-\lam_3}\in \frkH_{(\lam_1,0;\lam_3)}^\vee\boxtimes\tau_1^\vee\subset\frkH^{\bvep^{-1},\addchar}_{V,-\HH} 
\end{align*} 
(cf. Proposition 6.1 of \cite{KV}). 
Put 
\begin{align*}
w^{}_{n_2^*}&=\frac{(-k_1')!}{(\lam_1-1)!}\ome^{\bvep,\addchar}_{V,\HH}(U_{2,1})^{n_2^*}w^{}_0=x_{11}^{-k_1'-1}x_{21}^{n_2^*}y_2^{k_3-1}\in \frkH_{(\lam_1,0;\lam_3)}\boxtimes\tau_1, \\
\bar w^{}_{n_2^*}&=\frac{(-k_1')!}{(\lam_1-1)!}\ome^{\bvep^{-1},\addchar}_{V,-\HH}(U_{1,2})^{n_2^*}\bar w^{}_0=\bar x_{11}^{-k_1'-1}\bar x_{21}^{n_2^*}\bar y_2^{k_3-1}\in \frkH_{(\lam_1,0;\lam_3)}^\vee\boxtimes\tau_1^\vee, \\
\upsilon_{n_1^*,n_2^*}&=\ome^{\bvep,\addchar}_{V,\HH}(Y_+)^{n_1^*}w_{n_2^*}=(-x^{}_{21}y^{}_1-x^{}_{22}y^{}_2)^{n_1^*}x_{11}^{-k_1'-1}x_{21}^{n_2^*}y_2^{k_3-1}\in \pi\boxtimes\tau_1, \\
\bar \upsilon_{n_1^*,n_2^*}&=\ome^{\bvep,\addchar}_{V,-\HH}(X_+)^{n_1^*}\bar w_{n_2^*}=(-\bar x^{}_{21}\bar y^{}_1-\bar x^{}_{22}\bar y^{}_2)^{n_1^*}\bar x_{11}^{-k_1'-1}\bar x_{21}^{n_2^*}\bar y_2^{k_3-1}\in \pi^\vee\boxtimes\tau_1^\vee, \\
\Ups_{n_1^*,n_2^*}&=\upsilon_{n_1^*,n_2^*}\otimes\bar\upsilon_{n_1^*,n_2^*}. 
\end{align*}
Note that $\bar\upsilon_{n_1^*,n_2^*}\in\calp_{b,n_3,c}(\CC)$. 
Given $Q\otimes P\in\calp_{b,n_1^*,c}(\CC)\otimes\calp_{b,n_3,c}(\CC)$, we set 
\[(Q\otimes P)^\natural=Q\otimes P^\natural\in\calp_{b,n_1^*,c}(\CC)\otimes\calp_{c,n_3,b}(\CC). \]

Since $\ell_{\ome^{\bvep,\addchar}_{V,\HH}}(w_{n_2^*}\otimes\bar w_{n_2^*})=b!n_2^*!(k_3^{}-1)!$, 
\[\ell_{\ome^{\bvep,\addchar}_{V,\HH}}((g,h),w_{n_2^*}\otimes\bar w_{n_2^*})=b!n_2^*!(k_3^{}-1)!\binom{-k_1-1}{n_2^*}\ell_\pi(g,\frkv_\pi^{(n_2^*)})\ell_{\tau_1}(h,\frkv_{\tau_1}^{(0,0)}) \]
by (\ref{tag:b3}) and (\ref{tag:b4}), and hence 
\[\ell_{\ome^{\bvep,\addchar}_{V,\HH}}((g,h),\Ups_{n_1^*,n_2^*})=(k_3-1)!(-k_1-1)!\ell_\pi(g,Y_+^{n_1^*}v_\pi^{(n_2^*)}\otimes Y_+^{n_1^*}\bar v_\pi^{(n_2^*)})\ell_{\tau_1}(h,\frkv_{\tau_1}^{(0,0)}). \]
By the definition of the formal degree of $\tau_1$, for $g\in\U(V)$ we have
\[\int_{\U(\HH)}\ell_{\tau_1^\vee}(h,\frkv_{\tau_1^\vee}^{(n_1,n_1)})\frac{\ell_{\ome^{\bvep,\addchar}_{V,\HH}}((g,h),\Ups_{n_1^*,n_2^*})}{(k_3-1)!(-k_1-1)!}\,\d h
=\frac{\ell_\pi(g,Y_+^{n_1^*}v_\pi^{(n_2^*)}\otimes Y_+^{n_1^*}\bar v_\pi^{(n_2^*)})}{\d(\tau_1)}, \]
where $\d(\tau_1)=n_1+1$.  
We get 
\begin{align*}
&(k_3-1)!(-k_1-1)!\frac{J(Y_+^{n_1^*}v_\pi^{(n_2^*)},\bar v_{\sig^\vee})}{n_1+1}\\
=&\int_{\U(V')}\int_{\U(\HH)}\ell_{\tau_1^\vee}(h,\frkv_{\tau_1^\vee}^{(n_1,n_1)})\cdot\ell_{\ome^{\bvep,\addchar}_{V,\HH}}((g,h),\Ups_{n_1^*,n_2^*})\,\d h\cdot\ell_{\sig^\vee}(g,\bar\frkv_{\sig^\vee})\,\d g\\
=&\int_{\U(\HH)}\ell_{\tau_1^\vee}(h,\frkv_{\tau_1^\vee}^{(n_1,n_2)})\int_{\U(V')}\ell_{\ome^{\bvep,\addchar}_{V,\HH}}((g,h),\Ups_{n_1^*,n_2^*})\cdot\ell_{\sig^\vee}(g,\bar\frkv_{\sig^\vee})\,\d g\d h,  
\end{align*}
switching the order of integration. 

Recall the orthogonal decomposition $V=V'\oplus L_0$ and  
\begin{align*}
\ome^{\bvep,\addchar}_{V,\HH}|_{\U(V')\times\U(\HH)}&\simeq\ome^{1,\addchar}_{V',\HH}\otimes \ome^{\bvep,\addchar}_{L_0,\HH}, & 
(\ome^{1,\addchar}_{V',\HH})^\vee&\simeq \ome^{1,\addchar}_{V',-\HH}. 
\end{align*}
We write $\upsilon_{n_1^*,n_2^*}^{}=\sum_{i=0}^{n_1^*}\binom{n_1^*}{i}\upsilon_i'\otimes\upsilon''_i$ and $\bar \upsilon_{n_1^*,n_2^*}=\sum_{j=0}^{n_1^*}\binom{n_1^*}{j}\bar\upsilon_j'\otimes\bar\upsilon''_j$ with  
\begin{align*}
\upsilon_i'&=x_{11}^by_1^iy_2^{c-i}\in\calp'_{b,c}(\CC), &
\upsilon''_i&=x_{21}^{n_2^*+i}x_{22}^{n_1^*-i}\in\bvep^{n_3+1}\boxtimes\tau_3\subset\calp^{\bvep,\addchar}_{L_0,\HH}, \\
\bar\upsilon_j'&=\bar x_{11}^b\bar y_1^j\bar y_2^{c-j}\in\calp'_{b,c}(\CC), & 
\bar\upsilon''_j&=\bar x_{21}^{n_2^*+j}\bar x_{22}^{n_1^*-j}\in\bvep^{-n_3-1}\boxtimes\tau_3^\vee\subset\calp^{\bvep,\addchar}_{L_0,-\HH}.  
\end{align*}
Then $\ell_{\ome^{\bvep,\addchar}_{V,\HH}}((g,h),\Ups_{n_1^*,n_2^*})$ is equal to  
\[\sum_{i,j=0}^{n_1^*}\binom{n_1^*}{i}\binom{n_1^*}{j}\ell_{\ome^{1,\addchar}_{V',\HH}}((g,h),\upsilon_i'\otimes\bar\upsilon_j')\ell_{\ome^{\bvep,\addchar}_{L_0,\HH}}(h,\upsilon''_i\otimes\bar\upsilon''_j) \]
for $g\in\U(V')$ and $h\in\U(\HH)$. 
Observe that 
\[\ell_{\ome^{\bvep,\addchar}_{L_0,\HH}}(h,\upsilon''_i\otimes\bar\upsilon''_j)=n_3!\ell_{\tau_3}(h,\frkv_{\tau_3}^{(n_1^*-i,n_1^*-j)}). \] 

Since 
\begin{align*}
\wp(\upsilon'_i)&=X^{b+c-i}(-Y)^i=(-1)^iv_{\tau_2}^{(i)}, \\
\wp(\bar\upsilon^{\prime\natural}_j)&=\wp(y_1^bx_{11}^jx_{12}^{c-j})
=(-1)^bX^{\prime j}Y^{\prime c-j+b}=(-1)^{c-j}\bar v_{\tau_2}^{(j)}, 
\end{align*}
Lemma \ref{lem:b3} gives 
\begin{align*}
&\int_{\U(V')}\ell_{\ome^{1,\addchar}_{V',\HH}}((g,h),\upsilon'_i\otimes\bar\upsilon'_j)\frac{\ell_{\sig^\vee}(g,\bar\frkv_{\sig^\vee})}{b!c!}\,\d g
=\frac{(-1)^{i+j}}{\d(\sig)}\ell_{\tau_2}(h,\frkv_{\tau_2}^{(i,j)}). 
\end{align*} 
We conclude that 
\begin{multline*}
\int_{\U(V')}\ell_{\ome^{\bvep,\addchar}_{V,\HH}}((g,h),\Ups_{n_1^*})\frac{\ell_{\sig^\vee}(g,\bar\frkv_{\sig^\vee})}{b!c!n_3!}\,\d g\\
=\frac{1}{\d(\sig)}\sum_{i,j=0}^{n_1^*}(-1)^{i+j}\binom{n_1^*}{i}\binom{n_1^*}{j}\ell_{\tau_2}(h,\frkv_{\tau_2}^{(i,j)})\ell_{\tau_3}(h,\frkv_{\tau_3}^{(n_1^*-i,n_1^*-j)}) 
\end{multline*}
for $h\in\U(\HH)$. 
Hence $\frac{\d(\sig)(k_3-1)!(-k_1-1)!}{b!c!n_3!(n_1+1)}J(Y_+^{n_1^*}\frkv_\pi^{(n_2^*)},\bar\frkv_{\sig^\vee})$ is the sum of 
\[(-1)^{i+j}\binom{n_1^*}{i}\binom{n_1^*}{j}\int_{\U(\HH)}\ell_{\tau_1^\vee}(h,\frkv_{\tau_1^\vee}^{(n_1,n_1)})\ell_{\tau_2}(h,\frkv_{\tau_2}^{(i,j)})\ell_{\tau_3}(h,\frkv_{\tau_3}^{(n_1^*-i,n_1^*-j)})\,\d h. \]
Since the product of the central characters of $\tau_2$ and $\tau_3$ coincides with the central character of $\tau_1$, we can replace $\U(\HH)$ by $\SL(\HH)$. 
\end{proof}


\subsection{Computation of the trilinear form}\label{ssec:b12}
Put 
\begin{align*}
\call_{\underline{n}}(A)&=\call_{n_1}(A)\otimes \call_{n_2}(A)\otimes \call_{n_3}(A), & 
\ell_{\underline{n}}&=\ell_{n_1}\otimes\ell_{n_2}\otimes\ell_{n_3}.  
\end{align*}
Therefore 
\[\dim_\CC\Hom_{\HH^\times}(\tau_1\otimes\tau_2\otimes\tau_3,\CC)
=\dim_\CC\Hom_{\SL(\HH)}(\rho_{\underline{n}},\CC)=1. \]
Define $P_{\underline{n}}\in \call_{\underline{n}}(\CC)$ by 
\[P_{\underline{n}}=(X_1Y_2-X_2Y_1)^{n_3^*}(X_3Y_1-X_1Y_3)^{n_2^*}(X_2Y_3-X_3Y_2)^{n_1^*}. \]
Then $P_{\underline{n}}$ is a basis of the line $\call_{\underline{n}}(\CC)$ fixed by $\SL(\HH)$.\footnote{There are misprints in (4.9) of \cite{MH}. } 

\begin{lemma}\label{lem:b5}
Notations and assumptions being as above, we have 
\[\int_{\SL(\HH)}\ell_{\tau_1^\vee}(h,\frkv_{\tau_1^\vee}^{(n_1,n_1)})\ell_{\tau_2}(h,\frkv_{\tau_2}^{(i,j)})\ell_{\tau_3}(h,\frkv_{\tau_3}^{(n_1^*-i,n_1^*-j)})\,\d h
=\frac{(-1)^{i+j}a_ia_j}{\ell_{\underline{n}}(P_{\underline{n}}\otimes P_{\underline{n}})}, \]
where $a_i=\binom{n_1^*}{i}\binom{n_2}{i}^{-1}\binom{n_3}{n_1^*-i}^{-1}$. 
\end{lemma}

\begin{proof}
Recall that 
\begin{multline*}
\frkv_{\tau^\vee_1}^{(n_1,n_1)}\otimes\frkv_{\tau_2}^{(i,j)}\otimes\frkv_{\tau_3}^{(n_1^*-i,n_1^*-j)} \\
=(-1)^{\frac{n_1+n_2+n_3}{2}}Y_1^{n_1}X_2^{n_2-i}Y_2^iX_3^{n_2^*+i}Y_3^{n_1^*-i}\otimes X_1^{n_1}X_2^jY_2^{n_2-j}X_3^{n_1^*-j}Y_3^{n_2^*+j}. 
\end{multline*}
Since $\call_{\underline{n}}(\CC)^{\SL(\HH)}=\CC P_{\underline{n}}$, the left hand side equals 
\begin{align*}
\frac{\ell_{\underline{n}}(Y_1^{n_1}X_2^{n_2-i}Y_2^iX_3^{n_2^*+i}Y_3^{n_1^*-i}\otimes P_{\underline{n}})\ell_{\underline{n}}(P_{\underline{n}}\otimes X_1^{n_1}X_2^jY_2^{n_2-j}X_3^{n_1^*-j}Y_3^{n_2^*+j})}{(-1)^{\frac{n_1+n_2+n_3}{2}}\ell_{\underline{n}}(P_{\underline{n}}\otimes P_{\underline{n}})}
\end{align*}
(cf. (4.20) of \cite{MH}). 
Since 
\begin{align*}
P_{\underline{n}}
=&\sum_{l_1=0}^{n_1^*}\sum_{l_2=0}^{n_2^*}\sum_{l_3=0}^{n_3^*}\binom{n_1^*}{l_1}\binom{n_2^*}{l_2}\binom{n_3^*}{l_3}(-1)^{n_1^*+n_2^*+n_3^*-l_1^{}-l_2^{}-l_3^{}}\\
&\times X_1^{n_2^*-l_2^{}+l_3^{}}Y_1^{n_3^*+l_2^{}-l_3^{}}X_2^{n_3^*+l_1^{}-l_3^{}}Y_2^{n_1^*-l_1^{}+l_3^{}}X_3^{n_1^*-l_1^{}+l_2^{}}Y_3^{n_2^*+l_1^{}-l_2^{}}, 
\end{align*}
we have 
\begin{align*}
\ell_{\underline{n}}(Y_1^{n_1}X_2^{n_2-i}Y_2^iX_3^{n_2^*+i}Y_3^{n_1^*-i}\otimes P_{\underline{n}})&=\frac{\binom{n_1^*}{i}(-1)^{n_3+i}}{(-1)^{n_2-i}\binom{n_2}{i}(-1)^{n_2^*+i}\binom{n_3}{n_1^*-i}}, \\
\ell_{\underline{n}}(P_{\underline{n}}\otimes X_1^{n_1}X_2^jY_2^{n_2-j}X_3^{n_1^*-j}Y_3^{n_2^*+j})&=\frac{\binom{n_1^*}{n_1^*-j}(-1)^{n_3^*+j}}{(-1)^{n_2-j}\binom{n_2}{j}(-1)^{n_2^*+j}\binom{n_3}{n_1^*-j}},
\end{align*} 
from which we get the stated formula.  
\end{proof}


\subsection{Proof of Proposition \ref{prop:b1}}\label{ssec:b13}

Lemmas \ref{lem:b4} and \ref{lem:b5} give 
\[J(Y_+^{n_1^*}v_\pi^{(n_2^*)},\bar v_{\sig^\vee}^{})
=\frac{b!c!n_3!(n_1+1)\Big(\sum_{i=0}^{n_1^*}\binom{n_1^*}{i}a_i\Big)^2}{\d(\sig)(k_3-1)!(-k_1-1)!\ell_{\underline{n}}(P_{\underline{n}}\otimes P_{\underline{n}})}. \]

For non-negative integers $A$, $B$ and $C$, Lemma 5 of \cite{DO} gives the formula 
\[\sum_{i=0}^{A}\binom{A+B-i}{B}\binom{C+i}{C}
=\binom{A+B+C+1}{A}. \]
This formula applied with $A=n_1^*$, $B=n_3^*$ and $C=n_2^*$ gives 
\begin{align*}
\sum_{i=0}^{n_1^*}\binom{n_1^*}{i}a_i
&=\sum_{i=0}^{n_1^*}\binom{n_1^*}{i}^2\binom{n_2}{i}^{-1}\binom{n_3}{n_1^*-i}^{-1}\\
&=\frac{\Gam(n_1^*+1)^2\Gam(n_2^*+1)\Gam(n_3^*+1)}{\Gam(n_2+1)\Gam(n_3+1)}\sum_{i=0}^{n_1^*}\binom{n_2-i}{n_3^*}\binom{n_2^*+i}{n_2^*}\\
&=\frac{\Gam\bigl(\frac{n_1+n_2+n_3}{2}+2\bigl)\Gam(n_1^*+1)\Gam(n_2^*+1)\Gam(n_3^*+1)}{\Gam(n_1+2)\Gam(n_2+1)\Gam(n_3+1)}. 
\end{align*}
Lemma 4.11 of \cite{MH} gives the formula
\[\ell_{\underline{n}}(P_{\underline{n}}\otimes P_{\underline{n}})
=\frac{\Gam\bigl(\frac{n_1+n_2+n_3}{2}+2\bigl)\Gam(n_1^*+1)\Gam(n_2^*+1)\Gam(n_3^*+1)}{\Gam(n_1+1)\Gam(n_2+1)\Gam(n_3+1)}. \]
Put $c_H'=\frac{\d(\sig)}{n_2+1}$. 
Since 
\beq
\ell_{\ulk}(\bfP_{\ulk})=\dim\frkH_{(\lam_1,\lam_2;\lam_3)}=-k_1\label{tag:b5}
\eeq
by Lemma A.1 of \cite{HY}, we finally get  
\begin{align*}
J(Y_+^{n_1^*} v_\pi^{(n_2^*)},\bar v_{\sig^\vee}^{})
&=\frac{\Gam\bigl(\frac{n_1+n_2+n_3}{2}+2\bigl)\Gam(n_1^*+1)\Gam(n_2^*+1)\Gam(n_3^*+1)\Gam(-k_1')\Gam(k_2')}{c_H'\Gam(n_1+2)\Gam(n_2+2)\Gam(-k_1)\Gam(k_3)}\\
&=\frac{(2\pi)^{2n_1^*}}{4\pi^5 c_H'}\ell_{\ulk}(\bfP_{\ulk})\frac{L\left(\frac{1}{2},\pi\times\sig^\vee\right)}{L(1,\pi,\mathrm{Ad})L(1,\sig^\vee,\mathrm{Ad})}.  
\end{align*}


\section{The Petersson norm of newforms on $\U(1,1)$}\label{sec:c}


\subsection{The formula of Lapid and Mao}\label{ssec:c1}

Recall the maximal compact subgroup $K'=\prod_lK_l'$ of $H(\widehat{\QQ})$ defined by $K_l'=\GL_2(\frko_{E,l})\cap H(\QQ_l)$. 
Define the Haar measure of $H(\AA)$ by $\d_{K'} h=\d h_\infty\prod_l\d h_l$, where $\d h_\infty$ is defined in (\ref{tag:74}) and where $\d h_l$ is the Haar measure of $H(\QQ_l)$ giving $K'_l$ volume 1.
On the other hand, let $\d^\tau h$ be the Tamagawa measure on $H(\AA)$. 
Lemma \ref{lem:70} gives 
\[\d_{K'}h=2^{-2}D_EL(1,\eps_{E/\QQ})\zeta_\QQ(2)\,\d^\tau\! h. \]
We normalize the Petersson norm of $\vph\in\scra^0(H)$ by 
\[\|\vph\|^2_{K'}=2^{-2}D_EL(1,\eps_{E/\QQ})\zeta_\QQ(2)\,\|\vph\|^2. \] 

Let $\sigma\simeq\otimes_v'\sigma_v^{}$ be an irreducible $\addchar$-generic unitary cuspidal automorphic representation of $H(\AA)$ generated by a cusp form in $S_{\ulk'}^H(p^\ell\frkN',\chi';\CC)$. 
Then $\sig$ is tempered by the Ramanujan--Petersson conjecture in view of Remark \ref{rem:51}. 

Fix a perfect Hermitian paring $\calb_v:\sigma_v\times\sigma_v\to \CC$ for each place $v$ of $\QQ$. 
We define the local Whittaker period $I_v:\sigma_v\times \sigma_v\to \CC$ by 
\[I_v^{}(W_v^{},W_v')=\int_{\QQ_v}^{\rm st}\calb_v(\sigma_v(\bfn'(z_v))W_v^{},W_v')\addchar_v(z_v)^{-1}\d z_v\]
Define the normalized local period by 
\[I^\sharp_v(W_v)=\frac{L(1,\sigma_v,\Ad)}{L(1,\eps_{E_v/\QQ_v})\zet_{\QQ_v}(2)}I_v(W_v,W_v). \]

The following result is a special case of the conjecture of Lapid and Mao, studied in \S 6.2 of \cite{LM15}. 

\begin{theorem}[Lapid-Mao]\label{thm:c1}
Notation being as above, if $\varphi=\otimes_vW_v\in\sig$ is not zero, then 
\[\frac{|W_{\addchar}(\varphi)|^2}{\|\vph\|^2}=2^{-1-\vka_\sig}\frac{L(1,\eps_{E/\QQ})\zeta_\QQ(2)}{L(1,\sigma,\Ad)}\prod_v \frac{I^\sharp_v(W_v)}{\calb_v(W_v,W_v)},\]
where $\vka_\sig=0$ if the functorial lift of $\sig$ to $\GL_2(\EE)$ is cuspidal and $\vka_\sig=1$ otherwise. 
\end{theorem}

Let $\vph_\sig\in\sig$ be the normalized newform associated to $\sig$ defined in Definition \ref{def:54}. 
Let $l$ a split prime. 
We denote by $\scrw_{\addchar_l}(\sig_l)$ the Whittaker model of $\sig_l$ with respect to $\addchar_l$, by $W_{\sig_l}\in\scrw_{\addchar_l}(\sig_l)$ the normalized essential Whittaker vector (see Definition \ref{def:81}) and by $\La\;,\;\Ra_l':\scrw_{\addchar_l^{}}(\sig_l^{})\times\scrw_{\addchar_l^{-1}}(\sig_l^\vee)\to\CC$ the pairing defined in (\ref{tag:83}). 
The quantity $\calb_{\sig_l}$ is defined in (\ref{tag:79}). 

\begin{corollary}\label{cor:c1}
Notation being as above, we have  
\[\|\vph_\sig\|^2_{K'}= 2^{-2+\vka_\sig+k_1'-k_2'}D_EL(1,\sigma,\Ad)\prod_{l|N_\sig}\calb_{\sigma_l}. \]
\end{corollary}

\begin{remark}
Though $\vka_\sig=0$ by (splt), we include the factor $2^{\vka_\sig}$. 
\end{remark}


\subsection{Proof of Corollary \ref{cor:c1}}\label{ssec:c2}

\begin{lemma}\label{lem:c1}
If $q$ does not split in $E$, then $I_q^\sharp(W_{\sig_q})=\calb_q^{}(W_{\sig_q},W_{\sig_q})$. 
\end{lemma}

\begin{proof}
Write $E=E_q$, $F=\QQ_q$, $\addchar=\addchar_q$ and $\sigma=\sigma_q$ for simplicity. 
Then $\sigma$ is the unramified component of an unramified principal series $I'(\chi)$ for some unramified character $\chi$ of $E^\times$ (see \S \ref{ssec:52}). 
Define the Whittaker functional $J_\addchar: I'(\chi)\to \CC$ by  
\[J_\addchar(f)=\int_F f(w\bfn'(z))\addchar(z)^{-1}\d z.\]
 We also define the invariant pairing $\calb':I'(\chi)\times I'(\chi^{-1})\to \CC$ by 
\[\calb'(f,f')=\int_F f(w\bfn'(z))f'(w\bfn'(z))\,\d z.\]
By uniqueness of Whittaker functionals we have 
\[\frac{I^\sharp(f,f')}{\calb(f,f')}=\frac{L(1,\sigma,\Ad)}{L(1,\eps_{E/F})\zeta_F(2)}\cdot\frac{J_\addchar(f)J_{\addchar^{-1}}(f')}{\calb'(f,f')}\]
for $f\in I'(\chi)$ and $f'\in I'(\chi^{-1})$, where
\[L(s,\sigma,\Ad)=\zeta_F(s)L(s,\eps_{E/F})L(s,\chi|_{F^\times})L(s,\chi^{-1}|_{F^\times}).\]
If $f$ and $f'$ are both spherical elements in $I'(\chi)$ and $I'(\chi^{-1})$ with $f(\ono_2)=f'(\ono_2)=1$, then $J_\addchar(f)=1-\chi(q)|q|_E$ and $\calb'(f,f')=\frac{\zeta_F(1)}{\zeta_F(2)}$.
\end{proof}

Next we consider the split case. 

\begin{lemma}\label{lem:c2}
If $l$ is split in $E$, then $I_l^\sharp(W_{\sig_l})=\calb_{\sig_l}^{-1}\calb_l^{}(W_{\sig_l},W_{\sig_l})$. 
\end{lemma}

\begin{proof}
In this case $\sig_l$ is a unitary generic irreducible representation of $H(\QQ_l)\stackrel{\sim}{\to}\GL_2(\QQ_l)$. 
Define $\calb_l$ by 
\[\calb_l(W,W')=\frac{1}{\zet_l(1)}\int_{\QQ_l^\times}W\biggl(\begin{bmatrix} t_l & \\ & 1 \end{bmatrix}\biggl)\overline{W'\biggl(\begin{bmatrix} t_l & \\ & 1 \end{bmatrix}\biggl)}\,\d^\times t_l \]
for $W,W'\in\scrw_\addchar(\sig_l^{})$. 
Then Lemma 4.4 of \cite{LM15} gives 
\[I_l(W,W')=W(\ono_2)W'(\ono_2).\]
Note that 
\[\calb_l^{}(W_{\sig_l}^{},W_{\sig_l}^{})=\frac{1}{\zet_l(1)}\La W_{\sig_l}^{},W_{\sig_l^\vee}\Ra_l'. \] 

It follows that 
\[I^\sharp(W_{\sig_l})=\frac{L(1,\sig_l,\Ad)}{\zeta_l(1)\zet_l(2)}\cdot \frac{I_l^{}(W_{\sig_l}^{},W_{\sig_l}^{})}{\calb_l(W_{\sig_l}^{},W_{\sig_l}^{})}=\frac{L^\GL(1,\sigma_l^{}\times\sig_l^\vee)}{\zeta_l(2)\La W_{\sig_l}^{},W_{\sig_l^\vee}\Ra_l'}=\calb_{\sig_l}^{-1} \]
(cf. Remark \ref{rem:71}). 
\end{proof}

Finally, we consider the archimedean case. 

\begin{lemma}\label{lem:c3}
$I^\sharp_\infty(W_{\sig_\infty})=2^{1+k_2'-k_1'}e^{-4\pi}\calb_\infty(W_{\sig_\infty},W_{\sig_\infty})$.
\end{lemma}

\begin{proof}
Put $k=k_2'-k_1'$ and $w=-(k_1'+k_2')$. 
The normalized Whittaker function $W_{\sigma_\infty}\in\scrw_{\addchar_\infty}(\sigma_\infty)$ of minimal $K$-type $-\ulk'$ is given by 
\[W_{\sig_\infty}\left(\bfn'(z)\bfd'(t)\begin{bmatrix} \cos\theta_1 & \sin\theta_1\\-\sin\theta_1 & \cos\theta_1\end{bmatrix} e^{\sqrt{-1}\theta_2}\right)=\addchar_\infty(z)t^ke^{-2t^2\pi}e^{\sqrt{-1}(k\theta_1+w\theta_2)} \]
for $z,\tht_1,\tht_2\in\RR$ and $t\in\RR^\times_+$ (cf. \cite[(2.10)]{MH}). 
Define $\calb_\infty$ by 
\[\calb_\infty(W,W')=2\int_{\RR^\times_+}W\biggl(\begin{bmatrix} t_\infty & \\ & t_\infty^{-1} \end{bmatrix}\biggl)\overline{W'\biggl(\begin{bmatrix} t_\infty & \\ & t_\infty^{-1} \end{bmatrix}\biggl)}\,\d^\times t_\infty \]
for $W,W'\in\scrw_{\addchar_\infty}(\sigma_\infty)$. 
Then 
\begin{align*}
\calb_\infty(\sigma_\infty(\bfn'(z_\infty))W_{\sig_\infty},W_{\sig_\infty})
&=\int_{\RR^\times_+}t_\infty^k e^{-4\pi t_\infty}\addchar_\infty(t_\infty z_\infty)\, \d^\times t_\infty\\
&=\Gam(k)(2\pi(2-\sqrt{-1}z_\infty))^{-k} 
\end{align*}
and 
\[I_\infty^{}(W_{\sigma_\infty},W_{\sigma_\infty})=(-2\pi\sqrt{-1})^{-k}\Gamma(k)\int_\RR \frac{e^{-2\pi\sqrt{-1}z_\infty}}{(z_\infty+2\sqrt{-1})^k}\,\d z_\infty=e^{-4\pi}. \]
We have 
\[\frac{I^\sharp_\infty(W_{\sig_\infty}^{})}{\calb_\infty(W_{\sig_\infty},W_{\sig_\infty})}
=\frac{L(1,\sigma_\infty,\Ad)}{\Gam_\RR(1)\Gam_\RR(2)}\cdot \frac{e^{-4\pi}}{(4\pi)^{-k}\Gam(k)}
=2^{1+k}e^{-4\pi}\]
by (\ref{tag:b0}). 
\end{proof}

We are now ready to prove Corollary \ref{cor:c1}. 
Theorem \ref{thm:c1} gives
\[\frac{|W_{\addchar}(\varphi_\sig)|^2}{\|\vph_\sig\|_{K'}^2}=\frac{2^{1-\vka_\sig}}{D_EL(1,\sigma,\Ad)}\prod_v \frac{I^\sharp_v(W_{\sig_v})}{\calb_v^{}(W_{\sig_v},W_{\sig_v})}. \]
This combined with Lemmas \ref{lem:c1}, \ref{lem:c2} and \ref{lem:c3} yield
\[\frac{e^{-4\pi}}{\|\vph_\sig\|_{K'}^2}=\frac{2^{1-\vka_\sig}}{D_EL(1,\sigma,\Ad)}2^{1+k_2'-k_1'}e^{-4\pi}\prod_{l|N_\sig}\calb_{\sig_l}^{-1}, \]
which complete our proof of Corollary \ref{cor:c1}. 


\section{Comparison of periods}\label{sec:d}

Let $\Pi$ and $\Pi'$ be cohomological conjugate self-dual cuspidal automorphic representations of $\GL_3(\EE)$ and $\GL_2(\EE)$. 
In this appendix we explicate the periods introduced in \cite{GHL} for $\Pi\times \Pi'$ and make a comparison with the period in Theorem \ref{cor:71} when $\Pi^{}_\infty$ and $\Pi'_\infty$ satisfy the intertwining relation in Remark \ref{rem:b1}. 
For $a,b\in \CC$ we write $a\sim b$ if $b\neq 0$ and $a/b\in \overline{\QQ}$. 
We write 
\[L(s,\Pi\times\Pi')=L(s,\Pi_\infty^{}\times\Pi'_\infty)L_{\rm fin}(s,\Pi\times\Pi')\] 
for the Rankin-Selberg $L$-function defined in \cite{JPSS2}.  

Let $\chi$ be a Hecke character with conductor $\frkf$ and infinity type $(k,-k)$, where $k$ is a non-zero integer.  
In other words  $\chi((\alp))=(\alp/\alp^c)^k$ if $\alp\equiv 1\pmod{\frkf}$. 
The Hecke $L$-series associated to $\chi$ is defined by
\[L_{\rm fin}(s,\chi)=\sum_{(\frka,\frkf)=1}\frac{\chi(\frka)}{N(\frka)^s}, \]
where the sum is over integral ideals of $\frko_E$ prime to $\frkf$. 
This series is absolutely convergent if $\Re s > 1$. 
The Hecke $L$-series is known to have a meromorphic continuation to the whole complex $s$-plane. 

We fix embeddings $\iot=\iot_\infty: E\hookrightarrow\overline{\QQ}\hookrightarrow \CC$. 
Let $\Omega_\infty$ be the CM period defined in \S \ref{ssec:67}. 
Let $p(\chi,\iota)$ be the period associated to $\chi$ in \cite[\S 2.1.1]{GHL}. 
Theorem 2.3 of \cite{GHL} says that if $k>0$, then 
\[p(\widecheck{\chi},\iota^c)\sim L_{\rm fin}(0,\chi), \]
where $\widecheck{\chi}:=(\chi^c)^{-1}$. 
If $k<0$, then 
\[p(\widecheck{\chi},\iota)\sim p(\widecheck{\chi}^c,\iota^c)\sim L_{\rm fin}(0,\chi^c). \]
By Damerell's theorem \cite{D70,D71}, if $k>0$, then 
\[L_{\rm fin}(0,\chi)\sim \pi^{-k}\Omega_\infty^{2k} \]
(cf. \cite[Corollary 2.12]{BK10Duke}). 
In any case we have  
\[p(\widecheck{\chi},\iota)\sim \pi^k\Omega_\infty^{-2k}. \]

We denote the central characters of $\Pi$ and $\Pi'$ by $\ome_\Pi$ and $\ome_{\Pi'}$. 
Under certain assumptions, Theorem 2.6 of \cite{GHL} associates to $\Pi$ local arithmetic automorphic periods $\{P^{(i)}(\Pi,\iot)\}_{i=0,1,2,3}$ such that 
\begin{align*}
P^{(0)}(\Pi,\iot)&\sim p(\widecheck{\ome}_\Pi,\iot^c), & 
P^{(3)}(\Pi,\iot)&\sim p(\widecheck{\ome}_\Pi,\iot)  
\end{align*}
and such that  
\[L_{\rm fin}(1,\Pi,\mathrm{As}^-)\sim(2\pi\sqrt{-1})^6P^{(0)}(\Pi,\iot)P^{(1)}(\Pi,\iot) P^{(2)}(\Pi,\iot)P^{(3)}(\Pi,\iot) \]
by Theorem 5.14 of \cite{GHL}. 
Similarly, there are periods $\{P^{(i)}(\Pi',\iot)\}_{i=0,1,2}$ such that 
\begin{align*}
P^{(0)}(\Pi',\iot)&\sim p(\widecheck{\ome}_{\Pi'},\iot^c), \qquad\quad  
P^{(2)}(\Pi',\iot)\sim p(\widecheck{\ome}_{\Pi'},\iot), \\
L_{\rm fin}(1,\Pi',\mathrm{As})&\sim(2\pi\sqrt{-1})^3P^{(0)}(\Pi',\iot)P^{(1)}(\Pi',\iot) P^{(2)}(\Pi',\iot). 
\end{align*}

Put $G=\U(2,1)$. 
Let $\pi(0)$ be a square-integrable automorphic representations of $G(\AA)$ whose functorial lift is $\Pi$ and whose archimedean part $\pi_{\lam,0}$ is holomorphic (see Lemma 4.9 of \cite{GHL}).
Let $Q(\pi(0))$ be the automorphic $Q$-period defined in \cite[\S 4.3]{GHL}. 
By definition $Q(\pi(0))$ is the Petersson norm of a \emph{deRham rational} automorphic form in $\pi(0)$. 
Theorem 6.3 and Lemma 2.2(c) of \cite{GHL} give 
\[P^{(1)}(\Pi,\iot)=P^{(0)}(\Pi,\iot)Q(\pi(0))p(\widecheck{\ome}_\Pi,\iot)=Q(\pi(0)). \] 
It follows that 
\begin{align*}
P^{(0)}(\Pi,\iot)P^{(2)}(\Pi,\iot)&\sim\frac{L_{\rm fin}(1,\Pi,\mathrm{As}^-)}{\pi^6P^{(1)}(\Pi,\iot)P^{(3)}(\Pi,\iot)}
\sim\frac{L_{\rm fin}(1,\Pi,\mathrm{As}^-)}{\pi^6Q(\pi(0))p(\widecheck{\ome}_\Pi,\iot)}, \\
P^{(0)}(\Pi',\iot)P^{(1)}(\Pi',\iot)^2&\sim\frac{L_{\rm fin}(1,\Pi',\mathrm{As})^2}{\pi^6P^{(0)}(\Pi',\iot)P^{(2)}(\Pi',\iot)^2}
\sim\frac{L_{\rm fin}(1,\Pi',\mathrm{As})^2}{\pi^6p(\widecheck{\ome}_{\Pi'},\iot)}. 
\end{align*}

Let
\[(a_1,a_2,a_3)=(-k_1,-k_2-1,1-k_3);\quad (b_1,b_2)=\left(k_2'-\frac{1}{2},k_1'+\frac{1}{2}\right)\]
be the Langlands parameters of $\Pi_\infty^{}$ and $\Pi'_\infty$. 
Note that 
\begin{align*}
\ome_{\Pi,\infty}(z/z^c)&=(z/z^c)^{-(k_1+k_2+k_3)}, & 
\ome_{\Pi',\infty}(z/z^c)&=(z/z^c)^{k_1'+k_2'} 
\end{align*}
for $z\in\CC^\times$. 
The interlacing relation is 
\[a_1>-b_2>a_2>a_3>-b_1.\]
The split indices in \cite[Definition 2.12]{GHL} are given by 
\begin{gather*}
{\rm sp}(0,\Pi;\Pi',\iot)={\rm sp}(2,\Pi;\Pi',\iot)=1,\qquad 
{\rm sp}(1,\Pi;\Pi',\iot)={\rm sp}(3,\Pi;\Pi',\iot)=0; \\
{\rm sp}(0,\Pi';\Pi,\iot)=1,\qquad
{\rm sp}(1,\Pi';\Pi,\iot)=2,\qquad 
{\rm sp}(2,\Pi';\Pi,\iot)=0.
\end{gather*}
Theorem 5.19 of \cite{GHL} gives 
\[L_{\rm fin}\biggl(\frac{1}{2},\Pi\times\Pi'\biggl)\sim(2\pi\sqrt{-1})^{3}P^{(0)}(\Pi,\iot)P^{(2)}(\Pi,\iot) P^{(0)}(\Pi',\iot)P^{(1)}(\Pi',\iot)^2. \]
We conclude that 
\begin{align*}
L_{\rm fin}\biggl(\frac{1}{2},\Pi\times\Pi'\biggl)&\sim\pi^3\frac{L_{\rm fin}(1,\Pi,\mathrm{As}^-)}{\pi^6Q(\pi(0))p(\ome^\vee_\Pi,\iota)}\cdot\frac{L_{\rm fin}(1,\Pi',\mathrm{As})^2}{\pi^6p(\ome_{\Pi'}^\vee,\iota)}\\
&\sim\pi^{-9}\frac{L_{\rm fin}(1,\Pi,\mathrm{As}^-)}{(\pi^{-1}\Omega_\infty^2)^{k_1^{}+k_2^{}+k_3^{}}Q(\pi(0))}\cdot\frac{L_{\rm fin}(1,\Pi',\mathrm{As})^2}{(\pi^{-1}\Omega_\infty^2)^{-k_1'-k_2'}}. 
\end{align*}

By (\ref{tag:b1}) and (\ref{tag:b2}) we have
\[\frac{L\bigl(\frac{1}{2},\Pi_\infty\times\Pi'_\infty\bigl)}{L(1,\Pi_\infty,\mathrm{As}^-)L(1,\Pi_\infty',\mathrm{As})^2}\sim \pi^{6-2(a_3+b_1)}\pi^{3+b_1-b_2}
=\pi^{7+2k_3-k_1'-k_2'}. \]
Putting these together, we find that 
\begin{align*}
\frac{L\bigl(\frac{1}{2},\Pi\times\Pi'\bigl)}{L(1,\Pi,\mathrm{As}^-)L(1,\Pi',\mathrm{As})^2}
&\sim\pi^{2k_3-2-(k_1'+k_2')}\frac{(\pi^{-1}\Omega_\infty^2)^{k_1'+k_2'-k_1^{}-k_2^{}-k_3^{}}}{Q(\pi(0))}\\
&=\pi^{-2+k_3-k_1-k_2}\frac{(\pi^{-2}\Omega_\infty^2)^{k_1'+k_2'-k_1^{}-k_2^{}-k_3^{}}}{Q(\pi(0))}. 
\end{align*} 

Suppose that $\pi(0)$ is cuspidal and that $\pi(0)$ and $\Pi'$ are unramified at all non-split primes. 
Let $f^\circ\in S^G_{\ulk}(\frkN_{\pi(0)},\chi,\overline{\QQ})$ be a $\overline{\QQ}$-rational newform associated to $\pi(0)$. 
Theorem \ref{cor:71}, Corollary \ref{cor:c1}, (\ref{tag:711}) and Remark \ref{rem:71} give
\[\frac{L\bigl(\frac{1}{2},\Pi\times\Pi'\bigl)}{L(1,\Pi,\mathrm{As}^-)L(1,\Pi',\mathrm{As})^2}\sim
\frac{(\pi^{-2}\Omega_\infty^2)^{k_1'+k_2'-k_1^{}-k_2^{}-k_3^{}}}{(f^\circ,f^\circ)_{\frkN_{\pi(0)}}}. \]
We therefore see that 
\[\pi^{k_1+k_2-k_3}Q(\pi(0))\sim\pi^{-2}(f^\circ,f^\circ)_{\frkN_{\pi(0)}}\sim\La f^\circ,f^\circ\Ra, \]
where $\La f^\circ,f^\circ\Ra$ is the pairing defined by replacing the classical measure by the Tamagawa measure (see Lemma \ref{lem:70}). 

\begin{remark}
Let $\sig$ be a cuspidal automorphic representation of $\GL_2(\AA)$ whose archimedean component $\sig_\infty$ is a discrete series with extremal weight $\pm k$. 
Let $g^\circ$ be the normalized newform associated to $\sig$. 
Then $\pi^{k/2}g^\circ$ is deRham rational in view of \cite[(1.6.5)]{Harris89}. 
We expect that $\pi^{(k_3-k_1-k_2)/2}f^\circ$ is deRham rational. 
\end{remark}


\printindex

\bibliographystyle{amsalpha}
\bibliography{GGP}

\newcommand{\etalchar}[1]{$^{#1}$}
\providecommand{\bysame}{\leavevmode\hbox to3em{\hrulefill}\thinspace}
\providecommand{\MR}{\relax\ifhmode\unskip\space\fi MR }
\providecommand{\MRhref}[2]{%
  \href{http://www.ams.org/mathscinet-getitem?mr=#1}{#2}
}
\providecommand{\href}[2]{#2}
\begin{thebibliography}{BPLZZ21}

\bibitem[AOY24]{AOY}
Hiraku Atobe, Masao Oi, and Seidai Yasuda, \emph{Local newforms for generic
  representations of unramified odd unitary groups and the fundamental lemma},
  Duke Math. J. \textbf{173} (2024), no.~12, 2447--2479. \MR{4800616}

\bibitem[AP19]{AP}
Jeffrey~D. Adler and Dipendra Prasad, \emph{Multiplicity upon restriction to
  the derived subgroup}, Pacific J. Math. \textbf{301} (2019), no.~1, 1--14.
  \MR{4007368}

\bibitem[BC04]{BC04}
Jo\"{e}l Bella\"{\i}che and Ga\"{e}tan Chenevier, \emph{Formes non
  temp\'{e}r\'{e}es pour {$\rm U(3)$} et conjectures de {B}loch-{K}ato}, Ann.
  Sci. \'{E}cole Norm. Sup. (4) \textbf{37} (2004), no.~4, 611--662.
  \MR{2097894}

\bibitem[BDP13]{BDP}
Massimo Bertolini, Henri Darmon, and Kartik Prasanna, \emph{Generalized
  {H}eegner cycles and {$p$}-adic {R}ankin {$L$}-series}, Duke Math. J.
  \textbf{162} (2013), no.~6, 1033--1148, With an appendix by Brian Conrad.
  \MR{3053566}

\bibitem[BK10]{BK10Duke}
Kenichi Bannai and Shinichi Kobayashi, \emph{Algebraic theta functions and the
  {$p$}-adic interpolation of {E}isenstein-{K}ronecker numbers}, Duke Math. J.
  \textbf{153} (2010), no.~2, 229--295. \MR{2667134}

\bibitem[BPCZ22]{BCZ}
Rapha\"{e}l Beuzart-Plessis, Pierre-Henri Chaudouard, and Micha\l Zydor,
  \emph{The global {G}an-{G}ross-{P}rasad conjecture for unitary groups: the
  endoscopic case}, Publ. Math. Inst. Hautes \'{E}tudes Sci. \textbf{135}
  (2022), 183--336. \MR{4426741}

\bibitem[BPLZZ21]{BLZZ}
Rapha\"{e}l Beuzart-Plessis, Yifeng Liu, Wei Zhang, and Xinwen Zhu,
  \emph{Isolation of cuspidal spectrum, with application to the
  {G}an-{G}ross-{P}rasad conjecture}, Ann. of Math. (2) \textbf{194} (2021),
  no.~2, 519--584. \MR{4298750}

\bibitem[BPS16]{BPS}
St\'{e}phane Bijakowski, Vincent Pilloni, and Beno\^{i}t Stroh,
  \emph{Classicit\'{e} de formes modulaires surconvergentes}, Ann. of Math. (2)
  \textbf{183} (2016), no.~3, 975--1014. \MR{3488741}

\bibitem[Car79]{Cartier79Corvallis}
P.~Cartier, \emph{Representations of {$p$}-adic groups: a survey}, Automorphic
  forms, representations and {$L$}-functions ({P}roc. {S}ympos. {P}ure {M}ath.,
  {O}regon {S}tate {U}niv., {C}orvallis, {O}re., 1977), {P}art 1, Proc. Sympos.
  Pure Math., XXXIII, Amer. Math. Soc., Providence, R.I., 1979, pp.~111--155.

\bibitem[Car12]{Cara12}
Ana Caraiani, \emph{Local-global compatibility and the action of monodromy on
  nearby cycles}, Duke Math. J. \textbf{161} (2012), no.~12, 2311--2413.
  \MR{2972460}

\bibitem[CEF{\etalchar{+}}16]{CEFMV}
Ana Caraiani, Ellen Eischen, Jessica Fintzen, Elena Mantovan, and Ila Varma,
  \emph{{$p$}-adic {$q$}-expansion principles on unitary {S}himura varieties},
  Directions in number theory, Assoc. Women Math. Ser., vol.~3, Springer,
  [Cham], 2016, pp.~197--243. \MR{3596581}

\bibitem[Coa89a]{Coe2}
John Coates, \emph{On {$p$}-adic {$L$}-functions}, no. 177-178, 1989,
  S\'{e}minaire Bourbaki, Vol. 1988/89, pp.~Exp. No. 701, 33--59. \MR{1040567}

\bibitem[Coa89b]{Coe}
\bysame, \emph{On {$p$}-adic {$L$}-functions attached to motives over {${\bf
  Q}$}. {II}}, Bol. Soc. Brasil. Mat. (N.S.) \textbf{20} (1989), no.~1,
  101--112. \MR{1129081}

\bibitem[CPR89]{CPR}
John Coates and Bernadette Perrin-Riou, \emph{On {$p$}-adic {$L$}-functions
  attached to motives over {${\bf Q}$}}, Algebraic number theory, Adv. Stud.
  Pure Math., vol.~17, Academic Press, Boston, MA, 1989, pp.~23--54.
  \MR{1097608}

\bibitem[Dam70]{D70}
R.~M. Damerell, \emph{{$L$}-functions of elliptic curves with complex
  multiplication. {I}}, Acta Arith. \textbf{17} (1970), 287--301. \MR{285540}

\bibitem[Dam71]{D71}
\bysame, \emph{{$L$}-functions of elliptic curves with complex multiplication.
  {II}}, Acta Arith. \textbf{19} (1971), 311--317. \MR{399103}

\bibitem[Del73]{Deligne73}
P.~Deligne, \emph{Les constantes des \'{e}quations fonctionnelles des fonctions
  {$L$}}, Modular functions of one variable, {II} ({P}roc. {I}nternat. {S}ummer
  {S}chool, {U}niv. {A}ntwerp, {A}ntwerp, 1972), Lecture Notes in Math., Vol.
  349, Springer, Berlin-New York, 1973, pp.~501--597. \MR{349635}

\bibitem[DGdO13]{DO}
Rui Duarte and Ant\'{o}nio Guedes~de Oliveira, \emph{Note on the convolution of
  binomial coefficients}, J. Integer Seq. \textbf{16} (2013), no.~7, Article
  13.7.6, 9. \MR{3102652}

\bibitem[DR14]{DD14}
Henri Darmon and Victor Rotger, \emph{Diagonal cycles and {E}uler systems {I}:
  {A} {$p$}-adic {G}ross-{Z}agier formula}, Ann. Sci. \'Ec. Norm. Sup\'er. (4)
  \textbf{47} (2014), no.~4, 779--832. \MR{3250064}

\bibitem[DR17]{DD17}
\bysame, \emph{Diagonal cycles and {E}uler systems {II}: {T}he {B}irch and
  {S}winnerton-{D}yer conjecture for {H}asse-{W}eil-{A}rtin {$L$}-functions},
  J. Amer. Math. Soc. \textbf{30} (2017), no.~3, 601--672. \MR{3630084}

\bibitem[dSG16]{SG16}
Ehud de~Shalit and Eyal~Z. Goren, \emph{A theta operator on {P}icard modular
  forms modulo an inert prime}, Res. Math. Sci. \textbf{3} (2016), Paper No.
  28, 65. \MR{3543240}

\bibitem[dSG19]{SG19}
\bysame, \emph{Theta operators on unitary {S}himura varieties}, Algebra Number
  Theory \textbf{13} (2019), no.~8, 1829--1877. \MR{4017536}

\bibitem[EFMV18]{EFMV}
Ellen Eischen, Jessica Fintzen, Elena Mantovan, and Ila Varma,
  \emph{Differential operators and families of automorphic forms on unitary
  groups of arbitrary signature}, Doc. Math. \textbf{23} (2018), 445--495.
  \MR{3846052}

\bibitem[Eis12]{Eischen12}
Ellen~E. Eischen, \emph{{$p$}-adic differential operators on automorphic forms
  on unitary groups}, Ann. Inst. Fourier (Grenoble) \textbf{62} (2012), no.~1,
  177--243. \MR{2986270}

\bibitem[Fin06]{Finis}
Tobias Finis, \emph{Divisibility of anticyclotomic {$L$}-functions and theta
  functions with complex multiplication}, Ann. of Math. (2) \textbf{163}
  (2006), no.~3, 767--807. \MR{2215134}

\bibitem[Ger19]{Geraghty}
D.~Geraghty, \emph{Modularity lifting theorems for ordinary {G}alois
  representations}, Math. Ann. \textbf{373} (2019), no.~3-4, 1341--1427.

\bibitem[GGP12]{GGP}
Wee~Teck Gan, Benedict~H. Gross, and Dipendra Prasad, \emph{Symplectic local
  root numbers, central critical {$L$} values, and restriction problems in the
  representation theory of classical groups}, no. 346, 2012, Sur les
  conjectures de Gross et Prasad. I, pp.~1--109. \MR{3202556}

\bibitem[GHL]{GHL}
Harald Grobner, Michael Harris, and Jie Lin, \emph{Factorization of periods,
  construction of automorphic motives and {D}eligne's conjecture over
  {CM}-fields}, {\tt arXiv}:2509.02303.

\bibitem[GP92]{GP92}
Benedict~H. Gross and Dipendra Prasad, \emph{On the decomposition of a
  representation of {${\rm SO}_n$} when restricted to {${\rm SO}_{n-1}$}},
  Canad. J. Math. \textbf{44} (1992), no.~5, 974--1002. \MR{1186476}

\bibitem[Gre94]{Gre}
Ralph Greenberg, \emph{Iwasawa theory and {$p$}-adic deformations of motives},
  Motives ({S}eattle, {WA}, 1991), Proc. Sympos. Pure Math., vol.~55, Amer.
  Math. Soc., Providence, RI, 1994, pp.~193--223. \MR{1265554}

\bibitem[Har86]{Harris86}
Michael Harris, \emph{Arithmetic vector bundles and automorphic forms on
  {S}himura varieties. {II}}, Compositio Math. \textbf{60} (1986), no.~3,
  323--378. \MR{869106}

\bibitem[Har90]{Harris89}
\bysame, \emph{Period invariants of {H}ilbert modular forms. {I}. {T}rilinear
  differential operators and {$L$}-functions}, Cohomology of arithmetic groups
  and automorphic forms ({L}uminy-{M}arseille, 1989), Lecture Notes in Math.,
  vol. 1447, Springer, Berlin, 1990, pp.~155--202. \MR{1082965}

\bibitem[Har93]{H93}
\bysame, \emph{{$L$}-functions of {$2\times 2$} unitary groups and
  factorization of periods of {H}ilbert modular forms}, J. Amer. Math. Soc.
  \textbf{6} (1993), no.~3, 637--719. \MR{1186960}

\bibitem[Har14]{NHarris}
R.~Neal Harris, \emph{The refined {G}ross-{P}rasad conjecture for unitary
  groups}, Int. Math. Res. Not. IMRN (2014), no.~2, 303--389. \MR{3159075}

\bibitem[Har21a]{MHarris}
Michael Harris, \emph{Square root {$p$}-adic {$L$}-functions {I}:
  {C}onstruction of a one-variable measure}, Tunis. J. Math. \textbf{3} (2021),
  no.~4, 657--688. \MR{4331439}

\bibitem[Har21b]{H21}
\bysame, \emph{Shimura varieties for unitary groups and the doubling method},
  Relative trace formulas, Simons Symp., Springer, Cham, [2021] \copyright
  2021, pp.~217--251. \MR{4611947}

\bibitem[He17]{He}
Hongyu He, \emph{On the {G}an-{G}ross-{P}rasad conjecture for {$U(p,q)$}},
  Invent. Math. \textbf{209} (2017), no.~3, 837--884. \MR{3681395}

\bibitem[Hid93]{Hid93}
Haruzo Hida, \emph{Elementary theory of {$L$}-functions and {E}isenstein
  series}, London Mathematical Society Student Texts, vol.~26, Cambridge
  University Press, Cambridge, 1993. \MR{1216135}

\bibitem[Hid04]{Hid04}
\bysame, \emph{{$p$}-adic automorphic forms on {S}himura varieties}, Springer
  Monographs in Mathematics, Springer-Verlag, New York, 2004. \MR{2055355}

\bibitem[Hsi14]{MH2}
Ming-Lun Hsieh, \emph{Eisenstein congruence on unitary groups and {I}wasawa
  main conjectures for {CM} fields}, J. Amer. Math. Soc. \textbf{27} (2014),
  no.~3, 753--862. \MR{3194494}

\bibitem[Hsi21]{MH}
\bysame, \emph{Hida families and {$p$}-adic triple product {$L$}-functions},
  Amer. J. Math. \textbf{143} (2021), no.~2, 411--532. \MR{4234973}

\bibitem[HT93]{HT93}
H.~Hida and J.~Tilouine, \emph{Anti-cyclotomic {K}atz {$p$}-adic
  {$L$}-functions and congruence modules}, Ann. Sci. \'{E}cole Norm. Sup. (4)
  \textbf{26} (1993), no.~2, 189--259. \MR{1209708}

\bibitem[HY24a]{HY3}
Ming-Lun Hsieh and Shunsuke Yamana, \emph{Bessel periods and anticyclotomic
  {$p$}-adic spinor {$L$}-functions}, Trans. Amer. Math. Soc. \textbf{377}
  (2024), no.~8, 5617--5672. \MR{4771233}

\bibitem[HY24b]{HY2}
\bysame, \emph{Four-variable {$p$}-adic triple product {$L$}-functions and the
  trivial zero conjecture}, Math. Ann. \textbf{390} (2024), no.~1, 889--976.
  \MR{4800929}

\bibitem[HY25]{HY}
\bysame, \emph{Five-variable {$p$}-adic {$L$}-functions for {$\rm U(3)\times
  U(2)$}}, Adv. Math. \textbf{476} (2025), Paper No. 110355, 87. \MR{4908664}

\bibitem[Ich07]{Ichino07}
Atsushi Ichino, \emph{On the {S}iegel-{W}eil formula for unitary groups}, Math.
  Z. \textbf{255} (2007), no.~4, 721--729. \MR{2274532}

\bibitem[Ich22]{Ichino}
\bysame, \emph{Theta lifting for tempered representations of real unitary
  groups}, Adv. Math. \textbf{398} (2022), Paper No. 108188, 70. \MR{4372665}

\bibitem[II10]{II}
Atsushi Ichino and Tamutsu Ikeda, \emph{On the periods of automorphic forms on
  special orthogonal groups and the {G}ross-{P}rasad conjecture}, Geom. Funct.
  Anal. \textbf{19} (2010), no.~5, 1378--1425. \MR{2585578}

\bibitem[JPSS81]{JPSS}
H.~Jacquet, I.~I. Piatetski-Shapiro, and J.~Shalika, \emph{Conducteur des
  repr\'{e}sentations du groupe lin\'{e}aire}, Math. Ann. \textbf{256} (1981),
  no.~2, 199--214. \MR{620708}

\bibitem[JPSS83]{JPSS2}
H.~Jacquet, I.~I. Piatetskii-Shapiro, and J.~A. Shalika, \emph{Rankin-{S}elberg
  convolutions}, Amer. J. Math. \textbf{105} (1983), no.~2, 367--464.
  \MR{701565}

\bibitem[JS85]{JS}
H.~Jacquet and J.~A. Shalika, \emph{A lemma on highly ramified
  {$\epsilon$}-factors}, Math. Ann. \textbf{271} (1985), no.~3, 319--332.

\bibitem[KK07]{KK}
Takuya Konno and Kazuko Konno, \emph{On doubling construction for real unitary
  dual pairs}, Kyushu J. Math. \textbf{61} (2007), no.~1, 35--82. \MR{2317282}

\bibitem[KL19]{KL19}
Andrew Knightly and Charles Li, \emph{On the distribution of {S}atake
  parameters for {S}iegel modular forms}, Doc. Math. \textbf{24} (2019),
  677--747. \MR{3960116}

\bibitem[Kot92]{Kottwitz92}
Robert~E. Kottwitz, \emph{Points on some {S}himura varieties over finite
  fields}, J. Amer. Math. Soc. \textbf{5} (1992), no.~2, 373--444. \MR{1124982}

\bibitem[KS97]{KS}
Stephen~S. Kudla and W.~Jay Sweet, Jr., \emph{Degenerate principal series
  representations for {${\rm U}(n,n)$}}, Israel J. Math. \textbf{98} (1997),
  253--306. \MR{1459856}

\bibitem[KV78]{KV}
M.~Kashiwara and M.~Vergne, \emph{On the {S}egal-{S}hale-{W}eil representations
  and harmonic polynomials}, Invent. Math. \textbf{44} (1978), no.~1, 1--47.
  \MR{463359}

\bibitem[Lab11]{Lab}
J.-P. Labesse, \emph{Changement de base {CM} et s\'{e}ries discr\`etes}, On the
  stabilization of the trace formula, Stab. Trace Formula Shimura Var. Arith.
  Appl., vol.~1, Int. Press, Somerville, MA, 2011, pp.~429--470. \MR{2856380}

\bibitem[Lan90]{GTM121}
Serge Lang, \emph{Cyclotomic fields {I} and {II}}, second ed., Graduate Texts
  in Mathematics, vol. 121, Springer-Verlag, New York, 1990, With an appendix
  by Karl Rubin. \MR{1029028}

\bibitem[Lan12]{Lan12}
Kai-Wen Lan, \emph{Comparison between analytic and algebraic constructions of
  toroidal compactifications of {PEL}-type {S}himura varieties}, J. Reine
  Angew. Math. \textbf{664} (2012), 163--228. \MR{2980135}

\bibitem[Lan13]{Lan13}
\bysame, \emph{Arithmetic compactifications of {PEL}-type {S}himura varieties},
  London Mathematical Society Monographs Series, vol.~36, Princeton University
  Press, Princeton, NJ, 2013. \MR{3186092}

\bibitem[LM14]{LM}
Erez Lapid and Zhengyu Mao, \emph{On a new functional equation for local
  integrals}, Automorphic forms and related geometry: assessing the legacy of
  {I}. {I}. {P}iatetski-{S}hapiro, Contemp. Math., vol. 614, Amer. Math. Soc.,
  Providence, RI, 2014, pp.~261--294. \MR{3220931}

\bibitem[LM15]{LM15}
\bysame, \emph{A conjecture on {W}hittaker-{F}ourier coefficients of cusp
  forms}, J. Number Theory \textbf{146} (2015), 448--505. \MR{3267120}

\bibitem[Mat13]{NM}
Nadir Matringe, \emph{Essential {W}hittaker functions for {$GL(n)$}}, Doc.
  Math. \textbf{18} (2013), 1191--1214. \MR{3138844}

\bibitem[MV10]{MV}
Philippe Michel and Akshay Venkatesh, \emph{The subconvexity problem for {${\rm
  GL}_2$}}, Publ. Math. Inst. Hautes \'{E}tudes Sci. (2010), no.~111, 171--271.
  \MR{2653249}

\bibitem[MVW87]{MVW}
Colette M{\oe}glin, Marie-France Vign\'{e}ras, and Jean-Loup Waldspurger,
  \emph{Correspondances de {H}owe sur un corps {$p$}-adique}, Lecture Notes in
  Mathematics, vol. 1291, Springer-Verlag, Berlin, 1987. \MR{1041060}

\bibitem[MY14]{MY14}
Michitaka Miyauchi and Takuya Yamauchi, \emph{An explicit computation of
  {$p$}-stabilized vectors}, J. Th\'eor. Nombres Bordeaux \textbf{26} (2014),
  no.~2, 531--558. \MR{3320491}

\bibitem[Pau00]{P2}
Annegret Paul, \emph{Howe correspondence for real unitary groups. {II}}, Proc.
  Amer. Math. Soc. \textbf{128} (2000), no.~10, 3129--3136. \MR{1664375}

\bibitem[Pil12]{P12BSMF}
Vincent Pilloni, \emph{Sur la th\'{e}orie de {H}ida pour le groupe {${\rm
  GSp}_{2g}$}}, Bull. Soc. Math. France \textbf{140} (2012), no.~3, 335--400.

\bibitem[Pra92]{Prasad92}
Dipendra Prasad, \emph{Invariant forms for representations of {${\rm GL}_2$}
  over a local field}, Amer. J. Math. \textbf{114} (1992), no.~6, 1317--1363.
  \MR{1198305}

\bibitem[Rog90]{Rog90}
Jonathan~D. Rogawski, \emph{Automorphic representations of unitary groups in
  three variables}, xii+259.

\bibitem[RSZ20]{RSZ}
M.~Rapoport, B.~Smithling, and W.~Zhang, \emph{Arithmetic diagonal cycles on
  unitary {S}himura varieties}, Compos. Math. \textbf{156} (2020), no.~9,
  1745--1824. \MR{4167594}

\bibitem[Sch93]{Schmidt}
Claus-G. Schmidt, \emph{Relative modular symbols and {$p$}-adic
  {R}ankin-{S}elberg convolutions}, Invent. Math. \textbf{112} (1993), no.~1,
  31--76.

\bibitem[Sha83]{Shahidi83}
Freydoon Shahidi, \emph{Local coefficients and normalization of intertwining
  operators for {${\rm GL}(n)$}}, Compositio Math. \textbf{48} (1983), no.~3,
  271--295. \MR{700741}

\bibitem[Shi86]{Shimura86}
Goro Shimura, \emph{On a class of nearly holomorphic automorphic forms}, Ann.
  of Math. (2) \textbf{123} (1986), no.~2, 347--406. \MR{835767}

\bibitem[Shi97]{Shimura97}
\bysame, \emph{Euler products and {E}isenstein series}, CBMS Regional
  Conference Series in Mathematics, vol.~93, Published for the Conference Board
  of the Mathematical Sciences, Washington, DC; by the American Mathematical
  Society, Providence, RI, 1997. \MR{1450866}

\bibitem[Shi00]{Shimura00}
\bysame, \emph{Arithmeticity in the theory of automorphic forms}, Mathematical
  Surveys and Monographs, vol.~82, American Mathematical Society, Providence,
  RI, 2000. \MR{1780262}

\bibitem[Sil09]{Sil09}
Joseph~H. Silverman, \emph{The arithmetic of elliptic curves}, second ed.,
  Graduate Texts in Mathematics, vol. 106, Springer, Dordrecht, 2009.
  \MR{2514094}

\bibitem[SY25]{Chen25}
Chen Shih-Yu, \emph{{A}lgebraicity of adjoint {$L$}-functions for quasi-split
  groups}, {\tt arXiv}: 2509.23940 (2025), preprint.

\bibitem[TU99]{TU99}
J.~Tilouine and E.~Urban, \emph{Several-variable {$p$}-adic families of
  {S}iegel-{H}ilbert cusp eigensystems and their {G}alois representations},
  Ann. Sci. \'{E}cole Norm. Sup. (4) \textbf{32} (1999), no.~4, 499--574.

\bibitem[Wed99]{Wed99}
Torsten Wedhorn, \emph{Ordinariness in good reductions of {S}himura varieties
  of {PEL}-type}, Ann. Sci. \'Ecole Norm. Sup. (4) \textbf{32} (1999), no.~5,
  575--618. \MR{1710754}

\bibitem[Xue23]{HX}
Hang Xue, \emph{Bessel models for real unitary groups: the tempered case}, Duke
  Math. J. \textbf{172} (2023), no.~5, 995--1031. \MR{4568696}

\bibitem[Zha14]{Z}
Wei Zhang, \emph{Automorphic period and the central value of {R}ankin-{S}elberg
  {L}-function}, J. Amer. Math. Soc. \textbf{27} (2014), no.~2, 541--612.
  \MR{3164988}

\end{thebibliography}

\end{document}